\documentclass[onefignum,onetabnum]{siamart220329}

\usepackage{appendix}
\usepackage{amsfonts}
\usepackage{graphicx,epstopdf}
\usepackage{bm}
\usepackage{enumitem}
\usepackage{algorithm}
\usepackage{algpseudocode}
\usepackage{framed,multirow}
\usepackage{booktabs}
\usepackage{makecell}
\usepackage{subcaption}
\usepackage{xcolor}

\usepackage{tikz}

\usetikzlibrary{patterns} 
\usepackage{wrapfig}
\usetikzlibrary{calc}  
\newtheorem{example}{Example}[section] 
\newtheorem{assumption}{Assumption}

\newtheorem{thm}{Theorem}[section] 
 
\newtheorem{remark}[thm]{Remark}

\definecolor{darkgreen}{rgb}{0.0, 0.5, 0.0}

\newcommand{\dx}{\Delta x}
\newcommand{\dy}{\Delta y}

\newcommand{\dt}{\Delta t}

\newcommand{\bu}{{\bf u}}
\newcommand{\bv}{{\bf v}}

\newcommand{\bmm}{{\bm  m}}
\newcommand{\bn}{{\bf n}}
\newcommand{\bc}{{\bf c}} 
\newcommand{\rd}{\mbox{d}}
\newcommand{\bx}{{\bm x}}
\newcommand{\bff}{{\bf f}}
\headers{IDP High Order Schemes}{K. Wu, X. Zhang, and C.-W. Shu}

\title{High Order Numerical Methods Preserving Invariant Domain for Hyperbolic and Related Systems 
\thanks{Submitted on May 1, 2025; accepted for publication (in revised form) December 8, 2025.
\funding{
Kailiang Wu is partially supported by  Science Challenge Project (No.~TZ2025007), Shenzhen Science and Technology Program 
(Grant No.~RCJC20221008092757098), and National Natural Science Foundation of China 
(Grant No.~12171227). 
Xiangxiong Zhang is partially supported by NSF grant DMS-2208518 and Purdue University Seed Funding for review papers. 
Chi-Wang Shu is partially supported by NSF grant DMS-2309249.} }
}

\author{ Kailiang Wu \thanks{Department of Mathematics, Southern University of Science and Technology, Shenzhen, Guangdong
518055, China (\email{WUKL@sustech.edu.cn}).}
\and Xiangxiong Zhang\thanks{Department of Mathematics, Purdue University, West Lafayette, IN 47907, USA
  (\email{zhan1966@purdue.edu}).}
  \and Chi-Wang Shu \thanks{Division of Applied Mathematics,
    Brown University
    Providence, RI 02906, USA (\email{Chi-Wang\_Shu@brown.edu}).}
  }
 
\begin{document}

\maketitle
\begin{abstract} 
Admissible states in hyperbolic systems and related equations often form a convex invariant domain. Numerical violations of this domain can lead to  loss of hyperbolicity, resulting in ill-posedness and severe numerical instabilities.   
It is therefore crucial for numerical schemes to preserve the invariant domain to ensure both physically meaningful solutions and robust computations. For complex systems, constructing invariant-domain-preserving (IDP) schemes is highly nontrivial and particularly challenging for high-order accurate methods.  
This paper presents a comprehensive survey of IDP schemes for hyperbolic and related systems, with a focus on the most popular approaches for constructing provable IDP schemes.  
We first give a systematic review of the fundamental approaches for establishing the IDP property in first-order accurate schemes, covering finite difference, finite volume, finite element, and residual distribution methods. 
Then we focus on  
two widely used and actively developed classes of high order IDP schemes 
as well as their recent developments, most of which have emerged in the past decade.
The first class of methods seeks an intrinsic weak IDP property in high-order schemes and then designs polynomial limiters to enforce a strong IDP property 
 at the points of interest. 
This generic approach applies to high-order finite volume and discontinuous Galerkin schemes. 
The second class is based on the flux limiting approaches, which originated from the flux-corrected transport method and can be adapted to a broader range of spatial discretizations, including finite difference and continuous finite element methods. 
In this survey, we 
elucidate the main ideas and underlying principles that guide the construction of IDP schemes, unify several existing IDP analysis techniques and theories, and provide some new perspectives and insights on the existing approaches.  
We also illustrate these approaches through extensive examples, such as positivity-preserving schemes for the gas dynamics equations, and present numerical experiments drawn from interdisciplinary applications in gas dynamics and magnetohydrodynamics. 
\end{abstract}




\section{Introduction}
\subsection{Motivation}
Consider the initial value problem of a time-dependent PDE system of $N$ equations in $d$ spatial dimensions,
\begin{equation}\label{eq:gPDE}
	\partial_t {\bf u} +  \mathcal{L} ({\bf u}) = {\bf 0}, ~~~
	\qquad {\bf u}( {\bm x}, 0) = {\bf u}_0 ( {\bm x} ),\quad {\bm x} \in  \mathbb R^d, \quad {\bf u}={\bf u}({\bm x},t)\in \mathbb R^N,
\end{equation} 
defined in a bounded domain with appropriate boundary conditions,  where   
$\mathcal{L}$ denotes 
the spatial differential operator. 
A set  $G \subset \mathbb R^N$ is called an {\it invariant domain} of \eqref{eq:gPDE}
if   well posed solutions to \eqref{eq:gPDE} satisfy ${\bf u} ( {\bm x}, t ) \in G$ for all $\bm x$ and $t >0$ 
as long as ${\bf u} ( {\bm x}, 0 ) \in G$ for all $\bm x$.
In many problems, $G$ can be represented by the positivity or non-negativity of 
several functions of ${\bf u}$: 
\begin{equation}\label{eq:ASS-G}
	G = \left \{ {\bf u}  \in \mathbb R^N:~g_i ({\bf u}) >0~~\forall i \in \mathbb I, \quad g_i({\bf u})\ge 0~~\forall i \in \widehat {\mathbb I} \, \right \}. 
\end{equation}
For a given PDE, its invariant domain is not necessarily convex, but the invariant domain used in stabilizing many interesting and important hyperbolic systems is often a convex set.   
In this  paper, we only consider convex invariant domains.
Such convex invariant domains are typically found in hyperbolic conservation laws 
\begin{equation}\label{eq:hPDE}
\partial_t {\bf u} +   \nabla \cdot {\bf f} ( {\bf u} ) = {\bf 0}, 
\end{equation}
as well as in other time-dependent PDEs, such as 
convection-diffusion equations \cite{chueh1977positively}, semilinear parabolic equations \cite{DuJuLiQiao2021}, and  reaction-diffusion equations \cite{EstepMAMS2000,smoller2012shock}, etc.

For example, for gas dynamics equations consisting of conservation of mass, momentum and total energy, 
the set of admissible states is defined by   positive density and internal energy (or pressure for many common equations
of state).
 Such a set of admissible states is needed for not only physically meaningful solutions but also maintaining the hyperbolicity of the governing equations, 
which will be reviewed in  \Cref{sec:gasdynamics}.
 More importantly, numerically preserving positivity of density and pressure is critical for stabilizing computations of challenging problems such as high speed flows, especially for high order numerical schemes. Preserving an invariant domain defined by positivity of density and pressure in a conservative scheme can ensure $L^1$ stability of mass and total energy \cite{perthame1992second}.
  
For a generic hyperbolic problem, the set of admissible states such as positive water height in shallow water equations,  usually defines a convex invariant domain, preserving which  renders numerical schemes more robust. 
 There are many such examples  in applications including but not limited to
 weather modeling \cite{Nair2011-book}, radiative transfer \cite{buet2006asymptotic,olbrant2012realizability,buet2012asymptotic},
 kinetic equations \cite{cheng2013study,alldredge2015realizability},
neutrino transport in core-collapse supernovae \cite{mezzacappa2020physical},  relativistic
hydrodynamics \cite{radice2014high,WU2015539, qin2016bound,wu2017admissible,wu2017design, wu2021provably}, hydraulic engineering \cite{xing2010positivity}, astrophysics \cite{KIDDER201784,klingenberg2017numerical}, 
chemically reactive flows \cite{lv2017high-chem}, etc.

\subsection{Scalar conservation laws and order barriers}
\label{sec:barrier}
As a simplified model problem, 
consider a scalar conservation law $\partial_t u +   \nabla \cdot {\bf f} ( u) = 0$, whose entropy solution $u({\bm x},t)$ is total-variation-diminishing (TVD)   and satisfies the maximum principle $\min_{ {\bm x} } u({\bm x},t) \le u({\bm x},s)\le \max_{ {\bm x} } u({\bm x},t)$ for any $s>t\geq 0$, implying a bound-preserving property $U_{\min}:=\min_{ {\bm x} \in \Omega} u_0( {\bm x} ) \le u({\bm x},t) \le \max_{ {\bm x} } u_0( {\bm x} ) =: U_{\max}$. Hence, the scalar conservation law admits a convex invariant domain $G = [U_{\min},U_{\max}]$.
 
For scalar conservation laws, there is extensive literature on enforcing stronger properties, such as the TVD property and monotonicity, which go beyond merely preserving   the invariant domain $G$. 
Consider the one-dimensional scalar equation $\partial_t u + \partial_x f(u) = 0$. A numerical scheme of the form 
$u^{n+1}_j=H(u^{n}_{j-k}, u^{n}_{j-k+1}, \cdots, u^{n}_{j+k})$ 
is called \emph{monotone} if the function $H$ is monotone increasing (non-decreasing) with respect to each of its arguments. 
This definition naturally extends to multi-dimensional problems and to convection–diffusion equations: the numerical solution at the next time level must be a monotone increasing function of each of
its arguments within the stencil at the current time level.

It is well known that such a monotone scheme is at most first order accurate \cite{godunov1959finite,harten1976finite}.
There is a common misconception that a monotone scheme never produces new extrema.
A monotone scheme is also   TVD thus   {\it monotonicity-preserving}:
\[\mbox{monotone} ~ \Rightarrow ~ \mbox{TVD} ~ \Rightarrow ~ \mbox{monotonicity-preserving}, \]
none of which, however, implies that no new extrema can be generated. 
The definition of {\it monotonicity-preserving} is that if $u^n_{j+1}\geq u^n_j$ for all $j$, then $u^{n+1}_{j+1}\geq u^{n+1}_j$ for all $j$. Thus   a monotone scheme cannot generate new extrema on an infinite domain if the initial condition is a monotone profile. 
 For instance,
the Lax-Friedrichs scheme (a.k.a. Rusanov scheme)  is monotone, and
one can easily construct a non-monotone initial condition for which the Lax-Friedrichs scheme  produces new extrema \cite{breuss2004correct}.

The order barrier that a monotone scheme can only be first order accurate is also called {\it Godunov Theorem}.
There are similar order barriers stated in the literature. For example, a TVD finite difference scheme satisfying $\sum_j|u^{n+1}_{j+1}-u^{n+1}_{j}|\leq \sum_j|u^{n}_{j+1}-u^{n}_{j}|$ can be at most first order accurate in two dimensions \cite{goodman1985accuracy}. Similarly, if seeking maximum principle in the form $\min_j u^n_j \leq u^{n+1}_j\leq \max_j u^n_j$ in a finite difference or finite volume scheme, then such a scheme can be most second order accurate in spatial truncation error analysis, see \cite{zhang2011maximum} for a simple counterexample accredited to Harten. Central schemes  \cite{jiang1998nonoscillatory,kurganov2000new} satisfy such maximum principle and achieve second order accuracy.

 To enforce the maximum principle or the TVD property in high order finite volume and finite difference schemes, 
various limiters can be designed, which however causes loss of high order accuracy near extrema due to the  above-mentioned order barriers.
Nonetheless, such schemes can still achieve high-order accuracy for smooth solutions that are monotonically increasing or decreasing, while delivering good resolution of discontinuities.

 \subsection{First order schemes for systems}
Consider a one-dimensional hyperbolic system $\bu_t+\bff(\bu)_x={\bf 0}$ as an example, then a 3-point-stencil first order {\it locally conservative} scheme can be written as 
\begin{equation}
 \bu^{n+1}_j=\bu^{n}_j-\lambda\left(\hat{\bff}(\bu^n_{j}, \bu^n_{j+1})-\hat{\bff}(\bu^n_{j-1}, \bu^n_{j})\right),\quad \lambda=\frac{\Delta t}{\Delta x},
 \label{scheme:1Dsystem}
\end{equation} 
where $\bu^n_j$ denotes the solution at a  grid point $x_j$ and $n$-th time step,   $\hat{\bff}(\cdot, \cdot)$ denotes the numerical flux function, and $\Delta t$ and $\Delta x$ are the temporal and spatial step-sizes. 
 Stability properties of scalar equations such as TVD and maximum principle do not directly extensible to systems of conservation laws. For instance, in a blast wave solution to gas dynamics equations,   both the total variation and upper bound of density  become much larger than their initial values (see \Cref{fig:sedov}). 
Instead,   first order monotone schemes like the Godunov scheme and the Lax-Friedrichs scheme (a.k.a.~Rusanov scheme) for scalar conservation laws can  be shown invariant-domain-preserving (IDP) for hyperbolic systems, which will be reviewed in Section \ref{sec:firstorderIDP}.

For a given hyperbolic system, it is usually difficult to show that the exact solution preserves a given invariant domain   from PDE analysis. If  numerical solutions to a locally conservative scheme converge in the sense of bounded variation when refining the meshes, then the limit must be a weak solution by the   Lax-Wendroff Theorem \cite{lax1960systems}. 
Schemes like the Godunov scheme also satisfy the entropy inequality, thus the converged limit is  also an entropy solution. Therefore, numerical solutions of the first order Godunov scheme can be used to as a numerical evidence for whether the hyperbolic system should preserve a particular chosen convex domain \cite{hoff1979finite, hoff1985invariant}. If $\bu^n_j$ produced by a first order IDP scheme \eqref{scheme:1Dsystem} converges as mesh refines, then it is a strong numerical evidence that at least one exact solution to the PDE should preserve the same invariant domain.

\subsection{High order schemes for systems}

Although almost all desired stability properties can be proven for a first-order scheme \eqref{scheme:1Dsystem} using either the Godunov flux or the Lax–Friedrichs flux, their numerical resolution of fine structures is often unsatisfactory. 
High-resolution and high-order schemes are preferred for achieving better resolution. 
High-order schemes are typically defined as those that are at least third-order accurate for smooth solutions. 
Even though any Eulerian scheme on a uniform mesh can be at most half-order accurate in the $L^2$-norm for discontinuous solutions \cite[\S 11.2]{leveque1992numerical}, high-order schemes such as 
WENO (weighted essentially non-oscillatory) schemes \cite{shu2009high} or discontinuous Galerkin (DG) methods \cite{cockburn1990runge} generally produce better resolution than low-order schemes for discontinuities such as shocks, due to the reduced artificial viscosity present in high-order schemes.

However, popular and practical high order accurate finite difference, finite volume, and finite element schemes 
are not robust for very challenging hyperbolic problems (e.g., high speed flows involving low density and pressure), often due to violation of the invariant domains (e.g., loss of positivity of density or pressure in compressible flows). 
With proper modifications or limiters, high order schemes can be rendered to preserve the invariant domain of admissible states, which may improve their robustness.

We emphasize that the order barriers such as the Godunov Theorem hold only in the sense as stated in Section \ref{sec:barrier}. Under different definitions of discrete total variation or maximum, it is possible to avoid these  order barriers.
For example,  if defining the total variation as the total variation of the piecewise polynomials reconstructed  in a finite volume scheme, a very high order accurate finite volume TVD scheme  can be constructed for scalar equations in one dimension \cite{zhang2010genuinely}. 
If seeking a scheme preserving a simple invariant domain $U_{\min}\leq u^n_j \leq U_{\max}$ instead of a strict maximum principle $\min_j u^n_j \leq u^{n+1}_j\leq \max_j u^n_j$  for scalar equations, then in general it is still possible to achieve high order accuracy for a smooth solution. 
For a properly defined invariant domain \eqref{eq:ASS-G},  there should be no order barriers  for systems.  

For the sake of  stabilizing high order Eulerian schemes without adding excessive artificial viscosity, we should consider a method that is high order accurate (at least for smooth solutions), conservative, and IDP. For a special scheme solving a special system, there might be many different ways to construct such a method. For general problems, there are two popular and flexible approaches which are  described briefly as follows.

Take a high order accurate finite volume scheme on an interval $I_j=[x_{j-\frac12},x_{j+\frac12}]$ as an example. With forward Euler time stepping, it can be written as 
 \begin{equation}
 \bar \bu^{n+1}_j=\bar \bu^{n}_j-\lambda\left(\hat{\bff}(\bu^-_{j+\frac12}, \bu^+_{j+\frac12})-\hat{\bff}(\bu^-_{j-\frac12}, \bu^+_{j-\frac12})\right),\quad \lambda=\frac{\Delta t}{\Delta x},
 \label{scheme:1Dsystem-2}
\end{equation} 
where $\bar\bu_j$ denotes the cell average, $\hat\bff$ is a numerical flux function and $\bu^-_{j+\frac12}, \bu^+_{j+\frac12}$ are reconstructed values at $x_{j+\frac12}$ from the left and from the right, respectively.  
Let $\hat{\bff}_{j+\frac12}=\hat{\bff}(\bu^-_{j+\frac12}, \bu^+_{j+\frac12})$ denote the numerical flux.

The first popular and easy-to-use approach is to consider a modification or limiting of the   numerical fluxes $\hat{\bff}_{j\pm\frac12}$ so that  $ \bar \bu^{n+1}_j$ is in the invariant domain, and the idea of flux limiting traces back to  the seminal work of the flux corrected transport (FCT) by Boris and Book \cite{boris1973flux, book1975flux, boris1976flux} and also Zalesak \cite{zalesak1979fully}.

The second popular approach is to modify only the reconstructed values $\bu^\pm_{j\pm\frac12}$ so that $ \bar \bu^{n+1}_j$ is in the invariant domain under a suitable time step, and such an approach has been popularized by Zhang and Shu \cite{zhang2010positivity}, which 
 builds upon the idea by Perthame and Shu in \cite{perthame1996positivity} and the simple polynomial limiter analyzed in \cite{liu1996nonoscillatory}. 
 By the Godunov Theorem, a high order scheme like \eqref{scheme:1Dsystem-2} cannot be monotone for solving scalar equations, but it can still be {\it weakly monotone} (\Cref{thm:1} in \Cref{sec:zhang-shu}),  which is the key of such an approach.

\subsection{Scope and organization of this paper}

The preservation of invariant domain is merely a partial characterization of nonlinear stability, which is not sufficient for convergence. For convergence to entropy solutions, discrete entropy inequality should also be considered. Moreover, more properties might also be desired such as energy stability, and well-balancedness for shallow water equations.  
It is possible to combine discussions of other properties with the IDP property. For simplicity,
we focus on only how to preserve a convex invariant domain, and we do not discuss boundary conditions. We only discuss the numerical scheme in the interior of the domain, e.g., assuming periodic boundary conditions on a rectangular domain or zero inflow boundary conditions.
 For the organization of the rest of this paper,
we first list some representative examples of   invariant domains in Section \ref{sec:domain}. In Section \ref{sec:firstorderIDP}, we discuss how to show IDP  in classical first order schemes, on which IDP techniques in high order schemes depend heavily. 
In Section \ref{sec:zhang-shu} and Section \ref{sec:FCT}, we review two popular approaches for enforcing invariant domain in high order schemes, which can be used for many hyperbolic systems such as gas dynamics and shallow water equations. In Section \ref{sec:other}, we briefly discuss
other approaches and extensions, then survey recent breakthroughs and developments for the much more challenging MHD systems.
 The concluding remarks will be given in   Section \ref{sec:remarks}.

\section{Examples of invariant domains}
\label{sec:domain}

For a system $\partial_t\bu+\partial_x\bff(\bu)={\mathbf 0}$, it is hyperbolic if the Jacobian matrix $\bff'(\bu)$ has 
real eigenvalues and a complete set of eigenvectors, from which usually an  
invariant domain like \eqref{eq:ASS-G} can be defined. 
For nonlinear hyperbolic systems in multiple dimensions, it is in general difficult to prove that the exact solutions satisfy $\bu (\bx, t)\in G$, but violation of such an invariant domain usually causes blow-ups in numerical computation   due to loss of hyperbolicity. 
Next, we list some examples of \eqref{eq:ASS-G}.

\subsection{Gas dynamics equations}
\label{sec:gasdynamics}
\begin{example}[Positivity density and pressure for ideal gas]
	Consider the compressible Euler equations, which can be written in the form of \eqref{eq:hPDE}:
\begin{subequations}
    \label{eq:Euler}
\begin{equation}
		\partial_t	
		\begin{pmatrix}
			\rho  
			\\
			{\bm m}
			\\
			E	
		\end{pmatrix}
		+ 	 \nabla \cdot
		\begin{pmatrix}
			{\bm m}
			\\
			\rho^{-1}{\bm m} \otimes {\bm m} 
			+ p {\bf I}  
			\\
			\rho^{-1}( E + p ) {\bm m}
		\end{pmatrix} = {\bf 0},\quad \bx\in \mathbb R^d, \quad d=1,2,3
	\end{equation}	
where   $\rho$ is the density, $\bm m$ denotes the momentum vector, $p$ is the pressure, and 
$E= \rho e + \frac{|{\bm m}|^2}{2\rho}$ is the total energy with $e$ being the specific internal energy.  
The system \eqref{eq:Euler} is closed with an equation of state (EOS), e.g., the  EOS for ideal gas is   
\begin{equation}
\label{EOS}
    p=(\gamma-1)\rho e,
\end{equation}
where $\gamma>1$ is the ratio of specific heats. 
\end{subequations}
The commonly considered invariant domain for  \eqref{eq:Euler} is defined by positive density and positive pressure
\begin{equation}\label{eq:G-Euler}
	G = \left\{ 
	{\bf u}=(\rho, {\bm m},  E)^\top \in \mathbb R^{d+2}:~ \rho>0,~ 
	p=(\gamma-1)(E-\frac{|{\bm m}|^2}{2\rho})> 0 
	\right\}.
\end{equation}
The pressure function $p({\bu})$ is concave in $\bu$, which can be verified via its Hessian matrix. Thus $p$ satisfies the Jensen's inequality 
\begin{equation}\label{eq:Jensen4g}
p(\lambda \bu_1+(1-\lambda)\bu_2)\geq \lambda  p(\bu_1)+(1-\lambda)p(\bu_2) \quad \forall \lambda \in (0,1),~\forall \bu_1,\bu_2\in G, 
\end{equation}
which implies the convexity of $G$.
\end{example}

Negative density or pressure leads not only to physically meaningless solutions but also to loss of hyperbolicity, 
since the local sound speed $\sqrt{\frac{\gamma p}{\rho}}$ becomes imaginary and the eigenvalues of the Jacobian matrix of system \eqref{eq:Euler} become complex. 
In practice, negative density or pressure almost always cause significant numerical instabilities when solving \eqref{eq:Euler}, for example in simulations of high-speed flows (see \Cref{fig:astro-Mach2000-Euler}).
Therefore, the set \eqref{eq:G-Euler} is also known as the {\it set of admissible states}.


\begin{example}[Invariant domain with minimum entropy principle]
\label{ex:entropy}
For compressible Euler equations with ideal gas EOS \eqref{eq:Euler}, one may also consider adding 
the minimum entropy principle \cite{tadmor1986minimum, khobalatte1992maximum},   $S({\bf u}) \ge S_{min}:=\min_{\bm x} S ( {\bf u}_0({\bm x}) )$, for the specific entropy $S=\ln \frac{p}{\rho^{\gamma}}$, 
which is not a concave function but instead a quasi-concave function \cite[Lemma 2.1]{zhang2012minimum} thus satisfies
\begin{equation}
    \label{Jensen-quasiconcave}
     S(\lambda \bu_1+(1-\lambda)\bu_2)\geq \min\{  S(\bu_1), S(\bu_2)\}\quad \forall \lambda \in (0,1), \quad \forall \bu_1,\bu_2\in G.
\end{equation}
One can obtain another convex invariant domain:
\begin{equation}
	G_S = \left\{ 
	{\bf u}=(\rho, {\bm m},  E)^\top \in \mathbb R^{d+2}:~ \rho>0,~ 
	p> 0,~S=\ln \frac{p}{\rho^{\gamma}}\geq S_{\min} 
	\right\}.
  \label{G-Euler-2}
\end{equation}

\end{example}

\begin{example}[Invariant domain for any EOS]
The pressure function may no longer be a concave function in a general EOS, for which the internal energy
$\rho e=E-\frac{|\bm m |^2}{2\rho}$ is always concave for $\rho>0$. The invariant domain defined by positivity of internal energy can be considered for  both compressible Euler and Navier--Stokes equations \cite{grapsas2016unconditionally, zhang2017positivity, guermond2021second-NS}:
\begin{equation}
G = \left\{ 
	{\bf u}=(\rho, {\bm m},  E)^\top \in \mathbb R^{d+2}:~ \rho>0,~ 
	\rho e=E-\frac{|{\bm m}|^2}{2\rho}> 0 
	\right\}.
  \label{G-Euler-3}
\end{equation}
\end{example}

\subsection{Single layer shallow water equations}
The   single layer
 shallow water equations with a  bottom topography, which has been  used to model
flows in rivers and coastal areas for ocean and hydraulic engineering,
 can be written as 
 \begin{equation}
     \partial_t	
		\begin{pmatrix}
			h  
			\\
			{\bm m} 
		\end{pmatrix}
		+ 	 \nabla \cdot
		\begin{pmatrix}
			{\bm m}
			\\
			h^{-1}{\bm m} \otimes {\bm m} 
			+ \frac12 g h {\bf I}  
		\end{pmatrix} = \begin{pmatrix}
			0
			\\
			-gh \nabla b  
		\end{pmatrix}
        ,\quad \bx\in \mathbb R^2, 
	\end{equation}	 
where $h$ is water height, $\bm m=h (u,v)^\top$ is the momentum, $b(\bx)$ is the bottom topography function,
and $g$ is the gravity constant. The eigenvalues of the Jacobian  are related to $\sqrt{gh}$, thus the non-negativity of the water height function defines a convex invariant domain, which is crucial for numerical stability \cite{berthon2008positive, castro2006numerical, kurganov2007second}.

\subsection{Two-layer shallow water equations}

Consider the more complicated   
two-layer shallow water equations, which are widely used in the study of stratified flow motions such as salinity-driven exchange flow motions and layered flows. In one dimension, the equations for conservative variables can be written as
\begin{equation}
\label{Eq:2layer}
\partial_t \begin{pmatrix}
h_1 \\ h_1 u_1 \\ h_2 \\ h_2 u_2 
\end{pmatrix}+  \partial_x \begin{pmatrix}
h_1 u_1 \\ h_1 u_1^2+\frac12 gh_1^2 \\ h_2 u_2 \\ h_2 u_2^2+\frac12 gh_2^2 
\end{pmatrix}=\begin{pmatrix}
0 \\ -gh_1(h_2+b)_x \\ 0 \\ -rgh_2(h_1)_x-gh_2b_x
\end{pmatrix},
\end{equation}
where the subscripts 1 and 2 denote the upper and the lower layers of the water respectively, $h_i$ is the height of the $i$th layer, $u_i$ denotes the velocity of the $i$th layer,  $b$ is the bottom topography,  and $r=\rho_{1}/\rho_{2} \in (0,1)$ is the ratio of density. 

By using the expansion-based first order approximation  \cite{mandli2013numerical}, the eigenvalues of the Jacobian can be given as the external wavespeeds
and internal wavespeeds
\begin{equation}\label{Eq:ext-vel}
\lambda_{ext}^{\pm}\approx\frac{h_{1}u_{1}+h_{2}u_{2}}{h_{1}+h_{2}}\pm\sqrt{g(h_{1}+h_{2})},
\end{equation} 
\begin{equation}\label{Eq:int-vel}
\lambda_{int}^{\pm}\approx\frac{h_{1}u_{2}+h_{2}u_{1}}{h_{1}+h_{2}}\pm\sqrt{g'\frac{h_{1}h_{2}}{h_{1}+h_{2}}\left[1-\frac{(u_{1}-u_{2})^{2}}{g'(h_{1}+h_{2})}\right]},\quad g'=g(1-r)>0.
\end{equation} 

One can first consider an {\color{black}approximated invariant} domain defined to ensure the approximated eigenvalues to be real numbers:
\begin{equation}
\label{G-twolayer-1}
\color{black}
 G=\left\{(h_1, h_1 u_1, h_2, h_2 u_2)^\top:  h_1> 0,~ h_2> 0,~  \frac{(u_1-u_2)^2}{h_1+h_2}\leq  g'
  \right\},
\end{equation}  which unfortunately is not a convex set because $\frac{(u_1-u_2)^2}{h_1+h_2}$ is not a convex function of the conserved variables $(h_1,  h_1 u_1, h_2, h_2  u_2)^\top$. 
However, the function $\frac{(u_1-u_2)^2}{h_1+h_2}$ 
is a convex function of the primitive variables $(
           h_1, u_1,  h_2,  u_2)^\top$, which can be easily verified via its Hessian matrix. 
Thus we may consider rewriting the two-layer equations in these variables:          
\begin{equation}\label{Eq:2layer-2}
\begin{pmatrix}
h_1 \\ u_1 \\ h_2 \\   u_2 
\end{pmatrix}_t+\begin{pmatrix}
h_1 u_1 \\ \frac12 u_1^2+ g(h_1+h_2+b) \\ h_2 u_2 \\  \frac12 u_2^2+ g(r h_1+h_2+b)
\end{pmatrix}_x=\begin{pmatrix}
0 \\ 0 \\ 0 \\ 0
\end{pmatrix}, 
\end{equation}
and consider the following invariant domain for the same constraints:
\begin{equation}
\label{G-twolayer-2}
\color{black}
 G=\left\{(h_1, u_1, h_2,   u_2)^\top: h_1 > 0,~ h_2 > 0,~  \frac{(u_1-u_2)^2}{h_1+h_2}\leq  g'
  \right\}.
\end{equation}
Although   \eqref{G-twolayer-1} and \eqref{G-twolayer-2} describe the same admissible states, the key difference is that only the set \eqref{G-twolayer-2} is convex, which facilitates the construction of IDP schemes for the equivalent system \eqref{Eq:2layer-2}; see \cite{du2024high-twolayer}. However, 
{\color{black}the systems \eqref{Eq:2layer} and \eqref{Eq:2layer-2} are equivalent only  for smooth solutions and they may produce different weak solutions.} Moreover, 
there are some drawbacks of solving a system for non-conservative variables \eqref{Eq:2layer-2}, e.g., it is nontrivial to obtain conservation of momentum.

\subsection{Ten-moment Gaussian closure model}

Gaussian closure models are alternatives of compressible Euler equations
for modeling compressible flows with nonlocal thermodynamic equilibrium assumed, especially for extremely low pressure rarefied gas flows. 
The ten-moment Gaussian closure model
\cite{levermore1998gaussian,berthon2006numerical}
in two dimensions are given by $\partial_t {\bf u} +  \partial_x {\bf f}_1 ( {\bf u} ) + \partial_y {\bf f}_2  ( {\bf u} ) = {\bf 0}$ with
\begin{align}\label{eq:Rephrased-Ten-Moment}
& {\bf u} = 	
\begin{pmatrix}
	\rho  
	\\
	m_1 
	\\
	m_2 
	\\
	E_{11}
	\\
	E_{12}
	\\
	E_{22}
\end{pmatrix}, \qquad 
{\bf f}_j ( {\bf u} ) = 	
\begin{pmatrix}
	m_j 
	\\
	m_1 v_j + p_{1j}
	\\
	m_2 v_j + p_{2j}
	\\
	E_{11} v_j + p_{1j} v_1
	\\
	E_{12} v_j + \frac{1}{2} (p_{1j} v_2 + p_{2j} v_1)
	\\
	E_{22} v_j + p_{2j} v_2
\end{pmatrix}, \quad j=1,2,
\end{align}
where $\rho$ denotes the mass density, ${\bm m} = (m_1, m_2)^\top$ is the momentum, and the velocity vector is given by ${\bm v} = {\bm m}/\rho$. The symmetric tensor ${\bf E} = (E_{ij})$ represents the energy, and ${\bf p} = (p_{ij})$ is the pressure tensor, which is also symmetric and anisotropic in nature.

The system \eqref{eq:Rephrased-Ten-Moment} is closed by 
the relation ${\bf p} = 2 {\bf E} - \rho {\bm v} \otimes {\bm v}$. 
This ensures that the model remains consistent with physical constraints. The  set of admissible states is characterized by the positivity of both the density and the pressure tensor,
\begin{align}\label{eq:Rephrased-G}
G &= \left\{ {\bf u} \in \mathbb{R}^6 : \rho > 0,\; {\bf E} - \frac{{\bm m} \otimes {\bm m}}{2\rho} \text{ is positive-definite} \right\},
\end{align}
which is a convex set \cite{meena2017positivity}.

\subsection{Ideal MHD equations}
The system for ideal compressible magnetohydrodynamics (MHD) 
   can be reformulated in conservative form as follows:
\begin{equation}\label{eq:rephrased-idealMHD}
\partial_t	
\begin{pmatrix}
	\rho  
	\\
	{\bm m}
	\\
	{\bf B}
	\\
	E	
\end{pmatrix}
+ \nabla \cdot
\begin{pmatrix}
	{\bm m}
	\\
	{\bm m} \otimes {\bm v} - {\bf B} \otimes {\bf B}
	+ p_{\text{tot}} {\bf I}
	\\
	{\bm v} \otimes {\bf B} - {\bf B} \otimes {\bm v}
	\\
	\left( E +  p + \frac{1}{2} | {\bf B} |^2 \right) {\bm v} 
	- ( {\bm v} \cdot {\bf B} ) {\bf B}	
\end{pmatrix}
= {\bf 0},
\end{equation}
where   $\rho$ denotes the mass density, ${\bm m}$ is the momentum, and the velocity field is computed via ${\bm v} = {\bm m}/\rho$. The magnetic field is represented by ${\bf B}$, satisfying the solenoidal constraint $\nabla \cdot {\bf B} = 0$ if it is satisfied
at $t=0$. 
The total energy $E$ includes contributions from internal, kinetic, and magnetic energy,
$E = \rho e + \frac{1}{2} \left( \rho | {\bm v} |^2 + | {\bf B} |^2 \right),$
where $e$ is the specific internal energy.  
Physically meaningful states must satisfy positivity of both the mass density and internal energy. This leads to a convex invariant domain:  
\begin{equation}\label{eq:rephrased-G-iMHD}
G = \left\{
{\bf u} = (\rho, {\bm m}, {\bf B}, E)^\top \in \mathbb{R}^{2d+2} :\; \rho > 0,\;
g({\bf u}) := E - \frac{|{\bm m}|^2}{2\rho} - \frac{|{\bf B}|^2}{2} > 0
\right\}.
\end{equation}
The positivity is crucial for stable computations \cite{cheng2013positivity-MHD, derigs2016novel-mhd, wu2018positivity}.

\subsection{Relativistic MHD equations}

The relativistic magnetohydrodynamic (RMHD) system \cite{wu2017admissible,wu2021provably} can be expressed in conservative form as follows:
\begin{equation}\label{eq:Rephrased-RMHD}
\partial_t 
\begin{pmatrix}
	D \\
	{\bm m} \\
	{\bf B} \\
	E
\end{pmatrix}
+ \nabla \cdot
\begin{pmatrix}
	D {\bm v} \\
	{\bm m} \otimes {\bm v} 
	- {\bf B} \otimes \left( W^{-2} {\bf B} + ( {\bm v} \cdot {\bf B} ) {\bm v} \right)
	+ p_{\text{tot}} {\bf I} \\
	{\bm v} \otimes {\bf B} - {\bf B} \otimes {\bm v} \\
	{\bm m}
\end{pmatrix}
= {\bf 0}.
\end{equation}
In this formulation, the relativistic mass density $D = \rho W$, where $\rho$ is the rest-mass density and $W = (1 - | {\bm v} |^2)^{-1/2}$ is the Lorentz factor. The momentum vector ${\bm m}$ is defined as
\[
{\bm m} = (\rho h W^2 + | {\bf B} |^2) {\bm v} - ( {\bm v} \cdot {\bf B} ) {\bf B},
\]
where $h$ is the specific enthalpy and ${\bm v}$ is the velocity. The velocity is normalized such that the speed of light is unity. The total energy is given by
\[
E = \rho h W^2 - p_{\text{tot}} + | {\bf B} |^2.
\]
The magnetic field ${\bf B}$ obeys the divergence-free condition $\nabla \cdot {\bf B} = 0$ if it is satisfied at $t=0$, 
as in ideal MHD. The total pressure $p_{\text{tot}} = p + p_{m}$ is composed of the thermal component $p$ and the magnetic contribution
\[
p_{m} = \frac{1}{2} \left( W^{-2} | {\bf B} |^2 + ( {\bm v} \cdot {\bf B} )^2 \right).
\] 

To ensure physical admissibility, the solution must lie within the invariant domain
\begin{equation}\label{eq:Rephrased-RMHD-G}
G = \left\{ {\bf u} = (D, {\bm m}, {\bf B}, E)^\top \in \mathbb{R}^{2d+2} : D > 0,\; p({\bf u}) > 0,\; | {\bm v}({\bf u}) | < 1 \right\},
\end{equation}
where both ${\bm v}({\bf u})$ and $p({\bf u})$ are nonlinear functions of the conserved variables and cannot be explicitly written in closed form.  
These nonlinear implicit functions are frequently expressed in terms of another auxiliary function $\hat \phi({\bf u})$, which is determined implicitly. Specifically, the expressions take the form:
\begin{equation}\label{eq:Rephrased-getprimformU}
p({\bf u}) = \frac{\gamma - 1}{Z_{\bf u}^2(\hat \phi)  \,  \gamma} \left( \hat \phi - D \, Z_{\bf u}(\hat \phi) \right), \quad
{\bm v}({\bf u}) = \frac{{\bm m} + ( {\bm m} \cdot {\bf B} ) {\bf B}/\hat \phi}{\hat \phi + | {\bf B} |^2}.
\end{equation}
Here, $\hat \phi = \hat \phi({\bf u})$ is defined as the unique positive root of a nonlinear equation:
\[
 \phi - E + | {\bf B} |^2 
- \frac{1}{2} \left( \frac{( {\bm m} \cdot {\bf B} )^2}{\phi^2} 
+ \frac{| {\bf B} |^2}{Z_{\bf u}^2(\phi)} \right)
+ \frac{\gamma - 1}{\gamma} \left( \frac{D}{Z_{\bf u}(\phi)} 
- \frac{\phi}{Z_{\bf u}^2(\phi)} \right) = 0,
\]
where $\gamma$ is the adiabatic index (ratio of specific heats), and   $Z_{\bf u}(\phi)$ is 
\[
Z_{\bf u}(\phi) := 
\left( 
\frac{ 
\phi^2(\phi + | {\bf B} |^2)^2 
- \left[ \phi^2 | {\bm m} |^2 
+ (2\phi + | {\bf B} |^2)( {\bm m} \cdot {\bf B} )^2 \right]
}
{\phi^2 (\phi + | {\bf B} |^2)^2}
\right)^{-1/2}.
\]



\section{First order schemes preserving invariant domains}
\label{sec:firstorderIDP}
We only consider schemes with forward Euler time stepping in this section. 

\subsection{Monotone schemes for 1D scalar conservation laws}
\label{sec:mono}
We first consider the simplest example of enforcing bounds in solving scalar conservation laws. 
Consider the one-dimensional (1D) version of the scalar conservation law 
\begin{equation}\label{eq:1stSCL}
	\partial_t u + \partial_x f(u) = 0. 
\end{equation}

A 3-point-stencil first order monotone scheme for \eqref{eq:1stSCL} can be written as 
\begin{equation}\label{eq:1stSCLmono}
	u_j^{n+1} = u_j^{n} - \lambda \left( \hat f( u_j^{n},u_{j+1}^{n} ) - \hat f( u_{j-1}^{n},u_{j}^{n} )  \right) =: H_{\lambda} ( u_{j-1}^{n},u_{j}^{n},u_{j+1}^{n} ),
\end{equation}
where $u_j^{n}$ denotes the numerical solution at the $j$-th grid point in finite difference or $j$-th cell in finite volume at the time level $n$,  $\lambda = \Delta t/\Delta x$ with $\Delta t$ and $\Delta x$ denoting the temporal and spatial mesh sizes\footnote{Uniform mesh size is assumed here for  simplicity in presentation, while all our discussions are extensible to non-uniform meshes.}, and 
$\hat f(u^-,u^+)$ is a monotone flux, i.e., it is non-decreasing in the first argument $u^-$ and non-increasing in its second argument $u^+$, and satisfies the consistency $\hat f(u,u) = f(u)$.   

Under a suitable Courant--Friedrichs--Lewy (CFL) condition, typically of the form:
\begin{equation}\label{eq:CFL1}
	 \lambda \alpha \le 1 \quad \mbox{with} \quad \alpha:= \max_{u} |f'(u)|,
\end{equation}
the function $H_{\lambda}(\cdot,\cdot,\cdot)$ is monotonically non-decreasing in all its three arguments. Note that the  consistency of $\hat f$ implies that 
$
H_{\lambda} (u,u,u) = u. 
$ 
Therefore, if 
$$
U_{\min} \le u_{j-1}^{n},~ u_{j}^{n},~ u_{j+1}^{n} \le U_{\max},
$$
then the monotonicity and consistency imply 
\begin{align*}
	u_j^{n+1} &= H_{\lambda} ( u_{j-1}^{n},u_{j}^{n},u_{j+1}^{n} ) \ge H_{\lambda} ( U_{\min}, U_{\min}, U_{\min} )= U_{\min},
	\\
	u_j^{n+1} &= H_{\lambda} ( u_{j-1}^{n},u_{j}^{n},u_{j+1}^{n} ) \le H_{\lambda} ( U_{\max}, U_{\max}, U_{\max} ) =  U_{\max}.
\end{align*}
Hence, the monotone scheme \eqref{eq:1stSCLmono} preserves the invariant domain $G = [U_{\min},U_{\max}]$ under the CFL condition \eqref{eq:CFL1}. 
Examples of monotone numerical fluxes include:
\begin{itemize} 
	\item the Lax--Friedrichs, a.k.a.~Rusanov flux 
\begin{equation}\label{eq:LF}
		\hat{f}^{\text{LF}}\left(u^-, u^+\right) = \frac{1}{2} \Big( f(u^-) + f(u^+) - \alpha (u^+ - u^-) \Big);
\end{equation}
	\item the Godunov flux 
	$$
	\hat{f}^G\left(u^-, u^+\right) = \begin{cases} 
		\min\limits_{u^- \leq u \leq u^+} f(u) \quad & \text{if } u^- \leq u^+; \\
		\max\limits_{u^+ \leq u \leq u^-} f(u) \quad & \text{if } u^- > u^+;
	\end{cases}
	$$
	\item the Engquist--Osher flux
$$
\hat{f}^{\text{EO}}\left(u^-, u^+\right) = \frac{1}{2} \left( f(u^-) + f(u^+) - \int_{u^-}^{u^+} |f'(u)| \, \mathrm{d}u \right).
$$
\end{itemize}
For instance,  the scheme \eqref{eq:1stSCLmono} with the Lax--Friedrichs flux \eqref{eq:LF}, can be reformulated as 
$$
H_\lambda(u_{j-1}^n, u_j^n, u_{j+1}^n) = (1 - \alpha \lambda) u_j^n + \frac{1}{2} \lambda \left( \alpha u_{j+1}^n - f(u_{j+1}^n) \right) + \frac{1}{2} \lambda \left( \alpha u_{j-1}^n + f(u_{j-1}^n) \right),
$$
which is  non-decreasing with respect to $u_{j-1}^n$, $u_j^n$, $u_{j+1}^n$, thus preserves the bounds under \eqref{eq:CFL1}. 
Although the monotone scheme is at most first order accurate (by the Godunov Theorem), it satisfies much stronger stability properties such as $L^1$ contraction and it converges to the unique entropy solution for scalar conservation laws in multiple dimensions \cite{crandall1980monotone}.

\begin{remark}
Although some implicit schemes can be shown to be monotone for scalar linear problems \cite{abgrall2003construction}, the extensions to nonlinear equations can be difficult due to the algebraic nonlinear systems involved. 
\end{remark}

\subsection{Basic assumptions for systems}

Many techniques of the first order IDP  methods have been well established since 1980s, e.g., \cite{hoff1979finite, hoff1985invariant,tadmor1986minimum, einfeldt1991godunov, perthame1996positivity, frid2001maps}.  
 For enforcing a convex invariant domain defined as \eqref{eq:ASS-G}, we need to make some basic assumptions about the flux function in \eqref{eq:hPDE}, which are used in almost all classical IDP methods.
We consider a nonlinear system \eqref{eq:hPDE} and   first state   basic assumptions for the 1D version:
\begin{equation}\label{eq:1stHsystem}
	\partial_t {\bf u} + \partial_x {\bf f}({\bf u}) = {\bf 0}.  
\end{equation}

Let ${\bf U}^{\rm RP}(x, t; {\bf u}_L, {\bf u}_R )$
be the exact solution of a Riemann problem to  \eqref{eq:1stHsystem} with the initial data 
\begin{equation}\label{eq:RPdata}
			{\bf u}_0 ( x ) = 
	\begin{cases}
		{\bf u}_L, \quad & \text{if } x \leq 0, \\
		{\bf u}_R, \quad & \text{if } x > 0. 
	\end{cases}
\end{equation}
The exact solution of a Riemann problem is self-similar, i.e.,  
${\bf U}^{\rm RP}(\alpha x, \alpha t; {\bf u}_L, {\bf u}_R )$ = ${\bf U}^{\rm RP}(x, t; {\bf u}_L, {\bf u}_R )$ for any
constant $\alpha >0$.  Therefore, we only need to consider ${\bf U}^{\rm RP}(x, 1; {\bf u}_L, {\bf u}_R )$.

\begin{assumption}
    \label{prop:RP}
	The exact solution of the Riemann problem  
 preserves the invariant domain:
if ${\bf u}_L,{\bf u}_R\in G$, then ${\bf U}^{\rm RP}\left(x, t; {\bf u}_L, {\bf u}_R \right) \in G$ for any $x \in \mathbb R$ and $t>0.$ 
And there
 exists a maximum wave speed
  $a_{\max}( {\bf u}_L, {\bf u}_R )>0$ such that
\begin{equation}\label{eq:finiteWaveSpeed}
	\begin{aligned}
		{\bf U}^{\rm RP}\left(x, t; {\bf u}_L, {\bf u}_R \right) &= {\bf u}_L \quad 
		  \quad \forall x/t \le - a_{\max},  \\
		{\bf U}^{\rm RP}\left(x, t; {\bf u}_L, {\bf u}_R \right) &= {\bf u}_R \quad 
		\quad \forall  x/t \ge a_{\max}. 
	\end{aligned}
\end{equation}
\end{assumption}

 Such an assumption can be verified for  most systems of hyperbolic conservation laws and well studied equations such as
 scalar conservation laws, shallow water equations \cite{ketcheson2020riemann}, and 
 compressible Euler equations   \cite{toro2013riemann}. \Cref{prop:RP} does not hold in certain cases, such as the multidimensional ideal MHD and relativistic MHD systems with a jump in the normal component of magnetic field \cite{janhunen2000positive} or when the divergence-free constraint of magnetic field is violated \cite{wu2018provably,wu2021provably}.

For any $\alpha \ge a_{\max}$, 
integrating \eqref{eq:1stHsystem} with the initial data \eqref{eq:RPdata} over $[-\alpha, \alpha] \times [0,1]$ in space-time domain with the Divergence Theorem yields  

\begin{center} 
\begin{tikzpicture}
    \fill[pattern=north east lines, pattern color=yellow] (-3,0) rectangle (3,3);
    \draw[dotted] (-4,0) -- (4,0) node[right] {$t=0$};
    \draw[dotted] (-4,3) -- (4,3) node[right] {$t=1$};

    \draw[dotted] (-3,0) -- (-3,3); 
    \draw[dotted] (3,0) -- (3,3); 
    \node[text=red] at (-1.5,.2) {{ $\bu_L$}};
 
    \node[text=red] at (1.5, .2) {$\bu_R$};

    
    \draw[thick,blue] (0,0) -- (-0.8,3);
    \draw[thick,blue] (0,0) -- (-0.1,3);
    \draw[thick,blue] (0,0) -- (0.5,3);
    \draw[thick,blue] (0,0) -- (0.8,3);
    \draw[thick,blue] (0,0) -- (1.3,3);

    \draw[thick,red] (-3,0) -- (3,0);

\node[draw,circle,fill=black,inner sep=1pt] at (-3,0) {};
\node[draw,circle,fill=black,inner sep=1pt] at (0,0) {};
\node[draw,circle,fill=black,inner sep=1pt] at (3,0) {}; 
    \node[below] at (-3,0) {$x=-\alpha$};
    \node[below] at (0,0) {$x=0$};
    \node[below] at (3,0) {$x=\alpha$};

    \fill[cyan, opacity=0.3] (-3,3.1) -- (3,3.1) -- (3,2.9) -- (-3,2.9) -- cycle;
    \node[text=cyan] at (1.5, 3.5) {${\bf U}^{\rm RP}\left(x, 1; {\bf u}_L, {\bf u}_R \right)$};

\end{tikzpicture}
\end{center}
\begin{footnotesize}
\begin{align*}
	\int_{-\alpha}^{ \alpha }  {\bf U}^{\rm RP}\left(x,1 ; {\bf u}_L, {\bf u}_R \right) \,\rd x &= 
	\int_{-\alpha}^{ \alpha }  	{\bf u}_0 ( x ) \,\rd x -  [{\bf f}( {\bf U}^{\rm RP}\left( \alpha, 1 ; {\bf u}_L, {\bf u}_R \right) ) + {\bf f}( {\bf U}^{\rm RP}\left( -\alpha, 1  ; {\bf u}_L, {\bf u}_R \right) ]
	\\ & = \alpha ( {\bf u}_L + {\bf u}_R  ) - {\bf f} ( {\bf u}_R ) + {\bf f} ( {\bf u}_L ),
\end{align*}
\end{footnotesize}
where we have used \eqref{eq:finiteWaveSpeed} in the second step. 
Dividing this identity by $2 \alpha$ implies 
$$
\frac{ {\bf u}_L + {\bf u}_R  }2 + \frac{ {\bf f} ( {\bf u}_L ) - {\bf f} ( {\bf u}_R )  }{2 \alpha } = \frac{1}{2 \alpha } \int_{-\alpha}^{ \alpha }  {\bf U}^{\rm RP}\left(x, 1; {\bf u}_L, {\bf u}_R \right) \,\rd x.
$$
Thus \Cref{prop:RP} implies:
\begin{equation}
   {\bf u}_L,{\bf u}_R\in G\Rightarrow \frac{ {\bf u}_L + {\bf u}_R  }2 + \frac{ {\bf f} ( {\bf u}_L ) - {\bf f} ( {\bf u}_R )  }{2 \alpha } \in G\quad \forall \alpha \ge a_{\max} ( {\bf u}_L, {\bf u}_R ).
   \label{assumption2}
\end{equation}

We state the second assumption without using any exact solutions:
\begin{assumption}\label{prop:LFS}
	There exists a  suitable function $\hat a({\bf u})>0$ such that 
\begin{equation}\label{eq:LFS}
		{\bf u}\in G \quad \Longrightarrow \quad  {\bf u} \pm \frac{ {\bf f} ( {\bf u} )}{\alpha} \in G \quad \forall \alpha \ge \hat a({\bf u}).
\end{equation}
\end{assumption}

 In \Cref{prop:LFS}, a suitable $\hat a$ should satisfy $\hat a({\bf u}) \le \eta 
\hat a_{\max}({\bf u})$ for some constant $\eta > 0$, where $\hat a_{\max}({\bf u})$ denotes the maximum wave speed at the state $\bu$. 
If this requirement is removed, then for any open set $G$, one could always construct such a function by selecting sufficiently large values; however, this would lead to arbitrarily small time steps in an IDP scheme and thus be practically useless.

\begin{remark}
    For many problems, it often suffices to take $\hat a({\bf u})$ as the maximum wave speed to satisfy
    \Cref{prop:LFS}. 
    For example, for the 1D compressible Euler equations with the ideal gas EOS \eqref{EOS}, the maximum wave speed is given by 
    the spectral radius of the Jacobian matrix $\bff'(\bu)$, which is simply $|u| + \sqrt{ \gamma\frac{ p}{\rho}}$. 
    For enforcing the invariant domain in \eqref{eq:G-Euler},
    by \cite[Lemma 6]{zhang2017positivity}, to satisfy \Cref{prop:LFS}, 
    one can take $\hat a({\bf u}) = |u| + \sqrt{\frac{\gamma-1}{2}\frac{p}{\rho}}<|u| + \sqrt{\gamma \frac{ p}{\rho}}$. 
    However, for other stability considerations such as entropy stability, the use of the maximum wave speed can be necessary. 
\end{remark}

 \Cref{prop:LFS} is not a universal property
and it may not hold for all convex invariant domains. For instance, it does not hold  for many scalar conservation laws 
with the bound-preserving property or for the compressible MHD systems.  When the entropy principle is considered for gas dynamics, e.g., the invariant domain \eqref{G-Euler-2}, the property  \Cref{prop:LFS} may not hold in general. 
In the relativistic hydrodynamic case, if the entropy principle is included in the invariant domain, then \Cref{prop:LFS} no longer holds, as observed in \cite{wu2021minimum}.

When the invariant domain involves highly nonlinear (or even implicit) constraints, e.g., \eqref{eq:Rephrased-RMHD-G}, it is often very challenging to verify the above assumptions. In \cite{wu2023geometric}, Wu and Shu proposed a  general approach, termed 
Geometric Quasi Linearization (GQL), which transforms the nonlinear constraints in \eqref{eq:ASS-G} into {\em equivalent linear} constraints: 
\begin{equation}\label{eq:1200}
    {G}^\star
    :=
    \left\{\,
    {\bf u} \in \mathbb{R}^N
    \, : \, ({\bf u}-{\bf u}^\star) \cdot \mathbf{n}^\star_i \succ 0 \; ~ 
    \forall {\bf u}^\star \in \mathcal{S}_i,
    \forall i \in \mathbb{I} \cup \hat{\mathbb{I}}
    \,\right\},
\end{equation}
where 
$\mathcal{S}_i := \partial  {G} \cap \partial  {G}_i$ with $\partial {G}_i := \left\{ {\bf u} \in \mathbb{R}^N : g_i({\bf u}) = 0\right\}$, $\mathbf{n}^\star_i := \nabla g_i({\bf u}^\star)$ is an inward normal vector of $\partial G$ at ${\bf u}^\star$, and the symbol $\succ$ represents $>$ for $i \in \mathbb{I}$ and $\ge $ for $i \in \hat{\mathbb I}$. 
Here, ${\bf u}^\star$ is independent of $\bf u$ and is referred to as a \emph{free auxiliary variable} in the GQL framework \cite{wu2023geometric}. These variables are introduced to lift the dimension for linearity.

Under the GQL framework \cite{wu2023geometric}, we introduce a more general and weaker assumption than 
 \Cref{prop:LFS}, which is extensible to more equations such as MHD systems. 

\begin{assumption}\label{prop:wLFS}
	There exists a suitable function $\hat a({\bf u})>0$ and ${\zeta}({\bf u^\star})$  such that 
\begin{equation}\label{eq:wLFS}
		{\bf u}\in G \quad \Longrightarrow \quad  
            \alpha({\bf u}-{\bf u}^\star)\cdot \mathbf{n}_i^\star 
    \pm  
    {\bf f}({\bf u})\cdot \mathbf{n}_i^\star
    \succ \pm \zeta({\bf u^\star})~~ \forall {\bf u}^\star \in \mathcal{S}_i, ~~\forall \alpha \ge \hat a({\bf u}).
\end{equation}
\end{assumption}

Based on the equivalence $G = G^\star$, one can show that \Cref{prop:wLFS} also implies \eqref{assumption2}, with  
$a_{\max}({\bf u}_L, {\bf u}_R) = \max\{ \hat a({\bf u}_L), \hat a({\bf u}_R) \}.$ 
Moreover, \Cref{prop:LFS} can be viewed as a special case of \Cref{prop:wLFS} by taking ${\zeta}({\bf u}^\star) = 0$.  
With or without the entropy principle included in the invariant domain, \Cref{prop:wLFS} holds for the relativistic hydrodynamic \cite{wu2021minimum}.  
For MHD systems, where \Cref{prop:RP} and \Cref{prop:LFS} are both inapplicable, a variant of \Cref{prop:wLFS} can still be established by carefully constructing ${\zeta}({\bf u}^\star)$ so that its related terms vanishes under the (discrete) divergence-free constraint \cite{wu2017admissible,wu2018positivity,wu2019provably,wu2021provably}.

\begin{remark}
For special systems and specific invariant domains, it is possible to construct IDP schemes without employing these assumptions, e.g., kinetic schemes for compressible Euler equations 
\cite{perthame1992second,perthame1994variant,estivalezes1996high,tao1999gas}.  
\end{remark}

In the rest of this section, for simplicity, we only focus on  how to use \Cref{prop:RP} or \Cref{prop:LFS} for proving that first order schemes are IDP for systems like  gas dynamics equations. 
For harder problems like compressible MHD and relativistic hydrodynamics equations,
\Cref{prop:wLFS} applies and will be reviewed in
  \Cref{sec:MHD} and \Cref{sec:RHD}.

\subsection{First order schemes for 1D hyperbolic systems}

The   monotonicity technique in Section \ref{sec:mono} is useful for maximum principles, but it does not apply to  
 system of hyperbolic conservation laws.
Next we demonstrate   three   techniques based on the above assumptions, for  provable IDP  schemes  in the form \eqref{scheme:1Dsystem} for solving 1D hyperbolic system \eqref{eq:1stHsystem}.
\begin{definition}[IDP numerical flux]\label{def:IDPflux}
Let $\lambda=\frac{\Delta t}{\Delta x}$.
	A numerical flux $\hat {\bf f} (\cdot,\cdot)$ is said to be IDP, if 
	the 
	corresponding 1D three-point
	first order scheme \eqref{scheme:1Dsystem} is IDP, 
	$$
	\bu^{n+1}_j :=   \bu^{n}_j - \lambda \left( \hat {\bf f}( \bu^{n}_j, \bu^{n}_{j+1} ) - \hat {\bf f}( \bu^{n}_{j-1}, \bu^{n}_{j} )  \right) \in G \quad \forall \,  \bu^{n}_{j-1}, \bu^{n}_{j}, \bu^{n}_{j+1} \in G, 
	$$
	under a suitable CFL condition $\lambda \alpha   \le c_0$, where $\alpha$ denote the estimated maximum wavespeed and $c_0$ is the IDP CFL number. 
\end{definition}

For compressible Euler equations with ideal gas EOS \eqref{eq:Euler}, the Lax--Friedrichs flux (a.k.a. Rusanov flux) can be shown IDP for $G$ in \eqref{eq:G-Euler} with 
  $c_0=1$  \cite{tang2000positivity} and  the Godunov  and HLLE fluxes are IDP with $c_0=\frac12$; see \cite{einfeldt1991godunov}. See  
    \cite{perthame1992second, estivalezes1996high, tao1999gas} for the CFL of kinetic schemes to be IDP.
Next, we demonstrate how the three assumptions can be used to show IDP in first order schemes by a few examples.

\subsubsection{Method 1 using Assumption 1} This technique is a traditional approach and has been widely used; e.g., \cite{einfeldt1991godunov, perthame1996positivity}. 
Based on \Cref{prop:RP}, we can investigate the IDP property of several  numerical fluxes defined as exact or approximate Riemann solvers, such as the Godunov scheme, Lax--Friedrichs scheme, and HLL type schemes.

\begin{example}[Godunov scheme]  
Let the   self similar solution to the Riemann problem be denoted by
 ${\bf U}^{\rm RP}(\xi; {\bf u}_L, {\bf u}_R )$ with   $\xi:={x}/{t}$.
Then the Godunov flux is  
	$$\hat {\bf f}( {\bf u}^-, {\bf u}^+ ) = {\bf f} ( {\bf U}^{\rm RP}(0; {\bf u}^-, {\bf u}^+ ) ).$$
Consider the scheme \eqref{scheme:1Dsystem}  using the Godunov flux  with $\bu^n_j$ denoting the cell average on the interval  $I_j=[x_{j-\frac12}, x_{j+\frac12}]$.  Let 
$a_{\max}=\max_j  a_{\max} ( \bu_j^n, \bu_{j+1}^n ) $ be the maximum wavespeed. 

\begin{tikzpicture}
    
\fill[pattern=north east lines, pattern color=yellow] (0,0) rectangle (3,3);
    \draw[dotted] (-4,0) -- (7,0) node[right] {$t_n$};
    \draw[dotted] (-4,3) -- (7,3) node[right] {$t^{n+1}$};

    \draw[dotted] (-3,0) -- (-3,3);
    \draw[dotted] (0,0) -- (0,3);
    \draw[dotted] (3,0) -- (3,3);
    \draw[dotted] (6,0) -- (6,3);
    \node[text=red] at (-1.5,.2) {{ $\bu_{j-1}^n$}};
 
    \node[text=red] at (1.5, .2) {$\bu_j^n$};
    \node[text=red] at (4.5,.2) {$\bu_{j+1}^n$};

    \draw[thick,blue] (-3,0) -- (-2.5,3);
    \draw[thick,blue] (-3,0) -- (-2.8,3);
    \draw[thick,blue] (-3,0) -- (-3.0,3);
    \draw[thick,blue] (-3,0) -- (-3.2,3);
    
    \draw[thick,blue] (0,0) -- (-0.8,3);
    \draw[thick,blue] (0,0) -- (-0.4,3);
    \draw[thick,blue] (0,0) -- (0.5,3);
    \draw[thick,blue] (0,0) -- (0.8,3);

    \draw[thick,blue] (3,0) -- (2.5,3);
    \draw[thick,blue] (3,0) -- (2.7,3);
    \draw[thick,blue] (3,0) -- (3.2,3);

    \draw[thick,blue] (6,0) -- (5.0,3);
    \draw[thick,blue] (6,0) -- (6.1,3);
    \draw[thick,blue] (6,0) -- (6.2,3);

    \draw[thick,red] (-3,0) -- (6,0);

\node[draw,circle,fill=black,inner sep=1pt] at (-3,0) {};
\node[draw,circle,fill=black,inner sep=1pt] at (0,0) {};
\node[draw,circle,fill=black,inner sep=1pt] at (3,0) {};
\node[draw,circle,fill=black,inner sep=1pt] at (6,0) {};
    \node[below] at (-3,0) {$x_{j-\frac32}$};
    \node[below] at (0,0) {$x_{j-\frac12}$};
    \node[below] at (3,0) {$x_{j+\frac12}$};
    \node[below] at (6,0) {$x_{j+\frac32}$};

    \fill[cyan, opacity=0.3] (0,3.1) -- (3,3.1) -- (3,2.9) -- (0,2.9) -- cycle;
    \node[text=cyan] at (1.5, 3.5) {$\bu_j^{n+1}$};

\end{tikzpicture}

By integrating \eqref{eq:1stHsystem} over a space-time rectangle $I_j\times [t^n, t^{n+1}]$ with the Divergence Theorem as shown in the figure above, we obtain 
\[\resizebox{0.99\textwidth}{!}{$ \frac{1}{\Delta x}\int_{I_j}\bu(x, t^{n+1}) \, \rd  x=\frac{1}{\Delta x}\int_{I_j}\bu(x, t^{n}) \, \rd x -\frac{\Delta t}{\Delta x}\int_{t^n}^{t^{n+1}}[\bff(\bu(x_{j+\frac12},t))-\bff(\bu(x_{j-\frac12},t))]\, \rd t $},\]
in which 
 $\bu(\bx, t^n)$ is piecewise constant $\bu_j^n$, 
 the interface value $\bu(x_{j+\frac12},t)$ is equal to the self similar exact  solution ${\bf U}^{\rm RP}(0; \bu^n_j, \bu^n_{j+1})$ 
 for any $t\in [t^n, t^{n+1}]$, if the local Riemann problems do not intersect, ensured by the CFL $\frac{\Delta t}{\Delta x} a_{\max}\leq\frac12$ as shown in the figure. 
Thus in the Godunov scheme 
\[ \bu_j^{n+1}=\bu_j^{n}-\lambda \hat [\bff (\bu^n_j, \bu^n_{j+1})-\bff  (\bu^n_{j-1}, \bu^n_{j})],\]
the numerical solution $\bu_j^{n+1}$ is exactly the average of the exact solution to two non-intersecting Riemann problems under the CFL $\frac{\Delta t}{\Delta x} a_{\max}\leq\frac12$. The integral operator $\frac{1}{\Delta}\int_{I_j}\cdot\,\rd x$ preserves a convex invariant domain, thus \Cref{prop:RP} implies $\bu_j^{n+1}\in G$.
\end{example}

\begin{example}[Global Lax--Friedrichs flux]
	For the Lax--Friedrichs flux \eqref{eq:LF} with a global estimate of the wavespeed $\alpha$ at time $t^n$, the scheme 
	\eqref{scheme:1Dsystem} can be rewritten as a convex combination:
\begin{equation}\label{eq:LF1D}
		\mathbf{u}_j^{n+1} = (1 - \lambda \alpha ) \mathbf{u}_j^n + \lambda \alpha  \left( \frac{ {\bf u}_{j-1}^n + {\bf u}_{j+1}^n  }2 + \frac{ {\bf f} ( {\bf u}_{j-1}^n ) - {\bf f} ( {\bf u}_{j+1}^n )  }{2 \alpha } 
	\right) 
\end{equation}
	under the CFL condition 
	$\lambda \alpha \le 1$. By \eqref{assumption2} and the convexity of $G$, we have $\mathbf{u}_j^{n+1} \in G$, if  $\alpha$ in the Lax--Friedrichs flux \eqref{eq:LF} satisfies 
	$\alpha \ge \max_{j} a_{\max} ( {\bf u}_{j-1}^n, {\bf u}_{j}^n). $
\end{example}

\begin{example}[Local Lax--Friedrichs flux]
\label{example:localLFscheme-Method2}
Consider the  scheme \eqref{scheme:1Dsystem} with a local Lax--Friedrichs flux defined by
\begin{equation}
    \label{1D-local-LF-flux}
    \hat{\bff}(\bu^n_{j-1}, \bu^n_{j})=\frac{1}{2}[\bff(\bu^n_{j-1})+\bff(\bu^n_{j})-\alpha_{j-\frac12}^n (\bu^n_{j}-\bu^n_{j-1})],\quad \alpha^n_{j-\frac12}=a_{\max} ( {\bf u}_{j}, {\bf u}_{j-1} ).
\end{equation} 
    Then the first order local Lax--Friedrichs scheme  is
    \[ \resizebox{0.99\textwidth}{!}{$ \bu^{n+1}_j=\bu^{n}_j-
\lambda\left[ \frac{\bff(\bu^n_{j})+\bff(\bu^n_{j+1})-\alpha_{j+\frac12}^n (\bu^n_{j+1}-\bu^n_{j})}{2}-\frac{\bff(\bu^n_{j-1})+\bff(\bu^n_{j})-\alpha_{j-\frac12}^n (\bu^n_{j}-\bu^n_{j-1})}{2}\right],$}\]
which is equivalent to
\begin{align}
 \notag  \bu^{n+1}_j=&\left(1- \lambda \alpha_{j+\frac12}^n-\lambda\alpha_{j-\frac12}^n \right)\bu^{n}_j+
\lambda \alpha_{j+\frac12}^n \left(\frac{\bu^n_{j+1}+\bu^n_{j}}{2}+\frac{\bff(\bu^n_{j})-\bff(\bu^n_{j+1})}{2\alpha_{j+\frac12}^n}\right)    \notag  \\
&
+\lambda\alpha_{j-\frac12}^n\left(\frac{ \bu^n_{j}+\bu^n_{j-1}}{2}+\frac{\bff(\bu^n_{j-1})-\bff(\bu^n_{j})}{2\alpha_{j-\frac12}^n}\right). 
\label{1D-local-LF-flux-method2}
\end{align}
For compressible Euler equations with ideal gas EOS \eqref{eq:Euler} and the invariant domain  $G_S$ including the minimum entropy principle  \eqref{G-Euler-2}, 
by \eqref{assumption2}, we have
\[ \bu^n_{j}, \bu^n_{j-1}\in G_S \quad \Longrightarrow \quad  \frac{\bu^n_{j}+\bu^n_{j-1}}{2}+\frac{\bff(\bu^n_{j-1})-\bff(\bu^n_{j})}{2\alpha_{j-\frac12}^n} \in G_S.\]
 Therefore, by the convexity of $G_S$, we get $\bu^{n+1}_j\in G_S$ because the scheme \eqref{1D-local-LF-flux-method2} is a convex combination under the CFL condition
 \begin{equation}
\label{localLF-CFL-1D-1} \lambda \max_j\alpha^n_{j+\frac12}  \le \frac12,  
 \end{equation}
 which is the same CFL derived as in \cite[Appendix]{perthame1996positivity} by a different yet essentially equivalent approach. 
\end{example}

We remark that   an estimate of a rigorous yet explicit upper bound $a_{\max}$ is needed for the maximum wave speed in the Riemann problem in Method 1.  A naive estimate based on the largest eigenvalues is not always adequate when fast shocks arise in the Riemann problem; see, e.g., \cite{guermond2016fast} for a discussion in the context of the Euler equations. {\color{black}On the other hand, the eigenvalue-based estimate can be overly restrictive in some cases, for example, when simulating the Sod shock tube problem.}

\subsubsection{Method 2 using Assumption 2} 
This approach is more flexible, which will be frequently used in this paper. The basic idea is to decompose a  scheme into a convex combination of several simpler functions of the form $\bu$ and $\bu\pm \bff(\bu)/\alpha$, which preserve the invariant domain $G$, then the convexity of $G$ implies that the target scheme is IDP.

\begin{example}[Local Lax--Friedrichs flux]
\label{example:localLFscheme-Method3}
For solving \eqref{eq:Euler},    
     the first order the local Lax-Friedrichs scheme  can also be rewritten as
\begin{equation}\label{eq:LF1Deuler}
\resizebox{0.99\textwidth}{!}{$	\mathbf{u}_j^{n+1} = \frac{2-\lambda \alpha^n_{j-\frac12}-\lambda \alpha^n_{j+\frac12}}{2} \mathbf{u}_j^n + 
	\frac{\lambda \alpha^n_{j-\frac12}}2 \left( {\bf u}_{j-1}^n + \frac{ {\bf f} ( {\bf u}_{j-1}^n )  }{ \alpha^n_{j-\frac12} }  \right) 
	+ \frac{\lambda \alpha^n_{j+\frac12}}2 \left( {\bf u}_{j+1}^n - \frac{ {\bf f} ( {\bf u}_{j+1}^n )  }{ \alpha^n_{j+\frac12} }  \right). $}
\end{equation}
For the invariant domain $G$ in \eqref{eq:G-Euler}, by taking 
$$\alpha^n_{j-\frac12}=\max_{i=j, j-1}|\bff'(\bu_i^n)|=\max_{i=j, j-1}|v_i| + \sqrt{ \frac{\gamma p_i}{\rho_i}}\quad \forall j,$$
 if ${\bf u}_{j}^n\in G$, then ${\bf u}_{j}^n \pm \frac{ {\bf f} ( {\bf u}_{j}^n )}{\alpha^n_{j\pm\frac12}} \in G$ for all $j$, see \cite[Remark 2.4]{zhang2010positivity}. Therefore, by the convexity of $G$, the scheme \eqref{eq:LF1Deuler} preserves the invariant domain $G$ under the CFL condition $ \lambda \max_j\alpha^n_{j+\frac12}  \le 1$, which allows a larger time step than  \eqref{localLF-CFL-1D-1} in Method 1. 
\end{example}

\Cref{example:localLFscheme-Method2} and \Cref{example:localLFscheme-Method3} are two different approaches on the same scheme. By comparing these two examples, we can see that each method has its own advantages. Although Method 2 allows a larger time step for provable  positivity-preserving property for density and pressure for the local Lax-Friedrichs scheme, it cannot be used for enforcing the minimum entropy principle since \Cref{prop:LFS} may not hold for $G_S$ in \eqref{G-Euler-2}.
On the other hand, for gas dynamics equations such as compressible Navier--Stokes equations with a generic EOS and $G$ in \eqref{G-Euler-3}, Method 2 is very flexible to use.

Note that \Cref{prop:LFS} is commonly used to construct IDP schemes with Lax--Friedrichs type fluxes. However, for IDP schemes employing other numerical fluxes, it may not directly apply and must be appropriately adapted.

\subsection{Basic assumptions in multiple dimensions}


For a multi-dimensional system  \eqref{eq:hPDE} and any given unit vector $\bn\in \mathbb R^d$, let ${\bf U}^{\rm RP}(x, t;\bn, {\bf u}_L, {\bf u}_R )$
be the exact solution of a Riemann problem to one dimensional equation $\bu_t+\partial_x[\bff(\bu)\cdot \bn]=0$ with the initial data 
\begin{equation}\label{eq:RPdata-2}
			{\bf u}_0 ( x ) = 
	\begin{cases}
		{\bf u}_L, \quad & \text{if } x \leq 0, \\
		{\bf u}_R, \quad & \text{if } x > 0. 
	\end{cases}
\end{equation}

\begin{assumption}[\Cref{prop:RP} in multiple dimensions] 
\label{prop:LF-2D}
The exact solution of the Riemann problem preserves the invariant domain, namely, ${\bf u}_L,{\bf u}_R\in G\Rightarrow {\bf U}^{\rm RP}(x,t; \bn, {\bf u}_L, {\bf u}_R ) \in G$ for any $x \in \mathbb R$ and $t>0$. 
There exists a maximum wave speed
  $a_{\max}(\bn, {\bf u}_L, {\bf u}_R )>0$ such that
\begin{equation*} 
	\begin{aligned}
		{\bf U}^{\rm RP}\left(x, t; \bn, {\bf u}_L, {\bf u}_R \right) &= {\bf u}_L \quad 
		  \quad \forall x/t \le - a_{\max},  \\
		{\bf U}^{\rm RP}\left(x, t; \bn, {\bf u}_L, {\bf u}_R \right) &= {\bf u}_R \quad 
		\quad \forall  x/t \ge a_{\max}. 
	\end{aligned}
\end{equation*}
  Similar to the derivation 
 of \eqref{assumption2}, this assumption yields
\[  {\bf u}_L,{\bf u}_R\in G \quad \Longrightarrow \quad  \frac{ {\bf u}_L + {\bf u}_R  }2 + \frac{ {\bf f} ( {\bf u}_L )\cdot\bn - {\bf f} ( {\bf u}_R )\cdot\bn  }{2 \alpha } \in G\quad \forall \alpha \ge a_{\max} (\bn,  {\bf u}_L, {\bf u}_R ).
\]
\end{assumption}
 
\begin{assumption}[\Cref{prop:LFS} in multiple dimensions] 
\label{prop:LFS-2D}
	There exists a suitable function $\hat a({\bf u},\bn)>0$ such that 
${\bf u}\in G \quad \Longrightarrow \quad  {\bf u} \pm \frac{ {\bf f} ( {\bf u} )\cdot\bn}{\alpha} \in G \quad \forall \alpha \ge \hat a({\bf u},\bn).$
\end{assumption}

 \Cref{prop:LFS-2D} does hold for many interesting systems and commonly considered convex invariant domains such as  the Euler and ten-moment Gaussian 
closure systems without considering minimum entropy principle.  
For instance,  the following result implies that this assumption holds for compressible Euler and also Navier--Stokes equations, with the invariant domain defined in \eqref{G-Euler-3}. 
\begin{lemma}[Lemma 6 in \cite{zhang2017positivity}]
\label{lf-fact-NS}
Consider any $\bu=(\rho,  \bmm, E)^\top$, and
\[\bff^a( \bu)=\begin{pmatrix}
\bmm\\\rho^{-1} \bmm\otimes\bmm+p\mathbf{I}\\
\rho^{-1}(E+p) \bmm
\end{pmatrix}, \bff^d( \bu)=\begin{pmatrix}
0\\ \boldsymbol{\tau}\\
\rho^{-1}\bmm\cdot{\boldsymbol{\tau}}-\mathbf{q}
\end{pmatrix},\]
where $p$, $\boldsymbol{\tau}$, and $\mathbf{q}$ are not necessarily dependent on $ \bu$.
Let $e=\rho^{-1} E-\frac12 \rho^{-2}|\bmm|^2$.
For any unit vector $\mathbf{n}$, let $v=\rho^{-1}\bmm \cdot\mathbf{n}$, $q=\mathbf{q}\cdot\mathbf{n}$
and $\tau=\mathbf{n}\cdot\boldsymbol{\tau}$.
Then we have the following for $G=\{ \bu: \rho>0, e=\rho^{-1} E-\frac12 \rho^{-2}|\bmm|^2 \geq 0\}$, 
\begin{itemize}
 \item [(a)] $\bu\pm\alpha^{-1}\bff^a( \bu)\cdot\mathbf{n}\in G$ if and only if $\alpha> |v|+\sqrt{\frac{p^2}{2\rho^2 e}}$,
\item [(b)] $ \bu\pm \beta^{-1}(\bff^a( \bu)-\bff^d( \bu))\cdot\mathbf{n}\in G$ 
if and only if 
$$\beta>|v|+\frac{1}{2\rho^2 e}\left(\sqrt{\rho^2q^2+2\rho^2 e | \tau-p\mathbf{n}|^2 }+\rho|q|\right).$$
\end{itemize}
\end{lemma}

\begin{remark}
For the compressible Navier--Stokes equations, we can write it as if it were a formal convection system $\partial_t \bu+ \nabla\cdot\bff(\bu)={\bf 0}$ where $\bff=\bff^a-\bff^b$ and $\bff^a, \bff^b$ are given in Lemma \ref{lf-fact-NS}. Then $\beta$ in Lemma \ref{lf-fact-NS} (b) gives one way to satisfy \Cref{prop:LFS} for such a formal system. However, here $\beta$ is not any approximation to  wave speeds, but instead merely a quantity designed to satisfy \Cref{prop:LFS}. After all, the concept of wavespeed is not well defined for a convection diffusion system like the Navier--Stokes equations. 
\end{remark}

\begin{assumption}[\Cref{prop:wLFS} in multiple dimensions] \label{prop:wLFS-2D}
	There exists a suitable function $\hat a({\bf u},\bn)>0$ and a vector function ${\bm \zeta}({\bf u^\star})$  such that, for any  $\alpha \ge \hat a({\bf u}, {\bf n})$,  
\begin{equation}\label{eq:wLFS-2D}
		{\bf u}\in G \quad \Longrightarrow \quad  
            \alpha({\bf u}-{\bf u}^\star)\cdot \mathbf{n}_i^\star 
    \pm 
    ({\bf f}({\bf u})\cdot{\bf n})\cdot \mathbf{n}_i^\star
    \succ \pm {\bm \zeta}({\bf u^\star}) \cdot {\bf n}~~ \forall {\bf u}^\star \in \mathcal{S}_i,
\end{equation}
where $\mathbf{n}^\star_i$ is an inward normal vector of $\partial G$ at ${\bf u}^\star$; see \eqref{eq:1200}. 

\end{assumption}
  
\Cref{prop:LFS-2D}  can be regarded as a special case of \Cref{prop:wLFS-2D}  
with 
${\bm \zeta}({\bf u^\star})\equiv{\bf 0}$. 
Hence, \Cref{prop:wLFS-2D} is generally weaker than \Cref{prop:LFS-2D}. 
The application  of using \Cref{prop:wLFS-2D}   will be reviewed in 
  \Cref{sec:IDPopt}, \Cref{sec:RHD} and \Cref{sec:MHD}.

\subsection{Finite volume scheme in multiple dimensions}
 For simplicity, 
consider the two dimensional    system  
\eqref{eq:hPDE} and a polygonal mesh as shown in the following figure.

\begin{center}
    \begin{tikzpicture}[scale=1.5]
    \coordinate (A) at (0,0);
    \coordinate (B) at (1.5,0.2);
    \coordinate (C) at (3,0.2);
    \coordinate (D) at (0.7,1);
    \coordinate (E) at (2.3,0.9);
    \coordinate (F) at (0,2);
    \coordinate (G) at (1.8,1.8);
    \coordinate (H) at (3,2);
    \coordinate (I) at (2.5,2.8);
    

    \draw[line width=1pt] (A) -- (B) -- (C) -- (E) -- (B);
    \draw[line width=1pt] (A) -- (D) -- (B);
    \draw[line width=1pt] (D) -- (E) -- (B);
    \draw[line width=1pt] (A) -- (D) -- (F);
    \draw[line width=1pt] (D) -- (G) -- (E);
    \draw[line width=1pt] (H) -- (C);
    \draw[line width=1pt] (D) -- (G);
    \draw[line width=1pt] (E) -- (G) -- (H);
    \draw[line width=1pt] (G) -- (I) -- (H);
   \draw[line width=1pt] (A) -- (F)-- (I); 

    \node at (1.5,0.6) {$T_1$}; 
    \node at (1.7,1.2) {$T$};
    \node at (2.4,1.3) {$T_2$}; 
    \node at (1.0,1.8) {$T_3$};  
\end{tikzpicture}
\end{center}

Let $T$ be a polygonal cell with edges $E_i$ ($i=1,\cdots,T_E$)  and $T_i$ be the adjacent cell which shares the edge 
$E_i$ with $T$.  
Let $|E_i|$ denote the length of the edge $E_i$.
For solving \eqref{eq:hPDE}, 
consider a first order finite volume scheme
on the cell $T$,
\[ \bu^{n+1}_T=\bu^{n}_T-\frac{\Delta t}{|T|}\sum_{i=1}^{T_E}|E_i|\widehat{\mathbf{f}\cdot\mathbf{n}}(\bu^{n}_T, \bu^{n}_{T_i}),\]
with the Lax-Friedrichs or Rusanov flux defined by
\[\widehat{\mathbf{f}\cdot\mathbf{n}}(\bu^{n}_T,  \bu^{n}_{T_i})
=\frac12\left[\mathbf{f}(\bu^n_T)\cdot\mathbf{n}_i+\mathbf{f}(\bu^n_{T_i})\cdot\mathbf{n}_i-
\alpha_{i}(\bu^n_{T_i}-\bu^n_{T})\right],\]
where $\bu^{n}_T$ is the approximation to the average of $\mathbf{u}$ on $T$ at time level $n$,
$\mathbf{n}_i$ is the unit vector normal to the edge $E_i$ pointing outward of $T$,
and $\alpha_{i}$ is a positive number dependent on $\bu^{n}_T$ and $\bu^{n}_{T_i}$.
With the assumption $\bu^{n}_T, \bu^{n}_{T_i}\in G$, we want to find a proper $\alpha_{i}$ and a CFL condition so that $\bu^{n+1}_T\in G$.
A simple fact for any polygon $T$ is
\begin{equation}
    \sum_{i=1}^{T_E} \mathbf{n}_i|E_i|={\bf 0}.
    \label{discrete-div}
\end{equation}

\begin{example}[Method 1]
    By \eqref{discrete-div}, we obtain $ \sum_{i=1}^{T_E} \mathbf{f}(\bu^n_{T})\cdot\mathbf{n}_i|E_i|=0$, thus the right hand side of the first order Lax-Friedrichs   scheme can be rewritten as 
 \begin{align*}
      &\bu^{n}_T-\frac{\Delta t}{|T|}\sum_{i=1}^{T_E}|E_i|\widehat{\mathbf{f}\cdot\mathbf{n}}(\bu^{n}_T, \bu^{n}_{T_i})+\frac{\Delta t}{|T|}\sum_{i=1}^{T_E} \mathbf{f}(\bu^n_{T})\cdot\mathbf{n}_i|E_i|\\
 =&    \bu^{n}_T+\frac{\Delta t}{|T|}\sum_{i=1}^{T_E}|E_i|\frac{\mathbf{f}(\bu^n_T)\cdot\mathbf{n}_i-\mathbf{f}(\bu^n_{T_i})\cdot\mathbf{n}_i+
\alpha_{i}(\bu^n_{T_i}-\bu^n_{T})}{2}\\
 =&   \left(1-\frac{\Delta t}{|T|}\sum_{i=1}^{T_E}|E_i|\alpha_i\right) \bu^{n}_T+ \frac{\Delta t}{|T|}\sum_{i=1}^{T_E}|E_i|\alpha_{i}\left(
\frac{\bu^n_{T_i}+\bu^n_{T}}{2}+\frac{\mathbf{f}(\bu^n_T)\cdot\mathbf{n}_i-\mathbf{f}(\bu^n_{T_i})\cdot\mathbf{n}_i}{2\alpha_{i}}\right),
 \end{align*}   
which  is a convex combination  under the CFL constraint 
\begin{equation}
    \label{example-FV-2D-method2-CFL}\Delta t\frac{|\partial T|}{|T|}\max_i\alpha_i\leq 1,
\end{equation}
with $|\partial T|=\sum_{i=1}^{T_E}|E_i|$. Now Method 1 can be applied whenever \Cref{prop:LF-2D} holds. 
\end{example}

\begin{example}[Method 2]
    With \eqref{discrete-div}, the first order Lax-Friedrichs or Rusanov finite volume scheme can be rewritten as 
\begin{equation}
\label{FV-firstorder-2D}
    \bu^{n+1}_T=\left(1-\frac12\frac{\Delta t}{|T|}\sum_{i=1}^{T_E}|E_i|\alpha_i\right)\bu^{n}_T+\frac12\frac{\Delta t}{|T|}\sum_{i=1}^{T_E}|E_i|\alpha_i\left[\bu^{n}_{T_i}-\alpha_i^{-1}\mathbf{f}(\bu^n_{T_i})\cdot\mathbf{n}_i\right].
\end{equation} 
For gas dynamics equations with a generic EOS and $G$ in \eqref{G-Euler-3}, by Lemma \ref{lf-fact-NS}, 
 we have 
$\bu^{n}_{T_i}-\alpha_i^{-1}\mathbf{f}(\bu^n_{T_i})\cdot\mathbf{n}_i\in G$
if we use any viscosity parameter
$$\alpha_i>\max\limits_{\bu^{n}_T, \bu^{n}_{T_i}}|\rho^{-1}\bmm\cdot \bn_i|+\sqrt{\frac{p^2}{2\rho^2 e}}.$$  
Notice that \eqref{FV-firstorder-2D}  is a convex combination of $\bu^{n}_T$ and $\bu^{n}_{T_i}-\alpha_i^{-1}\mathbf{f}(\bu^n_{T_i})\cdot\mathbf{n}_i$ thus $\bu^{n+1}_T\in G$, under the CFL constraint $\Delta t\frac{|\partial T|}{|T|}\max_i\alpha_i\leq 2$, which twice of the CFL as in  \eqref{example-FV-2D-method2-CFL}.
\end{example}

\subsection{Continuous finite element method}
\label{sec:FEM}
We consider
a first order accurate IDP  continuous finite element method for hyperbolic problems 
\cite{guermond2016invariant}.   
Such a scheme is also known as the group finite element method \cite{fletcher1983group,selmin1993node,selmin1996unified,barrenechea2017analysis}.
As an example, we consider  solving \eqref{eq:hPDE} on a triangular mesh in two dimensions, and 
it can be  extended to a mesh consisting of tetrahedra, parallelepipeds, and triangular prisms in three dimensions.  

\subsubsection{Definition of the scheme}
Let $\Omega$ be the two dimensional domain and $\mathcal T_h$ be a triangular mesh. Let $V^h$ be the continuous piecewise linear polynomial space on $\mathcal T_h$.
 Let $\varphi_i(\bx)\in V^h $ be the Lagrangian basis at $i$-th vertex $\bx_i$ ($i=1,\cdots, N$) of the triangular mesh, then $\sum_{i=1}^N\varphi_i\equiv 1$.   
 Define $M$ as the mass matrix and $M^L$ as the lumped mass matrix, i.e., $M=[M_{ij}]$ with $M_{ij}=\iint_\Omega \varphi_i(\bx) \varphi_j(\bx)\, \rd \bx$, and $M^L$ is a diagonal matrix with the diagonal entries $ m_{i}=\iint_\Omega \varphi_i(\bx)\, \rd \bx=\sum_j M_{ij}.$ Define 
 $\mathcal N_i=\{j: \varphi_i(\bx)\varphi_j(\bx) \mbox{ is not constant zero}\}$, i.e., $\mathcal N_i=\{i\}\cup \{j: \mbox{$\bx_j$ is connected $\bx_i$ by an edge}\}$, as shown in the \Cref{fig:FEM}.

\begin{figure}[htbp]
    \centering
    \begin{tikzpicture}[scale=1]

    \coordinate (A) at (0,0);
    \coordinate (B) at (1.5,0.2);
    \coordinate (C) at (3,0);
    \coordinate (D) at (0.7,1);
    \coordinate (E) at (2.3,0.9);
    \coordinate (F) at (0,2);
    \coordinate (G) at (1.8,1.8);
    \coordinate (H) at (3,2);
    \coordinate (I) at (2.5,2.8);

    \draw[line width=1pt] (A) -- (B) -- (C) -- (E) -- (B);
    \draw[line width=1pt] (A) -- (D) -- (B);
    \draw[line width=1pt] (D) -- (E) -- (B);
    \draw[line width=1pt] (A) -- (D) -- (F)-- (A);
    \draw[line width=1pt] (D) -- (G) -- (E);
    \draw[line width=1pt] (E) -- (H) -- (C);
    \draw[line width=1pt] (D) -- (G) -- (F)-- (I);
    \draw[line width=1pt] (E) -- (G) -- (H);
    \draw[line width=1pt] (G) -- (I) -- (H);

    \node [red] at (E) {\textbullet};
    \node [blue] at (B) {\textbullet};
    \node [blue] at (C) {\textbullet};
    \node [blue] at (D) {\textbullet};
    \node [blue] at (G) {\textbullet};
    \node [blue] at (H) {\textbullet}; 
     
    \node[below] at (B) {$\bx_{j_1}$};
    \node[below] at (C) {$\bx_{j_2}$};
    \node[ left] at (D) {$\bx_{j_3}$};
    \node[right] at (E) {$\bx_i$}; 
    \node[above] at (G) {$\bx_{j_4}~~$};
    \node[right] at (H) {$\bx_{j_5}$}; 

\end{tikzpicture} \qquad \qquad 
\begin{tikzpicture}[scale=0.5]
    \coordinate (A) at (0,0);
    \coordinate (B) at (5,0);
    \coordinate (C) at (2,4);
    \coordinate (D) at (6,2.5);

    \draw [line width=1.1pt, fill=gray!10] (A) -- (B) -- (C) -- cycle;
    
    \draw [line width=1.1pt, fill=gray!10] (B) -- (C) -- (D) -- cycle;
    
    \node[left] at (A) {};
    \node[below] at (B) {$\bx_j$};
    \node[above] at (C) {$\bx_i$};
    \node[right] at (D) {};

    \draw[fill=red!30] (A) -- (0:0.75cm) arc (0:63:.75cm);
    \node at (1.3cm,0.5cm) {$\theta^1_{ij}$};

    \begin{scope}[shift={(6cm,2.5cm)}]
        \draw[fill=red!30] (0,0) -- (160:0.75cm) arc (160:245:0.75cm);
        \node at (-1.3cm,-0.2cm) {$\theta^2_{ij}$};
    \end{scope}
    
    \draw [line width=1.1pt] (A) -- (B) -- (C) -- cycle;
\end{tikzpicture}
    \caption{Notation for continuous finite element method.
     {Left: $\mathcal N_i=\{i, j_1, j_2, j_3, j_4, j_5\}$}.}
    \label{fig:FEM}
\end{figure}
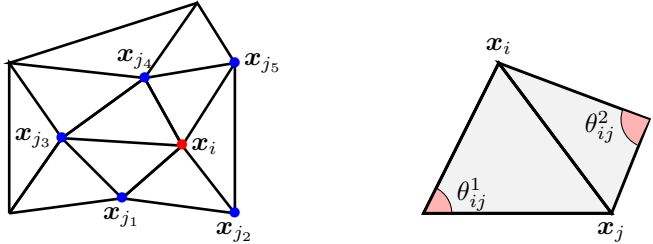

\begin{remark}
   Mass lumping is used not only in finite element methods for enforcing positivity \cite{lohner1987finite} and but also used in the Petrov--Galerkin formulation of residual distribution approach to recover a monotone residual distribution scheme, see \cite[Section 2]{abgrall2003construction}. For the Lagrange basis of piecewise  linear polynomials, such a row-sum lumped mass matrix is also equal to approximating integrals $\iint_\Omega \varphi_i(\bx) \varphi_j(\bx)\, \rd \bx$ by the simple quadrature using only the Lagrange basis points, e.g., using the quadrature of only three vertices on each triangle in a triangular mesh for approximating the integral $\iint_\Omega \varphi_i(\bx) \varphi_j(\bx)\, \rd \bx$. 
\end{remark}
Let $\bu_h^n$ denote the finite element solution at time step $n$ on a mesh of size $h$, and $\bu^n_j=\bu_h^n(\bx_j)$, then $$\bu_h^n(\bx)=\sum_i \bu^n_i \varphi_i(\bx),\quad \iint_\Omega \bu_h^n(\bx)\,\rd \bx=\sum_i \bu^n_i m_i.$$
By the group finite element approximation   $\bff(\bu^n_h)\approx \sum_{j} \bff(\bu^n_j) \varphi_j(\bx)$,
we have 
\[\iint_\Omega \nabla \cdot \bff(\bu^n_h) \varphi_i(\bx) \, \rd \bx\approx  \sum_{j\in \mathcal N_i}\bff(\bu^n_j)\cdot \bc_{ij},\quad  \bc_{ij}= \iint_\Omega \varphi_i\nabla \varphi_j\,\rd \bx,\]
then one version of IDP continuous finite element method with forward Euler time stepping can be given as
\begin{subequations}
        \label{FEM-firstorder}
\begin{equation}
    m_{i}\frac{\bu^{n+1}_i-\bu^{n}_i}{\Delta t}+\sum_{j\in \mathcal N_i} [\bff(\bu^n_j)\cdot \bc_{ij}-d_{ij}^n\bu^n_j]=0,
\end{equation}
where $d_{ij}^n$ is the artificial viscosity coefficients designed to ensure stability including the IDP properties. 
  Let $D^n=[d^n_{ij}]$ be a sparse matrix with entries $d_{ij}^n=0$ for $j\notin \mathcal N_i$. Then
 $D^n$ should be a symmetric matrix with zero row sum and non-negative off-diagonal entries, 
\begin{equation}
\label{FEM-diffusion-condition}
    d_{ij}^n\geq 0,  \quad d_{ij}^n= d_{ji}^n,\quad \forall i\neq j; \quad \sum_{j\in \mathcal N_i} d_{ij}^n=0.
\end{equation}
\end{subequations}
{\color{black}   It is possible to discuss more properties for a first order scheme like \eqref{FEM-firstorder}  by choosing  suitable $d_{ij}$ or adding suitable viscosity terms, e.g.,  \cite{guermond2011entropy,guermond2016invariant}.} Here we only review viscosity coefficients to achieve the IDP properties.

\subsubsection{Examples of artificial viscosity}

We first give two examples of \eqref{FEM-diffusion-condition}.

\begin{example}[Discrete Laplacian for Artificial Viscosity]
    \label{FEM-Sij}
  Note that the conditions  in \eqref{FEM-diffusion-condition} are met by the discrete Laplacian matrices of the linear finite element method on a simplicial mesh under some mild mesh constraints.
On a 2D triangular mesh, for the edge connecting two interior vertices $\bx_i, \bx_j$, there are two angles $\theta^1_{ij}$ and $\theta^2_{ij}$ as shown in \Cref{fig:FEM}. Let $S$ { with 
$S_{ij}=\iint_{\Omega}\nabla\varphi_i\cdot\nabla\varphi_j d\bx$}  be  the stiffness matrix in the continuous finite element method of Lagrange $P^1$ basis on a triangular mesh for solving Laplace equation $-\Delta u=0$ in two dimensions, then $S$ 
is a sparse symmetric matrix with zero row sums:
\[  S_{ij}=\begin{cases}
0, &\quad j\notin \mathcal N_i,\\
-\frac{\cot \theta^1_{ij}+\cot \theta^2_{ij}}{2}, & \quad j \in \mathcal N_i, j\neq i,\\
-\sum_{j\neq i}   S_{ij}&\quad j=i.
\end{cases}\]
Since the necessary and sufficient condition for $\cot \theta^1_{ij}+\cot \theta^2_{ij}\geq 0$ is $\theta^1_{ij}+\theta^2_{ij}\leq \pi$, for satisfying \eqref{FEM-diffusion-condition} up to a sign,
it suffices to have 
$\theta^1_{ij}+\theta^2_{ij}\leq \pi$, which can be achieved in a Delaunay triangulation in two dimensions. 
See \cite[Section 2]{xu1999monotone} for similar formulae of simplicial meshes in higher dimensions.   
 Then one choice of $D^n$ is to set $D^n=-\varepsilon S$, which is an approximation to $\varepsilon\Delta $ for some parameter $\varepsilon>0.$
 With such a choice of  $d^n_{ij}$, we can see that the scheme \eqref{FEM-firstorder} is a first order approximation to $\partial_t \bu+\nabla \cdot \bff(\bu)=0$ but a formally second order accurate approximation to the modified equation 
with extra artificial viscosity $\partial_t \bu+\nabla \cdot \bff(\bu)=\varepsilon\Delta \bu$. 
Such a modified equation approach is a standard tool for analyzing classical first order finite difference and finite volume schemes for conservation laws  \cite[Chapter 11]{leveque1992numerical}.

\end{example}

 {\color{black} The discrete Laplacian added here is essentially a multi-dimensional version of the Lax-Friedrichs type numerical dissipation, e.g., see the Lax-Friedrichs type scheme on unstructured grids in \cite[Section 2.3.1.2]{remi-10}.}

\begin{example}[Graph Laplacian for Artificial Viscosity]
\label{FEM-graphLaplacian}
    In the previous example, mesh constraints such as Delaunay triangulation are necessary in two dimensions, and such a mesh constraint will become more stringent in higher dimensions \cite{xu1999monotone}. To remove mesh constraints for satisfying \eqref{FEM-diffusion-condition}, a graph Laplacian  can   be considered \cite{guermond2014maximum}; see also \cite[Section 2.3.1.2]{remi-10} and \cite{selmin1993node, kuzmin-turek-2002flux, kuzmin2024property}.
    For a given triangular mesh with nodes $\bx_i$ and edges $E_{ij}$ connecting nodes $\bx_i$ and $\bx_j$, we regard it as a weighted graph by defining the weight for $E_{ij}$ as 
    \begin{equation}
     w_{ij}=\sum_{T\ni E_{ij}}\frac{a_T |T|}{2},
     \label{edge-weight}
    \end{equation}
      which is the weighted average of areas of two triangles sharing the edge $E_{ij}$ with $a_T$ denoting a viscosity constant for each cell $T$. Let $j\sim i$ denote that $\bx_j$ is connected to $\bx_i$ by an edge. Then the graph Laplacian matrix $L$ for such a weighted undirected graph can be given as 
    \[L_{ij}=\begin{cases}
        -w_{ij}, & j\sim i\\
       \sum_{j\in \mathcal N_i}w_{ij}, &j=i
    \end{cases}.\]
    That is, $L$ is a symmetric sparse matrix with zero row sums, positive diagonal entries, and non-positive off-diagonal entries.
    The advantage of using graph Laplacian is the easiness to satisfy \eqref{FEM-diffusion-condition} on any meshes, although graph Laplacian is a  less accurate approximation to Laplacian compared to $S_{ij}$.  
\end{example}

\begin{figure}[htbp]
    \centering
    \begin{tikzpicture}[scale=0.5]
    \coordinate (A) at (0,0);
    \coordinate (B) at (5,0);
    \coordinate (C) at (2,4);
    \coordinate (D) at (6,2.5);
    \draw [line width=1.5pt, fill=gray!10] (A) -- (B) -- (C) -- cycle;  
    \draw [line width=1.5pt, fill=gray!10] (B) -- (C) -- (D) -- cycle;   
    \node[below] at (B) {$\bx_j$};
    \node[above] at (C) {$\bx_i$};
    \node at (2.5,1.8) {$T_1$};
    \node at (4.7,2) {$T_2$};
    \draw [line width=1.5pt] (A) -- (B) -- (C) -- cycle;
\end{tikzpicture}
\qquad \qquad 
\begin{tikzpicture}[scale=0.5]
    \coordinate (A) at (0,0);
    \coordinate (B) at (5,0);
    \coordinate (C) at (2,4); 

    \coordinate (D) at (1,2); 
    \coordinate (E) at (-1,3); 
\draw [->] (D) -- (E);

    \draw [line width=1.5pt, fill=gray!10] (A) -- (B) -- (C) -- cycle;
     
    \node[left] at (A) {$\bx_k$};
    \node[below] at (B) {$\bx_j$};
    \node[above] at (C) {$\bx_i$}; 
    \node[above] at (E) {$\bn^T_{ik}$};
 \node[left] at (2.8,1.8) {$T$};   
    \draw [line width=1.5pt] (A) -- (B) -- (C) -- cycle;
\end{tikzpicture}
    \caption{An illustration of notations for computing $\bc_{ij}^T$ and $\bc_{ij}$. }
    \label{Fig:FEM-triangle-cij}
\end{figure}
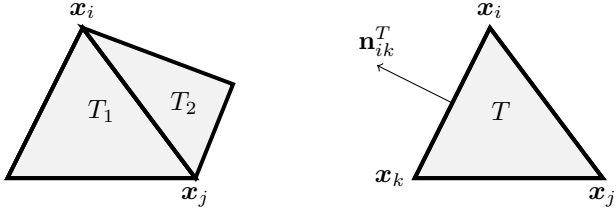

The viscosity here renders the scheme only first order accurate.
See \cite{guermond2017invariant} for better viscosity constructed by the FCT method \cite{kuzmin-turek-2002flux}. {\color{black}The residual distribution (RD)  methods~\cite{remi-10,remi-7,abgrall2003construction} can also provide improved viscosity constructions, see the next section for the connection between RD and the group finite element method.}

\subsubsection{Explicit expressions of the scheme}

Next, we give more explicit expressions of the scheme.
Let $\bc_{ij}^T=\iint_T \varphi_i\nabla \varphi_j\, \rd \bx$, where $T$ is one triangle containing the edge $E_{ij}$ connecting two nodes $\bx_i$ and $\bx_j$. Then 
     \[\bc_{ij}=\iint_\Omega \varphi_i\nabla \varphi_j\, \rd \bx=\sum_{T\ni E_{ij}} \iint_T \varphi_i\nabla \varphi_j\, \rd \bx=\sum_{T\ni E_{ij}}\bc_{ij}^T=\bc_{ij}^{T_1}+\bc_{ij}^{T_2}, \quad j\neq i, \]
     where $T_1$ and $T_2$ are two triangles sharing the edge $E_{ij}$ as shown in Figure \ref{Fig:FEM-triangle-cij} (Left).
    Let $T$ be a triangle with vertices $\bx_i, \bx_j, \bx_k$, and let $\bn^T_{ik}$ be unit normal vector to the edge $E_{ik}$ outward to the triangle $T$,
    as shown in Figure \ref{Fig:FEM-triangle-cij} (Right).   Then straightforward calculation (see  \cite{selmin1996unified,kuzmin2012flux,kuzmin2024property}) gives
\begin{subequations}
        \label{cij_properties}        
\begin{equation}
    \bc^T_{ij}=-\frac16 |E_{ik}|\bn^T_{ik}, \quad \bc_{ij}=\begin{cases}
             \sum_{T\ni \bx_i}\bc_{ii}^T={\mathbf 0}, & i=j\\
             \sum_{T\ni E_{ij}}\bc_{ij}^T=-\bc_{ji}, & i\neq j
         \end{cases}. 
\end{equation}  
\begin{equation}
\label{cij_property-totalsum}
 \sum_{j\in\mathcal N_i, j\neq i}\bc_{ij}=\sum_{j\in\mathcal N_i}\bc_{ij}=\mathbf{0},\quad \sum_{j\in\mathcal N_i}d^n_{ij}={\mathbf 0}.   
\end{equation}
\end{subequations}
Using  \eqref{cij_property-totalsum}, the finite element scheme \eqref{FEM-firstorder} can equivalently rewritten in a flux form as 
\begin{equation}
    \label{FEM-firstorder-LF3}
    m_{i}\frac{\bu^{n+1}_i-\bu^{n}_i}{\Delta t}+\sum_{j\in \mathcal N_i} [(\bff(\bu^n_j)+\bff(\bu^n_i))\cdot \bc_{ij}-d_{ij}^n(\bu^n_j-\bu^n_i)]={\mathbf 0},
    \end{equation}  
  and also rewritten as  
\begin{equation}
    \label{FEM-firstorder-LF}
    m_{i}\frac{\bu^{n+1}_i-\bu^{n}_i}{\Delta t}+\sum_{j\in \mathcal N_i} [(\bff(\bu^n_j)-\bff(\bu^n_i))\cdot \bc_{ij}-d_{ij}^n(\bu^n_j-\bu^n_i)]=0.
    \end{equation}   
Below are a few special cases of   the finite element scheme \eqref{FEM-firstorder}. 

 \begin{example}[Linear Advection]
    Consider a linear scalar PDE, i.e., $\bu=u$ is a scalar and $\bff(u)=u\bv$ with a constant vector $\bv$, then \eqref{FEM-firstorder-LF} reduces to 
     \[ m_{i}\frac{u^{n+1}_i-u^{n}_i}{\Delta t}=\sum_{j\in \mathcal N_i}  e_{ij}(u^n_j-u^n_i),\quad e_{ij}=:d_{ij}^n-\bv \cdot \bc_{ij},\]
which is a monotone scheme \cite{kuzmin-turek-2002flux, remi-10} if $e_{ij}$ is non-negative for $j\neq i$, under the CFL condition $\frac{\Delta t}{m_i} \sum\limits_{j\in \mathcal N_i, j\neq i}e_{ij} \leq 1$. The non-negativity of $e_{ij}$ $(j\neq i)$ can be   achieved by taking $$d^n_{ij}=\max\{0, \bv\cdot \bc_{ij}, \bv\cdot \bc_{ji}\},\quad j\neq i,$$
which gives {\color{black}an upwind feature} \cite{kuzmin2001positive,csik2002conservative, kuzmin2004high}. 
Such a scheme is a popular choice in algebraic flux correction (AFC) method for linear advection, which  
is naturally connected with a convection diffusion problem \cite{barrenechea2017edge}. The condition $e_{ij}\geq 0$, $j\neq i$ is also used in local extremum diminishing schemes \cite{jameson1995positive} and residual distribution schemes \cite{remi-10}. {\color{black} A truly multidimensional upwind method is the N scheme  \cite{zbMATH001505028,deconinck1993compact,deconinck1993multidimensional}.} 

 \end{example}
 
 \begin{example}[1D Problem] For the finite element scheme \eqref{FEM-firstorder} with a graph Laplacian viscosity as in \Cref{FEM-graphLaplacian}, we can consider a uniform mesh of intervals for a 1D problem with the viscosity coefficients $a_{i+\frac12}=\frac{1}{\Delta x}\alpha_{i+\frac12}$ for the interval $I_{i+\frac12}=[x_{i}, x_{i+1}]$, then the graph Laplacian is a tridiagonal matrix
 \[ L_{ij}=\begin{cases}
     -\frac12  \alpha_{i+\frac12}, & j=i+1\\
    -\frac12   \alpha_{i-\frac12}, & j=i-1\\
      \frac12   [\alpha_{i-\frac12}+\alpha_{i+\frac12}] , &j=i
    \end{cases}, \]
 and
 \eqref{FEM-firstorder-LF3} becomes   a first order finite difference or finite volume scheme with a local Lax--Friedrichs flux:
\[ \resizebox{0.99\textwidth}{!}{$ \bu^{n+1}_i=\bu^{n}_i-
\frac{\Delta t}{\Delta x}\left[ \frac{\bff(\bu^n_{i})+\bff(\bu^n_{i+1})-\alpha_{i+\frac12}^n (\bu^n_{i+1}-\bu^n_{i})}{2}-\frac{\bff(\bu^n_{i-1})+\bff(\bu^n_{i})-\alpha_{i-\frac12}^n (\bu^n_{i}-\bu^n_{i-1})}{2}\right].$}\]
 \end{example}
 
Similarly, the finite element scheme \eqref{FEM-firstorder}  on a uniform mesh of intervals for a 1D problem with viscosity defined in \Cref{FEM-Sij} reduces to the first order scheme \eqref{eq:LF1D}.

\subsubsection{Basic properties of the scheme}

We discuss some basic properties of \eqref{FEM-firstorder}.

\noindent\noindent \paragraph{\bf \bf Global conservation} Let $\tilde \bff=\sum_{j} \bff(\bu^n_j) \varphi_j(\bx)\in V^h$, then 
by summing $i$ in \eqref{FEM-firstorder}, we obtain conservation in the following sense:
\begin{align*}
 \frac{\iint_\Omega \bu_h^{n+1}\,\rd \bx-\iint_\Omega \bu_h^n\,\rd \bx}{\Delta t}  =- \sum_i\iint_{T_i} \nabla \cdot \tilde {\bf f} \,\rd \bx 
 = - \sum_i\oint_{\partial T_i}  \tilde {\bf f}\cdot \bn \,\rd s= - \oint_{\partial \Omega}  \tilde {\bf f}\cdot \bn \,\rd s, 
 \end{align*}
where we have used the fact that $\tilde {\bf f}\cdot \bn$ is continuous across each edge of triangles.

\noindent \paragraph{\bf Local conservation} It is not very obvious in what sense \eqref{FEM-firstorder} is locally conservative. 
As a matter of fact,  it is proven in \cite{selmin1993node,selmin1996unified} that the group finite element method \eqref{FEM-firstorder} can be written as a finite volume scheme on the median dual mesh shown in Figure \ref{Fig:RD}, see \cite[Section 6]{selmin1996unified}, thus the general version of Lax-Wendroff Theorem in \cite{shi2018local-LW}
 can be applied to show the convergence to weak solutions. In the next subsection,  we will also show that \eqref{FEM-firstorder} is exactly the first order Lax-Friedrichs scheme defined on unstructured grids via a definition of residual distribution scheme \cite[Section 2.3.1.2]{remi-10}, thus a Lax--Wendroff Theorem for residual distribution schemes can also be used \cite{Remi-Roe, remi-11, remi-3}.   {  Another proof of Lax--Wendroff Theorem for the continuous finite element method was given in \cite{kuzmin2025-consistency}.}

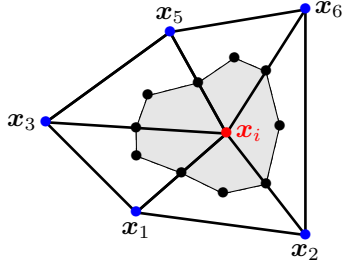
\begin{figure}[htbp] 
\centering   

   \begin{tikzpicture}[scale=1.5]

 
    \coordinate (B) at (1.5,0.2);
    \coordinate (C) at (3,0);
    \coordinate (D) at (0.7,1);
    \coordinate (E) at (2.3,0.9);
    \coordinate (G) at (1.8,1.8);
    \coordinate (H) at (3,2); 

    \coordinate (E1) at ($ (E)!0.5!(B) $);
    \coordinate (E2) at ($ (E)!0.5!(C) $);
    \coordinate (E3) at ($ (E)!0.5!(D) $);
    \coordinate (E4) at ($ (E)!0.5!(G) $);
    \coordinate (E5) at ($ (E)!0.5!(H) $); 
    \coordinate (M1) at ($ 0.334*(B)+0.334*(C) +0.334*(E) $);
    \draw[dashed] (M1) -- (E1);
    \draw[dashed] (M1) -- (E2);
    \coordinate (M2) at ($ 0.334*(B)+0.334*(D) +0.334*(E) $);
    \draw[dashed] (M2) -- (E3);
    \draw[dashed] (M2) -- (E1);

      \coordinate (M3) at ($ 0.334*(G)+0.334*(D) +0.334*(E) $);
    \draw[dashed] (M3) -- (E3);
    \draw[dashed] (M3) -- (E4);

         \coordinate (M4) at ($ 0.334*(G)+0.334*(H) +0.334*(E) $);
    \draw[dashed] (M4) -- (E5);
    \draw[dashed] (M4) -- (E4);

         \coordinate (M5) at ($ 0.334*(C)+0.334*(H) +0.334*(E) $);
    \draw[dashed] (M5) -- (E5);
    \draw[dashed] (M5) -- (E2);

 \fill [fill=gray!20, draw=black] (E2) -- (M1) -- (E1)-- (M2)-- (E3)-- (M3)-- (E4)-- (M4)-- (E5)-- (M5)-- (E2);
    
    \draw[line width=1pt] (B) -- (C) -- (E) -- (B);
    \draw[line width=1pt] (D) -- (B);
    \draw[line width=1pt] (D) -- (E) -- (B);
    \draw[line width=1pt] (D) -- (G) -- (E);
    \draw[line width=1pt] (E) -- (H) -- (C);
    \draw[line width=1pt] (D) -- (G);
    \draw[line width=1pt] (E) -- (G) -- (H); 

    \node [red] at (E) {\textbullet};
    \node [blue] at (B) {\textbullet};
    \node [blue] at (C) {\textbullet};
    \node [blue] at (D) {\textbullet};
    \node [blue] at (G) {\textbullet};
    \node [blue] at (H) {\textbullet}; 

    \node [black] at (E1) {\textbullet};
    \node [black] at (E2) {\textbullet};
    \node [black] at (E3) {\textbullet};
    \node [black] at (E4) {\textbullet};
    \node [black] at (E5) {\textbullet};
    \node [black] at (M1) {\textbullet};
    \node [black] at (M2) {\textbullet};
    \node [black] at (M3) {\textbullet};
    \node [black] at (M4) {\textbullet};
    \node [black] at (M5) {\textbullet};
    
    \node[below] at (B) {$\bx_1$};
    \node[below] at (C) {$\bx_2$};
    \node[ left] at (D) {$\bx_3$};
    \node[right] at (E) {\textcolor{red}{$\bx_i$}}; 
    \node[above] at (G) {$\bx_5$};
    \node[right] at (H) {$\bx_6$}; 

\end{tikzpicture}
\caption{The shaded polygon is the median dual cell $C_i$ constructed by connecting triangle centroids to edges containing $\bx_i$.  }
\label{Fig:RD}
\end{figure}

\noindent \paragraph{\bf Invariant domain} Since we may regard the finite element scheme \eqref{FEM-firstorder} as a Lax-Friedrichs type scheme on unstructured grids \cite[Section 2.3.1.2]{remi-10},  it is not a surprise that discussions in \Cref{example:localLFscheme-Method2} and \Cref{example:localLFscheme-Method3} can be applied here.

\begin{example}[Method 1]
\label{example:FEM-method1}
Since $\bc_{ii}={\mathbf 0}$, the scheme \eqref{FEM-firstorder-LF} can be rewritten as
\[      m_{i}\frac{\bu^{n+1}_i-\bu^{n}_i}{\Delta t}+\sum_{\substack{j\in \mathcal N_i \\ j\neq i}} [(\bff(\bu^n_j)-\bff(\bu^n_i))\cdot \bc_{ij}-d_{ij}^n(\bu^n_j-\bu^n_i)]=0.\]
Since $\sum_{j\in \mathcal N_i} d^n_{ij}=0\Rightarrow d_{ii}^n=-\sum_{\substack{j\in \mathcal N_i \\ j\neq i}} d_{ij}^n$, we have
\[\sum_{j\in \mathcal N_i} d_{ij}^n(\bu^n_j+\bu^n_i)=\sum_{\substack{j\in \mathcal N_i \\ j\neq i}} d_{ij}^n(\bu^n_j+\bu^n_i)+\sum_{j=i}  d_{ij}^n 2\bu^n_i=\sum_{\substack{j\in \mathcal N_i \\ j\neq i}} d_{ij}^n(\bu^n_j+\bu^n_i)- 2\bu^n_i \sum_{\substack{j\in \mathcal N_i \\ j\neq i}}  d_{ij}^n, \]
thus
the scheme can now be written as a convex combination,
\begin{align*}
\bu^{n+1}_i& =\bu^{n}_i  \left (1-\sum_{\substack{j\in \mathcal N_i \\ j\neq i}}\frac{2\Delta t d_{ij}^n}{m_i}\right)+\sum_{\substack{j\in \mathcal N_i \\ j\neq i}} \frac{2\Delta t d_{ij}^n}{m_i}\overline \bu_{ij}^{n+1},\\
\overline \bu_{ij}^{n+1}& = \frac12(\bu^n_j+\bu^n_i)-[\bff(\bu^n_j)-\bff(\bu^n_i)]\cdot \frac{\bc_{ij}}{2d^n_{ij}}= \frac{\bu^n_j+\bu^n_i}{2}-\frac{[\bff(\bu^n_j)-\bff(\bu^n_i)]\cdot \bn_{ij}}{2\alpha_{ij}},
\end{align*} 
where $\bn_{ij}=\frac{\bc_{ij}}{|\bc_{ij}|}$ and $\alpha_{ij}=d^n_{ij}/|\bc_{ij}|$.

For \Cref{prop:LF-2D} to hold for compressible Euler equation \eqref{eq:Euler} and invariant domain $G_S$ in \eqref{G-Euler-2}, we 
can take $d^n_{ij}/|\bc_{ij}|=\alpha_{ij}$ with a good estimate of the maximum wave speed $\alpha_{ij}$.
One convenient estimate is 
$$\alpha_{ij}=\max\left\{\left|\bff'(\bu^n_i)\cdot \bn_{ij}\right|,\left|\bff'(\bu^n_j)\cdot \bn_{ij}\right| \right\},\quad j\neq i,$$
where $\bff'(\bu)\cdot \bn$ denotes the Jacobian matrix of the flux along a given unit direction $\bn$ and $|\bff'(\bu)\cdot \bn|$ denotes its spectral radius. However, such an estimate may not ensure the IDP property in certain cases, see \cite{guermond2016fast} for a way to rigorously estimate the maximum wave speed for compressible Euler equations. And the scheme is IDP since it is a convex combination of states in $G_S$ under the CFL
\[ \frac{\Delta t}{m_i}(-d^n_{ii})=\frac{\Delta t}{m_i}\sum_{\substack{j\in \mathcal N_i \\ j\neq i}} d_{ij}^n\leq \frac12. \]
\end{example}

\begin{example}[Method 2]
 By \eqref{cij_property-totalsum} and the fact $\bc_{ii}={\bf 0}$, the scheme \eqref{FEM-firstorder} can also be rewritten as
\[ m_{i}\frac{\bu^{n+1}_i-\bu^{n}_i}{\Delta t}+\sum_{\substack{j\in \mathcal N_i \\ j\neq i}} [\bff(\bu^n_j) \cdot \bc_{ij}-d_{ij}^n(\bu^n_j-\bu^n_i)]=0,\]
which is equivalent to 
\[ \bu^{n+1}_i =\bu^{n}_i  \left (1-\sum_{\substack{j\in \mathcal N_i \\ j\neq i}}\frac{\Delta t d_{ij}^n}{m_i}\right)+\sum_{\substack{j\in \mathcal N_i \\ j\neq i}} \frac{\Delta t d_{ij}^n}{m_i}[\bu^n_j-\alpha_{ij}^{-1}\bff(\bu^n_j)\cdot \bn_{ij}],\]
where $\bn_{ij}=\frac{\bc_{ij}}{|\bc_{ij}|}$ and $\alpha_{ij}=d^n_{ij}/|\bc_{ij}|$.

Therefore, to have $\bu^{n}_i \in G\Rightarrow \bu^{n+1}_i \in G$ for compressible Euler equations with invariant domain \eqref{G-Euler-3}, 
by Lemma \ref{lf-fact-NS},
it suffices to take
\[\frac{\Delta t}{m_i} (-d_{ii}^n) = \frac{\Delta t}{m_i} \sum_{\substack{j\in \mathcal N_i \\ j\neq i}}  d_{ij}^n\leq 1, \quad \frac{d^n_{ij}}{|\bc_{ij}|}> \max\limits_{\bu^{n}_i, \bu^{n}_j}\left(|\rho^{-1}\bmm\cdot \bn_{ij}|+\sqrt{\frac{p^2}{2\rho^2 e}}\right). \] 
\end{example}

\subsection{A special residual distribution scheme: the Lax--Friedrichs scheme on unstructured grids}

The concept of residual distribution schemes traces back to early work in 1980s \cite{ni1982multiple},
see also  \cite{struijs1991fluctuation, deconinck1993compact, deconinck1993multidimensional}. 
Many schemes such as the streamline diffusion method,
the streamline upwind Petrov–Galerkin 
 finite element methods and the cell vertex finite volume methods, as well as discontinuous Galerkin methods, can be rewritten as residual distribution schemes  \cite{remi-11}. 
See  \cite{remi-7, remi-9} for higher order accurate residual distribution schemes,    \cite{Remi-Roe, remi-10} for non-oscillatory residual distribution schemes 
and \cite{remi-1,remi-2} for entropy satisfying residual distribution schemes.
Extensions to systems can be found in \cite{remi-4, remi-5, remi-6, ricchiuto2005residual}. 
 Convergence including Lax-Wendroff Theorem for residual distribution schemes  was discussed  \cite{Remi-Roe, remi-11, remi-3}.
 See \cite{abgrall2017construction} for higher order polynomial basis. 
In most references of residual distribution schemes, monotonicity for linear problems has been well studied, with which IDP can also be achieved using the methods reviewed in this section. 
 As an example of illustrating the main ideas,  we demonstrate how to construct $P^1$ continuous finite element method \eqref{FEM-firstorder} as a residual distribution scheme on a triangular mesh.

We use the same notation from the previous subsection.   
 Let $\bu^n_i$ be the numerical solution value at each node $\bx_i$ at time step $n$.  The dual cell in \Cref{Fig:RD} is constructed by connecting centroids of triangles to edges centers, which is also called median dual cell \cite{selmin1993node,selmin1996unified}. Let $|C_i|$ be the area of the dual cell volume of $C_i$ around $\bx_i$, then its area coincides with mass lumping in FEM:
$$|C_i|=\frac13 \sum_{T\ni \bx_i} |T|=\iint_\Omega \varphi_i(\bx)\,\rd \bx=m_i.$$
A residual distribution scheme is of the form 
\begin{equation}
  |C_i|\frac{\bu^{n+1}_i-\bu^n_i}{\Delta t} { =\sum_{T\ni \bx_i} \phi_i^T,}
  \label{scheme-RD}
\end{equation}
where $T\ni \bx_i$ refers to summation over all triangles containing the given vertex $\bx_i$, and $\phi_i^T$ is the  residual at $\bx_i$ for the triangle $T$ satisfying 
\begin{equation}
    \label{RD-localconservation}
    \sum_{\bx_i\in T} \phi_i^T=-\oint_{\partial T} \widehat{\bff \cdot \mathbf n} \,\rd s\approx   -\oint_{\partial T} \bff(\bu)\cdot \mathbf n \,\rd s,
\end{equation}
where $\widehat{\bff \cdot \mathbf n}$ is some numerical flux.
  In \cite{remi-11}, it is shown that any scheme satisfying
the local conservation relation \eqref{RD-localconservation} can be 
rewritten  in flux form (algebraically equivalent), and the flux can be explicitly
computed. 

  \subsubsection{The Lax-Friedrichs scheme on unstructured grids}

There are many methods to construct the  residual for the same dual cell,  e.g., see \cite{abgrall2001toward} for a finite volume approach.
Here we give one special construction of $\phi_i^T$ to recover exactly the same scheme  as \eqref{FEM-firstorder}.
By the Lax-Friedrichs scheme on unstructured grids in \cite[Section 2.3.1.2]{remi-10}, with the edge weight \eqref{edge-weight},
  We define the following   residual 
    \begin{align*}
  \phi^T_i &= -\frac13\left[ \frac{\bff(\bu^n_i)+\bff(\bu^n_k)}{2}\cdot \bn_{ik}^T |E_{ik}|-\frac{a_T |T|}{2}(\bu^n_k-\bu^n_i) \right]\\
 & -\frac13\left[ \frac{\bff(\bu^n_i)+\bff(\bu^n_j)}{2}\cdot \bn_{ij}^T |E_{ij}|-\frac{a_T |T|}{2}(\bu^n_j-\bu^n_i) \right]- \frac13 \frac{\bff(\bu^n_j)+\bff(\bu^n_k)}{2}\cdot \bn_{jk}^T |E_{jk}|. 
\end{align*} 
{ Here, $(i,j,k)$ are the indices of the vertices of the triangle $T$ as depicted in the right panel of \Cref{Fig:FEM-triangle-cij}.}  
Summing over all three vertices of a given triangle, we have
  \begin{align*}
  &\phi^T:=\sum_{\bx_i\in T}\phi^T_i \\
 &=- \frac{\bff(\bu^n_i)+\bff(\bu^n_k)}{2}\cdot \bn_{ik}^T |E_{ik}|-  \frac{\bff(\bu^n_i)+\bff(\bu^n_j)}{2}\cdot \bn_{ij}^T |E_{ij}|- \frac{\bff(\bu^n_j)+\bff(\bu^n_k)}{2}\cdot \bn_{jk}^T |E_{jk}| \\
 & =-\oint_{\partial T} \widehat {\bff \cdot\bn} \,\rd s\approx -\oint_{\partial T} \bff(\bu)\cdot\bn \,\rd s,
\end{align*} 
where the flux is
$$\widehat{\bff\cdot \bn}|_{E_{ik}}=\frac{1}{2}[\bff(\bu^n_i)+\bff(\bu^n_k)]\cdot \bn_{ik}^T.$$

 By \eqref{cij_properties}, in a triangle $T$ with vertices $\bx_i, \bx_j, \bx_k$, we have $\bc^T_{ij}=-\frac16 \bn_{ik} |E_{ik}|$. With \eqref{discrete-div}, we have
 $\bc^T_{ij}+\bc^T_{ik}+\bc^T_{ii}={\mathbf 0}$, 
 thus 
 $\bc^T_{ij}=-\bc^T_{ik}-\bc^T_{ii}.$

The  residual can be rewritten as
\begin{small}
    \begin{align*}
  \phi^T_i =& \frac{a_T |T|}{6}(\bu^n_k-\bu^n_i)+\frac{a_T |T|}{6}(\bu^n_j-\bu^n_i)\\
  &+ [\bff(\bu^n_i)+\bff(\bu^n_k)]\cdot \bc_{ij} + [\bff(\bu^n_i)+\bff(\bu^n_j)]\cdot \bc_{ik} + [\bff(\bu^n_j)+\bff(\bu^n_k)]\cdot \bc_{ii} \\
  =&\frac{a_T |T|}{6}(\bu^n_k-\bu^n_i)+\frac{a_T |T|}{6}(\bu^n_j-\bu^n_i)\\
    &+ [\bff(\bu^n_i)+\bff(\bu^n_k)]\cdot (-\bc^T_{ik}-\bc^T_{ii}) + [\bff(\bu^n_i)+\bff(\bu^n_j)]\cdot (-\bc^T_{ij}-\bc^T_{ii}) + [\bff(\bu^n_j)+\bff(\bu^n_k)]\cdot \bc_{ii}^T \\
  =&\frac{a_T |T|}{6}(\bu^n_k-\bu^n_i)+\frac{a_T |T|}{6}(\bu^n_j-\bu^n_i)\\
    &- [\bff(\bu^n_i)+\bff(\bu^n_k)]\cdot \bc^T_{ik} - [\bff(\bu^n_i)+\bff(\bu^n_j)]\cdot \bc^T_{ij} + 2\bff(\bu^n_i)\cdot \bc_{ii}^T.  
\end{align*} 
\end{small}

Summing over all triangles $T$ containing $\bx_i$, with $\bc_{ii}={\mathbf 0}$, by \eqref{cij_properties}, we have
\begin{align*}
    -\sum_{T\ni \bx_i} \phi_i^T& =\sum_{T\ni \bx_i} \left(\sum_{\substack{\bx_j\in T \\ j\neq i}} [\bff(\bu^n_j)+\bff(\bu^n_i)]\cdot \bc_{ij}^T-2f(\bu^n_i)\cdot \bc_{ii}^T-\frac{a_T|T|}{2} \sum_{\bx_j\in T} (\bu^n_j-\bu^n_i)\right)\\
    &=\sum_{\substack{j\in\mathcal N_i \\ j\neq i}}[\bff(\bu^n_j)+\bff(\bu^n_i)]\cdot \bc_{ij}-2f(\bu^n_i)\cdot \bc_{ii}-\sum_{\substack{j\in\mathcal N_i \\ j\neq i}}w_{ij}(\bu^n_j-\bu^n_i)\\
     &=\sum_{\substack{j\in\mathcal N_i }}[\bff(\bu^n_j)+\bff(\bu^n_i)]\cdot \bc_{ij}-\sum_{\substack{j\in\mathcal N_i}}d_{ij}^n(\bu^n_j-\bu^n_i)=\sum_{\substack{j\in\mathcal N_i }}[\bff(\bu^n_j)\cdot \bc_{ij}- d_{ij}^n\bu^n_j],
\end{align*}
 where   {$w_{ij}$ is defined in \eqref{edge-weight}} and $d^n_{ij}$ is the same graph Laplacian in \Cref{FEM-graphLaplacian}.  
The residual distribution scheme \eqref{scheme-RD} with such a  residual   is exactly the finite element scheme \eqref{FEM-firstorder} with $d_{ij}^n$ in  \Cref{FEM-graphLaplacian}. 
  The Lax-Wendroff type Theorem can be proven to show convergence to weak solutions for residual distribution schemes \cite{Remi-Roe, remi-11, remi-3, abgrall2002lax, barth}.   
Since an IDP finite element method (FEM) like \eqref{FEM-firstorder} can be derived as a residual distribution scheme, the Lax--Wendroff Theorem for residual distribution schemes can be applied now to show its convergence to weak solutions.  
  {  See also \cite{kuzmin2025-consistency} for another proof of the Lax--Wendroff Theorem for FEM.}

\section{Polynomial limiters for  high order finite volume and DG schemes}
\label{sec:zhang-shu}

The Zhang--Shu approach introduced in \cite{ZHANG20103091,zhang2010positivity,zhang2012maximum} is a flexible IDP approach which can be easily applied to finite volume (FV) type and discontinuous Galerkin (DG) high order schemes. 
In this section,
we review this approach and some recent related advances, e.g., \cite{cui2023classic,cui2024optimal,ding2025robust}.

\subsection{The main idea of the Zhang-Shu approach}
We   
demonstrate the  basic idea for 1D problems.
In a high order  accurate finite volume (FV) scheme with forward Euler time stepping \eqref{scheme:1Dsystem-2}, the algorithmic structure follows these steps:  
\begin{enumerate}
	\item Given the cell averages $\bar{\bu}_j^n$ on intervals $I_j$.
	\item Reconstruct a piecewise polynomial function $\bu_h^n(x)$ such that its cell average on  $I_j$ is $\bar{\bu}_j^n$.
	\item  Evaluate the piecewise polynomial  $\bu_h^n(x)$ at the cell ends $x_{j+\frac12}$  to obtain $\bu^{\pm}_{j+\frac12}$, using which \eqref{scheme:1Dsystem-2} gives the cell averages at next time step $\bar{\bu}_j^{n+1}$. 
\end{enumerate}  
A high order DG scheme follows a similar algorithmic structure:  
\begin{enumerate}
	\item Given  a piecewise polynomial solution $\bu_h^n(x)$ with cell averages equal to $\bar{\bu}_j^n$.
	\item Evolve the solution using a time discretization method to obtain $\bu_h^{n+1}(x)$ with the cell averages equal to $\bar{\bu}_j^{n+1}$. In particular, with forward Euler time stepping, the cell average updates satisfy the same  scheme   \eqref{scheme:1Dsystem-2}.
\end{enumerate}

For a given convex invariant domain $G$, instead of seeking $\bu_h^{n+1}(x)\in G$ for all $x$, 
the Zhang--Shu approach seeks to enforce the following IDP property in a finite
volume or DG scheme IDP,
\begin{equation}\label{eq:IDPdg}
 \bu_h^n(x) \in G ~~ \forall x \in \mathbb S_j, \forall j  \quad \Longrightarrow \quad   \bu_h^{n+1}(x) \in G ~~ \forall x \in \mathbb S_j, \forall j,
\end{equation}
where $\mathbb S_j \subset I_j$ is a set of   points for each cell $I_j$ to be specified later. 
 A flowchart for achieving \eqref{eq:IDPdg} in a high order IDP finite volume or discontinuous Galerkin (DG) scheme follows three steps:
\begin{enumerate}
	\item Start with $\bu_h^n(x)$, which is high order accurate and satisfies  
	\[
	\bar{\bu}_j^n \in G \quad \forall j, \qquad \bu_h^n(x) \in G \quad \forall x \in \mathbb{S}_j.    
	\]
	\item Evolve the solution forward in time to ensure  
	\begin{equation}\label{eq:cell_IDP}
		\bar{\bu}_j^{n+1} \in G \quad \forall j. 
	\end{equation}
	This guarantees that the updated cell averages remain IDP, referred to as the {\bf weak IDP} property. In general, such a step is nontrivial to achieve, which will be reviewed in this section.  
	\item Given the weak IDP condition \eqref{eq:cell_IDP}, modify $u_h^{n+1}(x)$ without losing high order accuracy to enforce the {\bf pointwise IDP at finitely many points}:  
	\[
	\bu_h^{n+1}(x) \in G \quad \forall x \in \mathbb{S}_j,
	\] 
	which often can be enforced by a simple scaling limiter of polynomials. 
\end{enumerate}

Most importantly, it can be proven that 
$\bu_h^n(x) \in G ~ \forall x \in \mathbb{S}_j$ is a sufficient condition to ensure \eqref{eq:cell_IDP}, which holds for any high order finite volume scheme and DG scheme with an IDP numerical flux on any polygonal mesh in any dimension for a convex set $G$.  This fact implies that any high order finite volume scheme and DG scheme with an IDP numerical flux using SSP time discretizations can be rendered   IDP in the sense of \eqref{eq:IDPdg}, by adding a simple limiter to limit solution polynomials at some  points
 within each cell, which allows not only easy implementation but also easy justification of the  accuracy.

\subsection{One dimensional scalar conservation law}

We first demonstrate the method for solving 1D scalar conservation laws $u_t+f(u)_x=0.$
For convenience, 
we will first focus on the forward Euler time discretization, while high order time stepping methods will be discussed later in \Cref{sec:zhang-shu-alo}. 

\subsubsection{Loss of monotonicity in high order  schemes}\label{sec:challenges}

Let  $p_j(x)$ be a polynomial of degree $k$ either evolved in a DG scheme or reconstructed in a FV method on a cell $I_{j}=[x_{j-\frac12}, x_{j+\frac12}]$ with its cell average $\bar u_{j}^n$. 
The evolution equation of the cell averages in a high order FV or DG scheme can be written in a unified form as 
\begin{equation}\label{eq:cellaverage}
\resizebox{0.99\textwidth}{!}{$ 
	\begin{aligned}
		\bar {u}_j^{n+1} &= \bar u_{j}^n - \lambda 
		\left(  \hat f( u_{j+\frac12}^-, u_{j+\frac12}^+  ) - \hat f( u_{j-\frac12}^-, u_{j-\frac12}^+  )  \right)
		=: \hat{H}_\lambda \left( \bar u_{j}^n, u_{j+\frac12}^-, u_{j+\frac12}^+, u_{j-\frac12}^-, u_{j-\frac12}^+ \right),
	\end{aligned}$}
\end{equation}
where  
$
u_{j-\frac12}^+ = p_j( x_{j-\frac12} ),   u_{j+\frac12}^- = p_j( x_{j+\frac12} )$ are   shown in the figure below.
\begin{center}
    
\scalebox{0.7}{
\centering{
\begin{tikzpicture}[samples=100, domain=-4:8,place/.style={circle,draw=blue!50,fill=blue!20,thick,
inner sep=0pt,minimum size=2mm},
transition/.style={circle,draw=red,fill=red,thick,inner sep=0pt,minimum size=2mm}]

\draw[color=blue, domain=0:4] plot (\x,{-0.1*(\x-3)^2+1.8}) ;
\draw[color=blue, domain=-4:0] plot (\x,{-0.1*(\x-1)^2+0.5}) ;
\draw[color=blue, domain=4:7] plot (\x,{-0.05*(\x-2)^2+1.5}) ;
\draw[color=red] (0,0)--(4, 0);
\draw[color=red] (7,0)--(4, 0);
\draw[color=red] (0,0)--(-4, 0);
\node at ( 0,0) [place] {};
\node at ( 4,0) [place] {};
\node at ( 7,0) [place] {};
\node at ( -4,0) [place] {};
\node at ( 0,0.9) [transition] {};
\node at ( 4,1.7) [transition] {};
\node at ( 0,1.4) {$u^+_{j-\frac 12}$};
\node at ( 4.0,2.2) {$u^-_{j+\frac 12}$};
\node at (0,-0.5) {$x_{j-\frac12}$};
\node at (4,-0.5) {$x_{j+\frac12}$};
\node at (-4,-0.5) {$x_{j-\frac32}$};
\node at (6.8,-0.5) {$x_{j+\frac32}$};
\node at (2,0.3) {$I_j$};
\node at (2,1.99) {$p_j(x)$};
 \end{tikzpicture}}
}
\end{center}

Assume the numerical flux $\hat f ( \cdot, \cdot )$ is monotone as defined \Cref{sec:mono}, so that the corresponding first order scheme \eqref{eq:1stSCLmono} is monotone under the CFL condition \eqref{eq:CFL1}. 
By order barriers in \Cref{sec:barrier}, the high order scheme \eqref{eq:cellaverage} is in general not monotone.
In particular, the function $\hat{H}_\lambda$ is monotonically non-decreasing 
   w.r.t. $\bar u_{j}^n$ and   $u_{j+\frac12}^+, u_{j-\frac12}^-$, but non-increasing 
w.r.t. $u_{j+\frac12}^-, u_{j-\frac12}^+$. 
This induces one challenge for achieving \eqref{eq:cell_IDP}, which is explained by the following simple example. 
\begin{example} 
\label{example:advection_upwind}
	 For the linear advection equation $u_t+u_x = 0$, the high order scheme \eqref{eq:cellaverage}  with the upwind flux reduces to 
	\begin{align} \label{eq:cellaverageUpwind}
		\bar {u}_j^{n+1}  = \bar u_{j}^n - \lambda 
		\left(   u_{j+\frac12}^--  u_{j-\frac12}^- \right),
	\end{align}
	which is decreasing w.r.t. $u_{j+\frac12}^-$. Assume $G:= [m, M]=[0,1]$, $\bar u_{j}^n= u_{j-\frac12}^-=0$, and $u_{j+\frac12}^-=1$, then $\bar {u}_j^{n+1} = - \lambda <0$ for any time step $\Delta t>0$. 
\end{example}

\subsubsection{Weak monotonicity of high order schemes}

From \Cref{example:advection_upwind}, we can see that the loss of monotonicity in the high order scheme \eqref{eq:cellaverage} implies that 
 the scheme \eqref{eq:cellaverage} may fail to preserve the bounds for any positive time step $\Delta t>0$ even if  $\bar u_{j}^n, u_{j+\frac12}^-, u_{j+\frac12}^+, u_{j-\frac12}^-, u_{j-\frac12}^+$ are within the desired bounds. 
 
 In other words, since the scheme \eqref{eq:cellaverage} has the property $\bar u^{n+1}_j=\hat{H}_\lambda(\uparrow,\downarrow,\uparrow,\uparrow,\downarrow)$,  requiring all of its input values to be in $G=[m, M]$ is not enough to achieve $\bar u^{n+1}_j\in G$. 
 Even though $\bar u^{n+1}_j=\hat{H}_\lambda(\uparrow,\downarrow,\uparrow,\uparrow,\downarrow)$ is not a monotone function w.r.t. independent degree of freedoms $\bar u_{j}^n, u_{j+\frac12}^-, u_{j+\frac12}^+, u_{j-\frac12}^-, u_{j-\frac12}^+$,
the key observation by   Zhang and Shu in \cite{ZHANG20103091}  is that \eqref{eq:cellaverage} can still be rewritten as a monotone function w.r.t. some point values, which might be dependent degree of freedoms in general. 

Notice that $\hat{H}_\lambda$ is decreasing only w.r.t. $u_{j+\frac12}^-, u_{j-\frac12}^+$, which are degree of freedoms within the cell $I_j$ thus can be controlled by $\bar u_{j}^n$ 
if the cell average $\bar u_{j}^n$ is decomposed into a convex combination of several point values including  $u_{j+\frac12}^-, u_{j-\frac12}^+$. Such a decomposition can be achieved via $L$-point Gauss--Lobatto quadrature, which is exact for integrating polynomials of degree $k$ with positive quadrature weights if $L\geq  \frac{k+3}{2}$. 
This implies 
\begin{equation}\label{eq:GL1D}
	\bar u_j^n = \frac{1}{\Delta x} \int_{ I_{j} } p_j(x) {\rm d} x = \sum_{\mu = 1}^L  {\omega}_\mu p_j(  x_j^{(\mu)} ) =  \sum_{\mu = 2}^{L-1}{\omega}_\mu p_j(  x_{j}^{(\mu)} )  
	+ {\omega}_1   u_{j-\frac12}^+  +  {\omega}_L   u_{j+\frac12}^-,
\end{equation}
where  $ {\omega}_\mu>0$ are the Gauss--Lobatto quadrature weights for the interval $[-\frac12, \frac12]$ satisfying $\sum_\mu{\omega}_\mu=1$ with $ \omega_1  = \omega_L = \frac{1}{L(L-1)}$, and $\{ x_j^{(\mu)} \}$ are the quadrature nodes for $I_j$ with $x_j^{(1)} = x_{j-\frac12}$ and $ x_j^{(L)} = x_{j+\frac12}$. 
Then the high order scheme \eqref{eq:cellaverage} using  the cell average decomposition \eqref{eq:GL1D} can be rewritten as follows,
\begin{equation}
\label{weak-mono-proof}
	\begin{aligned}
		\bar{u}_j^{n+1} = & \sum_{\mu=2}^{L-1} {\omega}_\mu p_j(x_j^{(\mu)}) + {\omega}_L \left( u_{j+\frac{1}{2}}^- - \frac{\lambda}{{\omega}_L} \left( \hat{f}\big( u_{j+\frac{1}{2}}^-, u_{j+\frac{1}{2}}^+ \big) - \hat{f}\big( u_{j-\frac{1}{2}}^+, u_{j+\frac{1}{2}}^- \big) \right) \right) \\
		& + {\omega}_1 \left( u_{j-\frac{1}{2}}^+ - \frac{\lambda}{{\omega}_1} \left( \hat{f}\big( u_{j-\frac{1}{2}}^+, u_{j+\frac{1}{2}}^- \big) - \hat{f}\big( u_{j-\frac{1}{2}}^-, u_{j-\frac{1}{2}}^+ \big) \right) \right) \\
		= & \sum_{\mu=2}^{L-1} {\omega}_\mu p_j(x_j^{(\mu)}) + {\omega}_L H_{ \frac{\lambda}{{\omega}_L}} \left( u_{j-\frac{1}{2}}^+, u_{j+\frac{1}{2}}^-, u_{j+\frac{1}{2}}^+ \right) + {\omega}_1 H_{\frac{\lambda}{{\omega}_1}} \left( u_{j-\frac{1}{2}}^-, u_{j-\frac{1}{2}}^+, u_{j+\frac{1}{2}}^- \right),
	\end{aligned}
\end{equation}
which is a convex combination of 
$$
p_j(x_j^{(\mu)}),~~  H_{ \frac{\lambda}{{\omega}_L}} \left( u_{j-\frac{1}{2}}^+, u_{j+\frac{1}{2}}^-, u_{j+\frac{1}{2}}^+ \right), ~~ H_{\frac{\lambda}{{\omega}_1}} \left( u_{j-\frac{1}{2}}^-, u_{j-\frac{1}{2}}^+, u_{j+\frac{1}{2}}^- \right).
$$
Recall that $H_{\lambda}(\cdot,\cdot,\cdot)$, defined in the first order scheme \eqref{eq:1stSCLmono}, is monotonically non-decreasing  in all its three arguments under the CFL condition \eqref{eq:CFL1}. 
Thus, $H_{\frac{\lambda}{{\omega}_1}}$ and  $H_{ \frac{\lambda}{{\omega}_L}}$ 
are monotonically non-decreasing under a reduced CFL condition 
\begin{equation}\label{eq:CFL1Dhigh}
	\lambda \alpha \le  \omega:=  \omega_1 =  \omega_L = \frac{1}{L(L-1)}. 
\end{equation}
Therefore, $\bar{u}_j^{n+1}$ is a monotonically non-decreasing function of 
$$
u_{j-\frac12}^-,~~ u_{j+\frac12}^+,~~ p_j( x_j^{(\mu)} ),~~ 1\le \mu \le N. 
$$
It can be stated as the following theorem, as a sufficient condition for achieving \eqref{eq:cell_IDP}.
\begin{theorem}[Weak monotonicity of high order schemes]\label{thm:1}
	For a finite volume scheme or the scheme satisfied by the cell averages of the DG method with forward Euler time discretization \eqref{eq:cellaverage} using a monotone flux $\hat f$, let $p_j(x)$ be  the reconstructed or DG solution polynomial of degree $k$  satisfying
	\[
	\bar{u}_j^n = \frac{1}{\Delta x} \int_{I_j} p_j(x) \, dx, \quad u_{j-\frac{1}{2}}^+ = p_j \left( x_{j-\frac{1}{2}} \right) \quad \text{and} \quad u_{j+\frac{1}{2}}^- = p_j \left( x_{j+\frac{1}{2}} \right).
	\]
Then under the CFL condition \eqref{eq:CFL1Dhigh}, $\bar{u}^{n+1}_j$ is monotone w.r.t. 
$$u_{j-\frac12}^-=p_{j-1}(x_{j-\frac12}),~ u_{j+\frac12}^+=p_{j+1}(x_{j+\frac12}),~ p_j( x_j^{(\mu)} ),~ \mu=1,\cdots, L$$
 Therefore, with \eqref{eq:CFL1Dhigh}, a sufficient condition for   $\bar{u}_j^{n+1} \in G:= [m, M]$ for all $j$ is
\begin{equation}\label{eq:1DBPcondition}
		 p_j( x_j^{(\mu)} )\in G,~~ 1\le \mu \le N\quad \forall j.
\end{equation}
\end{theorem}
 
\Cref{thm:1} can be regarded as a complementary result to the classical Godunov Theorem. In a high order accurate linear finite volume scheme \eqref{eq:cellaverage}, $\bar u^{n+1}_j$ is not monotone w.r.t. $\bar u^{n}_j$, but $\bar u^{n+1}_j$  can still be a monotone function w.r.t. some quadrature point values.

\begin{remark}[Optimality of Gauss--Lobatto quadrature]
	The cell average decomposition 
	\eqref{eq:GL1D} based on the Gauss--Lobatto quadrature plays a critical role in revealing the weak monotonicity. In fact, any quadrature rule with positive weights and nodes including the two endpoints:
	\begin{equation}\label{eq:g1DCAD}
		 \frac{1}{\Delta x} \int_{ I_{j} } p(x) {\rm d} x =  \sum_{\mu = 2}^{N-1}  {\omega}_\mu p( x_{j}^{(\mu)} )  
		+  {\omega}_1   p(x_{j-\frac12})  +   {\omega}_N   p(x_{j+\frac12})
	\end{equation}
	is 
	 applicable, 
	as long as it is exact for integrating polynomials of degree $k$. 
	Different feasible quadrature rules would lead to different IDP CFL conditions, for example, \eqref{eq:g1DCAD} leads to $\lambda \alpha \le \min\{ {\omega}_1, {\omega}_N \}$. 
	It is highly desirable to choose the optimal feasible quadrature to maximize the CFL number $\min\{ {\omega}_1, {\omega}_N \}$. 
	 The Gauss--Lobatto quadrature was  proven to the optimal one in the sense of the largest CFL in 1D \cite{cui2024optimal}. 
\end{remark}

\begin{remark}
For high order DG methods, the CFL condition \eqref{eq:CFL1Dhigh} is close to the CFL condition needed for linear stability. For high order FV schemes,
the CFL condition \eqref{eq:CFL1Dhigh} is usually smaller than the commonly used ones, e.g., for FV WENO schemes. On the other hand,  \eqref{eq:CFL1Dhigh} is only a convenient sufficient condition but not a necessary condition for achieving \eqref{eq:cell_IDP}.
\end{remark}

\subsubsection{A simple scaling limiter for enforcing   pointwise bounds}\label{sec:limiter} 

The weak monotonicity in \Cref{thm:1} indicates that a
high order conservative finite volume or DG scheme \eqref{eq:cellaverage} preserves  $\bar u^{n+1}_j\in [m, M]$ if 
the approximation polynomials $p_j(x)$ satisfy \eqref{eq:1DBPcondition}, which can be enforced by 
the following simple limiter for a polynomial with cell average $\bar u_j^n\in [m, M]$:
\begin{subequations}
    \label{eq:ZSlimiter}
\begin{align}
\tilde p_j(x)   = \theta(p_j(x) - \bar u_j^n) + \bar u_j^n,\\
\theta   = \min\left\{ \left| \frac{M - \bar u_j^n}{M_j - \bar u_j^n} \right|,~ \left| \frac{m - \bar u_j^n}{m_j - \bar u_j^n}\right|,~ 1\right\},
M_j= \max_{x \in \mathbb S_{j}} p_j(x), m_j = \min_{x \in \mathbb S_{j}} p_j(x). 
\end{align}
\end{subequations}

The limiter is a simplified version of those used in \cite{BarthJespersen1989,liu1996nonoscillatory},
and it satisfies the following properties.

\paragraph{Conservation} Since $\tilde p_j(x)$ is a convex combination of $p_j(x)$ with its own cell average, $\tilde p_j(x)$ has the same cell average $\bar u_j^n$ on $I_j$.

\paragraph{Bounds at $\mathbb S_j$} By the definition of $\theta$ in \eqref{eq:ZSlimiter}, $\tilde p_j(x)$ satisfies  \eqref{eq:1DBPcondition}, if the cell average is within the bounds  $\bar u_j^n\in [m, M]$.

\paragraph{Easiness of implementation} This limiter is local to each cell and only requires point values at $\mathbb S_j$, thus its implementation is easy and friendly for parallel computing. 

\paragraph{High order accuracy}
This limiter does not destroy the high order approximation accuracy of $p_j(x)$ under suitable assumptions.

\begin{theorem}[Accuracy of Zhang--Shu limiter]\label{thm:limiter_acc}
Assume the cell average of $p_j(x)$ is $\bar u^n_j\in [m, M]$, then for any function $u(x)\in [m, M]$,
the limiter \eqref{eq:ZSlimiter} for 
 polynomials of degree $k$   satisfies
\[
\left| p_j(x) - \tilde{p}_j(x) \right| \leq C_k \max_{x \in I_j} |p_j(x) - u(x)|,
\]
where $C_k$ is a constant that depends only on the polynomial degree $k$. 
\end{theorem}
\begin{remark}
\label{rmk-limiter-2D}
  If replacing $M_j$ and $m_j$ in \eqref{eq:ZSlimiter} by the maximum and minimum of $p_j(x)$ in the cell $I_j$, then  it a more restrictive thus less accurate limiter. For such a more restrictive limiter on a reference cell of any shape in any dimension, the same result holds with $C_k$ depending only on the polynomial degree $k$ and the reference cell.
\end{remark}
\begin{proof}
 We review the key arguments of the proof here. Without loss of generality, we only need to discuss the case that $p_j(x)$ is not a constant and $\theta=\frac{M - \bar u_j^n}{M_j - \bar u_j^n}$ with $M_j>M$. For convenience, let $\overline{p}_j=\bar u^n_j$, then 
$\widetilde{p}_j(x)-p_j(x) =(M-M_j)\frac{p_j(x)-\overline{p}_j}{M_j-\overline{p}_j}.$

First, $|M-M_j|\leq \max\limits_{x \in I_j} |p_j(x) - u(x)|$ because $\max\limits_{x\in \mathbb S_j} p_j(x)=M_j>M\geq u(x)$, see \cite{xu2017bound}. 
Thus we only need to prove  
$\left |\frac{p_j(x)-\overline{p}_j}{M_j-\overline{p}_j}\right | \leq C_k$. 
Define $q(\xi)=p_j \left(\xi \Delta x +x_{j-\frac12}\right)-\overline{p}_j$ with $\xi \in [0,1]$,
then $\overline q=\int_0^1 q(\xi)\,\rd \xi=0$, $\max\limits_{\xi\in[0,1]} q(\xi)=\max\limits_{x\in I_j}p_j(x)-\overline{p}_j$
and $\min\limits_{\xi\in[0,1]} q(\xi)=\min\limits_{x\in I_j}p_j(x)-\overline{p}_j$.
We have
\[\left |\frac{p_j(x)-\overline{p}_j}{M_j-\overline{p}_j}\right |\leq  \frac{|q(\xi)|}{\max\limits_{\xi\in[0,1]} q(\xi)}
\leq \frac{\max\limits_{\xi \in [0,1]}|q(\xi )|}{\max\limits_{\xi \in[0,1]} q(\xi )}.\]
Thus we only need to prove $ \frac{\max\limits_{\xi \in [0,1]}|q(\xi)|}{\max\limits_{\xi \in[0,1]} q(\xi)} \leq C_k$ for any polynomial of degree $k$ satisfying $\int_0^1 q(\xi)\, \rd \xi=0$.
For quadratic polynomials in one dimension, $C_2=3$ was proven by explicit calculations in \cite{liu1996nonoscillatory}.
For higher order polynomials in one dimension,  
$C_k\leq (k^2+k-1)\Lambda_{k+1}[0,1], $
 where $\Lambda_{k+1}[0,1]$ is the Lebesgue constant on the interval $[0,1]$, see \cite[Lemma 7]{zhang2017positivity}.
For general $k$ and higher dimensions, the existence of the constant $C_k$ can be established by an abstract proof similar to proving the equivalence of two norms in a finite-dimensional Banach space \cite[Lemma 8]{zhang2017positivity}, which can be used to prove the multi-dimensional case as mentioned \Cref{rmk-limiter-2D}.
\end{proof}

\subsubsection{The bound-preserving algorithm flowchart}
\label{sec:zhang-shu-alo}

Assuming $\bar{u}^n_j\in [m, M]$, then using the simple limiter \eqref{eq:ZSlimiter} at time step $n$ can achieve the sufficient condition in \Cref{thm:1} to ensure $\bar{u}^{n+1}_j\in [m, M]$, with which the simple limiter \eqref{eq:ZSlimiter} can again enforce bounds at time step $n+1$.

Such a method can be easily extended from forward Euler to  high order strong stability preserving (SSP) explicit Runge--Kutta or multistep methods \cite{gottlieb2001strong, gottlieb2009high}, which are convex combinations of forward Euler steps.
For  a semi-discrete scheme $\frac{d}{dt} u_h = \mathcal{L}(u_h)$, e.g., $\mathcal L$ denotes high order spatial discretization, 
 the classic third order explicit SSP Runge–Kutta method is
\begin{align}
	u_h^{n,*} &= u_h^n + \Delta t \mathcal{L} (u_h^n), \notag \\
	u_h^{n,**} &= \frac{3}{4} u_h^n + \frac{1}{4} \left( u_h^{n,*} + \Delta t \mathcal{L} \left( u_h^{n,*} \right) \right), \label{ssp-RK-3rd} \\
	u_h^{n+1} &= \frac{1}{3} u_h^n + \frac{2}{3} \left( u_h^{n,**} + \Delta t \mathcal{L} \left( u_h^{n,**} \right) \right), \notag
\end{align}  
and the explicit SSP third order multistep method is
\[
u_h^{n+1} = \frac{16}{27} \left( u_h^n + 3\Delta t \mathcal{L} (u_h^n) \right) 
+ \frac{11}{27} \left( u_h^{n-3} + \frac{12}{11} \Delta t \mathcal{L} (u_h^{n-3}) \right),
\]  
where $u_h^n$ represents the numerical solution at the $n$-th time step.

To construct a high order accurate bound-preserving scheme, we can use 
SSP high order time discretizations with high order FV or DG methods in space with a monotone numerical flux, with the limiter  \eqref{eq:ZSlimiter} applied in each time step in a SSP multistep methods or each time stage in a SSP RK method. Then under suitable CFL conditions,  \Cref{thm:1} and \Cref{thm:limiter_acc} imply that the full scheme is conservative, bound-preserving and high order accurate.

The main advantages of such a method include easy extensions to systems and higher dimensions, easy implementation and easy justification of accuracy, which is due to not only \Cref{thm:limiter_acc} but also, more importantly, the fact that this approach is built upon an intrinsic weak monotonicity property of the high order spatial discretization \eqref{eq:cellaverage}. 

On the other hand, although the weak monotonicity can be established for FV and DG schemes, in general it does not hold for high order finite difference (FD) schemes. For special compact FD schemes, weak monotonicity may hold and bound-preserving schemes can be constructed \cite{li2018high}.

\subsubsection{A simplified weak monotonicity and limiter}
\label{sec:zhanglimiter-FV}
The limiter  \eqref{eq:ZSlimiter} involves the polynomial $p_j(x)$ which is not available in  ENO (essentially non-oscillatory)  and WENO (weighted ENO) finite volume reconstructions. One way is to use interpolation to construct an approximation polynomial $p_j(x)$ in ENO and WENO schemes to apply \eqref{eq:ZSlimiter}. An easier alternative  provided in 
\cite{zhang2011maximum} is  to avoid explicitly using  point values $p_j(x^{(\mu)}_j)$ for $\mu=2,\cdots,N-1$. 
Since $\sum_{\mu=2}^{N-1}\frac{{\omega}_{\mu}}{1-2 \omega_1}p_j(
x^{(\mu)}_j)$ is a convex combination of point values $p_j(x^{(\mu)}_j)$ for $\mu=2,\cdots,N-1$,
by the Mean Value Theorem, there exists some point $x^*_j\in I_j$ such that
$\sum_{\mu=2}^{N-1}\frac{{\omega}_{\mu}}{1-2 \omega_1}p_j(
x^{(\mu)}_j)=p_j(x^*_j).$
We can rewrite \eqref{weak-mono-proof} as
\[\overline{u}^{n+1}_j=(1-2 \omega_1)p_j(x^*_j)+{\omega_1} H_{ \frac{\lambda } { \omega_1}} ( u^+_{j-\frac
12}, u^-_{j+\frac 12}, u^+_{j+\frac 12} ) +{\omega_1} H_{ \frac{\lambda} { \omega_1}} ( u^-_{j-\frac 12}, u^+_{j-\frac 12}, u^-_{j+\frac
12} ),\]
thus in \Cref{thm:1} we can use the following weaker sufficient condition to replace  \eqref{eq:1DBPcondition},
\begin{equation}
 \label{suff-condition-2}
 u^\pm_{j-\frac
12}, u^\pm_{j+\frac 12}, \frac{\overline{u}^{n}_j- \omega_1 u^+_{j-\frac
12}- \omega_1 u^-_{j+\frac
12}}{1-2 \omega_1}=\sum_{\mu=2}^{N-1}\frac{{\omega}_{\mu}}{1-2 \omega_1}p_j(
x^{(\mu)}_j)=p_j(x^*_j)\in [m,M].
\end{equation}

A simplified limiter to enforce \eqref{suff-condition-2} in ENO and WENO finite volume schemes 
is 
\begin{subequations}
\label{limiter-scalar-simple-2}
\begin{align}
    \widetilde p_j(x)=\theta\left( p_j(x)-\overline p_j \right) +\overline p_j,\quad \theta=\min\left\{1, \left|\frac{M-\overline{p}_j}{M_j-\overline{p}_j}\right|,
 \left|\frac{m-\overline{p}_j}{m_j-\overline{p}_j}\right|\right\},\\ 
M_j=\max\limits_{{x}^{(1)}_j, {x}^{(N)}_j, {x}^*_j}p_j(x), m_j=\min\limits_{{x}^{(1)}_j, {x}^{(N)}_j, {x}^*_j}p_j(x),\quad 
p_j(x^*_j)=\frac{\overline{u}^{n}_j- \omega_1 u^+_{j-\frac
12}- \omega_1 u^-_{j+\frac
12}}{1-2 \omega_1}.
\end{align}
\end{subequations}

\Cref{thm:limiter_acc}
still applies to \eqref{limiter-scalar-simple-2}
since it is a more relaxed limiter than 
\eqref{eq:ZSlimiter}.

\begin{remark}
    If \eqref{suff-condition-2} is replaced by requiring 
    $u^\pm_{j-\frac{1}{2}}, u^\pm_{j+\frac{1}{2}}, \frac{\overline{u}^{n}_j - a u^+_{j-\frac{1}{2}} - a u^-_{j+\frac{1}{2}}}{1 - a} \in [m, M]$ 
    with $a \in (0, \frac{1}{2})$, then this still provides a sufficient condition for ensuring $\bar{u}^{n+1}_j \in [m, M]$, as first proven in \cite{perthame1996positivity}. 
    However, it is difficult to justify the accuracy of enforcing 
    $\frac{\overline{u}^{n}_j - a u^+_{j-\frac{1}{2}} - a u^-_{j+\frac{1}{2}}}{1 - a} \in [m, M]$, 
    unless $a$ corresponds to a quadrature weight such that 
    $\frac{\overline{u}^{n}_j - a u^+_{j-\frac{1}{2}} - a u^-_{j+\frac{1}{2}}}{1 - a} = p_j(x^*_j)$, 
    which allows one to invoke \Cref{thm:limiter_acc}.
\end{remark}

\subsection{One dimensional hyperbolic systems}

We now discuss the extension to 1D  systems of the form \eqref{eq:1stHsystem} with a convex invariant domain $G$ defined in \eqref{eq:ASS-G}. 
Consider the evolution equation of the cell averages for a high order finite volume or DG scheme for the system \eqref{eq:1stHsystem}, which can be written as \eqref{scheme:1Dsystem-2}
with a numerical flux $\hat {\bf f} (\cdot,\cdot)$.
Assume  $\hat {\bf f}$ is IDP as defined in \Cref{def:IDPflux} 
under a CFL condition $\lambda \alpha   \le c_0$. 

\subsubsection{The weak IDP property of high order schemes}

The monotonicity is no longer directly applicable for hyperbolic systems. 
Instead, we can use either the convex decomposition technique \cite{ZHANG20103091,zhang2010positivity} or the GQL approach \cite{wu2023geometric}. 
In particular, let ${\bf p}_j(x)$ be the approximation polynomial in $I_j$, then \eqref{weak-mono-proof} implies that 
a high order scheme \eqref{scheme:1Dsystem-2} is a convex combination of two formal first order IDP schemes:
\begin{equation*} 
	\begin{aligned}
	&	\bar{\bu}_j^{n+1} =  \sum_{\mu=2}^{L-1} {\omega}_\mu {\bf p}_j(x_j^{(\mu)}) + {\omega}_L \left( \bu_{j+\frac{1}{2}}^- - \frac{\lambda}{{\omega}_L} \left( \hat{\bff}\big( \bu_{j+\frac{1}{2}}^-, \bu_{j+\frac{1}{2}}^+ \big) - \hat{\bff}\big( \bu_{j-\frac{1}{2}}^+, \bu_{j+\frac{1}{2}}^- \big) \right) \right) \\
		& + {\omega}_1 \left( \bu_{j-\frac{1}{2}}^+ - \frac{\lambda}{{\omega}_1} \left( \hat{\bff}\big( \bu_{j-\frac{1}{2}}^+, \bu_{j+\frac{1}{2}}^- \big) - \hat{\bff}\big( \bu_{j-\frac{1}{2}}^-, \bu_{j-\frac{1}{2}}^+ \big) \right) \right)\\
        &=  \sum_{\mu=2}^{L-1} {\omega}_\mu {\bf p}_j(x_j^{(\mu)}) + {\omega}_L {\bf H}_{ \frac{\lambda}{{\omega}_L}} \left( \bu_{j-\frac{1}{2}}^+, \bu_{j+\frac{1}{2}}^-, \bu_{j+\frac{1}{2}}^+ \right) + {\omega}_1{\bf H}_{\frac{\lambda}{{\omega}_1}} \left( \bu_{j-\frac{1}{2}}^-, \bu_{j-\frac{1}{2}}^+, \bu_{j+\frac{1}{2}}^- \right),
	\end{aligned}
\end{equation*}
Since ${\bf H}$ is a first order scheme thus is IDP under suitable CFL, e.g.,
\[ \bu_{j-\frac{1}{2}}^+, \bu_{j+\frac{1}{2}}^-, \bu_{j+\frac{1}{2}}^+ \in G\Rightarrow {\bf H}_{ \frac{\lambda}{{\omega}_L}} \left( \bu_{j-\frac{1}{2}}^+, \bu_{j+\frac{1}{2}}^-, \bu_{j+\frac{1}{2}}^+ \right)\in G,\qquad \mbox{if }\frac{\lambda}{{\omega}_L}\alpha\leq c_0.  \]
 thus \Cref{thm:1} can be extended as follows.                 
\begin{theorem}[Weak IDP property]\label{thm:1Dsystem}
	For a finite volume scheme or the scheme satisfied by the cell averages of the DG method in the form \eqref{scheme:1Dsystem-2} using an IDP flux $\hat \bff$, with an approximation polynomial vector ${\bf p}_j(x)$ of degree $k$   satisfying
	\[
	\bar{\bf u}_j^n = \frac{1}{\Delta x} \int_{I_j} {\bf p}_j(x) \, dx, \quad {\bf u}_{j-\frac{1}{2}}^+ = {\bf p}_j \left( x_{j-\frac{1}{2}} \right) \quad \text{and} \quad {\bf u}_{j+\frac{1}{2}}^- = {\bf p}_j \left( x_{j+\frac{1}{2}} \right),
	\]
then $\bar{\bf u}_j^{n+1} \in G$ 
	under the CFL condition $\lambda \alpha \le   \omega c_0$ with $\omega=\frac{1}{L(L-1)}$,	if 
		\begin{equation}\label{eq:1DBPcondition_system}
		{\bf p}_{j-1}(x_{j-\frac12})={\bf u}_{j-\frac12}^-,~ {\bf p}_{j+1}(x_{j+\frac12})={\bf u}_{j+\frac12}^+,~ {\bf p}_j( x_j^{(\mu)} )\in G,~~ 1\le \mu \le L.
	\end{equation}
\end{theorem}

\subsubsection{A simple scaling limiter}
 
Similar to the scalar case, given ${\bf p}_j(x)$ with a cell average $\bar \bu^n_j\in G$, for enforcing \eqref{eq:1DBPcondition_system}, 
on each cell $I_j$ a simple scaling limiter  can be  designed as follows:
\begin{equation}\label{eq:ZSlimiter_system}
	\tilde {\bf p}_j(x) = \theta( {\bf p}_j(x) - \bar {\bf u}_j^n) + \bar {\bf u}_j^n,
\end{equation}
$$
\theta = \min_{\mu} \{ \theta_j^{(\mu)} \}, \qquad \mbox{with} \qquad 
\theta_j^{(\mu)} = \begin{cases}
	1,\quad   &  \mbox{if}~~ {\bf p}_j(x_{j}^{(\mu)}) \in G,
	\\
	\frac{ \left| {\bf u}_* -  \bar {\bf u}_j^n \right| }{ \left| {\bf p}_j(x_{j}^{(\mu)}) - \bar {\bf u}_j^n \right| }, \quad & \mbox{otherwise}, 
\end{cases}
$$
where ${\bf u}_*$ is the  intersection point   of $\partial G$ and the line segment that  connects  $\bar {\bf u}_j^n\in G$ and  ${\bf p}_j(x_{j}^{(\mu)})\notin G$.

In many cases, the computation of $\theta_j^{(\mu)}$ is cumbersome and may require a root-finding procedure. 
If the invariant domain can be reformulated such that all the functions $g_i({\bf u})$ in \eqref{eq:ASS-G} are linear or concave with respect to $\bf u$, then a different but easier alternative of   $\theta$ can be considered, 
$$
\theta = \min_i \{ \theta_i\}, \qquad \mbox{with} \qquad  \theta_i = \min \left \{ \left| \frac{  g_i(\bar {\bf u}_j^n) - \varepsilon_i  }{  g_i(\bar {\bf u}_j^n) - \min_{\mu} g_i( {\bf p}_j(x_{j}^{(\mu)})  ) }  \right|, 1 \right \},
$$
where the small number $\varepsilon_i \ge 0 $ is   introduced to mitigate the effect of round-off errors.  
Jensen inequality for concave functions $g_i({\bf u})$ implies 
$
\tilde {\bf p}_j(x) \in G~~ \forall x \in \mathbb S_j,
$
e.g., see \cite{zhang2010positivity,zhang2011maximum,zhang2012minimum,wang2012robust,zhang2017positivity} for compressible Euler equations and \cite{xing2010positivity} for shallow water equations.
Such a  limiter does not increase entropy for compressible Euler equations \cite{chen2017entropy,lin2023positivity}.

A simplified limiter for finite volume schemes without directly using reconstruction polynomial ${\bf p}_j(x)$ can also be constructed, similar to \Cref{sec:zhanglimiter-FV}, see \cite[Section 5]{zhang2011maximum} for details. 

In practice, one may want to preserve the invariant domain of the numerical solution at more points, such as Gauss quadrature points used in DG methods. It can be easily enforced by the same limiter \eqref{eq:ZSlimiter_system} at these points.

\subsubsection{Robust and efficient implementations}

The full algorithm flowchart follows similarly as in \Cref{sec:zhang-shu-alo}.  
Since the exact solutions to scalar conservation laws satisfy the maximum principle, it is not difficult to estimate the maximum wave speed $\alpha$ in \Cref{def:IDPflux} for scalar conservation laws.
However, for a hyperbolic system, the maximum wave speed may grow in time.
It is nontrivial to enforce the CFL needed for \Cref{thm:1Dsystem} in each time stage of a SSP Runge-Kutta method by computing a time step $\Delta t$ based only on information of $\bu^n$. 

On the other hand, the CFL condition in \Cref{thm:1Dsystem} is only a sufficient but not a necessary condition for achieving $\bar\bu^{n+1}_j\in G$ thus one convenient and efficient implementation is described by the following two steps.

First,  use  a SSP Runge-Kutta method with a high order finite volume or DG method in space using  an IDP numerical flux, e.g., the Lax-Friedrichs flux
    \[   \hat{\bff}(\bu^-_{j+\frac12}, \bu^+_{j+\frac12})=\frac{1}{2}[\bff(\bu^-_{j+\frac12})+\bff(\bu^+_{j+\frac12})-\alpha_{j+\frac12} (\bu^+_{j+\frac12}-\bu^-_{j+\frac12})], \alpha_{j+\frac12}=\max \left|\bff'(\bu^\pm_{j+\frac12})\right|.\]
    
Second, use any commonly used time step in a SSP Runge-Kutta method to evolve from time step $n$ to step $n+1$, with  the simple limiter \eqref{eq:ZSlimiter_system} used on each time stage. If $\bar{\bu}^{n+1}_j\notin G$ happens at time step $n+1$ or any inner time step of the  Runge-Kutta method, then it means that the CFL condition in \Cref{thm:1Dsystem} is not met at that time stage or time step, thus go back to time step $n$ to recompute with halved time step.
See \cite{wang2012robust,zhang2017positivity} for more details. 

Notice that the implementation of recomputing with halved time step whenever invariant domain is violated can be used for any numerical scheme, which however may result in an infinite loop of recomputing, e.g., a high order with any positive time step can violate bounds in \Cref{example:advection_upwind}. Only when  the simple limiter  \eqref{eq:ZSlimiter_system} is added to control point values at the proper locations, this is not an infinite loop
because \Cref{thm:1Dsystem} ensures $\bar{\bu}^{n+1}_j\in G$ when time step is small enough. 

\begin{remark}
The Zhang-Shu approach can be used with any finite volume and DG methods to construct high order IDP schemes. For finite difference schemes, this approach can still be used if the FD scheme is defined via a pseudo finite volume scheme. For example, the classical Jiang-Shu FD WENO scheme can be rendered positivity-preserving by this approach   \cite{zhang2012positivity,fan2021positivity}. 
\end{remark}

\subsection{Multi-dimensional extensions}
 
We summarize the two key and essential ingredients for extending from 1D results above to  multiple dimensions:

 \paragraph{A first order IDP flux or scheme} It is available in multiple dimensions as reviewed in \Cref{sec:firstorderIDP}.
 
\paragraph{A decomposition of a cell average into a convex combination of point values including point values on the cell boundary} In one dimension, the cell boundary values are simply $\bu^{\pm}_{j+\frac12}$. In multiple dimensions, these cell boundary values should be the quadrature point values used for computing numerical flux along cell boundaries.

With these two ingredients, it is possible to extend \Cref{thm:1Dsystem} to multiple dimensions, which will be reviewed in the next two subsections, and the point values can be controlled and corrected by a similar simple limiter like \eqref{eq:ZSlimiter_system}. 
The desired decomposition of a cell average into point values is achieved by 
Gauss--Lobatto quadrature in 1D. In multiple dimensions, it can be done by a quadrature with positive quadrature weights. On a given cell, such a quadrature may or may not exist. On a polygonal cell, such a quadrature can be constructed and is not unique.
In \cite{zhang2012maximum,lv2015entropy,chen2017entropy}, different choices of such quadrature rules were constructed on triangles and simplices, by
 from which specialized quadrature rules on polygons can be obtained by partitioning the polygon into a union of triangles \cite{MR3471184}. {\color{black}Quadrature rules on polygons with fewer points can be found in \cite{zbMATH05590395}.} 
See \cite{endeve2015bound} for the curvilinear elements. \Cref{fig:specialquad} shows the special quadrature in two dimensions used in \cite{zhang2010positivity,zhang2012maximum}, which is however not an optimal choice.  
\begin{figure}[htbp]
    \centering 
   \begin{tikzpicture}[scale=1]
  \draw[line width=1pt] (-2,0) -- (2,0);
  \fill[cyan] (-2,0) circle (3pt);
  \fill[cyan] (2,0) circle (3pt);
  \fill[red] (0,0) circle (3pt);
\end{tikzpicture}
\begin{tikzpicture}[scale=0.8]
  \draw[line width=1pt] (-2,-2) -- (-2,2);
  \draw[line width=1pt] (-2,-2) -- (2,-2);
  \draw[line width=1pt] (-2,2) -- (2,2);
  \draw[line width=1pt] (2,2) -- (2,-2);
  
  \foreach \x/\y in { -2/-1.55, -2/0, -2/1.55, 2/-1.55, 2/0, 2/1.55, 
                      -1.55/-2, 0/-2, 1.55/-2, -1.55/2, 0/2, 1.55/2 } {
    \fill[cyan] (\x,\y) circle (3pt);
  }
  \foreach \x/\y in { 0/-1.55, 0/0, 0/1.55, -1.55/0, 1.55/0 } {
    \fill[red] (\x,\y) circle (3pt);
  }
\end{tikzpicture}
\begin{tikzpicture}[scale=0.5]
  \draw[line width=1.2pt] (-4,-2) -- (4,-2);
  \draw[line width=1.2pt] (-4,-2) -- (0,4.93);
  \draw[line width=1.2pt] (4,-2) -- (0,4.93);

  \foreach \x/\y in { 0/-2, -3.1/-2, 3.1/-2,
                      -2/1.465, -3.45/-1.04, -0.45/4.15,
                      2/1.465, 3.45/-1.04, 0.45/4.15 } {
    \fill[cyan] (\x,\y) circle (5pt);
  }

  \foreach \x/\y in { -1/-0.2675, -3.275/-1.52, 1.325/1.075,
                      0/1.465, 0/-1.04, 0/4.15,
                      1/-0.2675, 3.275/-1.52, -1.325/1.075 } {
    \fill[red] (\x,\y) circle (5pt);
  }
\end{tikzpicture}
    \caption{One example of the special quadrature for quadratic polynomials. Left: 1D cell. Middle: 2D rectangle. Right: 2D triangle. For 1D, it is simply 3-point Gauss-Lobatto quadrature. The  points in cyan color are the Gauss quadrature points for computing numerical flux integrals in high order schemes, and red points and cyan points together form a special quadrature that is exact for quadratic polynomials with positive weight.}
    \label{fig:specialquad}
\end{figure}
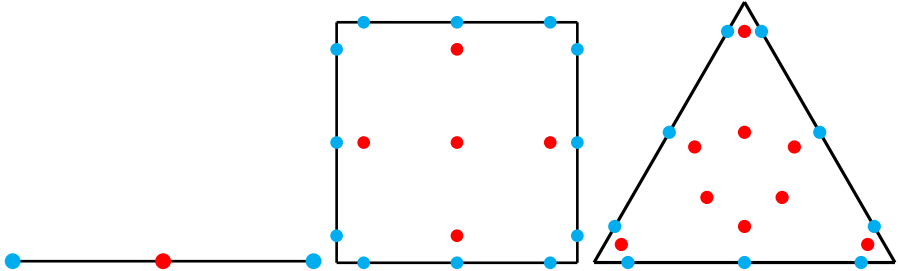
In  \Cref{fig:specialquad},
such a special quadrature consists of unnecessarily too many points 
 for integrating quadratic polynomials. We emphasize that such a quadrature is not used for computing any integrals, and we only need the following from this special quadrature: 
\begin{itemize}
\item We need its existence to establish weak IDP properties like \Cref{thm:1Dsystem}.
\item We need the quadrature points and weights for defining and implementing the limiter \eqref{eq:ZSlimiter_system}. 
\item The smallest weight on the cyan color points along the boundary also gives the CFL number in  \Cref{thm:1Dsystem}, which will be discussed later in this section. 
\end{itemize}
 
\begin{remark}
    If implementing the Zhang--Shu method using the simplified limiter as stated in \Cref{sec:zhanglimiter-FV}, then evaluation of DG or FV polynomials at the highly redundant red points in \Cref{fig:specialquad} can be avoided. Similar to \Cref{thm:1Dsystem}, only quadrature weights of the cyan points are needed to state a weak IDP theorem and define a simplified limiter. 
\end{remark}

\subsection{2D hyperbolic systems on rectangular meshes}

We now review the Zhang--Shu approach  for 2D hyperbolic system 
\begin{equation}\label{eq:2DSCL}
	\partial_t {\bf u} + \partial_x {\bf f}_1(u) + \partial_y {\bf f}_2(u) = {\bf 0}
\end{equation}
on rectangular meshes \cite{zhang2010positivity}, and some recent advances \cite{cui2023classic,cui2024optimal} on seeking optimal quadrature for IDP high order schemes.

We first review how to establish the weak IDP property for the updated cell averages. 
For \eqref{eq:2DSCL} with the forward Euler time discretization on a rectangular cell $I_{ij} := [x_{i-\frac{1}{2}}, x_{i+\frac{1}{2}}] \times [y_{j-\frac{1}{2}}, y_{j+\frac{1}{2}}]$,  
a finite volume scheme or the cell average of DG scheme 
 can be formulated as 
 \begin{equation}\label{eq:2Dscheme1}
 	\begin{aligned}
 \bar{\bf u}_{ij}^{n+1} &= \bar{\bf u}_{ij}^n - \frac{\Delta t}{\Delta x} \sum_{q=1}^Q \widetilde \omega_q  \left[ \hat{\bf f}_1 \big( {\bf u}_{i+\frac12,q}^{-}, {\bf u}_{i+\frac12,q}^{+}  \big) - \hat{\bf f}_1  \big( {\bf u}_{i-\frac12,q}^{-}, {\bf u}_{i-\frac12,q}^{+}  \big)  \right]
 \\
 & \qquad \quad 
 - \frac{\Delta t}{\Delta y} \sum_{q=1}^Q \widetilde \omega_q  \left[ \hat {\bf f}_2 \big( {\bf u}_{q,j+\frac12}^{-}, {\bf u}_{q,j+\frac12}^{+}  \big) - \hat {\bf f}_2 \big( {\bf u}_{q,j-\frac12}^{-}, {\bf u}_{q,j-\frac12}^{+}  \big) \right]
 \end{aligned}
 \end{equation}
with 
 \[
 \begin{aligned}
 	&
 {\bf u}_{i-\frac{1}{2},q}^{+} = {\bf p}_{ij}(x_{i-\frac{1}{2}}, \widetilde y_j^{(q)}), \quad
 {\bf u}_{i+\frac{1}{2},q}^{-} = {\bf p}_{ij}(x_{i+\frac{1}{2}}, \widetilde y_j^{(q)}), 
 \\
 &
 {\bf u}_{q,j-\frac{1}{2}}^{+} = {\bf p}_{ij}( \widetilde x_i^{(q)}, y_{j-\frac{1}{2}}), \quad
 {\bf u}_{q,j+\frac{1}{2}}^{-} = {\bf p}_{ij}( \widetilde x_i^{(q)}, y_{j+\frac{1}{2}}),
 \end{aligned}
 \]
where ${\bf p}_{ij}(x, y)$ is the approximate solution polynomial vector (reconstructed in finite volume methods or evolved in DG methods) on $I_{ij}$ at time level $n$ and its cell average over $I_{ij}$ equals $\bar{\bf u}_{ij}^n$. 
Here, $\{\widetilde x_i^{(q)}\}_{q=1}^{Q}$ and $\{\widetilde y_j^{(q)}\}_{q=1}^{Q}$ denote the nodes of a $Q$-point Gauss  quadrature of sufficiently high order accuracy in the intervals $[x_{i-\frac{1}{2}}, x_{i+\frac{1}{2}}]$ and $[y_{j-\frac{1}{2}}, y_{j+\frac{1}{2}}]$, respectively, with the normalized weights $\{ \widetilde \omega_q \}$ satisfying $\sum_{q=1}^{Q} \widetilde \omega_q = 1$. We use tildes to distinguish this Gauss quadrature from the Gauss–Lobatto quadrature introduced earlier; both rules will be employed below.

%
%

\subsubsection{Weak IDP property}

In order to obtain an IDP scheme, the numerical fluxes $\hat{\bff}_1$ and $\hat{\bff}_2$ in \eqref{eq:2Dscheme1} are taken as IDP fluxes,  with which the corresponding 1D three-point first order schemes are IDP, i.e., for any ${\bf u}_1, {\bf u}_2, {\bf u}_3 \in G$ it holds that
 \[
 {\bf u}_2 - \frac{\Delta t}{\Delta x} \left( \hat{\bf f}_1({\bf u}_2, {\bf u}_3) - \hat{\bf f}_1({\bf u}_1, {\bf u}_2) \right) \in G, \quad
 {\bf u}_2 - \frac{\Delta t}{\Delta y} \left( \hat{\bf f}_2({\bf u}_2, {\bf u}_3) - \hat{\bf f}_2({\bf u}_1, {\bf u}_2) \right) \in G
 \]
 under a suitable CFL condition $\max\{ \alpha_1 \Delta t / \Delta x, \alpha_2 \Delta t / \Delta y\} \leq c_0$, where $\alpha_1$ and $\alpha_2$ denote the maximum characteristic speeds in the $x$- and $y$-directions, and $c_0$ is the maximum allowable CFL number for the 1D first order schemes.

Similar to the role of the Gauss--Lobatto quadrature in the 1D case, in order to decompose the cell average $\bar {\bf u}_{ij}^n$ into a convex combination of some point values of ${\bf p}_{ij}$, we need a special quadrature on the cell $I_{ij}$ of the form 
	\begin{equation}\label{2Ddecomp}
		\begin{aligned}
 &  \frac{1}{\Delta x \Delta y}	\int_{I_{ij}} p(x,y) \, {\rm d}x {\rm d}y=  
	\sum_{q=1}^Q \widetilde \omega_{q}  \Big[  
	\omega_1^- p (x_{i-\frac{1}{2}}, \widetilde y_j^{(q)}) +
	\omega_1^+ p (x_{i+\frac{1}{2}}, \widetilde y_j^{(q)}) 
	\\
	& \qquad \qquad +
	\omega_2^- p ( \widetilde x_i^{(q)}, y_{j-\frac{1}{2}}) +
	\omega_2^+ p (  \widetilde  x_i^{(q)}, y_{j+\frac{1}{2}}) 
	\Big]
	+ \sum_{s=1}^S  \omega_s^* p (  x_s^*,  y_s^* ),
	\end{aligned}
\end{equation}
which should satisfy three requirements: 
	\begin{enumerate}[label=(\roman*)]
	\item  The quadrature \eqref{2Ddecomp} is exact for all polynomials $p \in \mathbb V$, where $\mathbb V$ is the approximate solution space, e.g., $\mathbb P^k$ or $\mathbb Q^k$;
	\item  The weights $\{{\omega}_1^\pm, {\omega}_2^\pm,  {\omega}_s^* \}$ are  positive and they sum to one;
    \item  The internal node set $\mathcal{I}_K = \left\{ ( x_s^*,  y_s^* ) \right\}_{s=1}^{S} \subset I_{ij}$.
\end{enumerate}

With such a quadrature, \Cref{thm:1Dsystem} can be easily extended to rectangular cells \cite{ZHANG20103091,zhang2010positivity,zhang2012maximum,cui2023classic,cui2024optimal}. We state one version of such an extension  in  \cite{cui2023classic} as follows.
\begin{theorem}[Weak IDP property on 2D rectangular mesh]\label{thm:CFL}
	If the solution polynomials $\{{\bf p}_{ij} \}$ satisfy 
	\begin{equation}\label{eq:2Dcond}
		{\bf p}_{ij} (x,y)\in G \qquad \forall (x,y)\in \mathbb S_{ij}, \quad \forall i,j,
	\end{equation}
	where $\mathbb S_{ij}$ denotes the set of all the quadrature points in \eqref{2Ddecomp}, 
	then the high order scheme \eqref{eq:2Dscheme1} preserves 
	$\bar {\bf u}_{ij}^{n+1}\in G$ under the CFL condition
	\begin{equation}\label{2Dhst_CFL_all}
		\Delta t \le c_0 \min \left\{ \frac{\omega_1^- \Delta x}{\alpha_1}, \frac{\omega_1^+ \Delta x}{\alpha_1} , \frac{\omega_2^- \Delta y}{\alpha_2}, \frac{\omega_2^+ \Delta y}{\alpha_2} \right\}.
	\end{equation}
\end{theorem}

\subsubsection{Quadrature and CFL}

The special 2D quadrature of the form \eqref{2Ddecomp} plays a critical role in constructing above 2D high order IDP schemes. Such quadrature rules are not unique. Below are several examples. 

\begin{example}[Zhang--Shu quadrature \cite{ZHANG20103091,zhang2010positivity}]
 Based on	the tensor product of the $L$-point Gauss--Lobatto quadrature  (with $L=\lceil \frac{k+3}{2}\rceil$) and the $Q$-point Gauss  quadrature, 
 Zhang and Shu \cite{ZHANG20103091,zhang2010positivity} proposed the following  quadrature:
\begin{equation}\label{eq:ZSquadrature}
	\begin{aligned}
   \frac{1}{\Delta x \Delta y}	\int_{I_{ij}} p(x,y) \, {\rm d}x {\rm d}y & =  
\sum_{q=1}^Q \widetilde \omega_{q}  \Big[  
\kappa_1  \omega_1 p (x_{i-\frac{1}{2}}, \widetilde  y_j^{(q)}) +
\kappa_1  \omega_1 p (x_{i+\frac{1}{2}}, \widetilde  y_j^{(q)}) 
\\
& \qquad +
\kappa_2  \omega_1 p ( \widetilde  x_i^{(q)},  y_{j-\frac{1}{2}}) +
\kappa_2  \omega_1 p ( \widetilde  x_i^{(q)}, y_{j+\frac{1}{2}}) 
\Big]
\\
 & \quad	+
	\sum_{\mu=2}^{L-1} \sum_{q=1}^{Q}  \omega_\mu \widetilde   \omega_q 
	\left[ 
	\kappa_1  p\left( x_{i}^{(\mu)}, \widetilde  y_{j}^{(q)} \right) + 
	\kappa_2  p\left( \widetilde  x_{i}^{(q)}, y_{j}^{(\mu)}  \right)
	\right],
	\end{aligned}
\end{equation}
where  $\{ x_i^{(\mu)}\}_{\mu=1}^{L}$ and $\{ y_j^{(\mu)}\}_{\mu=1}^{L}$ denote the nodes of the $L$-point Gauss-Lobatto quadrature in the intervals $[x_{i-\frac{1}{2}}, x_{i+\frac{1}{2}}]$ and $[y_{j-\frac{1}{2}}, y_{j+\frac{1}{2}}]$, respectively, with the weights $\{ \omega_\mu \}$ satisfying $\sum_{\mu=1}^{L} \omega_\mu = 1$, and 
\begin{equation*}
	\kappa_1 := \frac{{\alpha_{1}}/{\Delta x}}{{
			\alpha_1}/{\Delta x}+{\alpha_2}/{\Delta y}}, \qquad
	\kappa_2 := \frac{{\alpha_{2}}/{\Delta y}}{{\alpha_1}/{\Delta x}+{\alpha_2}/{\Delta y}}.
\end{equation*}
The set of quadrature nodes in  the   quadrature rule \eqref{eq:ZSquadrature} can be expressed as 
$$
\mathbb S_{ij}^{\tt ZS} =  \left(  {\mathbb S}_i^x \otimes  \widetilde { \mathbb S}^y_j  \right) \cup 
\left(  \widetilde {\mathbb S}_i^x \otimes   { \mathbb S}^y_j \right), 
$$
where  
$\widetilde  {\mathbb S}_i^x:=\{  \widetilde  x_{i}^{(q)} \}_{q=1}^Q,$ $\widetilde {\mathbb S}^y_j:=\{ \widetilde  y_{j}^{(q)} \}_{q=1}^Q,$ 
${\mathbb S}_i^x:=\{ x_{i}^{(\mu)} \}_{\mu=1}^L,$ and ${ \mathbb S}^y_j:=\{  y_{j}^{(\mu)} \}_{\mu=1}^L$ are the Gauss quadrature nodes for flux integration and the Gauss--Lobatto quadrature nodes for convex decomposition of cell averages, respectively.
The quadrature nodes of the Zhang--Shu quadrature \eqref{eq:ZSquadrature} are illustrated in \Cref{fig:374} for $\mathbb P^2$ and $\mathbb P^3$ based methods. 
\end{example}

\begin{figure}
	\centering
	\begin{subfigure}{0.48\textwidth}
		\centering
		\includegraphics[width=0.9\textwidth]{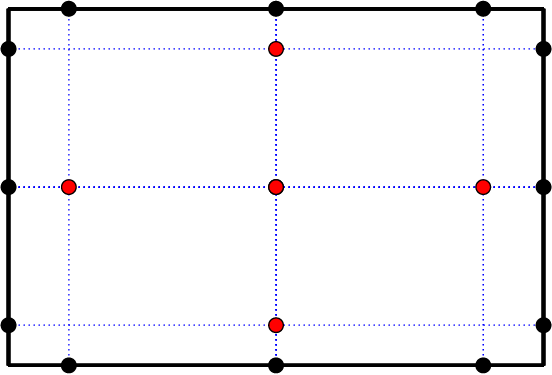}
		\caption{Zhang--Shu quadrature \eqref{eq:ZSquadrature} for $\mathbb{P}^2$}
	\end{subfigure}
	\hfill
	\begin{subfigure}{0.48\textwidth}
		\centering
		\includegraphics[width=0.9\textwidth]{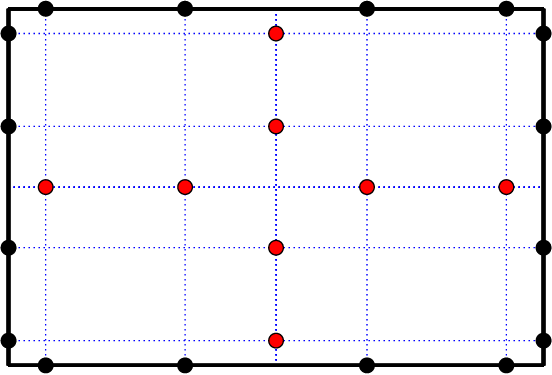}
		\caption{Zhang--Shu quadrature \eqref{eq:ZSquadrature} for $\mathbb{P}^3$}
	\end{subfigure}
	\\
	\begin{subfigure}{0.48\textwidth}
		\centering
		\includegraphics[width=0.9\textwidth]{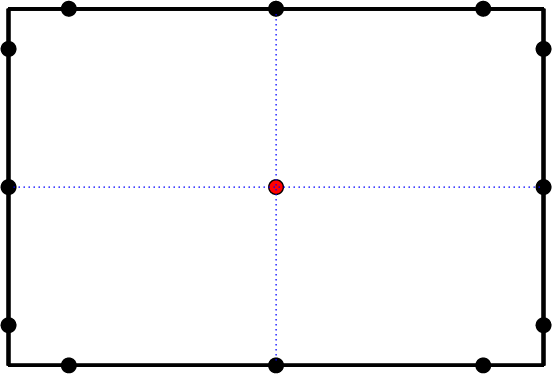}
		\caption{Optimal quadrature \eqref{eq:QR2D} for $\mathbb{P}^2$}
	\end{subfigure}
	\hfill
	\begin{subfigure}{0.48\textwidth}
		\centering
		\includegraphics[width=0.9\textwidth]{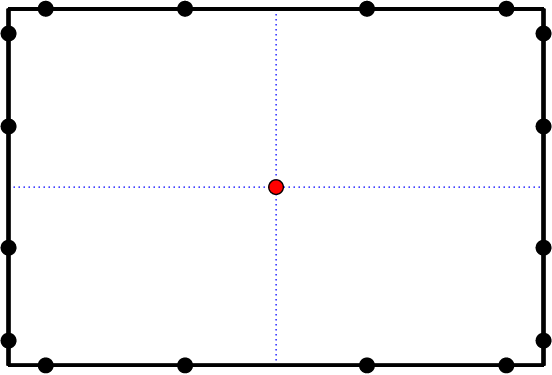}
		\caption{Optimal quadrature \eqref{eq:QR2D} for $\mathbb{P}^3$}
	\end{subfigure}
	\caption{Boundary nodes (black) and internal nodes (red) of the special 2D quadrature on a rectangular cell $I_{ij}$, for the $\mathbb{P}^2$- and $\mathbb{P}^3$-based methods, in the case of $\frac{\dx}{\alpha_1} = \frac{\dy}{\alpha_2}$. }
	\label{fig:374}  
\end{figure}

As a corollary of 
Theorem \ref{thm:CFL}, we have the following result.

\begin{theorem}[Weak IDP via Zhang--Shu quadrature]
	If the solution polynomials $\{{\bf p}_{ij} \}$ satisfy 
	\eqref{eq:2Dcond} with $\mathbb S_{ij} = \mathbb S_{ij}^{\tt ZS}$, 
	then the high order scheme \eqref{eq:2Dscheme1} preserves 
	$\bar {\bf u}_{ij}^{n+1}\in G$ under the CFL condition
	\begin{equation}\label{eq:CFL-ZS}
	 	\left( \frac{\alpha_1}{\Delta x}+\frac{\alpha_2}{\Delta y} \right) \Delta t \le  \omega_1\,c_0 
		=\frac{c_0}{ L(L-1) } \quad \mbox{with} \quad L=\left \lceil \frac{k+3}{2} \right\rceil.
	\end{equation}
\end{theorem}

 \begin{remark}
 	A similar quadrature was used by Jiang and Liu in \cite{jiang2018invariant}. 
 	The only difference from the Zhang--Shu quadrature is that $\kappa_1=\kappa_2=\frac12$. 
 	In this case, the CFL condition for the weak IDP property becomes 
 	\begin{equation}\label{eq:JL-CFL}
 		2 \max \left\{\frac{\alpha_1}{\Delta x},\frac{\alpha_2}{\Delta y}\right\}\Delta t \le  \omega_1c_0 ,
 	\end{equation}
 	which is more restrictive than \eqref{eq:CFL-ZS}. 
 \end{remark}

Since such special quadrature rules are not unique, it is natural to seek the {\em optimal quadrature} such that the resulting IDP CFL condition \eqref{2Dhst_CFL_all} is the mildest. 
It was proven in \cite{cui2024optimal} that the Zhang--Shu quadrature \eqref{eq:ZSquadrature} is optimal for high order schemes based on the $\mathbb Q^k$ space with any positive integer $k$. 
However, for the $\mathbb P^k$-based methods, the quadrature \eqref{eq:ZSquadrature} is generally not optimal \cite{cui2023classic}. The optimal 2D quadrature rules for $\mathbb P^k$-based methods were systematically studied in \cite{cui2023classic,cui2024optimal}.

\begin{example}[Optimal quadrature \cite{cui2023classic,cui2024optimal}]
	For the $\mathbb{P}^2$- and $\mathbb P^3$-based methods, the 
optimal quadrature for weak IDP property is given by 
	\begin{equation}\label{eq:QR2D}
	\begin{aligned}
		&  \frac{1}{\Delta x \Delta y}	\int_{I_{ij}} p(x,y) \, {\rm d}x {\rm d}y=  
		 \frac{\mu_1}{2}\sum_{q=1}^Q \widetilde \omega_{q}  \Big[  
		 p (x_{i-\frac{1}{2}}, \widetilde y_j^{(q)}) +
		 p (x_{i+\frac{1}{2}}, \widetilde y_j^{(q)}) \Big]
		\\
		& \qquad \quad + \frac{\mu_2}{2}\sum_{q=1}^Q \widetilde \omega_{q} \Big[
		 p ( \widetilde x_i^{(q)}, y_{j-\frac{1}{2}}) +
		 p ( \widetilde x_i^{(q)}, y_{j+\frac{1}{2}}) 
		\Big]
		+   \omega^* \sum_{s=1}^2  p ( x_s^*,  y_s^* )
	\end{aligned}
\end{equation}
with the internal nodes
\begin{equation}\label{eq:226}
	 \left\{ x_s^*,  y_s^* \right\}_{s=1}^2 =
	\begin{cases}
		\; \left( 
		x_i, 
		y_j \pm \frac{\dy}{2\sqrt{3}}\sqrt{\frac{\phi_*-\phi_2}{\phi_*}}
		\right),
		& \text{if} \; \phi_1 \ge \phi_2, \\
		\; \left(
		x_i \pm \frac{\dx}{2\sqrt{3}}\sqrt{\frac{\phi_*-\phi_1}{\phi_*}} ,
		y_j
		\right),
		& \text{if} \; \phi_1 < \phi_2,
	\end{cases}
\end{equation}
where
\begin{equation*}
	\begin{aligned}
	&\phi_1 = \frac{ \alpha_1}{\dx}, \quad
	\phi_2 = \frac{\alpha_2}{\dy}, \quad
	\phi^* = \max\{\phi_1,\phi_2\}, \\ 
	&
	\psi = \phi_1+\phi_2+2\phi^*, 
	\quad
	\mu_1 = \frac{\phi_1}{\psi}, \quad
	\mu_2 = \frac{\phi_2}{\psi}, \quad
	 \omega^* = \frac{\phi^*}{\psi}.
	\end{aligned}
\end{equation*}
There are only two internal nodes, which merge to a single node  $(x_i,y_j)$ in case of $\phi_1=\phi_2$. The quadrature nodes of the optimal quadrature \eqref{eq:QR2D} are illustrated in \Cref{fig:374} for $\mathbb P^2$- and $\mathbb P^3$-based methods. 
For  the optimal quadrature for the $\mathbb{P}^k$-based methods with higher $k\ge 4$, we refer the readers to \cite{cui2024optimal}. 
\end{example}

\begin{theorem}[Weak IDP via optimal quadrature \cite{cui2023classic,cui2024optimal}]
	If the solution polynomials $\{{\bf p}_{ij} \}$ are in the space $\mathbb{P}^2$ or $\mathbb P^3$ and satisfy 
	\eqref{eq:2Dcond} with $\mathbb S_{ij}$ being the set of all the nodes in \eqref{eq:QR2D}, 
	then the high order scheme \eqref{eq:2Dscheme1} preserves 
	$\bar {\bf u}_{ij}^{n+1}\in G$ under the CFL condition
\begin{equation}\label{eq:CFL2D}
	\left(   2\frac{\alpha_1}{\dx}+2\frac{\alpha_2}{\dy}+ 4 \max\left\{ 
	\frac{\alpha_1}{\dx},\frac{\alpha_2}{\dy}\right\} \right) \Delta t \le c_0.
\end{equation}
\end{theorem}

\Cref{tab:330} lists a comparison of the IDP CFL conditions and the internal nodes for  different 2D special quadrature rules.  

\begin{table}[htbp]
	\centering
	\caption{IDP CFL conditions and the numbers of internal nodes of different 2D quadrature  for the $\mathbb{P}^2$- and $\mathbb{P}^3$-based methods.}\label{tab:330}
	\renewcommand\arraystretch{1.7}
	\begin{tabular}{ lrrrr} 
		\toprule[1.5pt]
		& IDP CFL & IDP CFL  & \# Nodes &  \# Nodes \\
		& general case & $\frac{\dx}{\alpha_1} = \frac{\dy}{\alpha_2} = h$ &    $\mathbb{P}^2$ &   $\mathbb{P}^3$  \\ 
		
		\midrule[1.5pt]
		
		Optimal \cite{cui2023classic} & 
		\eqref{eq:CFL2D} 
		& $\Delta t \le \frac{c_0}{8} h$ & 1 $\sim$ 2 & 1 $\sim$ 2 \\ 
		Zhang-Shu \cite{zhang2010positivity} & \eqref{eq:CFL-ZS} & $\Delta t \le \frac{c_0}{12} h$ & 5 & 8 \\ 
		Jiang-Liu \cite{jiang2018invariant} & \eqref{eq:JL-CFL} & $\Delta t \le \frac{c_0}{12}h$ & 5 & 8\\ 
		
		\bottomrule[1.5pt]
	\end{tabular}
\end{table}

The simple limiter \eqref{eq:ZSlimiter_system} can be easily extended to multiple dimensions as follows \begin{equation*} 
	\tilde {\bf p}_{ij}(x,y) = \theta( {\bf p}_{ij}(x,y) - \bar {\bf u}_{ij}^n) + \bar {\bf u}_{ij}^n,
\end{equation*}
where $\theta$ is computed via either point values in the special quadrature $\mathbb S_{ij}$ or only the boundary quadrature point values in a simplified fashion as in \Cref{sec:zhanglimiter-FV}. See \cite{zhang2010positivity,zhang2012maximum} for details.

\subsection{Extensions to unstructured triangular meshes}

For simplicity, we focus only on the special quadrature and extensions of \Cref{thm:1Dsystem}.

\subsubsection{The special quadrature for convex decompositions of the cell average} 

Let $K$ denote an arbitrary triangular cell  with edge length denoted by $l_{K}^{(i)}$ ($i=1,2,3$). Let 
$(x^{i,q}_K,y^{i,q}_K)$ be the $q$th node of the $Q$-point Gauss quadrature on the $i$th edge $e_K^{(i)}$ and $\widetilde \omega_q$ be the weight. 
The first task is to find a special 2D quadrature on $K$:
	\begin{equation}\label{eq:980}
	\begin{aligned}
		\frac{1}{|K|}\iint_K p(x,y) ~ \textrm{d} x \textrm{d} y
		& =
		\sum_{i=1}^3
		\frac{w_i}{l_K^{(i)}}
		\int_{e_K^{(i)}} p(x,y) ~ \textrm{ds}
		+
		\sum_{s = 1}^{S}
		\omega_s^* p( x_s^*,y_s^*)
		\\
		& = 		\sum_{i=1}^3 \sum_{q=1}^Q
		w_i \widetilde \omega_q
		p( x_{K}^{i,q} ) + 
		\sum_{s = 1}^{S}
		 \omega_s^* p( x_s^*,y_s^*)
	\end{aligned}
\end{equation}
such that 
	\begin{enumerate}[label=(\roman*)]
	\item  The quadrature \eqref{eq:980} holds exactly for all $p(x,y) \in \mathbb{P}^k$;
	\item  The edge weights $\{w_i\}_{i=1}^3$ and the internal node weights $\{\omega_s^*\}_{s=1}^{S}$ are all positive, and $	\sum_{i=1}^3 \sum_{q=1}^Q
	w_i \widetilde \omega_q+ \sum_{s = 1}^{S}
	 \omega_s^*=1$;
	\item  The internal node set $\mathcal{I}_K = \left\{ ( x_s^*, y_s^* ) \right\}_{s=1}^{S} \subset K$.
\end{enumerate}

\begin{example}[Zhang--Xia--Shu quadrature \cite{zhang2012maximum}]
	Zhang, Xia, and Shu proposed in \cite{zhang2012maximum} the following quadrature for the $\mathbb{P}^k$ space on a triangular cell $K$:
	\begin{equation}\label{eq:ZXSquadrature}
		\begin{aligned}
	\frac{1}{|K|}\iint_K p(x,y) ~ \textrm{d} x \textrm{d} y
	& =
	\sum_{i=1}^3
	\frac{2 \,  \omega_1}{3 l^{(i)}_K}
	\int_{e_K^{(i)}} p(x,y) ~ \textrm{ds}
	+
	\sum_{s = 1}^{S^{\tt ZXS}}
	\omega^{\tt ZXS}_s p( x^{\tt ZXS}_s,y^{\tt ZXS}_s),
	\\
	& = 
	\sum_{i=1}^3 \sum_{q=1}^Q
	\frac{2 \,  \omega_1 \widetilde \omega_q  }{3  } 
	p( x_{K}^{i,q} )
	+
	\sum_{s = 1}^{S^{\tt ZXS}}
	\omega^{\tt ZXS}_s p( x^{\tt ZXS}_s, y^{\tt ZXS}_s),
	\end{aligned}
	\end{equation}
	where $\omega_1=\frac{1}{L(L-1)}$ is the first Gauss--Lobatto quadrature weight with $L=\lceil \frac{k+3}{2} \rceil$, $\left\{( x^{\tt ZXS}_s, y^{\tt ZXS}_s)\right\}$ denote the coordinates of $S^{\tt ZXS} = 3\lceil\frac{k-1}{2}\rceil(k+1)$ internal nodes (see \cite{zhang2012maximum} for more details). This quadrature was constructed by the average of three different mappings from the Zhang--Shu quadrature \eqref{eq:ZSquadrature} on a rectangular cell to the triangular cell $K$. In practice, one may want to use different 
   Gauss quadrature for each edge in the high order schemes. For instance, 
    see \Cref{fig:specialquad} for the Zhang--Xia--Shu quadrature for $\mathbb{P}^2$ with $3$-point Gauss quadrature for each edge, which is in general enough for the $\mathbb{P}^2$ DG method \cite{zhang2012maximum}. 
In
\cite{MR3095289}, Zhang-Xia-Shu special quadrature \eqref{eq:ZXSquadrature} with $4$-point Gauss quadrature for each edge of a triangle is used for the $\mathbb{P}^2$ DG method in order to achieve well-balanced property for shallow water equations.
\end{example}

\begin{example}[Chen--Shu quadrature]\label{rem:ChenShu}
	In \cite{chen2017entropy}, Chen and Shu used another series of quadrature rules on triangular cells  for entropy-stable DG methods, constructed by the quadrature method \cite{MR2493559},  which was also used in \cite{lv2015entropy}. 
	 These quadrature rules are also qualified for designing $\mathbb{P}^k$-based IDP schemes on triangular cells; see \cite[Appendix C]{chen2017entropy} for further details on these quadrature rules, which is much less redundant than  Zhang--Xia--Shu quadrature, as shown in \Cref{fig:ZXS}. 
\end{example}

	 \begin{figure}
        \centering
       \begin{tikzpicture}[scale=0.5]
  \draw[line width=1.2pt] (-4,-2) -- (4,-2);
  \draw[line width=1.2pt] (-4,-2) -- (0,4.93);
  \draw[line width=1.2pt] (4,-2) -- (0,4.93);

  \foreach \x/\y in { 0/-2, -3.1/-2, 3.1/-2,
                      -2/1.465, -3.45/-1.04, -0.45/4.15,
                      2/1.465, 3.45/-1.04, 0.45/4.15 } {
    \fill[cyan] (\x,\y) circle (5pt);
  }

  \foreach \x/\y in { -1/-0.2675, -3.275/-1.52, 1.325/1.075,
                      0/1.465, 0/-1.04, 0/4.15,
                      1/-0.2675, 3.275/-1.52, -1.325/1.075 } {
    \fill[red] (\x,\y) circle (5pt);
  }
\end{tikzpicture}
\qquad \qquad 
\begin{tikzpicture}[scale=0.5]
  \draw[line width=1.2pt] (-4,-2) -- (4,-2);
  \draw[line width=1.2pt] (-4,-2) -- (0,4.93);
  \draw[line width=1.2pt] (4,-2) -- (0,4.93);

  \foreach \x/\y in { 0/-2, -3.1/-2, 3.1/-2,
                      -2/1.465, -3.45/-1.04, -0.45/4.15,
                      2/1.465, 3.45/-1.04, 0.45/4.15 } {
    \fill[cyan] (\x,\y) circle (5pt);
  }

  \foreach \x/\y in { 0/0.31 } {
    \fill[red] (\x,\y) circle (5pt);
  }
\end{tikzpicture}
        \caption{Two special quadratures  for $\mathbb P^2$ on an equilateral triangle with area $1$. Left is  Zhang--Xia--Shu quadrature, with weights for three cyan points of each edge being $(\frac{5}{162},\frac{8}{162},\frac{5}{162})$.
        Right is  Chen--Shu quadrature \cite[Tale C.2 (b)]{chen2017entropy}  which is much less redundant, with weights for three cyan points of each edge being $(\frac{1}{24},\frac{1}{10},\frac{1}{24})$. }
        \label{fig:ZXS} 
    \end{figure}
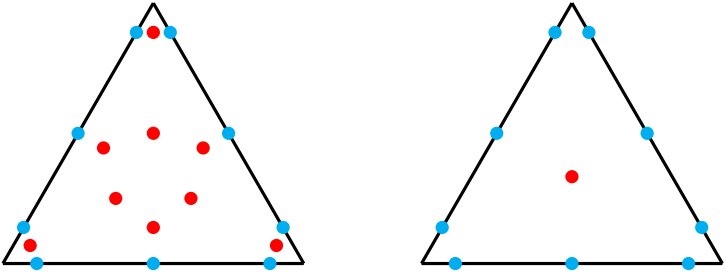

\begin{remark}
	Neither  Zhang--Xia--Shu quadrature nor Chen--Shu quadrature is optimal for IDP studies in general, as demonstrated in \cite{ding2025robust}, where the optimal quadrature rules for $\mathbb P^1$- and $\mathbb P^2$-based IDP schemes were found. 
    We emphasize that the interior nodes (red nodes in \Cref{fig:ZXS}) in the special quadrature can be avoided in the final implementation, similar to \Cref{sec:zhanglimiter-FV}; see  \cite{zhang2011maximum}.
\end{remark}

\subsubsection{The weak IDP property on triangular meshes}

The cell average scheme of a high order finite volume or DG scheme with forward Euler time discretization  on a triangular cell $K$ is given by
\begin{equation}\label{eq:1033}
	\bar{\bf u}_{K}^{n+1}  
	= 
	\bar{\bf u}_{K}^{n} 
	- 
	\frac{\Delta t}{2 |K|} 
	\sum_{i=1}^{3}  l_K^{(i)} 
	\sum_{q = 1}^{Q} \widetilde \omega_q  
	\hat {\bf f}({\bf u}^{\rm int}_{i,q},{\bf u}^{\rm ext}_{i,q},{\bf n}_K^{i}),
\end{equation}
where $|K|$ is the area of $K$, $l_K^{(i)}$ stands for the length of edge $e_K^{(i)}$ (the $i$-th edge of element $K$), $\mathbf{n}_{K}^{i}=\big(n_{K}^{i,1},n_{K}^{i,2}\big)$ is the outward unit normal vector of the edge $e_K^{(i)}$, and 
$${\bf u}^{\rm int}_{i,q}:={\bf u}^{n}_h(x^{i,q}_K,y^{i,q}_K)|_K, \qquad {\bf u}^{\rm ext}_{i,q}:={\bf u}^{n}_h(x^{i,q}_K,y^{i,q}_K) \big|_{K^{i}}$$
are approximations to edge values from interior and exterior of $K$ respectively,
with $K^{i}$ denoting the neighboring cell that shares the edge $e_K^{(i)}$ with $K$, 
$(x^{i,q}_K,y^{i,q}_K)$ being the $q$th node of the $Q$-point Gauss quadrature on $e_K^{(i)}$, 
and ${\bf u}^{n}_h$ denoting the piecewise polynomial solution either reconstructed in finite volume or evolved in DG methods.

Using the Zhang--Xia--Shu quadrature \eqref{eq:ZXSquadrature}, one obtains the following convex decomposition for the cell averages:
\begin{equation}\label{eq:ucellave}
\bar{\bf u}_{K}^{n}  =  \sum_{i=1}^3
\frac{2 \,  \omega_1}{3 }
\sum_{q = 1}^{Q} \widetilde \omega_q  
{\bf u}^{\rm int}_{i,q} 
+
\sum_{s = 1}^{S^{\tt ZXS}}
\omega^{\tt ZXS}_s {\bf u}_s^{\tt ZXS}. 
\end{equation}
Substituting  \eqref{eq:ucellave} into \eqref{eq:1033}, one can reformulate the scheme \eqref{eq:1033} as follows (cf.~\cite{zhang2012maximum}): 
\begin{equation}\label{eq:edefd}
		\bar{\bf u}_{K}^{n+1}  
	=  \sum_{s = 1}^{S^{\tt ZXS}}
	\omega^{\tt ZXS}_s {\bf u}_s^{\tt ZXS} +  \frac{2 \,  \omega_1}{3 }
	\sum_{q = 1}^{Q}   
 \widetilde \omega_q   \Big( {\bf \Pi}_{1,q} +  {\bf \Pi}_{2,q} +  {\bf \Pi}_{3,q} \Big), 
\end{equation}
where 
\begin{small}
\begin{align*}
		{\bf \Pi}_{1,q}&= \mathbf{u}^{\text{int}}_{2,q} 
	- \frac{3 \Delta t}{2 {\omega}_1 |K|} 
	\left[ \hat{\mathbf{f}}(\mathbf{u}^{\text{int}}_{2,q}, \mathbf{u}^{\text{int}}_{1,q}, {\bf n}_K^1) l_K^{(1)} 
	+ \hat{\mathbf{f}}(\mathbf{u}^{\text{int}}_{2,q}, \mathbf{u}^{\text{ext}}_{2,q}, {\bf n}_K^2) l_K^{(2)}  + \hat{\mathbf{f}}(\mathbf{u}^{\text{int}}_{2,q}, \mathbf{u}^{\text{int}}_{3,q}, {\bf n}_K^3) l_K^{(3)} \right], \\
	{\bf \Pi}_{2,q} &= \mathbf{u}^{\text{int}}_{1,q} 
	- \frac{3 \Delta t}{2 {\omega}_1 |K|} 
	\left[ \hat{\mathbf{f}}(\mathbf{u}^{\text{int}}_{1,q}, \mathbf{u}^{\text{ext}}_{1,q}, {\bf n}_K^1) l_K^{(1)} 
	+ \hat{\mathbf{f}}(\mathbf{u}^{\text{int}}_{1,q}, \mathbf{u}^{\text{int}}_{2,q}, -{\bf n}_K^1) l_K^{(1)} \right], \\
	{\bf \Pi}_{3,q} &= \mathbf{u}^{\text{int}}_{3,q} 
	- \frac{3 \Delta t}{2 {\omega}_1 |K|} 
	\left[ \hat{\mathbf{f}}(\mathbf{u}^{\text{int}}_{3,q}, \mathbf{u}^{\text{int}}_{2,q}, -{\bf n}_K^3) l_K^{(3)}
	+ \hat{\mathbf{f}}(\mathbf{u}^{\text{int}}_{3,q}, \mathbf{u}^{\text{ext}}_{3,q}, {\bf n}_K^3) l_K^{(3)} \right].
\end{align*}
\end{small}

As an example, consider the Lax--Friedrichs flux: 
$$
\hat {\bf f}({\bf u}^{\rm int}_{i,q},{\bf u}^{\rm ext}_{i,q},{\bf n}_K^{i}) = \frac12 
\Big(  {\bf f}({\bf u}^{\rm int}_{i,q}) \cdot {\bf n}_K^{i}  + {\bf f}({\bf u}^{\rm ext}_{i,q}) \cdot {\bf n}_K^{i} +  \alpha {\bf u}^{\rm int}_{i,q} - \alpha {\bf u}^{\rm ext}_{i,q}
   \Big),
$$
which is IDP if the CFL number is less than or equal to one and 
\begin{equation}\label{eq:alphaLF}
\alpha^{\rm LF} :=
        \max_{{\bm x}\in {e}^{(i)}_K, i \in \{1,2,3\}, K \in \mathcal{T}_h} 
        \tilde\alpha({\bm u}_h({\bm x}),{\bm n}_K^{i})
\end{equation}      
Note that ${\bf \Pi}_{1,q}$ is formally a first order IDP scheme on the triangular cell $K$ under 
the CFL condition
\begin{align}\label{eq:ZS-tri}
\alpha \frac{\Delta t}{|K|} \sum_{i=1}^3 l_K^{(i)} \le \frac23 {\omega}_1, 
\end{align}
and ${\bf \Pi}_{2,q}$ and ${\bf \Pi}_{3,q}$ are formally one-dimensional IDP schemes under the CFL conditions $\frac{3 \Delta t}{2 {\omega}_1 |K|} l_K^{(1)} \le 1$ 
and $\frac{3 \Delta t}{2 {\omega}_1 |K|} l_K^{(3)} \le 1$. 
Therefore, if 
$$
{\bf u}^{\rm int}_{i,q} \in G, ~~  {\bf u}^{\rm ext}_{i,q} \in G  \quad \forall i,q, 
$$
then ${\bf \Pi}_{i,q} \in G$ for all $i$ and $q$ under the CFL condition \eqref{eq:ZS-tri}. 
 
As observed from \eqref{eq:edefd}, 
$\bar{\bf u}_{K}^{n+1}$ is a convex combination of ${\bf u}_s^{\tt ZXS}$ and ${\bf \Pi}_{i,q}$. 
Thanks to the convexity of $G$, we obtain the following theorem \cite{zhang2012maximum}.  

\begin{theorem}[IDP via Zhang--Xia--Shu quadrature]
	If 
	$
	{\bf u}_h^n(x,y) \in G ~\forall (x,y) \in \mathbb S_K^{\tt ZXS},~\forall K, 
	$
	where $\mathbb S_K^{\tt ZXS}$ denotes the set of  quadrature nodes on cell $K$ in \eqref{eq:ZXSquadrature}, 
	then the scheme \eqref{eq:1033} preserves  $\bar{\bf u}_{K}^{n+1}\in G$ under the 
	CFL condition \eqref{eq:ZS-tri}. 
\end{theorem}

  With the weak IDP property, a simple limiter similar to the one presented in \Cref{sec:limiter} can be designed, and  
the extension of \Cref{thm:limiter_acc} to triangles or any cells is straightforward using \cite[Lemma 8]{zhang2017positivity}.

\subsubsection{IDP with larger CFL via optimal quadrature and GQL}\label{sec:IDPopt}

This subsection introduces 
the improvements of the theoretical IDP CFL condition by seeking an optimal quadrature \cite{ding2025robust} 
and the GQL approach \cite{wu2023geometric} under  
\Cref{prop:wLFS-2D}, which is weaker than \Cref{prop:LFS-2D}.

In order to precisely define the optimality of the special quadrature, 
we first show the IDP result \cite{ding2025robust} with an arbitrarily feasible quadrature 
\eqref{eq:980}.

\begin{theorem}[IDP via general feasible quadrature] \label{thm:1008}
	If the solution ${\bf u}_h^n$ satisfies 
	\begin{equation}\label{eq:1526}
	    {\bf u}_h^n(x,y) \in G \qquad \forall (x,y) \in \mathbb S_K,~~\forall K, 
	\end{equation}
	where $\mathbb S_K$ denotes the set of all the nodes in the quadrature \eqref{eq:980} on a triangle $K$, 
	then the scheme \eqref{eq:1033} preserves  $\bar{\bf u}_{K}^{n+1}\in G$ under the 
	CFL condition 
        \begin{equation}\label{eq:1009}
        \alpha \frac{\Delta t}{|K|} \le \mathcal{C}_{\tt IDP}:= \min 
        \left\{ 
        \frac{w_1}{l^{(1)}_K},
        \frac{w_2}{l^{(2)}_K},
        \frac{w_3}{l^{(3)}_K}
        \right\}. 
    \end{equation}
\end{theorem}

\begin{proof}
Using the equivalent linear GQL representation \eqref{eq:1200}
 of $G$, 
 we derive for any ${\bf u}^\star \in \mathcal{S}_j$ and $j \in \mathbb{I}\cup\hat{\mathbb{I}}$ that 
 	\begin{align*}
		&(\bar{\bf u}^{n+1}_K-{\bf u}^\star)\cdot \mathbf{n}^\star_j
		\\
		\stackrel{\eqref{eq:980}}{=} & 
		\sum_{i = 1}^3 l^{(i)}_K
		\sum_{q = 1}^Q \widetilde \omega_{q}
		\left[
			\left(\frac{w_i}{l_K^{(i)} }-\frac{\alpha\dt}{2|K|}\right)
			({\bf u}^{\rm int}_{i,q} - {\bf u}^\star) \cdot \mathbf{n}^\star_j
			-
			\frac{\dt}{2|K|}
			\left({\bf f}({\bf u}^{\rm int}_{i,q})\cdot{\bf n}_K^{i}\right)\cdot \mathbf{n}^\star_j
		\right]\\
		& +
		\sum_{i = 1}^3 l^{(i)}_K
		\sum_{q = 1}^Q \widetilde \omega_q
		\left[
		\frac{\alpha\dt}{2|K|}
		({\bf u}^{\rm ext}_{i,q} - {\bf u}^\star) \cdot \mathbf{n}^\star_j
		-
		\frac{\dt}{2|K|}
		\left({\bf f}({\bf u}^{\rm ext}_{i,q})\cdot{\bf n}_K^{i}\right)\cdot \mathbf{n}^\star_j
		\right]
        \\
		& + \sum_{s = 1}^{S}
		  \omega_s^* 
		\left( {\bf u}_h^n ( x_s^*, y_s^*) -{\bf u}^\star\right)\cdot \mathbf{n}^\star_j
		\\
		\stackrel{\eqref{eq:1526}, \eqref{eq:1009}}{\ge} & 
		\sum_{i = 1}^3 l^{(i)}_K
		\sum_{q = 1}^Q \widetilde \omega_q
		\left[
		\frac{\alpha\dt}{2|K|}
		({\bf u}^{\rm int}_{i,q} - {\bf u}^\star) \cdot \mathbf{n}^\star_j
		-
		\frac{\dt}{2|K|}
		\left({\bf f}({\bf u}^{\rm int}_{i,q})\cdot{\bf n}_K^{i}\right)\cdot \mathbf{n}^\star_j
		\right]\\
		& +
		\sum_{i = 1}^3 l^{(i)}_K
		\sum_{q = 1}^Q \widetilde \omega_{q}
		\left[
		\frac{\alpha\dt}{2|K|}
		({\bf u}^{\rm ext}_{i,q} - {\bf u}^\star) \cdot \mathbf{n}^\star_j
		-
		\frac{\dt}{2|K|}
		\left({\bf f}({\bf u}^{\rm ext}_{i,q})\cdot{\bf n}_K^{i}\right)\cdot \mathbf{n}^\star_j
		\right]
		\\
		\stackrel{\eqref{eq:wLFS-2D}}{\succ} & 
		\sum_{i = 1}^3 l^{(i)}_K
		\sum_{q = 1}^Q \widetilde \omega_{q}
		\frac{\dt}{2|K|}
		\left[
		{\bm \zeta}({\bf u}^\star)\cdot {\bf n}_K^{i}
		\right]
		+
		\sum_{i = 1}^3 l^{(i)}_K
		\sum_{q = 1}^Q \widetilde \omega_{q}
		\frac{\dt}{2|K|}
		\left[
		{\bm \zeta}({\bm u}^\star)\cdot {\bf n}_K^{i}
		\right]
		\\
		= &
		\frac{\dt}{|K|}
		{\bm \zeta}({\bf u}^\star)
		\cdot 
		\left(
			\sum_{i = 1}^3 l^{(i)}_K
			{\bm n}_K^{i}
		\right)
		\stackrel{\eqref{discrete-div}}{=}
		\frac{\dt}{2|K|}
		{\bm \zeta}({\bf u}^\star)
		\cdot 
		{\bf 0}
		= 0.
	\end{align*}
\end{proof}

As shown in \Cref{thm:1008}, the CFL condition \eqref{eq:1009} depends on  {\color{black}$\mathcal{C}_{\tt IDP}$}, which is determined by a chosen quadrature \eqref{eq:980}. It is therefore natural to seek the \textbf{optimal} quadrature of the form \eqref{eq:980} that \textbf{maximizes} {\color{black}$\mathcal{C}_{\tt IDP}$}, thereby yielding the most lenient IDP CFL condition. This allows for larger stable time step sizes and improves the efficiency of high-order IDP schemes.

Such optimal quadrature rules were recently discovered in \cite{ding2025robust} for $\mathbb{P}^1$- and $\mathbb{P}^2$-based IDP schemes on arbitrary triangular meshes.
For convenience, we consider an arbitrary triangular cell $K$ and rearrange the indices of its edges $\{e_K^{(i)}\}_{i=1}^3$ and vertices $\{{\bf V}_K^{(i)}\}_{i=1}^3$ such that $l_K^{(1)} \ge l_K^{(2)} \ge l_K^{(3)}$.

\begin{example}[Optimal quadrature \cite{ding2025robust} for $\mathbb{P}^1$]\label{ex:P1}
The optimal quadrature of the form \eqref{eq:980} for $\mathbb{P}^1$-based IDP schemes on any triangular cell $K$ is given by
\begin{equation}\label{eq:1032}
	w_i = \frac{2l_K^{(i)}}{3l_K^{(1)} + 3l_K^{(2)}}, \qquad i = 1, 2, 3,
\end{equation}
with at most one internal node ($S \le 1$), whose weight and location are given by:
\begin{equation}\label{eq:1048}
	 {\omega}_1^* = 
	\frac{l_K^{(1)} + l_K^{(2)} - 2l_K^{(3)}}{3l_K^{(1)} + 3l_K^{(2)}}, \quad
	(x_1^*, y_1^*) =
	\frac{
		(l_K^{(1)} - l_K^{(3)})\,{\bf V}_K^{(1)} +
		(l_K^{(2)} - l_K^{(3)})\,{\bf V}_K^{(2)}
	}
	{l_K^{(1)} + l_K^{(2)} - 2l_K^{(3)}}.
\end{equation}
Note that if the cell $K$ is equilateral, i.e., $l_K^{(1)} = l_K^{(2)} = l_K^{(3)}$, then the weight ${\omega}_1^*$ becomes zero, and the optimal quadrature contains \emph{no} internal node ($S = 0$).
\end{example}

\begin{example}[Optimal quadrature \cite{ding2025robust} for $\mathbb{P}^2$]\label{ex:P2}
The optimal quadrature of the form \eqref{eq:980} for $\mathbb{P}^2$-based IDP schemes on any triangular cell $K$ has the boundary weights 
	\begin{equation*}
		w_i = \frac{2l_K^{(i)}}{9 \bar l_K + 3 \hat l_K}, \qquad i = 1,2,3,
	\end{equation*}
	and two internal nodes with weights and coordinates
	\begin{equation*}
		\omega_s^* = \frac{\bar l_K + \hat l_K}{6 \bar l_K + 2 \hat l_K},
		~~ 
		(x_s^*, y_s^*)^\top 
		= 
		\sum_{i=1}^3 \beta_{s,i} \, {\bf V}_K^{(i)},
		~~ 
		\beta_{s,i}
		=
		\frac
		{{\bm l}_K^\top {\bm M}_{s,i} \, {\bm l}_K
			+ 2 \, c_{s,i} \, \hat l_K}
		{18(\bar l_K+\hat l_K)(l_K^{(2)}+\hat l_K)},
		\quad s = 1,2,
	\end{equation*}
	where ${\bm l}_K := (l_K^{(1)}, l_K^{(2)}, l_K^{(3)})^\top$, $\bar l_K := (l_K^{(1)} + l_K^{(2)} + l_K^{(3)})/3$, and
        \begin{equation*} 
	\begin{aligned}
	\hat l_K :
	= &
	\sqrt
	{
		\big(l_K^{(1)}\big)^2
		+
		\big(l_K^{(2)}\big)^2
		+
		\big(l_K^{(3)}\big)^2
		-
		\frac23\Big(
		l_K^{(1)} l_K^{(2)} + l_K^{(2)} l_K^{(3)} + l_K^{(3)} l_K^{(1)}
		\Big)
	}.
	\end{aligned}
        \end{equation*}
	The positive coefficients $c_{s,i}$ and the positive definite matrices ${\bm M}_{s,i}$ are given by
	\begin{equation*}
		\begin{aligned}
		c_{1,1}
		&=
		3 l_K^{(1)} + 3 l_K^{(2)} + \sqrt{3} l_K^{(2)} - \sqrt{3} l_K^{(3)}, \quad 
        		c_{2,1}
		=
		3 l_K^{(1)} + 3 l_K^{(2)} + \sqrt{3} l_K^{(3)} - \sqrt{3} l_K^{(2)},
        \\
		c_{1,2}
		&=
		6 l_K^{(2)} + \sqrt{3} l_K^{(3)} - \sqrt{3} l_K^{(1)}, \quad 
		c_{2,2}
		=
		6 l_K^{(2)} + \sqrt{3} l_K^{(1)} - \sqrt{3} l_K^{(3)},
        \\
		c_{1,3}
		&=
		3 l_K^{(2)} + 3 l_K^{(3)} + \sqrt{3} l_K^{(1)} - \sqrt{3} l_K^{(2)}, 
        \quad 
		c_{2,3}
		=
		3 l_K^{(2)} + 3 l_K^{(3)} + \sqrt{3} l_K^{(2)} - \sqrt{3} l_K^{(1)},
		\end{aligned}
	\end{equation*}
	{\small 
	\begin{equation*}
		\begin{aligned}
			{\bm M}_{1,1}&=
			\begin{bmatrix}
				6 & 1 & -2 \\
				1 & 2\sqrt{3}+6 & -\sqrt{3}-2 \\
				-2 & -\sqrt{3}-2 & 6
			\end{bmatrix},&
			{\bm M}_{2,1}&=
			\begin{bmatrix}
				6 & 1 & -2 \\
				1 & 6-2\sqrt{3} & \sqrt{3}-2 \\
				-2 & \sqrt{3}-2 & 6
			\end{bmatrix},&
			\\
			{\bm M}_{1,2}&=
			\begin{bmatrix}
				6 & -\sqrt{3}-2 & -2 \\
				-\sqrt{3}-2 & 12 & \sqrt{3}-2 \\
				-2 & \sqrt{3}-2 & 6
			\end{bmatrix},&
			{\bm M}_{2,2}&=
			\begin{bmatrix}
				6 & \sqrt{3}-2 & -2 \\
				\sqrt{3}-2 & 12 & -\sqrt{3}-2 \\
				-2 & -\sqrt{3}-2 & 6
			\end{bmatrix},&
			\\
			{\bm M}_{1,3}&=
			\begin{bmatrix}
				6 & \sqrt{3}-2 & -2 \\
				\sqrt{3}-2 & 6-2\sqrt{3} & 1 \\
				-2 & 1 & 6
			\end{bmatrix},&
			{\bm M}_{2,3}&=
			\begin{bmatrix}
				6 & -\sqrt{3}-2 & -2 \\
				-\sqrt{3}-2 & 2\sqrt{3}+6 & 1 \\
				-2 & 1 & 6
			\end{bmatrix}.
		\end{aligned}
	\end{equation*}
}
\end{example}

If the above optimal quadrature is used for cell average decomposition, 
we have the following results, 
as a corollary of \Cref{thm:1008}.

\begin{theorem}[IDP via optimal quadrature \cite{ding2025robust} for $\mathbb P^1$ and $\mathbb P^2$] \label{thm:1108}
	Assume $l_K^{(1)} \ge l_K^{(2)} \ge l_K^{(3)}$. 
	If  a  $\mathbb P^m$-based ($m=1$ or $2$) solution ${\bf u}_h^n$ satisfy 
	\begin{equation}\label{eq:1526P1}
	    {\bf u}_h^n(x,y) \in G \qquad \forall (x,y) \in \mathbb S_{K,m}^{\tt DCW},~~\forall K, 
	\end{equation}
	where $\mathbb S_{K,1}^{\tt DCW}$ and $\mathbb S_{K,2}^{\tt DCW}$ denote the set of all the nodes in the optimal quadrature in \Cref{ex:P1} and \Cref{ex:P2}, respectively, proposed by Ding, Cui, and Wu \cite{ding2025robust}, 
	then the scheme \eqref{eq:1033} preserves  $\bar{\bf u}_{K}^{n+1}\in G$ under the 
	CFL condition 
	\begin{equation}\label{eq:1692}
	\begin{aligned}
		& \alpha \frac{\dt}{|K|}  \le 
		\mathcal{C}_{K,m}^{\tt DCW} \qquad \forall K \in \mathcal{T}_h, 
		\\
		& \mathcal{C}_{K,1}^{\tt DCW} := \frac{2}{3(l^{(1)}_K+l^{(2)}_K)}, \qquad 
		\mathcal{C}^{\tt DCW}_{K,2}  := \frac{2}{9\bar l_K + 3 \hat l_K}, 
		\end{aligned}
	\end{equation}
    which is  optimal for using any   
    quadrature of the form \eqref{eq:980}. 
\end{theorem}


\subsubsection{Comparison of different quadrature rules for IDP}

\begin{remark}[IDP via Chen--Shu quadrature in \cite{chen2017entropy}]
	As a direct consequence of 
	\Cref{thm:1008}, if  the Chen--Shu quadrature in \cite[Table C.2]{chen2017entropy} is used for cell average decomposition, 
	then we obtain a $\mathbb P^m$-based ($m=1, 2, 3, 4$) high-order IDP scheme under the CFL condition:
	\begin{equation}\label{eq:1692cs}
		\begin{aligned}
			 \alpha \frac{\dt}{|K|}  \le   w_m^{\tt CS}    \min    \left\{ 
			 \frac{1}{l^{(1)}_K},
			 \frac{1}{l^{(2)}_K},
			 \frac{1}{l^{(3)}_K}
			 \right\} =: 
			\mathcal{C}_{K,m}^{\tt CS} \qquad \forall K \in \mathcal{T}_h, 
		\end{aligned}
	\end{equation}
	where $w_1^{\tt CS}=\frac13$, $w_2^{\tt CS}=\frac3{20}$, 
	$w_3^{\tt CS}\approx 0.086812$, and $w_4^{\tt CS}\approx0.05572449$.
\end{remark}

\begin{remark}[Improved IDP via Zhang--Xia--Shu quadrature \eqref{eq:ZXSquadrature}]
	Applying 
	\Cref{thm:1008} to the case of using Zhang--Xia--Shu quadrature \eqref{eq:ZXSquadrature},  we can improve the IDP CFL condition \eqref{eq:ZS-tri} to the following more relaxed one:
	\begin{equation}\label{eq:1692zxs}
		\begin{aligned}
			\alpha \frac{\dt}{|K|}  \le   \frac23 {\omega}_1   \min    \left\{ 
			\frac{1}{l^{(1)}_K},
			\frac{1}{l^{(2)}_K},
			\frac{1}{l^{(3)}_K}
			\right\} =: 
			\mathcal{C}_{K,m}^{\tt ZXS} \qquad \forall K \in \mathcal{T}_h, 
		\end{aligned}
	\end{equation}
	where  $ \omega_1=\frac{1}{L(L-1)}$ is the first weight of the $L$-point Gauss--Lobatto quadrature with $L=\lceil \frac{m+3}{2} \rceil$.
\end{remark}

\begin{remark}[Lv--Ihme approach]
	A different decomposition for cell averages was proposed by Lv and Ihme in \cite{lv2015entropy}. The idea is to begin with any quadrature rule that has positive weights and sufficiently high accuracy, and then solve an optimization problem to increase the decomposition weights at the boundary Gauss points. 
	The primary advantage of this approach is its applicability to arbitrary polygonal or polyhedral cells. However, it requires solving computationally expensive optimization problems tailored to each specific cell geometry. Moreover, on triangular meshes, the resulting IDP CFL number is generally smaller than those obtained using the quadrature rules discussed above.
\end{remark}

\Cref{tab:2019} presents a comparison of 
 the IDP CFL condition 
$\alpha \frac{\Delta t}{|K|} \le  
        \mathcal{C}_{\tt IDP}$ 
obtained using different quadrature rules for $\mathbb{P}^1$- and $\mathbb{P}^2$-based schemes on two representative triangular cells. As expected, $\mathcal{C}_{\tt IDP}$ derived using the optimal quadrature proposed in \cite{ding2025robust} is the largest among all the cases.  

    \begin{table}[!htb]
    \centering
    \caption{Comparison of the IDP CFL condition
    	$\alpha \frac{\Delta t}{|K|} \le \mathcal{C}_{\tt IDP}$
    	and the number of internal nodes for various quadrature rules on two example triangular cells.  (Note: The information in \cite[Table 1]{ding2025robust} for the Chen--Shu quadrature  was incorrect, and the corrected version is presented here.)}\label{tab:2019}
	\renewcommand\arraystretch{1.5}
	\begin{tabular}{@{\hskip 2pt}c|c@{\hskip 2pt}|c@{\hskip 2pt}c|@{\hskip 2pt}c@{\hskip 2pt}c@{\hskip 2pt}}
		\toprule[1.5pt]
		\multicolumn{2}{@{\hskip 2pt}c|}{\raisebox{2.5ex}{\shortstack[t]{example cell $K$ \\ ~\\~\\~\\~\\~\\~\\~\\~\\~ }}} &
		\multicolumn{2}{c|@{\hskip 2pt}}{\includegraphics[trim = 5 0 0 0, clip, width = 0.25\textwidth]{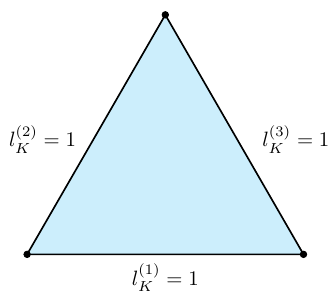}} &
		\multicolumn{2}{c@{\hskip 2pt}}{\includegraphics[trim = 4 0 0 0, clip, width = 0.26\textwidth]{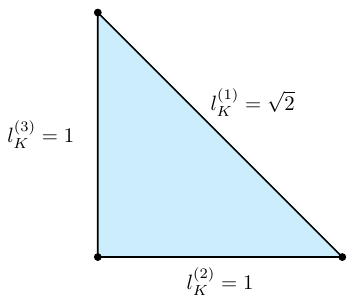}} \\
		\midrule[1pt]
		\multicolumn{2}{c|}{\multirow{1}{*}{quadrature}}
		& $\mathcal{C}_{\tt IDP}$ & \makecell{internal \\ nodes}  & $\mathcal{C}_{\tt IDP}$ & \makecell{internal \\ nodes} \\
		\midrule[.5pt]
		\multicolumn{1}{c|}{\multirow{4}{*}{$\mathbb{P}^1$}} & Optimal \cite{ding2025robust} & $\frac13 \approx 0.333$ & 0 & $\frac13 \approx 0.333$ & 1 \\
		& ZXS \cite{zhang2012maximum} & $\frac19 \approx 0.111$ & 0 & $\frac{1}{3(2+\sqrt{2})} \approx 0.0976$ & 0\\
		& ZXS \eqref{eq:1692zxs}  &  $\frac13 \approx 0.333$  &  0  &  $\frac1{3\sqrt{2}} \approx 0.236$ &  0 
		\\& Chen--Shu \cite{chen2017entropy} & $\frac13 \approx 0.333$ & 0 & $\frac1{3\sqrt{2}} \approx 0.236$ & 0 \\
		\midrule[.5pt]
		\multicolumn{1}{c|}{\multirow{4}{*}{$\mathbb{P}^2$}} & Optimal \cite{ding2025robust} & $\frac16 \approx 0.167$ & 1 & $\frac{2}{6+3\sqrt{2}+\sqrt{30-12\sqrt{2}}} \approx 0.144$ & 2 \\
		& ZXS \cite{zhang2012maximum} & $\frac{1}{27} \approx 0.037$ & 9 & $\frac{1}{9(2+\sqrt{2})} \approx 0.0325$ & 9\\
		& ZXS \eqref{eq:1692zxs} &  $\frac19 \approx 0.111$  &  9  &  $\frac1{9\sqrt{2}} \approx 0.0786$ &  9 \\
		& Chen--Shu \cite{chen2017entropy} & $\frac{3}{20}=0.15$ & 1 & $\frac{3}{20\sqrt{2}}\approx0.106$ & 1 \\
		\bottomrule[1.5pt]
	\end{tabular}
\end{table}

\subsection{Numerical examples of high order DG schemes for gas dynamics equations}

We list a few benchmark tests in  gas dynamics for verifying robustness of high order accurate schemes solving 
low density or low pressure problems,
all of which are challenging tests for high order accurate DG schemes. Below are numerical results of high order DG schemes with the third order SSP Runge-Kutta method with only the simple  limiter \eqref{eq:ZSlimiter_system} for enforcing positivity of density and pressure.

\begin{example} [Sedov blast wave]
 The blast wave generates low density and pressure. \Cref{fig:sedov} shows an IDP $\mathbb Q^6$ DG method on a rectangular mesh for compressible Navier--Stokes equations. 
{\color{black}The parameters are chosen so that, at the final time, the shock front is a circle of radius~1.} See \cite{zhang2010positivity,zhang2017positivity} for the problem setup, {\color{black}and the exact solution in Sedov’s monograph \cite{sedov1993-similarity}}.
\end{example}

\begin{example}[High speed astrophysical jets] The extremely high speed renders small internal energy in the computation, which is a tough test even for many second order schemes, e.g., even a second order MUSCL scheme may blow up if positivity is not preserved. \Cref{fig:astro-Mach2000-Euler}   shows an IDP $\mathbb Q^4$ DG method on a rectangular mesh for compressible Euler equations for a Mach 2000 jet with background pressure $0.4127$, see \cite{zhang2010positivity} for the problem set up.  \Cref{fig:astro-Mach2000}  shows  an IDP $\mathbb Q^6$ DG method on a rectangular mesh for compressible Navier--Stokes equations for a Mach 2000 jet with background pressure $10^{-6}$, see \cite{tong2023class, liu2024optimization} for the initial conditions.
\end{example}

\begin{example}[Mach 10 shock passing a sharp corner] In this test, a Mach 10 shock is first reflected, generating Kelvin-Helmholtz instability, exactly the same as those in the classical {\it double Mach reflection} test.  
    Then the shock is diffracted at a sharp corner, which induces low density and low pressure, causing numerical instabilities in high order DG schemes. This test involves strong shocks, low density/pressure, as well as fine structures such as roll-ups from Kelvin-Helmholtz instability,  which are often used as an indicator whether excessive artificial viscosity is added in numerical schemes for stabilization. \Cref{fig:astro-Mach2000}  shows results of IDP  high order DG methods for solving compressible Navier--Stokes equations.   \Cref{fig:astro-Mach2000} (a) and (b) show results of $\mathbb P^7$ DG on unstructured triangular meshes for a $60$ degree corner, see \cite{zhang2017positivity} for the problem set up.   \Cref{fig:astro-Mach2000} (c) and (d) show results of $\mathbb Q^6$ DG on rectangular meshes for a $90$ degree corner, see \cite{fan2022positivity,liu2024optimization} for the problem setup. 
    Limiters for enforcing non-oscillations can be added to reduce oscillations. 
\end{example}

\begin{figure}[htbp]
	\centering
	\begin{subfigure}{0.24\textwidth}
		\includegraphics[width=\textwidth]{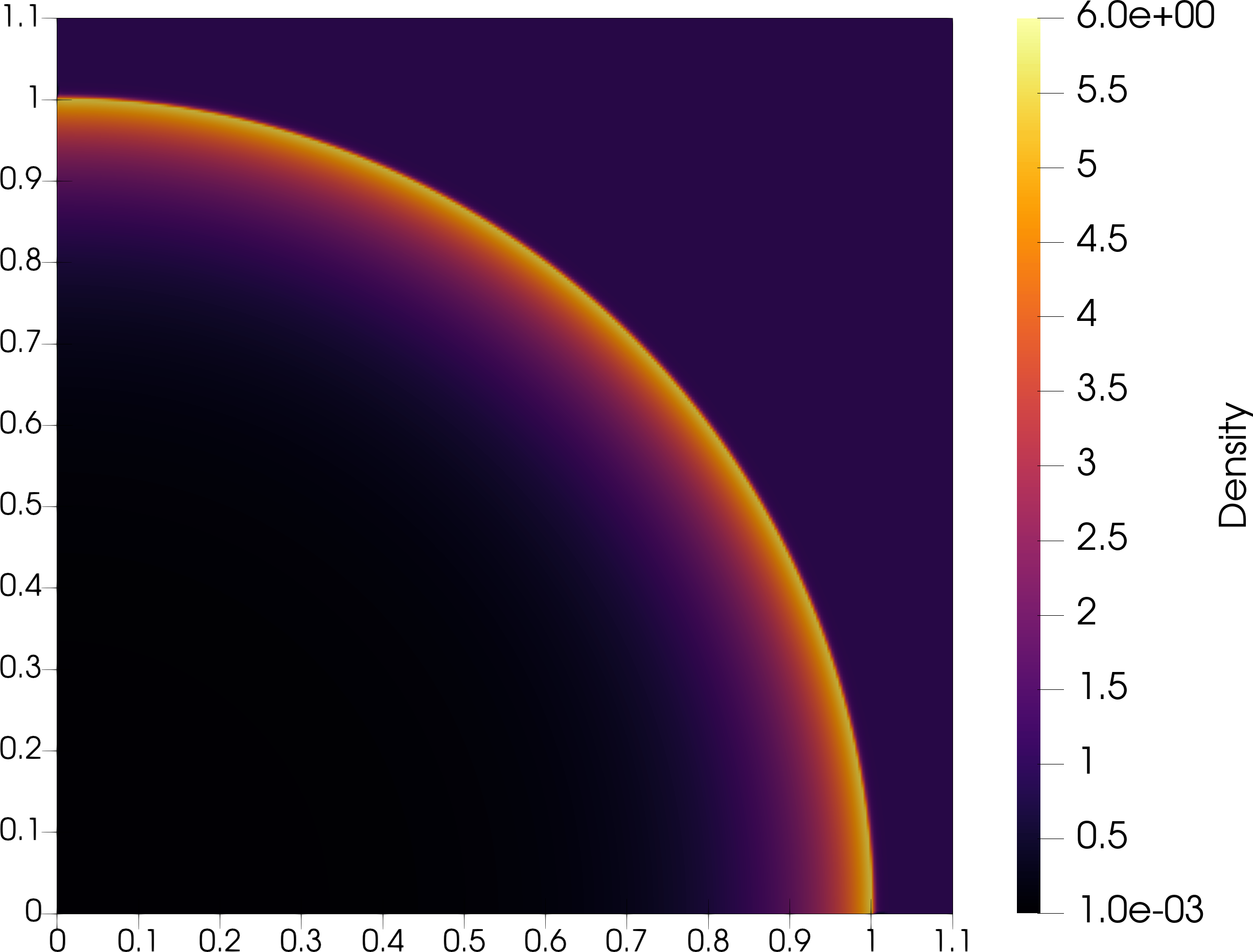}
		\caption{Density.}
	\end{subfigure}
	\hfill
	 	\begin{subfigure}{0.24\textwidth}
		\includegraphics[width=\textwidth]{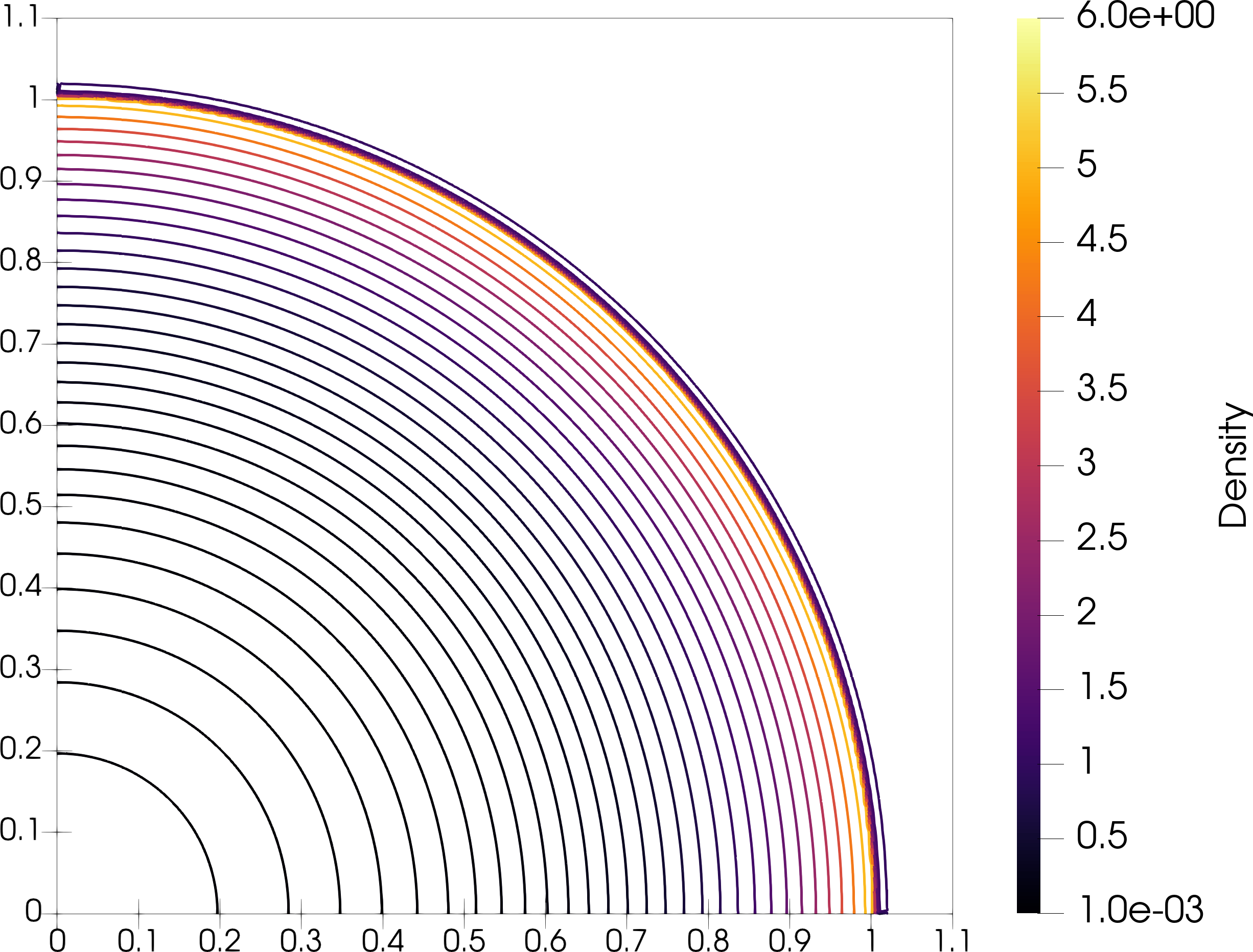}
		\caption{Density.}
	\end{subfigure}
    \begin{subfigure}{0.24\textwidth}
		\includegraphics[width=\textwidth]{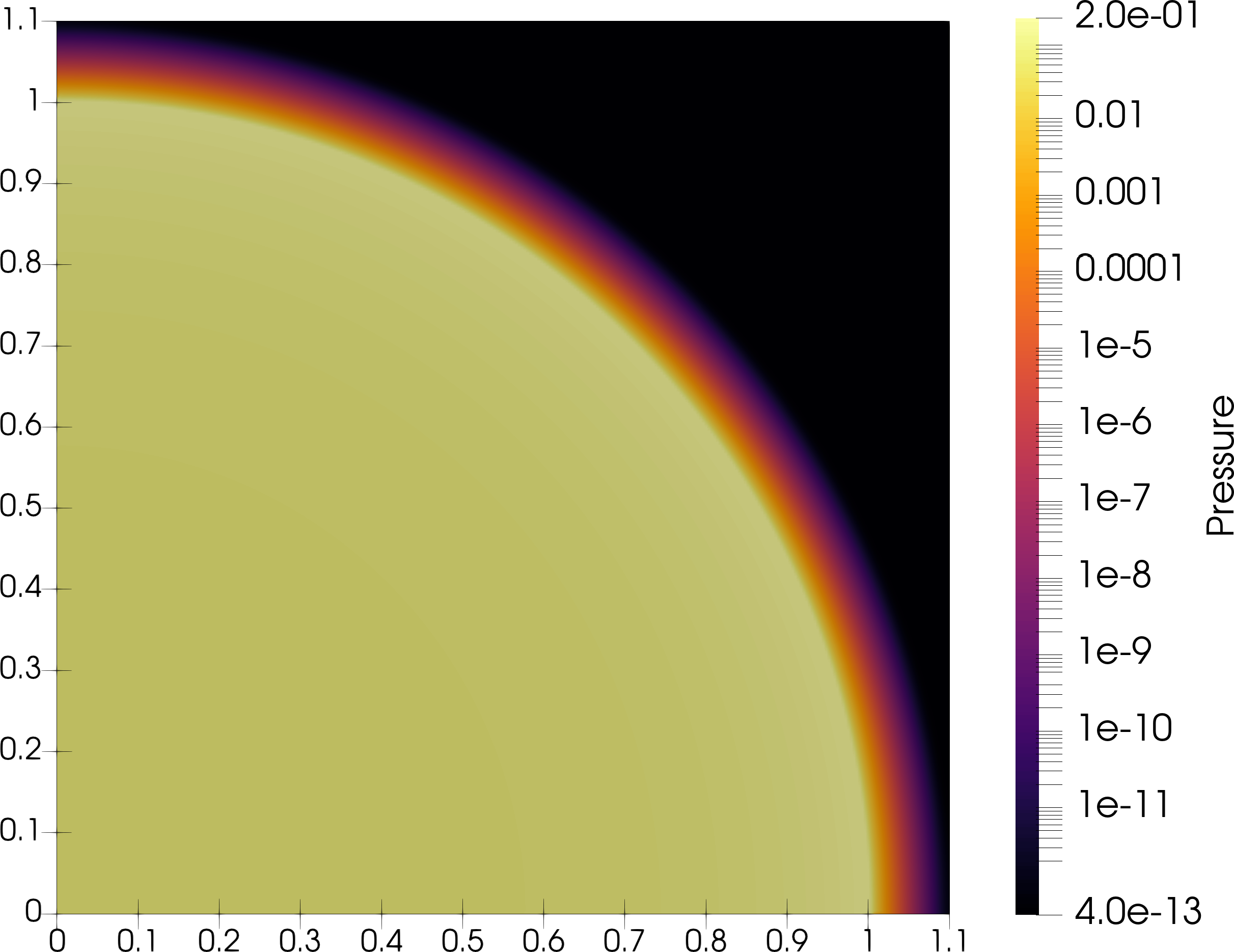}
		\caption{Pressure.}
	\end{subfigure}
	\hfill
	 	\begin{subfigure}{0.24\textwidth}
		\includegraphics[width=\textwidth]{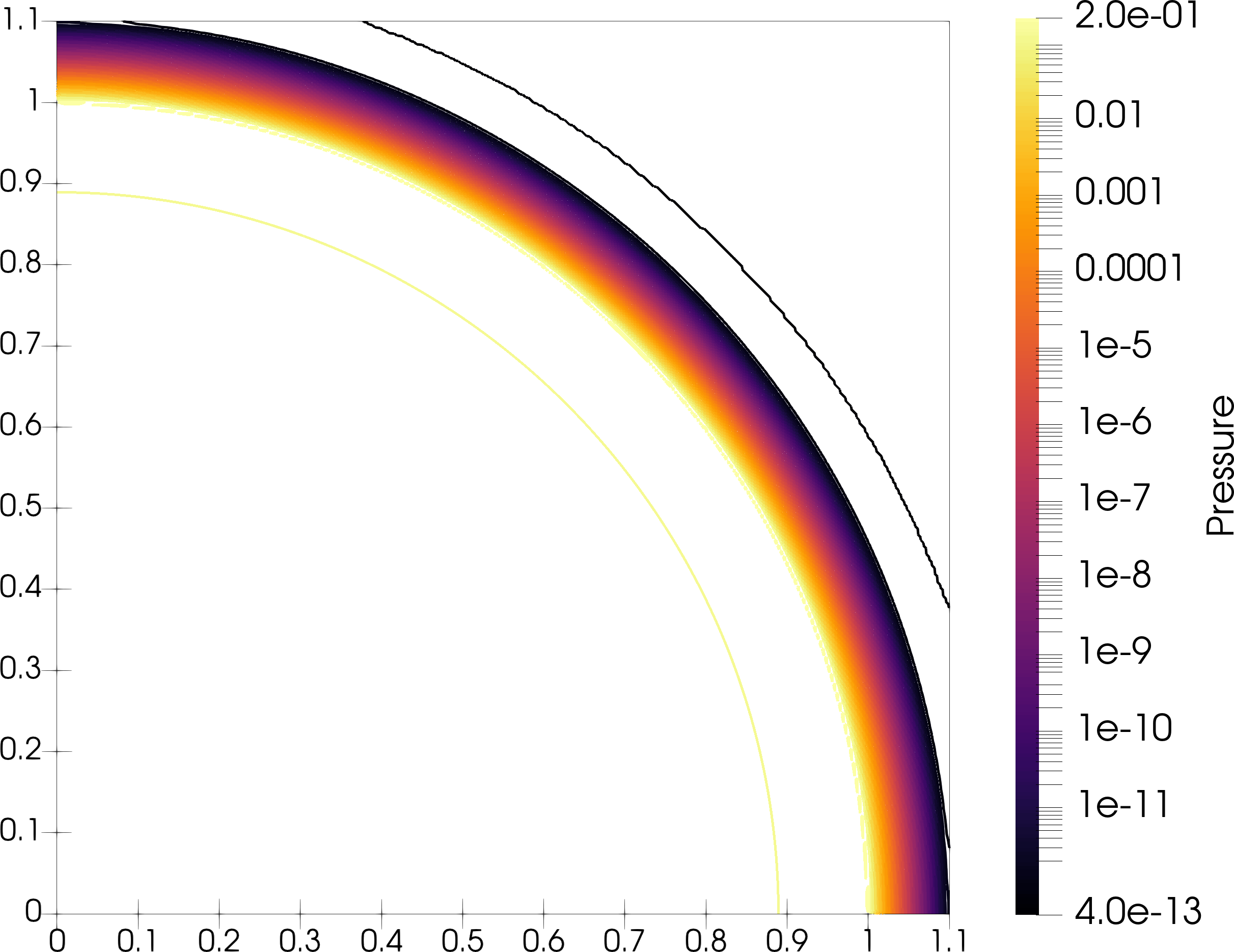}
		\caption{Pressure.}
	\end{subfigure} 
   \caption{2D Sedov blast wave. Numerical results of a positivity-preserving high order DG method with $\mathbb Q^6$ basis on a rectangular mesh of size   $\frac{1.1}{320}$  for  compressible Navier--Stokes with Reynolds number $1000$.  
   Only positivity-preserving limiter is used and no other limiters are used.}
	\label{fig:sedov}  
\end{figure}

\begin{figure}[htbp]
	\centering
	\begin{subfigure}{0.48\textwidth}
		\includegraphics[width=\textwidth]{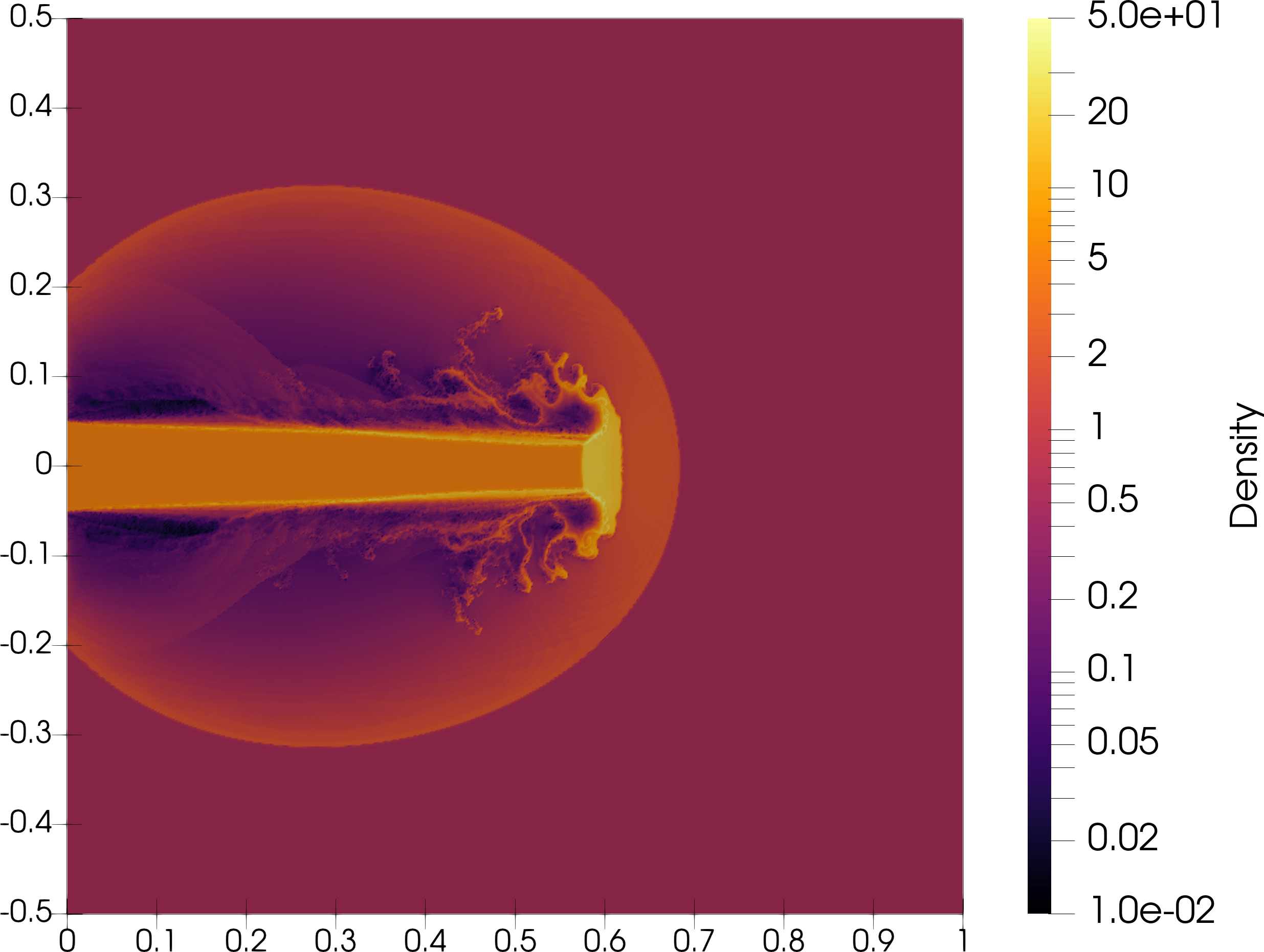}
		\caption{Density.}
	\end{subfigure}
	\hfill
	 	\begin{subfigure}{0.48\textwidth}
		\includegraphics[width=\textwidth]{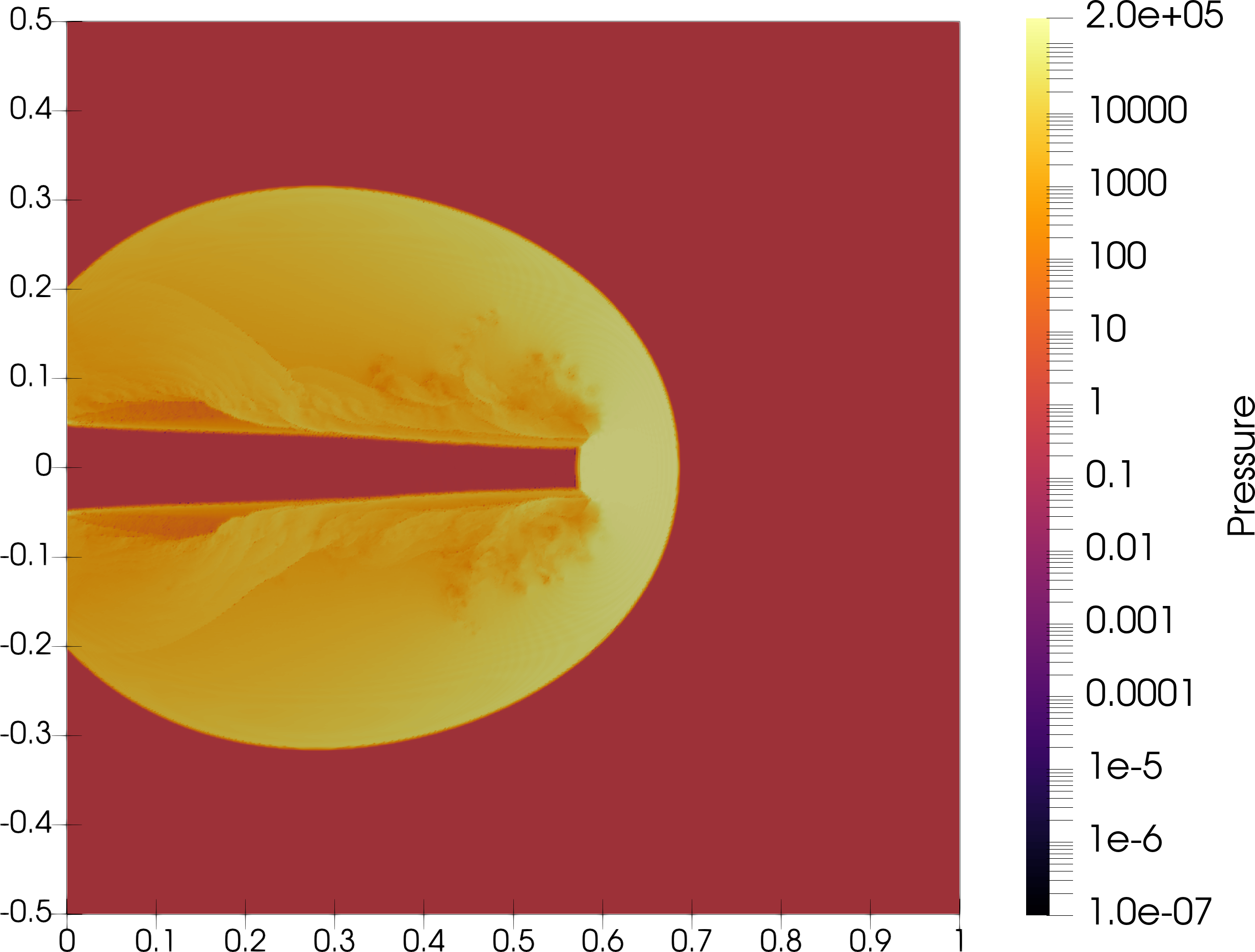}
		\caption{Pressure.}
	\end{subfigure} 
   \caption{{\it Mach 2000 jet with background pressure $0.4127$}. Numerical results of a positivity-preserving high order DG method  with $\mathbb Q^4$ basis on a rectangular mesh of size  $\frac{1}{640}$ for compressible Euler equations. Only positivity-preserving limiter is used and no other limiters are used. } 
	\label{fig:astro-Mach2000-Euler}  
\end{figure}

\begin{figure}[htbp]
	\centering
	\begin{subfigure}{0.48\textwidth}
		\includegraphics[width=\textwidth]{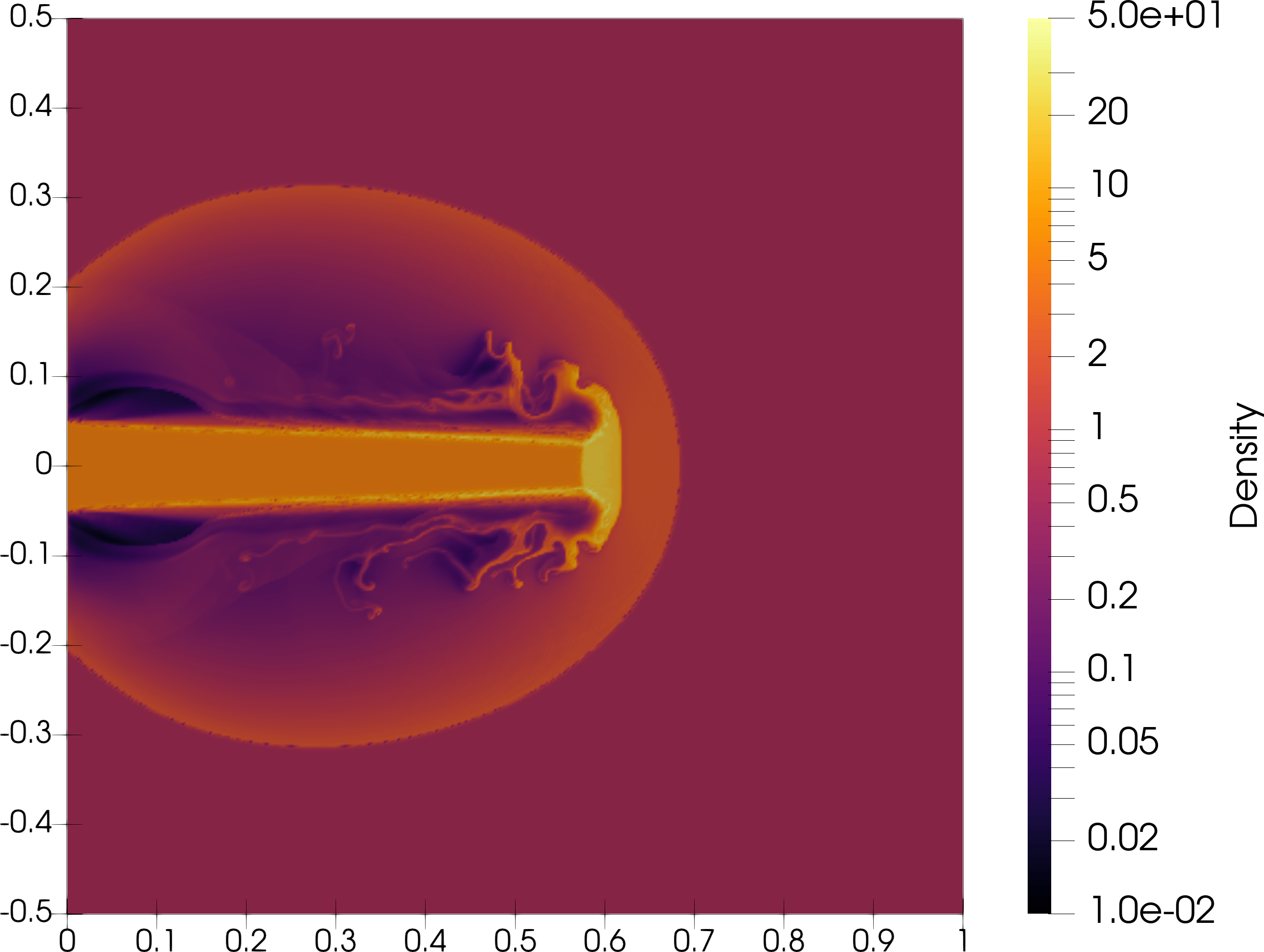}
		\caption{Density.}
	\end{subfigure}
	\hfill
	 	\begin{subfigure}{0.48\textwidth}
		\includegraphics[width=\textwidth]{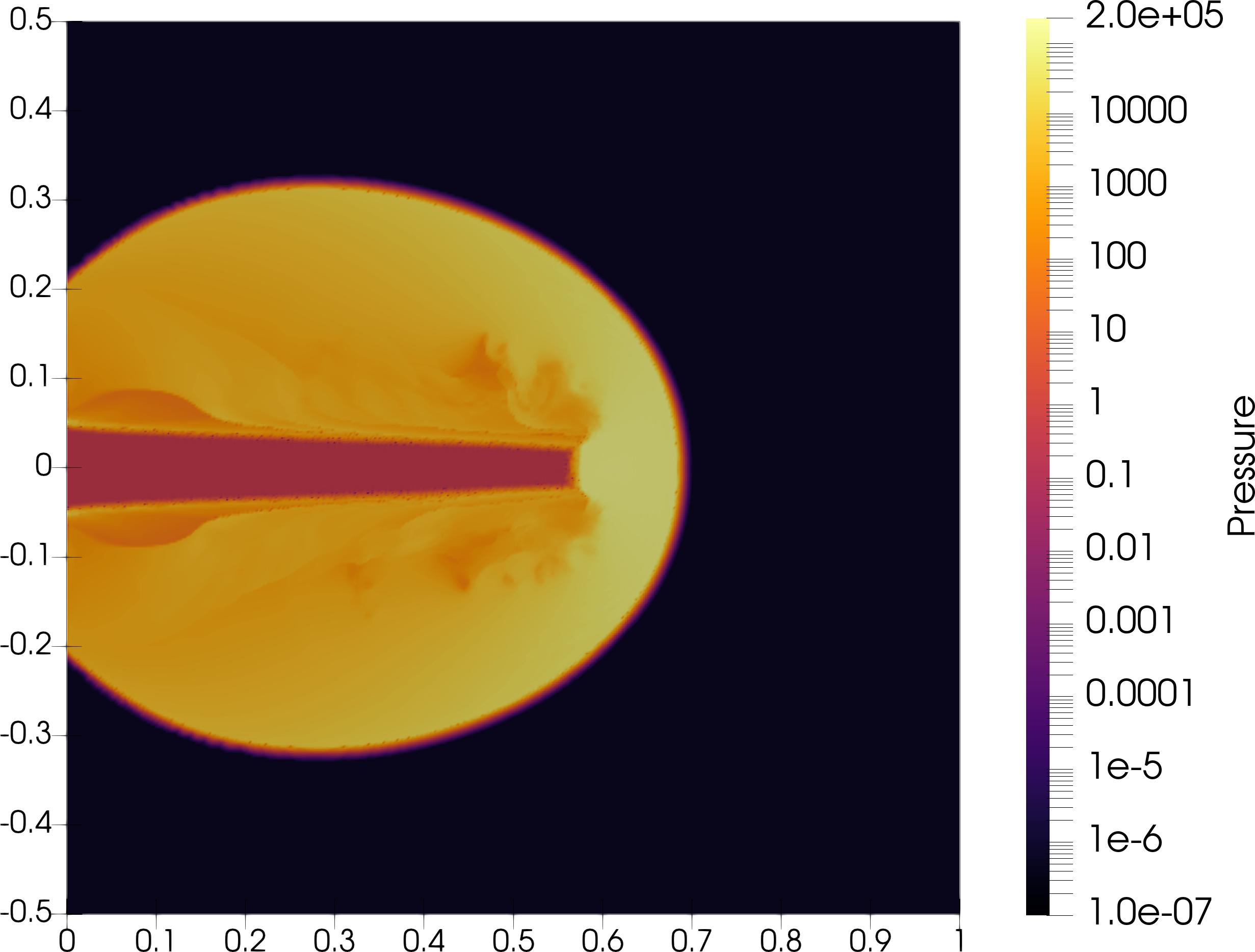}
		\caption{Pressure.}
	\end{subfigure} 
   \caption{{\it Mach 2000 jet with background pressure $10^{-6}$}. Numerical results of a positivity-preserving high order DG method with  $\mathbb Q^6$ basis on a rectangular mesh of size  $\frac{1}{400}$  for compressible Navier--Stokes with Reynolds number $1000$. Only positivity-preserving limiter is used and no other limiters are used. }
	\label{fig:astro-Mach2000}    
\end{figure}

\begin{figure}[htbp]
\centering
	\begin{subfigure}{0.48\textwidth}
		\includegraphics[width=\textwidth]{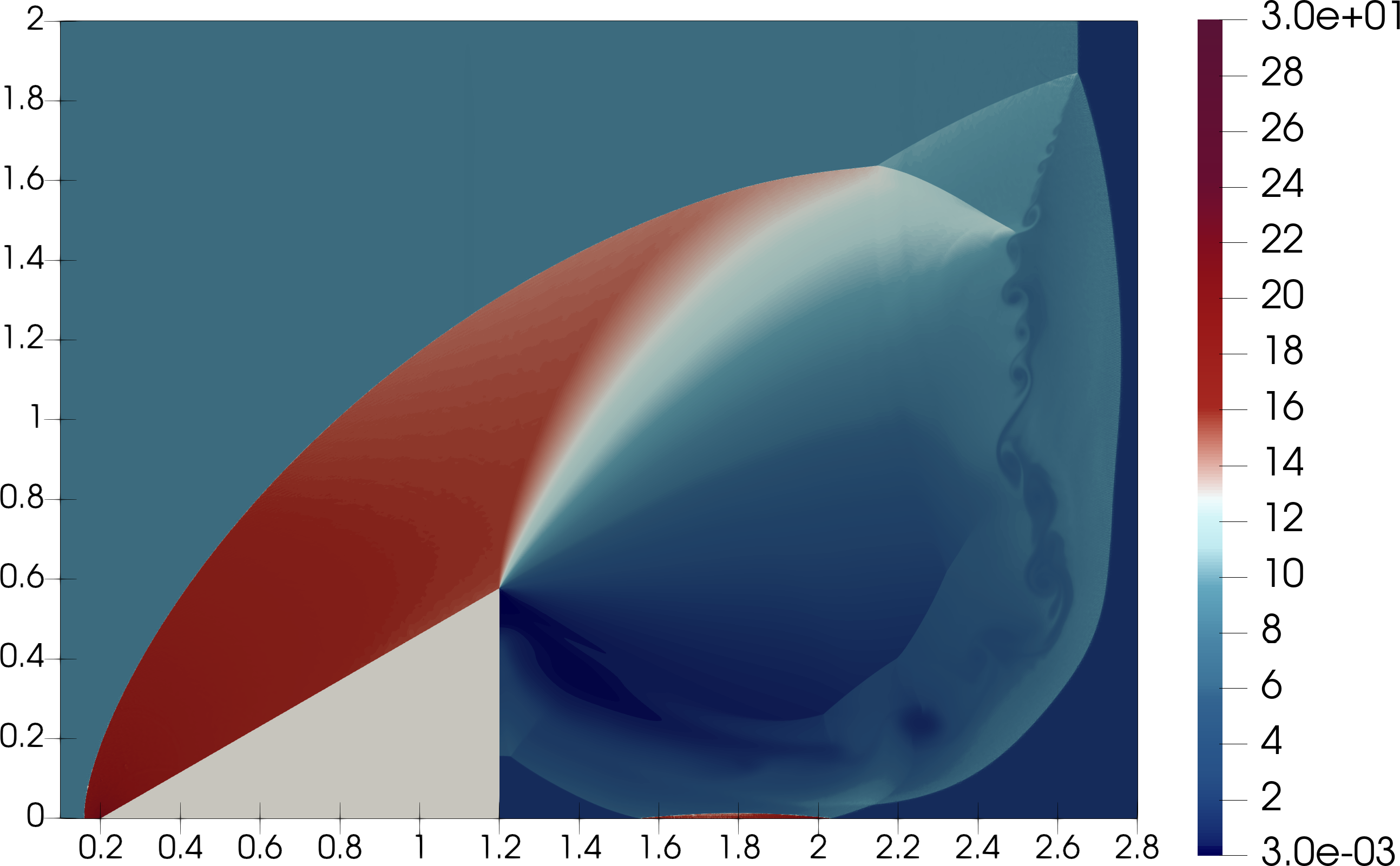}
		\caption{  $\mathbb P^7$ DG method on unstructured triangular mesh with a mesh size around $\frac{1}{160}$.}
	\end{subfigure}
	\hfill
	 	\begin{subfigure}{0.48\textwidth}
		\includegraphics[width=\textwidth]{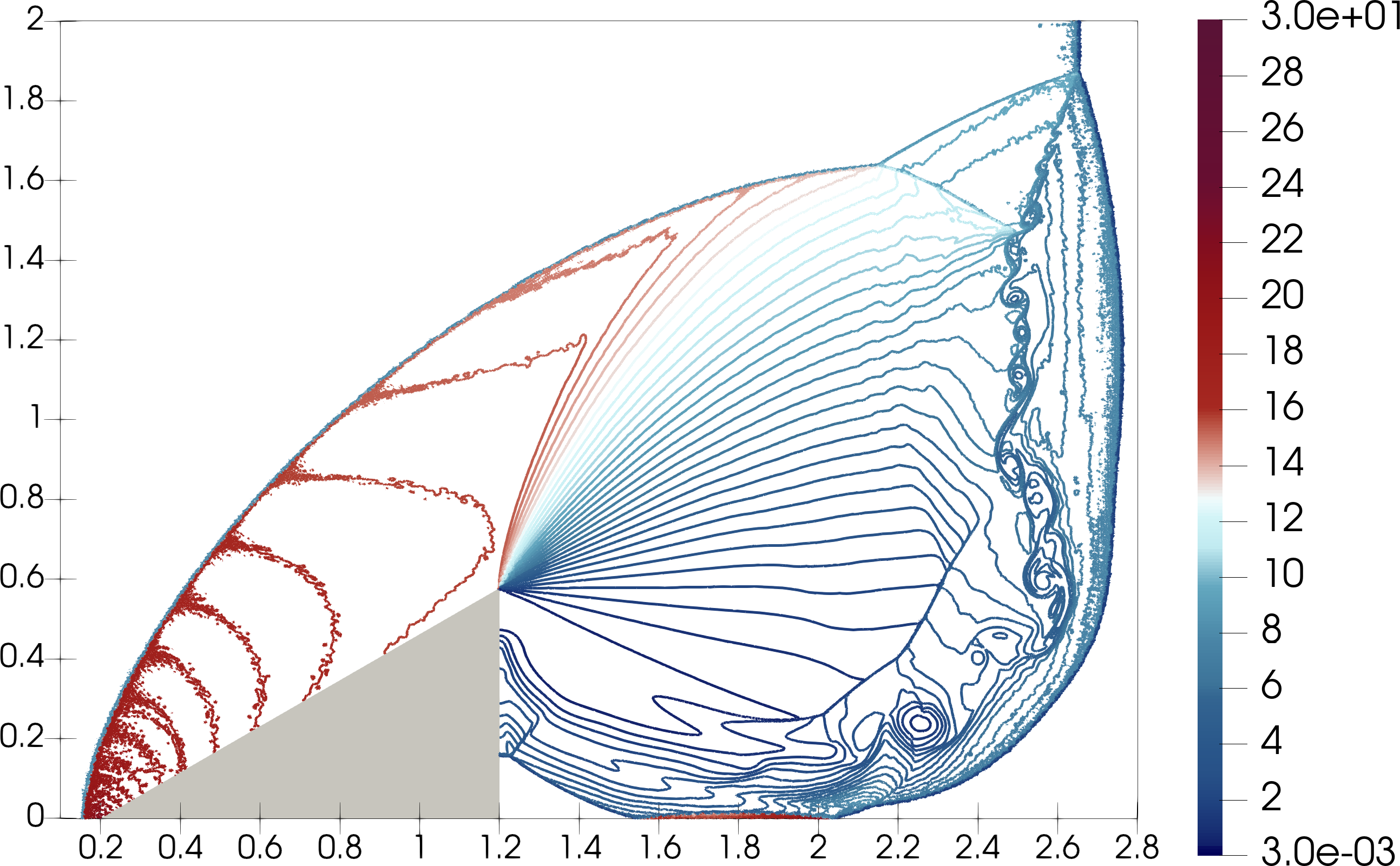}
		\caption{  $\mathbb P^7$ DG method on unstructured triangular mesh with a mesh size around $\frac{1}{160}$.}
	\end{subfigure}\\ 
\begin{subfigure}{0.48\textwidth}
		\includegraphics[width=\textwidth]{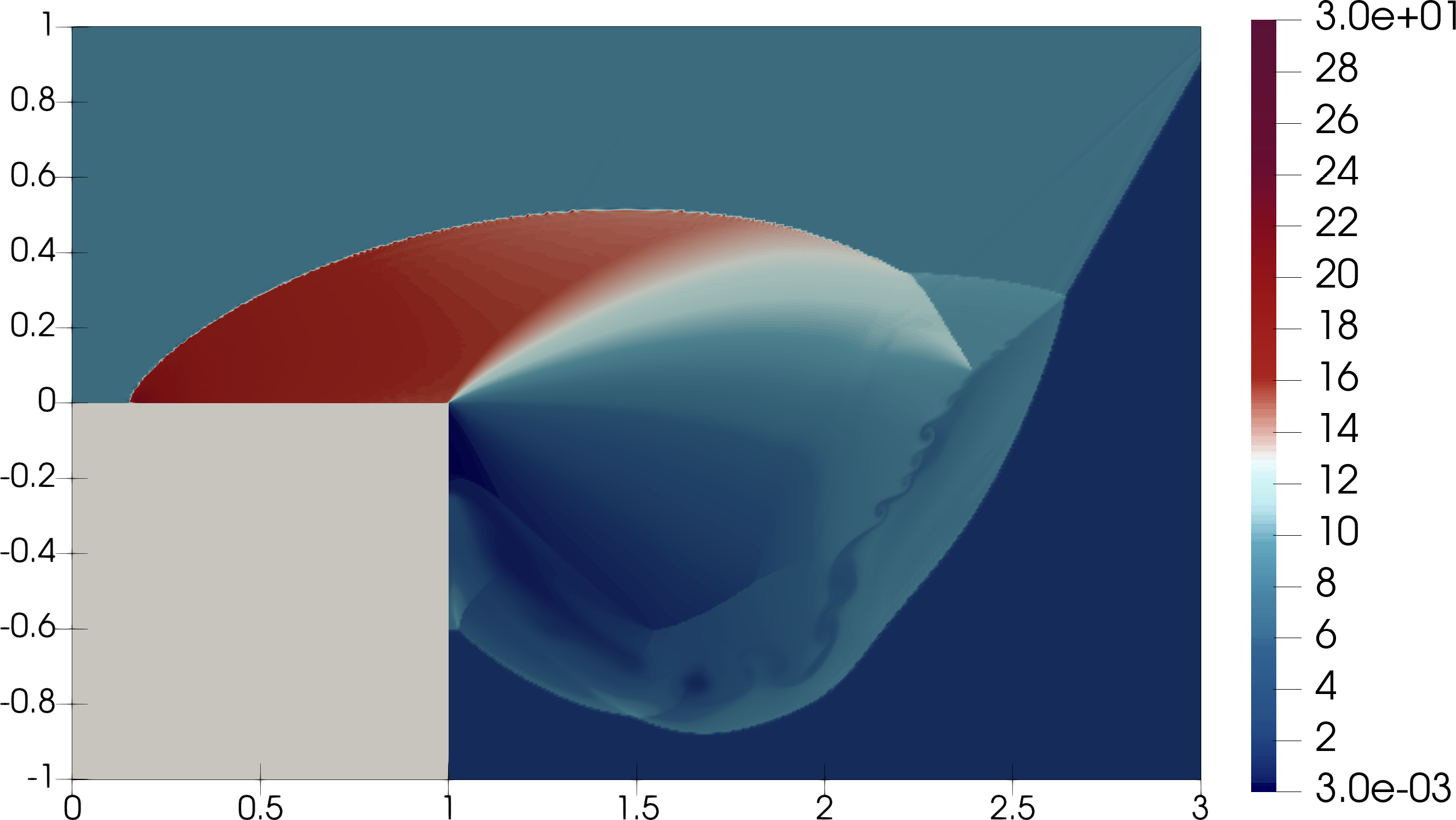}
		\caption{  $\mathbb Q^6$ DG method on a rectangular mesh with a mesh size $\frac{1}{180}$.}
        \end{subfigure}
	\hfill
    	\begin{subfigure}{0.48\textwidth}
		\includegraphics[width=\textwidth]{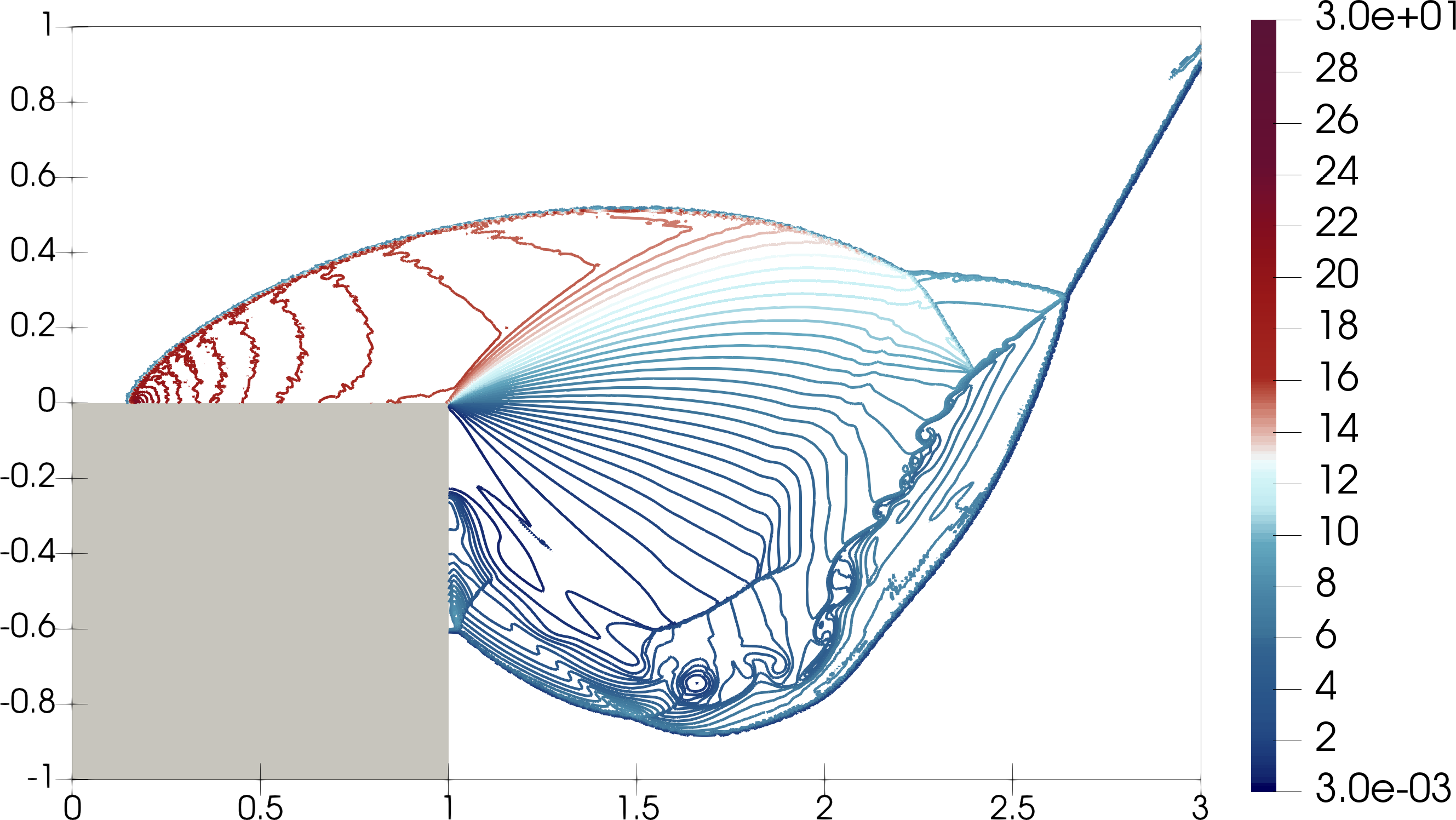}
		\caption{  $\mathbb Q^6$ DG method on a rectangular mesh with a mesh size $\frac{1}{180}$.}
	\end{subfigure}
	\label{fig:mach10-wedge}  
   \caption{ {\it Mach 10 shock reflection and diffraction.} Numerical results of positivity-preserving high order DG methods for solving Compressible Navier--Stokes with Reynolds number $1000$. Plot of Density. 
   Only positivity-preserving limiter is used and no other limiters are used.}
\end{figure}

\section{Flux correction limiters and convex limiting for high order schemes}
\label{sec:FCT}
Invoke flux correction limiters is also a very popular approach to enforce 
 convex invariant domains in high order schemes of many spatial discretizations, including continuous finite element methods. 
The methodology of flux correction limiters for IDP is build upon the ideas of FCT by Boris and Book \cite{boris1973flux,book1975flux,boris1976flux} and  Zalesak \cite{zalesak1979fully}. 
We refer to \cite{oran2001numerical,kuzmin-turek-2002flux,zalesak2012design,kuzmin2012flux,kuzmin2024property} and references therein for histories and developments as well as comprehensive reviews of FCT methods.
FCT type flux limiters can be considered for higher order PDEs and used to enforce more properties than a convex invariant domain. We mainly focus on the flux correction based methods for enforcing a convex invariant domain for systems of conservation laws, most of which emerged in the past 15 years. 

Flux correction limiters can be designed for many different time discretizations. For the sake of simplicity, we focus on a
high order accurate strong stability preserving (SSP) Runge-Kutta (RK) method \eqref{ssp-RK-3rd}, which is convex combination of forward Euler steps. Thus we only need to consider how to achieve IDP for the forward Euler step since a convex combination  preserves a convex invariant domain.

\subsection{The main idea of flux corrections}   

 Consider a high order   finite difference (or finite volume) spatial discretization with forward Euler time discretization as an example,
\begin{align}\label{eq:forward_Euler}
    {\bf u}_j^{n+1,H} = {\bf u}_j^n - \lambda \Big(\hat {\bf f}_{j+\frac12}^H - \hat {\bf f}_{j-\frac12}^H\Big),
\end{align}
which is in general not IDP. 
Assume that there is a first order IDP scheme, 
\begin{align}\label{eq:forward_EulerL}
    {\bf u}_j^{n+1,L} = {\bf u}_j^n - \lambda \Big(\hat {\bf f}_{j+\frac12}^L - \hat {\bf f}_{j-\frac12}^L\Big),
\end{align}
which is provably IDP under a CFL condition $a_{\max} \dt/\dx \le c_0$ as in Section \ref{sec:firstorderIDP}. 
The idea of flux correction for constructing IDP high order schemes is to seek suitable parameters $\theta_{j+\frac12} \in [0,1]$
such that 
the scheme using a modified numerical flux is IDP,
\begin{subequations}
    \label{eq:forward_EulerIDP}
\begin{align}
    {\bf u}_j^{n+1} = {\bf u}_j^n - \lambda \Big(\hat {\bf f}_{j+\frac12} - \hat {\bf f}_{j-\frac12}\Big),
\end{align} 
\begin{align}\label{flux_BP}
    \hat {\bf f}_{j+\frac12}
    = \theta_{j+\frac12} \Big(\hat {\bf f}^H_{j+\frac12} - \hat {\bf f}^L_{j+\frac12}\Big)
    + \hat {\bf f}^L_{j+\frac12}. 
\end{align}
\end{subequations}
Since $\hat \bff_{j+\frac12}$ is a convex combination of $\hat \bff^H_{j+\frac12}$ and $\hat \bff^L_{j+\frac12}$, the corrected numerical flux $\hat \bff_{j+\frac12}$ is a consistent and locally conservative flux, thus the scheme
\eqref{eq:forward_EulerIDP} is still a consistent and locally conservative scheme. 
In order to maintain the high order accuracy, $\theta_{j+\frac12} \in [0,1]$ should be as large as possible under the IDP constraint.

It is nontrivial to determine such parameters $\{ \theta_{j+\frac12} \}$, as they are coupled.   
For example, the IDP goal for preserving $g({\bf u})>0$ requires $\{ \theta_{j+\frac12} \}$ to satisfy the globally coupled inequalities for all $j$: 
\begin{align}\label{eq:MMP_NF_BP-2}
g \left( \bu_j^n-\lambda \Big[  (\theta_{j+\frac12} \Big(\hat {\bf f}^H_{j+\frac12} - \hat {\bf f}^L_{j+\frac12}\Big)
    + \hat {\bf f}^L_{j+\frac12})    -  (\theta_{j-\frac12} \Big(\hat {\bf f}^H_{j-\frac12} - \hat {\bf f}^L_{j-\frac12}\Big)
    + \hat {\bf f}^L_{j-\frac12})    \Big] \right) > 0. 
\end{align}

Defining ${\bm \delta}_{j+\frac12} := \hat {\bf f}^H_{j+\frac12} - \hat {\bf f}^L_{j+\frac12}$ and using the notation in \eqref{eq:forward_EulerL}, 
the IDP inequalities \eqref{eq:MMP_NF_BP-2} can be equivalently expressed as 
\begin{align}\label{eq:MMP_NF_BP}
g \left( \bu_j^{n+1,L}-\lambda \big(  \theta_{j+\frac12} {\bm \delta}_{j+\frac12}   -  \theta_{j-\frac12} {\bm \delta}_{j-\frac12}  \big)   \right) >  0, \quad \forall j. 
\end{align}

We list a simple fact implied by the Jensen's inequality:
\begin{lemma}
\label{lem:smallertheta}
   If $g(\bu)$ is a concave  or a quasi-concave function of $\bu\in \mathbb R^d$,
 and
   $\{\theta_{j+\frac12}, 
   j=1,2,\cdots, n\}$ satisfies \eqref{eq:MMP_NF_BP}, then    $g \left( \bu_j^{n+1,L}     \right) \geq  0$ implies that $\{a \theta_{j+\frac12},j=1,2,\cdots, n\}$ satisfies \eqref{eq:MMP_NF_BP} for any $a\in [0,1]$.
\end{lemma}
\begin{proof}
Since $\bu_j^{n+1}=\bu_j^{n+1,L}-\lambda \big(  \theta_{j+\frac12} {\bm \delta}_{j+\frac12}   -  \theta_{j-\frac12} {\bm \delta}_{j-\frac12}  \big),$
 we have a convex combination, 
\[\bu_j^{n+1,L}-\lambda \big( a \theta_{j+\frac12} {\bm \delta}_{j+\frac12}   - a \theta_{j-\frac12} {\bm \delta}_{j-\frac12}  \big)=(1-a)\bu_j^{n+1,L}+ a \bu_j^{n+1},\]
 If $g(\bu)$ is concave,  the Jensen's inequality gives
 \[g\left(\bu_j^{n+1,L}-\lambda \big( a \theta_{j+\frac12} {\bm \delta}_{j+\frac12}   - a \theta_{j-\frac12} {\bm \delta}_{j-\frac12}  \big)\right)=(1-a)g\left(\bu_j^{n+1,L}\right)+ a g\left(\bu_j^{n+1}\right)\geq 0.\]
 If $g(\bu)$ is quasi-concave, the same argument applies using the inequality \eqref{Jensen-quasiconcave}.
\end{proof}

For the method \eqref{eq:forward_EulerIDP}, the largest $\theta_{j+\frac12}\in [0,1]$ satisfying \eqref{eq:MMP_NF_BP} would give the most accurate scheme, and they can be found by maximizing  $\theta_{j+\frac12}\in [0,1]$ under the constraints \eqref{eq:MMP_NF_BP}, which is  
 a globally coupled constrained optimization thus
expensive to solve. As pointed out in \cite{liska2010optimization}, the flux correction approach can be regarded as seeking easier alternatives for parameters $\theta_{j+\frac12}$, to avoid solving the constrained global optimization problem.
There are several different ways to efficiently compute limiting parameters $ \theta_{j+\frac12}$ to satisfy  \eqref{eq:MMP_NF_BP}, resulting in different flux-correction limiters and IDP schemes. In the following subsections, we will review several popular methods.

\begin{remark}
\label{rmk:iFCT}
    To improve accuracy in the flux corrected scheme \eqref{eq:forward_EulerIDP}, one  method is to upgrade the low order IDP flux $\bff^L$ in \eqref{eq:forward_EulerIDP} by any higher order accurate IDP flux. 
The iterative FCT method \cite{schar1996synchronous} is such a simple  approach, by replacing the low order IDP flux $\bff^L$ in \eqref{eq:forward_EulerIDP} by the IDP flux  $\hat \bff$ in \eqref{flux_BP} recursively as follows. 
With a given first order IDP flux $\bff^L$, define $\bff^{L,0}=\bff^L$, and find $\theta^m$ such that
\[  \hat {\bf f}^{L,m+1}_{j+\frac12}
    = \theta^m_{j+\frac12} \Big(\hat {\bf f}^H_{j+\frac12} - \hat {\bf f}^{L,m}_{j+\frac12}\Big)
    + \hat {\bf f}^{L,m}_{j+\frac12},\]
is an IDP flux, for $m=0,1,2,\cdots, M$, then  ${\bf u}_j^{n+1} = {\bf u}_j^n - \lambda \Big(\hat {\bf f}_{j+\frac12} - \hat {\bf f}_{j-\frac12}\Big)$ {  with $\hat {\bf f}_{j+\frac12}:=\hat {\bf f}^{L,M+1}_{j+\frac12}$} gives a more accurate 
IDP scheme than \eqref{eq:forward_EulerIDP}, but with a higher computational cost.
\end{remark}

\subsection{Zalesak's FCT limiter for scalar conservation laws}
This method was originally proposed in \cite{zalesak1979fully} for preserving local maximum principle of scalar conservation laws. A detailed description of design principles behind FCT algorithms for structured grids can be found in Zalesak's book chapter \cite{zalesak2012design}. See \cite{kuzmin2012flux, zalesak1979fully,Kuzmin2022BoundPreserving} for the general version for multidimensional problems and unstructured meshes .  
As mentioned in Section \ref{sec:barrier}, if enforcing a maximum principle
like $\min_j u^n_j \leq u^{n+1}_j\leq \max_j u^n_j$, then it would be at most second order accurate for smooth solutions.
Thus for higher order accuracy, we consider enforcing the bound-preserving property, i.e., the invariant domain is a simple interval
 $G=[m,M]$ with  $m$ and $M$ being the lower and upper bounds of the initial condition. 
 
\subsubsection{The original version of Zalesak's  FCT limiter}

 Zalesak's \cite{zalesak1979fully} classical FCT algorithm determines the parameter 
 $\theta_{j+\frac12}$   as follows:
 \begin{itemize}
    \item Define 
    \[
    P^{+}_{j} = \max\left\{0, -\delta_{j+1/2}\right\} + \max\left\{0, \delta_{j-1/2}\right\},
    \]
    \[
    P^{-}_{j} = \min\left\{0, -\delta_{j+1/2}\right\} + \min\left\{0, \delta_{j-1/2}\right\}.
    \]

    \item Compute 
    \begin{equation}\label{eq:ZalesakQ}
    Q^{+}_{j} = \frac{1}{\lambda}\left(M- u^{n+1,L}_{j}\right), \quad 
    Q^{-}_{j} = \frac{1}{\lambda}\left(m - u^{n+1,L}_{j}\right).
    \end{equation}

    \item Calculate  
    \[
    R^{+}_{j} = \min\left\{1, \frac{Q^{+}_{j}}{P^{+}_{j}}\right\}, \quad 
    R^{-}_{j} = \min\left\{1, \frac{Q^{-}_{j}}{P^{-}_{j}}\right\}.
    \]

    \item Obtain the IDP limiting parameter
    \begin{equation}
       \theta_{j+1/2} = 
    \begin{cases}
        \min\left\{R^{+}_{j+1}, R^{-}_{j}\right\}& \text{if } \delta_{j+1/2} \geq 0, \\
        \min\left\{R^{-}_{j+1}, R^{+}_{j}\right\}& \text{if } \delta_{j+1/2} < 0.
    \end{cases}  \label{zelask-1d-scalar}
    \end{equation}
\end{itemize}

The formula \eqref{zelask-1d-scalar} can be derived from enforcing constraints, e.g., see \cite[Section 3]{liska2010optimization}. Thus \eqref{eq:forward_EulerIDP} with \eqref{zelask-1d-scalar} gives the bound-preserving property $u^{n+1}_j\in [m, M]$.

\subsubsection{Parametrized flux limiters}\label{sec:Xuscalar} 
In order to have a better understanding of the Zalesak's algorithm for computing $\theta_{j+\frac12}$, we first consider an alternative way to find $\theta_{j+\frac12}$ in \eqref{eq:MMP_NF_BP}.
It is possible to decouple the constraints of $\{ \theta_{j+\frac12} \}$ in \eqref{eq:MMP_NF_BP} by a parametrized method,  proposed in 
 \cite{jiang2013parametrized,xu2014parametrized}.  
Specifically, the parametrized method seeks a group of locally defined parameters $\Lambda_{j+\frac12}$ such that the IDP property is maintained for any
 $\theta_{j+\frac12} \in \Big[0, \Lambda_{j+\frac12}\Big]$. Then for the sake of minimal correction to maintain high order accuracy, one should simply take $\theta_{j+\frac12}=\Lambda_{j+\frac12}$.

For a finite difference scheme \eqref{eq:forward_Euler} at a grid point $x_j$, define an interval $I_j=[x_{j-\frac12}, x_{j+\frac12}]$. Parameters $\Lambda_{-\frac12,I_j}$ and $\Lambda_{+\frac12,I_j}$ will be constructed such that
\begin{align}\label{bound-theta}
    \theta_{j+\frac{1}{2}} \in \left[0, \Lambda_{+\frac12,I_j}\right] \cap
    \left[0, \Lambda_{-\frac12,I_{j+1}}\right],  
\end{align}
is sufficient for the scheme \eqref{eq:forward_EulerIDP} to preserve the desired bounds.  In \eqref{bound-theta}, 
 $\theta_{j+\frac{1}{2}}$ at $x_{j+\frac12}$ satisfies some constraints in $I_j=[x_{j-\frac12}, x_{j+\frac12}]$ and also $I_{j+1}=[x_{j+\frac12}, x_{j+\frac32}]$. The subscripts $+\frac12$ and $-\frac12$ in $\Lambda_{\pm\frac12,I_{j}}$ can be understood as a shift from the cell center of $I_j$, e.g., both  $\Lambda_{+\frac12,I_j}$ and  $\Lambda_{-\frac12,I_{j+1}}$ correspond to the grid point $x_{j+\frac12}$.

Define 
\begin{align}\label{eq:XuGamma}
    \Gamma_j^M= M - u^n_j + \lambda\left( \hat f^L_{j+\frac{1}{2}} - \hat f^L_{j-\frac{1}{2}} \right), \quad
    \Gamma_j^m= m - u^n_j + \lambda\left( \hat f^L_{j+\frac{1}{2}} - \hat f^L_{j-\frac{1}{2}} \right).
\end{align}
The IDP property of a first order monotone scheme yields
\begin{align}
   \Gamma_j^M \geq 0,  \quad \Gamma_j^m \leq 0.
\end{align}
To maintain $u_j^{n+1} \in [m, M]$,  $\theta_{j+\frac12}$ must satisfy \eqref{eq:MMP_NF_BP} 
with $g(u)=M-u$ and $g(u)=u-m$, respectively, i.e., 
\begin{align}
    L_M(\theta_{j-\frac12},\theta_{j+\frac12}):=& \lambda \theta_{j-\frac{1}{2}} \delta_{j-\frac{1}{2}}
    - \lambda \theta_{j+\frac{1}{2}} \delta_{j+\frac{1}{2}}
    - \Gamma_j^M \leq 0, \label{eq:MMP_NF_BP-1}\\
   L_m(\theta_{j-\frac12},\theta_{j+\frac12}):= & \lambda \theta_{j-\frac{1}{2}} \delta_{j-\frac{1}{2}}
    - \lambda \theta_{j+\frac{1}{2}} \delta_{j+\frac{1}{2}}
    - \Gamma_j^m \geq 0. \label{eq:MMP_NF_BP-2-Lm}
\end{align}
The parameters $\Lambda_{\pm\frac12,I_j}^M$ and $\Lambda_{\pm\frac12,I_j}^m$ are constructed as follows:

\paragraph{Upper Bound Preservation}  
   The parameters $\Lambda_{\pm\frac12,I_j}^M$ are defined to satisfy \eqref{eq:MMP_NF_BP-1}.  
    \begin{enumerate}[label=(\alph*)]
        \item If $\delta_{j-\frac{1}{2}} \leq 0$ and $\delta_{j+\frac{1}{2}} \geq 0$, define $ \left(\Lambda_{-\frac12,I_j}^M, \Lambda_{+\frac12,I_j}^M\right) = (1,1).$
        \item If $\delta_{j-\frac{1}{2}} \leq 0$ and $\delta_{j+\frac{1}{2}} < 0$, define  
        \begin{align}
            \left(\Lambda_{-\frac12,I_j}^M, \Lambda_{+\frac12,I_j}^M\right) = \left(1,\min\left\{1, \frac{\Gamma_j^M}{-\lambda \delta_{j+\frac{1}{2}}+\epsilon}\right\}\right).
        \end{align}
        \item If $\delta_{j-\frac{1}{2}} > 0$ and $\delta_{j+\frac{1}{2}} \geq 0$, define
        \begin{align}
            \left(\Lambda_{-\frac12,I_j}^M, \Lambda_{+\frac12,I_j}^M\right) = \left(\min\left\{1, \frac{\Gamma_j^M}{\lambda \delta_{j-\frac{1}{2}}+\epsilon}\right\},1\right).
        \end{align}
        \item If $\delta_{j-\frac{1}{2}} > 0$ and $\delta_{j+\frac{1}{2}} < 0$:  
                    \begin{align}
                \left(\Lambda_{-\frac12,I_j}^M, \Lambda_{+\frac12,I_j}^M\right) = \left(  \Lambda_0, \Lambda_0 \right),\quad \Lambda_0 = \min \left \{  1, \frac{\Gamma_j^M}{\lambda \delta_{j-\frac{1}{2}} - \lambda \delta_{j+\frac{1}{2}}+\epsilon} \right \}.
            \end{align}
    \end{enumerate}
        Here, $\epsilon$ is a small positive parameter, which is slightly above machine zero, to prevent division by zero. 
For convenience, they can also be equivalently  written as
\begin{subequations}
    \label{xu-limiter-equivalent-form}
    \begin{equation}\label{eq:OB1}
     \Lambda_{+\frac12,I_j}^M =\begin{cases}
         1, \qquad &\mbox{if }\quad  \delta_{j+\frac12} \ge 0\\
         \min\left\{ 1,  \frac{\Gamma_j^M}{\lambda \max \{0, \delta_{j-\frac{1}{2}} \} - \lambda \delta_{j+\frac{1}{2}}+\epsilon}   \right \}, \qquad &\mbox{if }\quad  \delta_{j+\frac12} < 0
     \end{cases},
    \end{equation}    
      \begin{equation}\label{eq:OB3}
     \Lambda_{-\frac12,I_j}^M =\begin{cases}
         1, \qquad &\mbox{if }\quad  \delta_{j-\frac12} \le 0\\
         \min\left\{ 1,  \frac{\Gamma_j^M}{\lambda  \delta_{j-\frac{1}{2}} - \lambda \min\{ 0, \delta_{j+\frac{1}{2}} \} +\epsilon}   \right \}, \qquad &\mbox{if }\quad  \delta_{j-\frac12} > 0
     \end{cases}.
    \end{equation}  
\end{subequations}

The first fact is that any smaller $\theta_{j\pm\frac12}$ than $\Lambda_{\pm\frac12,I_j}^M$ also satisfy \eqref{eq:MMP_NF_BP-1}. Notice that \Cref{lem:smallertheta} does not imply this result.
\begin{lemma}
\label{lem:xulimiter_upperbound}
The upper bound \eqref{eq:MMP_NF_BP-1} is satisfied for   
   $\theta_{j-\frac{1}{2}} \in \left[0,\Lambda_{-\frac12,I_j}^M\right]$ and   $\theta_{j+\frac{1}{2}} \in \left[0,\Lambda_{+\frac12,I_j}^M\right]$.
\end{lemma}
\begin{proof}
Regard $\theta_{j-\frac12}$ and $\theta_{j+\frac12}$ as unknowns and consider a line equation 
    \begin{align}
    \label{eqn:blueline}
        L_M(\theta_{j-\frac12},\theta_{j+\frac12}):=\lambda \theta_{j-\frac{1}{2}}  \delta_{j-\frac{1}{2}}  
        - \lambda \theta_{j+\frac{1}{2}}\delta_{j+\frac{1}{2}}  - \Gamma_j^M = 0.
    \end{align} 
Then we discuss it case by case. For case (a), $L_M$ is a decreasing function, thus  $L_M(\theta_{j-\frac12},\theta_{j+\frac12})\leq L_M(0,0)=-\Gamma^M_j\leq 0$.

\begin{center}
        \begin{tikzpicture}[scale=0.55, xscale=1.13,yscale=0.96]
            \draw [very thick,-latex] (-2,0) -- (3,0);
            \draw [very thick,-latex] (0,-2) -- (0,3);

            \draw [ultra thick, cyan] (-1.8,-1) -- (1,2.5);
            \filldraw [green,opacity=0.4] (0,0) rectangle (1,1.22);
            \node [black, scale=0.98] at (2.5,1.6) {$(\Lambda_{-\frac12,I_j}^M, \Lambda_{+\frac12,I_j}^M)$};
            \node [black, scale=0.98] at (1,-0.6) {$(1,0)$};
            \node [black, scale=0.98] at (3,-0.9) {$\theta_{j-\frac12}$};
            \node [black, scale=0.98] at (-0.4,3.5) {$\theta_{j+\frac12}$};
            
            \filldraw [black] (1,0) circle (2.pt)
                              (1,1.22) circle (2.pt);
        \end{tikzpicture}
        \begin{tikzpicture}[scale=0.55, xscale=1.13,yscale=0.96]
            \draw [very thick,-latex] (-2,0) -- (3,0);
            \draw [very thick,-latex] (0,-2) -- (0,3);

            \draw [ultra thick, cyan] (-1,-1.8) -- (2.5,1.3);
            \filldraw [green,opacity=0.4] (0,0) rectangle (1,1.22);
            \node [black, scale=0.98] at (2.2,1.8) {$(\Lambda_{-\frac12,I_j}^M, \Lambda_{+\frac12,I_j}^M)$};
            \node [black, scale=0.98] at (-0.8,1.2) {$(0,1)$};
            \node [black, scale=0.98] at (3,-0.9) {$\theta_{j-\frac12}$};
            \node [black, scale=0.98] at (-0.4,3.5) {$\theta_{j+\frac12}$};
            
            \filldraw [black] (0,1.22) circle (2.pt)
                              (1,1.22) circle (2.pt);
        \end{tikzpicture}
        \begin{tikzpicture}[scale=0.55, xscale=1.11,yscale=0.96]
            \draw [very thick,-latex] (-2,0) -- (3,0);
            \draw [very thick,-latex] (0,-2) -- (0,3);

            \draw [ultra thick, cyan] (-1,3) -- (2.65,-0.2);
            \filldraw [green,opacity=0.4] (0,0) rectangle (1,1.22);
            \filldraw [black] (1,1.22) circle (2.pt);
            \node [black, scale=0.98] at (3.2,1.2) {$(\Lambda_{-\frac12,I_j}^M, \Lambda_{+\frac12,I_j}^M)$};
            \node [black, scale=0.98] at (3,-0.9) {$\theta_{j-\frac12}$};
            \node [black, scale=0.98] at (-0.4,3.5) {$\theta_{j+\frac12}$};
            
        \end{tikzpicture}
    \end{center}
 The cases  (b), (c) and (d) are illustrated in the figures, in which
     the blue line \eqref{eqn:blueline} is the zero line, i.e., $L_M(\theta_{j-\frac12},\theta_{j+\frac12})=0$. Since $L_M(0,0)\leq 0$ and the green rectangle is on the same side as $(0,0)$ w.r.t. the zero line,
    any $(\theta_{j-\frac12}, \theta_{j+\frac12})$ in the green rectangle achieves $L_M(\theta_{j-\frac12},\theta_{j+\frac12})\leq 0$. The proof is concluded.
\end{proof}

\paragraph{Lower Bound Preservation}    Similarly,  
the parameters $\Lambda_{\pm\frac12,I_j}^m$ are defined to satisfy \eqref{eq:MMP_NF_BP-2-Lm}.   
    \begin{enumerate}[label=(\alph*)]
        \item If $\delta_{j-\frac{1}{2}} \geq 0$, $\delta_{j+\frac{1}{2}} \leq 0$, define $ \left(\Lambda_{-\frac12,I_j}^m, \Lambda_{+\frac12,I_j}^m\right) = (1,1).$
        \item If $\delta_{j-\frac{1}{2}} \geq 0$, $\delta_{j+\frac{1}{2}} > 0$, define  $\left(\Lambda_{-\frac12,I_j}^m, \Lambda_{+\frac12,I_j}^m\right) = \left(1,\min\left\{1, \frac{\Gamma_j^m}{-\lambda \delta_{j+\frac{1} {2}} - \epsilon}\right\}\right).$
   
        \item If $\delta_{j-\frac{1}{2}} < 0$,  $\delta_{j+\frac{1}{2}} \leq 0$,  define $\left(\Lambda_{-\frac12,I_j}^m, \Lambda_{+\frac12,I_j}^m\right) = \left(\min\left\{1, \frac{\Gamma_j^m}{\lambda \delta_{j-\frac{1}{2}} -\epsilon}\right\},1\right).$
        \item If $\delta_{j-\frac{1}{2}} < 0$, $\delta_{j+\frac{1}{2}} > 0$, define $\Big(\Lambda_{-\frac12,I_j}^m, \Lambda_{+\frac12,I_j}^m\Big) = \Bigg(  
                \Lambda_0, \Lambda_0 \Bigg)$ with           
              $ \Lambda_0 = \min \left \{ 1, \frac{\Gamma_j^m}{\lambda \delta_{j-\frac12} - \lambda \delta_{j+\frac12}-\epsilon} \right \}.$
    \end{enumerate}

Similar to \Cref{lem:xulimiter_upperbound},
any smaller $\theta$ would also enforce the lower bound, i.e.,
   $\theta_{j-\frac{1}{2}} \in \left[0,\Lambda_{-\frac12,I_j}^m\right]$ and $\theta_{j+\frac{1}{2}} \in \left[0,\Lambda_{+\frac12,I_j}^m\right]$ satisfy \eqref{eq:MMP_NF_BP-2-Lm}.

    The final limiting parameter combines upper and lower bounds  
    \begin{align}\label{eq:MPP_limiting-parameter}
         \Lambda_{+\frac12,I_j} = \min\Big\{\Lambda_{+\frac12,I_j}^M, \Lambda_{+\frac12,I_j}^m\Big\}, \quad
\Lambda_{-\frac12,I_{j+1}} = \min\Big\{\Lambda_{-\frac12,I_{j+1}}^M, \Lambda_{-\frac12,I_{j+1}}^m\Big\}.
    \end{align}   Implied by \Cref{lem:xulimiter_upperbound}, we have
\begin{thm}
\label{thm:xu-limiter-bp}
   Let $\Lambda_{j+\frac{1}{2}}: 
        = \min\{\Lambda_{+\frac12,I_j},\Lambda_{-\frac12,I_{j+1}}\}$, then  $\theta_{j+\frac12}\in [0,\Lambda_{j+\frac{1}{2}}]$  satisfy both \eqref{eq:MMP_NF_BP-1} and \eqref{eq:MMP_NF_BP-2-Lm}. In particular,
the modified flux \eqref{flux_BP} with  $\theta_{j+\frac12} =  \Lambda_{j+\frac{1}{2}}$ is IDP for $G=[m,M]$.
\end{thm}

\subsubsection{Equivalence of Zalesak's FCT limiter and the parametrized flux limiter} 

Although the  parametrized flux limiter has a different formula, 
its final result for $\theta_{j+\frac12}$ is actually equivalent to the Zalesak's FCT limiter for enforcing bounds.

\begin{theorem}
    The parameterized limiter   is equivalent to Zalesak's FCT limiter, i.e., \eqref{zelask-1d-scalar} is equal to $\Lambda_{j+\frac{1}{2}} 
        = \min\{\Lambda_{+\frac12,I_j},\Lambda_{-\frac12,I_{j+1}}\}$.
\end{theorem}
\begin{proof}
Without loss of generality, we only consider the preservation of the upper bound $u \le M$. For considering only the upper bound, the limiting parameter $\theta_{j+\frac12}$ given by Zalesak's FCT limiter becomes 
\begin{equation}\label{eq:Zalesak}
    \theta_{j+1/2} = 
    \begin{cases}
        R^{+}_{j+1},& \text{if } \delta_{j+1/2} \geq 0, \\
        R^{+}_{j},& \text{if } \delta_{j+1/2} < 0,
    \end{cases}
\end{equation}
where 
$
R^{+}_{j} = \min\left\{1, \frac{Q^{+}_{j}}{P^{+}_{j}}\right\},   
Q^{+}_{j} = \frac{\Delta x}{\Delta t}\left(M- u^{n+1,L}_{j}\right),
$
and 
    \[
    P^{+}_{j} = \max\left\{0, -\delta_{j+1/2}\right\} + \max\left\{0, \delta_{j-1/2}\right\}. 
    \]

    First, $\Gamma_j^M$ defined in \eqref{eq:XuGamma} simply
    $
\Gamma_j^M = \lambda Q^{+}_{j}
$, with which \eqref{xu-limiter-equivalent-form} gives 
$$
\begin{aligned}
\Lambda_{+\frac12,I_j}^M & =1, \quad \mbox{if }\quad  \delta_{j+\frac12} \ge 0\\ 
\Lambda_{-\frac12,I_{j+1}}^M & =\min\left\{ 1,  \frac{\Gamma_{j+1}^M}{\lambda  \delta_{j+\frac{1}{2}} - \lambda \min\{ 0, \delta_{j+\frac{3}{2}} \} +\epsilon}   \right \}
\\
&= \min\left\{ 1,  \frac{Q_{j+1}^+}{ \delta_{j+\frac{1}{2}} +  \max\{ 0, -\delta_{j+\frac{3}{2}} \}  }   \right \} = R_{j+1}^+,\qquad \mbox{if }\quad  \delta_{j+\frac12} \ge 0
\end{aligned}
$$
thus
$$
\min\Big\{\Lambda_{+\frac12,I_j}^M, \Lambda_{-\frac12,I_{j+1}}^M\Big\} = R_{j+1}^+,
$$
which implies the equivalence between the parameterized limiter and Zalesak's FCT limiter for the case $\delta_{j+\frac12}\geq 0$. 
Similarly,  \eqref{xu-limiter-equivalent-form} gives
$$
\begin{aligned}
\Lambda_{+\frac12,I_j}^M &=\min\left\{ 1,  \frac{\Gamma_j^M}{\lambda \max \{0, \delta_{j-\frac{1}{2}} \} - \lambda \delta_{j+\frac{1}{2}}+\epsilon}   \right \} \\
& = 
\min\left\{ 1,  \frac{Q_j^+}{ \max \{0, \delta_{j-\frac{1}{2}} \} - \delta_{j+\frac{1}{2}}}   \right \} = 
R_j^+,\quad \mbox{if }\quad  \delta_{j+\frac12} < 0\\
\Lambda_{-\frac12,I_{j+1}}^M &= 1,\quad \mbox{if }\quad  \delta_{j+\frac12} < 0,
\end{aligned}
$$
thus
$
\min\Big\{\Lambda_{+\frac12,I_j}^M, \Lambda_{-\frac12,I_{j+1}}^M\Big\} = R_{j}^+,
$
which implies the equivalence between the parameterized limiter and Zalesak's FCT limiter when $\delta_{j+\frac12} < 0$. 
The proof is completed. 
\end{proof}

\subsubsection{High order accuracy}
 By \Cref{thm:xu-limiter-bp}, the scheme \eqref{eq:forward_EulerIDP} with the Zalesak's FCT flux correction for scalar equations preserves $G=[m, M]$.
 In practice, such a flux correction can be applied to many high order schemes such as finite difference, finite volume and discontinuous Galerkin (DG) schemes, and high order accuracy can be observed numerically for sufficiently small time steps. For finite volume schemes solving 1D linear equation, the truncation error of the flux correction in this subsection can be shown high order accurate for smooth solutions under a reasonable time step constraint \cite{xu2014parametrized}.   

 For high order finite difference schemes with SSP Runge-Kutta methods, if applying the flux correction to each time stage in Runge-Kutta methods, then high order accuracy can be maintained only under extremely small time steps such as $\Delta t=\mathcal O(\Delta x^{1.5})$, see \cite{xu2014parametrized}. To recover high order accuracy under a reasonable time step, one simple remedy   is to apply the flux correction only at the final time stage of a Runge-Kutta method, with which  third order accuracy of a finite difference scheme can be proven for solving 1D linear equation with a practical time step \cite{xiong2013parametrized}. Such a flux correction    can be applied to any high order time discretization such as Lax-Wendroff time stepping, and can also be extended to explicit time stepping for a convection dominated diffusion equation \cite{xiong2015high-DG-confussion}.

\subsection{Parametrized flux limiters for hyperbolic systems}
\label{sec:parameterizedsystem}

The parametrized flux limiting method can be extended to several hyperbolic systems. 
As an example, we review how it can be applied to a finite difference scheme  for the compressible Euler equations \cite{xiong2016parametrized}, which is one kind of extension of Zalesak's FCT limiter to systems. See also 
\cite{lohmann2016synchronized} and references therein for another extension of Zalesak's FCT limiter to systems.
For simplicity, we only consider the one dimensional case since the multiple dimensional scheme can be  defined similarly in a dimension by dimension fashion for a finite difference scheme. 

For the positivity of density $\rho$ and pressure $p$ in the Euler equations, introduce two threshold parameters:  
$\epsilon_{\rho} = \min_j(\rho_j^{n+1,L},10^{-13})$ and $\epsilon_p = \min_j(p_j^{n+1,L},10^{-13})$, where $\rho_j^{n+1,L}$ and $p_j^{n+1,L}$ denote density and pressure computed by a first order IDP scheme. 
Let $(\hat f^{\rho,L}, \hat f^{m,L}, \hat f^{E,L})^\top$ represent the components of the first order IDP flux $\hat {\bf f}^L$. Similarly,  $\hat {\bf f}^H = (\hat f^{\rho,H}, \hat f^{m,H}, \hat f^{E,H})^\top$ denotes a high order numerical flux in a finite difference scheme, e.g., finite difference WENO scheme. Let $\hat {\bold f} = (\hat f^{\rho}, \hat f^{m}, \hat f^{E})^\top$ be the corrected flux.
The method proceeds in two steps:

  First,  follow the scalar case to determine 
    the limiting parameters $\theta_{j\pm\frac12}$ to enforce the positivity of density, i.e. to maintain 
    \begin{align*}
        \rho_j^{n+1} = \rho_j^n - \lambda\left( \hat f^{\rho}_{j+\frac12} - \hat f^{\rho}_{j-\frac12} \right) \geq \epsilon_\rho.
    \end{align*}  
     Obtain $(\Lambda_{-\frac12,I_j}^\rho, \Lambda_{+\frac12,I_j}^\rho)$ and define a rectangular box region:
    \begin{align}
        S_{\rho} = \left\{ (\theta_{j-\frac12}, \theta_{j+\frac12}): 0 \leq \theta_{j-\frac12} \leq \Lambda_{-\frac12,I_j}^{\rho},\ 0 \leq \theta_{j+\frac12} \leq \Lambda_{+\frac12,I_j}^{\rho} \right\}.
    \end{align}  
 Let $A^1 = (0, \Lambda_{+\frac12,I_j}^{\rho})$, $A^2 = (\Lambda_{-\frac12,I_j}^{\rho}, 0)$, and $A^3 = (\Lambda_{-\frac12,I_j}^{\rho}, \Lambda_{+\frac12,I_j}^{\rho})$ denote vertices of $S_{\rho}$, as shown in  Figure~\ref{Fig:decoupling_rectangle}.

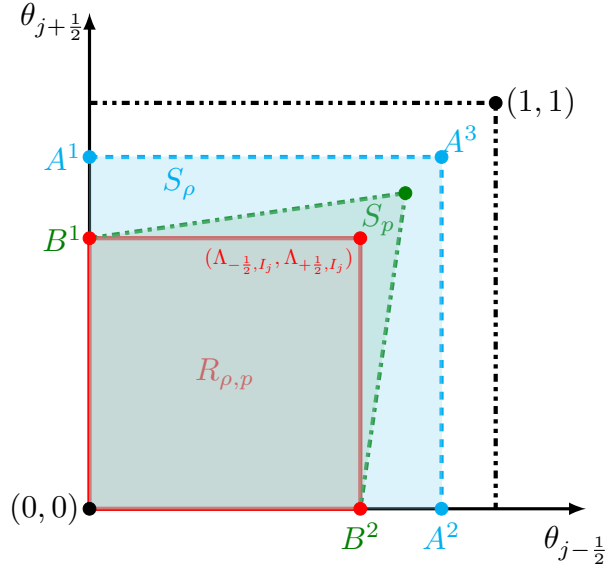
\begin{figure}[!htb]
    \begin{center}
        \begin{tikzpicture}[xscale=1.2,yscale=1.2]
            \draw [very thick,-latex] (0,0) -- (5.5,0);
            \draw [very thick,-latex] (0,0) -- (0,5.5);

            \draw [ultra thick, red] (0,0) -- (3,0) -- (3,3) -- (0,3) -- cycle;
            \filldraw [red!40,opacity=0.3] (0,0) rectangle (3,3);
            \node[red, scale=1.2] at (1.5,1.5) {$R_{\rho,p}$};

            \draw [ultra thick, dashdotted, darkgreen] (0,3) -- (3.5,3.5) -- (3,0);
            \filldraw [darkgreen!40,opacity=0.3] (0,3) -- (3.5,3.5) -- (3,0) -- (0,0) -- (0,3);
            \node[darkgreen, scale=1.2] at (3.2,3.2) {$S_p$};

            \draw [ultra thick, dashed, cyan] (0,3.9) -- (3.9,3.9) -- (3.9,0);
            \filldraw [cyan!40,opacity=0.3] (0,3.9) -- (3.9,3.9) -- (3.9,0) -- (0,0) -- (0,3.9);
            \node[cyan, scale=1.2] at (1.0,3.6) {$S_{\rho}$};

            \draw [ultra thick, dashdotdotted, black] (0,4.5) -- (4.5,4.5) -- (4.5,0);

            \filldraw [black] (0,0) circle (2pt);
            \filldraw [black] (4.5,4.5) circle (2pt);
            \filldraw [red] (3,0) circle (2pt);
            \filldraw [red] (3,3) circle (2pt);
            \filldraw [red] (0,3) circle (2pt);
            \filldraw [darkgreen] (3.5,3.5) circle (2pt);
            \filldraw [cyan] (3.9,0) circle (2pt);
            \filldraw [cyan] (3.9,3.9) circle (2pt); 
            \filldraw [cyan] (0,3.9) circle (2pt);

            \node[darkgreen, scale=1.2] at (-0.3,3) {$B^1$};
            \node[darkgreen, scale=1.2] at (3,-0.3) {$B^2$};
            
            \node[cyan, scale=1.2] at (3.9,-0.3) {$A^2$};
            \node[cyan, scale=1.2] at (4.1,4.1) {$A^3$};
            \node[cyan, scale=1.2] at (-0.3,3.9) {$A^1$};

            \node [red, scale=0.8] at (2.1,2.75) {$(\Lambda_{-\frac12,I_j}, \Lambda_{+\frac12,I_j})$};
            \node [black, scale=1.2] at (-0.5,0) {$(0,0)$};
            \node [black, scale=1.2] at (5,4.5) {$(1,1)$};
            \node [black, scale=1.2] at (5.4,-0.4) {$\theta_{j-\frac12}$};
            \node [black, scale=1.2] at (-0.4,5.4) {$\theta_{j+\frac12}$};
        \end{tikzpicture}
    \end{center}
    \caption{An illustration of 
the parameters for enforcing positivity of both density and pressure:
the first step is to find a box region $S_\rho$ (cyan color rectangle bounded by dashed lines) with four vertices $(0,0), A^1, A^2, A^3$, the second step is to find $B^i=rA^i$ with $r\in[0,1]$ such that $p(B^i)\geq\epsilon_p$ for each $i$. The convex hull of vertices 
$(0,0), B^1, B^2, B^3$ is the green polygon  $S_p\subset S_{\rho, p}$, and 
 the largest rectangle inside $S_p$ with two sides along axes is the red rectangle $R_{\rho,p}\subset S_{\rho, p}$.}
    \label{Fig:decoupling_rectangle}
\end{figure}

The second step is to enforce pressure positivity without losing positivity of density.
For the scheme with corrected flux \eqref{eq:forward_EulerIDP}, its numerical solution $\bu^{n+1}_j$ can be
regarded as a function of $\theta_{j\pm\frac12}$. For the ideal gas EOS \eqref{EOS}, the pressure function is concave w.r.t. $\bu=\begin{pmatrix}\rho & m & E\end{pmatrix}^T$, thus the proof of \Cref{lem:smallertheta}  implies the following set is a convex set, 
    \begin{align}
        S_{\rho,p} = \left\{ (\theta_{j-\frac12}, \theta_{j+\frac12}) \in S_{\rho}:  p_j^{n+1}(\theta_{j-\frac12}, \theta_{j+\frac12}) = (\gamma-1) \left( E_j^{n+1} - \frac{(m_j^{n+1})^2}{2\rho_j^{n+1}} \right) \geq \epsilon_p \right\}.
    \end{align}  

     Next, a box region $R_{\rho, p} $ in the plane of two variables $(\theta_{j-\frac12}, \theta_{j+\frac12}) $ for enforcing positivity of both density and pressure
   is constructed as follows and also illustrated in Figure~\ref{Fig:decoupling_rectangle}:  
    \begin{enumerate}
        \item For each vertex $A^\ell$, find $B^\ell = rA^\ell$ with $r \in [0,1]$ such that $p(B^\ell) \geq \epsilon_p$. The convex hull of $(0,0)$ and $B^\ell$ forms $S_p\subset S_{\rho, p}$.
        \item Define the rectangle $ R_{\rho, p} = \left[0,\Lambda_{-\frac12,I_j}\right] \times \left[0,\Lambda_{+\frac12,I_j}\right],$  
        with limiting parameters:  
        \begin{align}
            (\Lambda_{-\frac12,I_j}, \Lambda_{+\frac12,I_j}) = \left(\min(B_1^2, B_1^3), \min(B_2^1, B_2^3)\right).
        \end{align}  
    \end{enumerate}

The construction above and convexity of $ S_{\rho,p}$ ensure that  $R_{\rho, p}\subset S_{\rho,p}$.  
Finally, the limiting parameter is given by $ \theta_{j+\frac12} = \min\left(\Lambda_{-\frac12,I_{j+1}}, \Lambda_{+\frac12,I_j}\right)$, and such $\theta_{j+\frac 12}$ is inside all $R_{\rho, p}$ constructed thus IDP is achieved in \eqref{eq:forward_EulerIDP}.

Such a flux limiting method can be extended to other schemes such as  DG methods \cite{xiong2015high-DG-confussion}  and  finite volume schemes on unstructured meshes  \cite{christlieb2015high}.

\subsection{A simple flux limiting based on \Cref{prop:LFS}}

\label{sec:hu-shu}

In the literature, there are other simpler decoupling methods to find sufficient conditions for satisfying \eqref{eq:MMP_NF_BP-2}. We first review a simple flux limiting method
 introduced by Hu, Adams and Shu in \cite{hu-adam-2013positivity} for compressible Euler equations, which can be extended to
hyperbolic systems with $G$ satisfying \Cref{prop:LFS} (and  \Cref{prop:LFS-2D} in multiple dimensions).  It should be noted that scalar conservation laws with $G=[U_{\min}, U_{\max}]$ do not satisfy \Cref{prop:LFS} in general. 

For the compressible Euler equations, \Cref{prop:LFS} holds, implied by \Cref{lf-fact-NS}.
In the Hu--Adams--Shu method \cite{hu-adam-2013positivity}, the low order IDP flux is chosen as the first order Lax--Friedrichs flux, e.g.,
\begin{equation}
    \label{LF-flux-FLUXlimiter}
{\bf f}^{L}_{j+\frac12} = \frac{1}{2}[\bff(\bu^n_{j})+\bff(\bu^n_{j+1})-\alpha (\bu^n_{j+1}-\bu^n_{j})].
\end{equation}
The first order Lax--Friedrichs scheme can be written as 
\[ \bu^{n+1,L}_j=\frac12{\bf u}_j^{L,+}+\frac12{\bf u}_j^{L,-},\qquad {\bf u}_j^{L,\pm} := {\bf u}_j^n\mp2\lambda\hat{\bold f}_{j\pm\frac12}^{L}.\]
 Under the CFL condition $\alpha \lambda \le \frac12$ and the assumption $\bu^n_j, \bu^n_{j\pm 1}\in G$,  \Cref{prop:LFS} implies
 \begin{equation}
     \label{splitting_IDP_HuAdamShu}
     {\bf u}_j^{L,\pm} = (1-2\alpha \lambda)\bu^n_j + \alpha \lambda \left(  {\bf u}^n_j \mp \frac{ {\bf f} ({\bf u}^n_j) } {\alpha}    \right) + \alpha \lambda \left(  {\bf u}^n_{j\pm 1} \mp \frac{ {\bf f} ({\bf u}^n_{j \pm 1}) } {\alpha}    \right)\in G. 
 \end{equation}

In order to decouple the constraints of $\{ \theta_{j+\frac12} \}$ in \eqref{eq:MMP_NF_BP}, 
 the Hu--Adams--Shu method decomposes 
 the original high order scheme \eqref{eq:forward_Euler} into two parts: 
\begin{align}\label{eq:convex_combination-hu}
\bu_j^{n+1,H} = \frac12\Big(\bu_j^n + 2\lambda \hat {\bf f}_{j-\frac12}^H \Big)
                    + \frac12\Big(\bu_j^n - 2\lambda \hat {\bf f}_{j+\frac12}^H \Big)
                    := \frac12 \bu_j^{H,-} + \frac12 \bu_j^{H,+}.
\end{align} 

Then the scheme \eqref{eq:forward_EulerIDP} can be written as
\begin{equation} 
\label{hu-shu-decomp}
    \bu^{n+1}_j=\frac12\left[(1-\theta_{j-\frac12})\bu_j^{L,-}+ \theta_{j-\frac12}\bu_j^{H,-}\right]+\frac12\left[(1-\theta_{j+\frac12})\bu_j^{L,+}+ \theta_{j+\frac12}\bu_j^{H,+}\right],
\end{equation}
and the basic idea of the Hu--Adams--Shu method is to use ${\bf u}_j^{L,\pm} \in G$ to limit $\bu_j^{H,\pm}$.
 
For simplicity, assume $G$ in \eqref{eq:ASS-G} is defined by concave functions $g_i(\bu)>0, i=1,\cdots, N$, thus $g_i$ satisfies the Jensen inequality \eqref{eq:Jensen4g}, such as the positivity of density or pressure in the compressible Euler equations.  Define
\begin{subequations}
\label{alg:theta}
    \begin{equation}
       \theta_{j+\frac12}=\min_i \theta^i_{j+\frac12},\quad  \theta^i_{j+\frac12}=\min\left\{\theta_{j+\frac12}^{i,+}, \theta_{j+\frac12}^{i,-} \right\},
    \end{equation}
    \begin{equation}
        \theta_{j+\frac12}^{i,+}=\begin{cases}
    1, & \mbox{if } g_i(\bu_j^{H,+})\geq \epsilon\\
  \mbox{Solution to }  \big(1-\theta\big) g\big(\bu_j^{L,+}\big) + \theta g\big(\bu_j^{H,+}\big) = \epsilon, &\mbox{if }  g_i(\bu_j^{H,+})< \epsilon
\end{cases},
    \end{equation}
        \begin{equation}
        \theta_{j+\frac12}^{i,-}=\begin{cases}
    1, & \mbox{if } g_i(\bu_{j+1}^{H,-})\geq \epsilon\\
  \mbox{Solution to }  \big(1-\theta\big) g\big(\bu_{j+1}^{L,-}\big) + \theta g\big(\bu_{j+1}^{H,-}\big) = \epsilon, &\mbox{if }  g_i(\bu_{j+1}^{H,-})< \epsilon
\end{cases}.
    \end{equation}
\end{subequations}

Then \eqref{eq:forward_EulerIDP}  with $\theta_{j+\frac12}$ in  \eqref{alg:theta}  is IDP because of the following  facts.
 \begin{lemma}
\label{lem:fluxlimiter_smallertheta-simple}
Let $\bu^L,\bu^H$ be two given vectors satisfying $g(\bu^L)\geq \epsilon>0$
for 
  a concave function $g$. Let $\hat \theta\in[0,1]$ be a solution to
  $(1-\hat \theta) g(\bu^L)+\hat \theta g(\bu^H)=\epsilon$, then 
    \[ g\big[(1-\theta)\bu^L+ \theta\bu^H \big]\geq \epsilon,\quad \forall \theta\in [0,
    \hat\theta]. \]
\end{lemma}
\begin{proof}
    Let $a=\theta/\hat{\theta}$, then $a\in [0,1]$, thus Jensen's inequality from concavity gives
    \begin{align*}
     & g\big[ \big(1-\theta\big)\bu^{L} + \theta  \bu^{H}\big] =g\big[ (1-a)\bu^{L}+ a  \big(1-\hat \theta\big)\bu^{L} + a \hat \theta \bu^{H}\big]\\
     \geq&  (1-a) g\big(\bu^{L}\big)+ a \big(1-\hat\theta\big)g\big(\bu^{L}\big) + a \hat\theta g\big( \bu^{H}\big)\geq \epsilon.
\end{align*}
\end{proof}

\begin{lemma}
\label{lem:fluxlimiter_smallertheta}
    If $\theta^i_{j+\frac12}$ is obtained from \Cref{alg:theta}, then \eqref{eq:forward_EulerIDP}  with any non-negative $\theta_{j+\frac12}\leq \theta^i_{j+\frac12}$ satisfies $g_i(\bu^{n+1})> 0$ for a concave function $g_i$.
\end{lemma}
\begin{proof}
   \Cref{lem:fluxlimiter_smallertheta-simple} implies
      \[  g_i\big[ \big(1-\theta_{j+\frac12}\big)\bu_j^{L,+} + \theta_{j+\frac12}  \bu_j^{H,+}\big]>0,\quad  g_i\big[ \big(1-\theta_{j-\frac12}\big)\bu_{j}^{L,-} + \theta_{j-\frac12}  \bu_{j}^{H,-}\big]>0.\]
With \eqref{hu-shu-decomp},
Jensen's inequality gives $g\big(\bu^{n+1}_j\big)>0.$
\end{proof}

\subsection{The convex limiting based on \Cref{prop:RP}}  
The convex limiting by Guermond, Popov and Tomas in \cite{guermond2019invariant} is another flux limiting approach  
adapted from the convex limiting technique for the continuous finite element method \cite{guermond2017invariant,guermond2018second-euler}, and applies to the DG, finite volume and finite difference schemes \cite{guermond2019invariant} solving 
 general hyperbolic systems satisfying
 \Cref{prop:RP} (and \Cref{prop:LF-2D} in multiple dimensions).  {See also \cite{cotter2016embedded,kuzmin2010failsafe} for using flux limiters for conservation laws.}
In \cite{kuzmin2020monolithic}, Kuzmin introduced   the monolithic convex limiting which can be applied to a semi-discrete scheme.
 As an example, we  briefly review how these convex limiting methods apply to finite difference and continuous finite element methods.

\subsubsection{Finite difference schemes with monolithic convex limiting} 
\label{sec:kuzmin}

We describe the monolithic convex limiting approach in  \cite{kuzmin2020monolithic} which can be used with any IDP flux \cite[Section 2.5.6.2]{kuzmin2024property}.
By adding and subtracting $\lambda {\bf f}(\bu_j^n)$,  the original high order scheme \eqref{eq:forward_Euler} can be rewritten as
\[ \bu_j^{n+1,H} = \frac12\Big(\bu_j^n + 2\lambda \big(\hat {\bf f}_{j-\frac12}^H
                      - {\bf f}(\bu_j^n)\big)\Big)
                    + \frac12\Big(\bu_j^n - 2\lambda \big(\hat {\bf f}_{j+\frac12}^H
                      - {\bf f}(\bu_j^n)\big)\Big)
                       =: \frac12 \bu_j^{H,-} + \frac12 \bu_j^{H,+}.\]
This can be regarded as a decomposition of the high order scheme \eqref{eq:forward_Euler} into {\it  residuals} as in the residual distribution schemes \cite{remi-10, remi-11}, which is however different from the decomposition in the Hu-Adam-Shu approach.

Now consider any first order IDP scheme \eqref{eq:forward_EulerL}
with $\hat{\bf f}_{j+\frac{1}{2}}^{L}=\hat{\bff}(\bu^n_{j},\bu^n_{j+1})$ and $\hat{\bff}$ is any consistent IDP flux including Lax--Friedrichs, Godunov, and HLLE fluxes. Then it can be also rewritten as
\[ \bu_j^{n+1,L} = \frac12\Big(\bu_j^n + 2\lambda \big(\hat {\bf f}_{j-\frac12}^L
                      - {\bf f}(\bu_j^n)\big)\Big)
                    + \frac12\Big(\bu_j^n - 2\lambda \big(\hat {\bf f}_{j+\frac12}^L
                      - {\bf f}(\bu_j^n)\big)\Big)
                       =: \frac12 \bu_j^{L,-} + \frac12 \bu_j^{L,+}.\]
If the first order scheme \eqref{eq:forward_EulerL} is IDP under CFL condition $\lambda \max_{\bu}|\bff'(\bu)|\leq a_0$, then   the following holds under halved CFL condition,
\begin{equation}
    \label{fake-1d-scheme}
    {\bf u}_j^{L,\pm} := 
\bu_j^n \mp 2\lambda \big(\hat{\bf f}_{j\pm\frac{1}{2}}^{L} - {\bf f}(\bu_j^n)\big) \in G,\qquad \lambda \max_{\bu}|\bff'(\bu)|\leq \frac12 a_0.
\end{equation}
 To see why \eqref{fake-1d-scheme} is true, we focus on ${\bf u}_j^{L,-}$. Notice that it can be written as another first order scheme    
 \begin{equation}
    \label{fake-1d-scheme-2}
     {\bf u}_j^{L,-} = \bu_j^n + 2\lambda \big(\hat {\bf f}_{j-\frac12}^L
                      - {\bf f}(\bu_j^n)\big)=\bu_j^n -2\lambda \big(\hat{\bff}(\bu_j^n,\bu_j^n)-\hat {\bf f}(\bu^n_{j-1},\bu^n_{j})\big),\end{equation}
  which is implied by $\hat {\bf f}_{j-\frac12}^L=\hat {\bf f}(\bu^n_{j-1},\bu^n_{j})$ and the consistency of the flux function $\hat {\bf f}(\bu^n_{j},\bu^n_{j})=\bff(\bu^n_{j}).$     Since \eqref{fake-1d-scheme-2} is in the same form as \eqref{scheme:1Dsystem} but with doubled time step,
  it is also IDP under halved CFL.
  
Then the scheme \eqref{eq:forward_EulerIDP} can be written as
\begin{equation}
    \label{1D-FD-monolithic-convex-limiting}
    \bu^{n+1}_j=\frac12\left[(1-\theta_{j-\frac12})\bu_j^{L,-}+ \theta_{j-\frac12}\bu_j^{H,-}\right]+\frac12\left[(1-\theta_{j+\frac12})\bu_j^{L,+}+ \theta_{j+\frac12}\bu_j^{H,+}\right],
\end{equation}
and we can use ${\bf u}_j^{L,\pm} \in G$ to correct  ${\bf u}_j^{H,\pm}$  in the same way as in the Hu-Adam-Shu approach.
  For simplicity, assume $G$ in \eqref{eq:ASS-G} is defined by concave functions $g_i(\bu)>0, i=1,\cdots, N$. Then \eqref{eq:forward_EulerIDP}  with $\theta_{j+\frac12}$   
 in  \eqref{alg:theta}  is IDP due to  \Cref{lem:fluxlimiter_smallertheta}.

 \begin{remark}
    The limiting approaches in \Cref{sec:hu-shu} and \Cref{sec:kuzmin} are defined via quantities of $\bu^n_j$ without involving $\bu^{n+1,L}_j$, and such limiting methods are called monolithic  \cite{hajduk2020matrix,kuzmin2020monolithic}. 
See \cite{hajduk2021monolithic,rueda2024monolithic,moujaes2025monolithic} for more monolithic convex limiting techniques. 
{ Similar flux limiters were proposed in \cite{Gu2021A,Fu2025Bound} for the five-equation model.}
\end{remark}

\subsubsection{Finite difference schemes with convex limiting { of FCT type}} 

For a high order finite difference scheme \eqref{eq:forward_EulerIDP}, 
 the convex limiting method in \cite{guermond2018second-euler,guermond2017invariant,guermond2019invariant} provides a simpler sufficient solution for enforcing \eqref{eq:MMP_NF_BP-2}. 
 The main idea is to rewrite the scheme \eqref{eq:forward_EulerIDP} in the following form 
\begin{equation*}
\bu^{n+1}_j=    \frac12 \left[\bu_j^{n+1,L}-2\lambda \theta_{j+\frac12} \Big(\hat {\bf f}^H_{j+\frac12} - \hat {\bf f}^L_{j+\frac12}\Big)\right]
   +\frac12 \left[\bu_j^{n+1,L}+2\lambda \theta_{j-\frac12} \Big(\hat {\bf f}^H_{j-\frac12} - \hat {\bf f}^L_{j-\frac12}  \Big)\right].
\end{equation*}
For a given concave function $g(\bu)>0$, the idea is to find $\theta_{j+\frac12}$ such that
\begin{equation}
  \label{FD:convexlimiting-guermond}
  g \left[\bu_j^{n+1,L}-2\lambda \theta_{j+\frac12} \Big(\hat {\bf f}^H_{j+\frac12} - \hat {\bf f}^L_{j+\frac12}\Big)\right]\geq \epsilon,
   \quad g \left[\bu_j^{n+1,L}+2\lambda \theta_{j-\frac12} \Big(\hat {\bf f}^H_{j-\frac12} - \hat {\bf f}^L_{j-\frac12}  \Big)\right]\geq \epsilon.
\end{equation}
 
Such $\theta_{j+\frac12}$ can be found in a way similar to \eqref{alg:theta}, see  \cite{guermond2017invariant,guermond2019invariant}.
In such a convex limiting approach, one seeks to use $\bu_j^{n+1,L}\in G$ to correct $\pm 2\lambda \Big(\hat {\bf f}^H_{j\mp\frac12} - \hat {\bf f}^L_{j\mp\frac12}\Big)$, which is different from the previous approaches.
 As a comparison, the approaches in \Cref{sec:hu-shu} and \Cref{sec:kuzmin} use ${\bf u}_j^{L,\pm} \in G$ to correct  ${\bf u}_j^{H,\pm}$,  and they satisfy $${\bf u}_j^{H,+}-{\bf u}_j^{L,+}=-\lambda(\hat \bff^H_{j+\frac12}-2\hat \bff^L_{j+\frac12}),\quad {\bf u}_j^{H,-}-{\bf u}_j^{L,-}=2\lambda(\hat \bff^H_{j-\frac12}-\hat \bff^L_{j-\frac12}).$$

\subsubsection{Convex limiting for continuous finite element methods}
Since $\bu^{n+1,L}_j$ is needed in \eqref{FD:convexlimiting-guermond}, the convex limiting method \eqref{FD:convexlimiting-guermond} is different from a monolithic approach. On the other hand, such a method offers easiness for schemes on unstructured meshes in multiple dimensions. 
 We briefly review the main idea in \cite{guermond2018second-euler} for applying convex limiting to continuous finite element methods. 
 
 With the same notation as in \Cref{sec:FEM}, we consider the group finite element method with $\mathbb P^1$ basis defined on unstructured triangular meshes.
Recall 
  the first order  continuous finite element method with mass lumping \eqref{FEM-firstorder} can be written 
in the flux form \eqref{FEM-firstorder-LF3}. For convenience, we denote  \eqref{FEM-firstorder-LF3} as
\[m_{i}\frac{\bu^{n+1, L}_i-\bu^{n}_i}{\Delta t}+\sum_{j\in \mathcal N_i} [(\bff(\bu^n_j)+\bff(\bu^n_i))\cdot \bc_{ij}-d_{ij}^L(\bu^n_j-\bu^n_i)]=0,\]
where $d_{ij}^L$ is the artificial viscosity coefficient in the first order IDP scheme. 

A second order in space finite element scheme without mass lumping with forward Euler time discretization can be written as 
\begin{equation}
   \label{secondorderFEM}
   \sum_{j\in \mathcal N_i}M_{ij}\frac{\bu^{n+1,H}_j-\bu^{n}_j}{\Delta t}+\sum_{j\in \mathcal N_i} [(\bff(\bu^n_j)+\bff(\bu^n_i))\cdot \bc_{ij}-d_{ij}^H(\bu^n_j-\bu^n_i)]=0,
\end{equation}
with  smaller  artificial viscosity coefficients $d_{ij}^H$ satisfying the same symmetry and zero sum constraints \eqref{FEM-diffusion-condition}. We refer to 
\cite{guermond2014second,guermond2017invariant,kuzmin2020monolithic} and references therein for how to compute $d_{ij}^H$.  
Let $\delta_{ij}$ be the Kronecker delta, then $\sum_{j\in \mathcal N_i}(M_{ij}-\delta_{ij} m_i)=0$, thus
\begin{align*}
& \sum_{j\in \mathcal N_i}M_{ij}\frac{\bu^{n+1,H}_j-\bu^{n}_j}{\Delta t}=\frac{m_i}{\Delta t}(\bu^{n+1,H}_i-\bu^{n}_i)+\sum_{j\in \mathcal N_i}\frac{M_{ij}-\delta_{ij}m_i}{\Delta t}(\bu^{n+1,H}_j-\bu^{n}_j)\\
=&\frac{m_i}{\Delta t}(\bu^{n+1,H}_i-\bu^{n}_i)+\sum_{j\in \mathcal N_i}\frac{M_{ij}-\delta_{ij}m_i}{\Delta t}(\bu^{n+1,H}_j-\bu^{n}_j-\bu^{n+1,H}_i+\bu^{n}_i).
\end{align*}
Together with properties of $\bc_{ij}$, 
 \eqref{secondorderFEM} can be rewritten as
 \[\frac{m_i}{\Delta t}(\bu^{n+1,H}_i-\bu^{n}_i)+\sum_{j\in \mathcal N_i^* }\hat \bff^H_{ij}=0, \]
where $j\in \mathcal N^*_i$ denotes $j\in \mathcal N_i, j\neq i$, and the  numerical flux is
\[\hat \bff^H_{ij}=\frac{M_{ij}}{\Delta t}(\bu^{n+1,H}_j-\bu^{n}_j-\bu^{n+1,H}_i+\bu^{n}_i)+(\bff(\bu^n_j)+\bff(\bu^n_i))\cdot \bc_{ij}-d_{ij}^H(\bu^n_j-\bu^n_i).  \]

The first order IDP scheme can be written in a similar form:
 \[\frac{m_i}{\Delta t}(\bu^{n+1,L}_i-\bu^{n}_i)+\sum_{j\in \mathcal N_i^* }\hat \bff^L_{ij}=0, \] 
\[\hat \bff^L_{ij}=(\bff(\bu^n_j)+\bff(\bu^n_i))\cdot \bc_{ij}-d_{ij}^L(\bu^n_j-\bu^n_i).  \]
 
The flux corrected scheme can be written as
 \[ \bu^{n+1}_i=\bu^{n+1,L}_i-\frac{\Delta t}{m_i}\sum_{j\in \mathcal N_i^* }\theta_{ij}\left(\hat \bff^H_{ij}-\hat \bff^L_{ij}\right), \quad \theta_{ij}\in [0,1].\] 
As an easy approach to find $\theta_{ij}$ to enforce convex invariant domains,
the convex limiting method in
\cite{guermond2017invariant,guermond2018second-euler,guermond2019invariant} proposes
to consider rewrite it as
 \[ \bu^{n+1}_i=\sum_{j\in \mathcal N_i^* }a_j \left[\bu^{n+1,L}_i-\theta_{ij}\frac{\Delta t}{m_ia_j}\left(\hat \bff^H_{ij}-\hat \bff^L_{ij}\right)\right], \quad \theta_{ij}\in [0,1],\] 
 where $a_j>0$ are any convex combination coefficients such that $\sum_{j\in \mathcal N_i^* }a_j=1$.
We refer to \cite{guermond2019invariant} for how to find $\theta_{ij}$ ensuring $$g\left[\bu^{n+1,L}_i-\theta_{ij}\frac{\Delta t}{m_ia_j}\left(\hat \bff^H_{ij}-\hat \bff^L_{ij}\right)\right]>0,$$ which is  sufficient for ensuring $g( \bu^{n+1}_i)>0$ for a concave constraint function $g$.
\begin{remark}
   The convex limiting  becomes more diffusive with  higher order polynomial basis, see \cite[Figure 4]{guermond2024finite}. Instead, a better way is to mix a high order finite element method with an invariant-domain preserving low order method on the
closest neighbor stencil. Such continuous finite element methods with $\mathbb P^2$ and $\mathbb P^3$ bases on simplicial meshes were constructed in \cite{guermond2024finite}.  {See also \cite{lohmann2017flux,kuzmin2020subcell,hajduk2020matrix} for  limiters on higher order finite element method.}
\end{remark}

{\color{black} 
\subsubsection{Monolithic limiter via GQL representation}
A monolithic limiting method via GQL representation was presented in \cite{Abgrall2026PAMPA1D} for the
Point-Average-Moment-PolynomiAl-interpreted (PAMPA) scheme
\cite{Abgrall2025IDP-PAMPA-1D}. For simplicity, we review the main idea of this method on a finite difference scheme, for which we write an alternative wave-speed-based convex decomposition of \eqref{eq:forward_EulerIDP}:
\begin{align*}
	\mathbf{u}_j^{n+1}
	&= \mathbf{u}_j^n
	- \lambda\!\left(\hat{\mathbf{f}}_{j+\frac12}-\hat{\mathbf{f}}_{j-\frac12}\right)
	+ \lambda\!\left(\alpha_{j-\frac12}+\alpha_{j+\frac12}\right)\mathbf{u}_j^n
	- \lambda\!\left(\alpha_{j-\frac12}+\alpha_{j+\frac12}\right)\mathbf{u}_j^n \\
	&=
	\bigl(1-\lambda\alpha_{j-\frac12}-\lambda\alpha_{j+\frac12}\bigr)\mathbf{u}_j^n
	+ \lambda\,\alpha_{j-\frac12}\,\mathbf{u}_j^-
	+ \lambda\,\alpha_{j+\frac12}\,\mathbf{u}_j^+ ,
\end{align*}
where  $\hat{\bff}$ is any consistent IDP flux and 
\[
\mathbf{u}_j^-=
\mathbf{u}_j^n+\frac{\hat{\mathbf{f}}_{j-\frac12}^L-\mathbf{f}(\mathbf{u}_j)
	+\theta_{j-\frac12}\,\boldsymbol{\delta}_{j-\frac12}}{\alpha_{j-\frac12}}
=: \widehat{\mathbf{u}}_j^{L,-}
+ \theta_{j-\frac12}\,\frac{\boldsymbol{\delta}_{j-\frac12}}{\alpha_{j-\frac12}},
\]
\[
\mathbf{u}_j^+=
\mathbf{u}_j^n-\frac{\hat{\mathbf{f}}_{j+\frac12}^L-\mathbf{f}(\mathbf{u}_j)
	+\theta_{j+\frac12}\,\boldsymbol{\delta}_{j+\frac12}}{\alpha_{j+\frac12}}
=: \widehat{\mathbf{u}}_j^{L,+}
- \theta_{j+\frac12}\,\frac{\boldsymbol{\delta}_{j+\frac12}}{\alpha_{j+\frac12}}
\]
with ${\bm \delta}_{j+\frac12} := \hat {\bf f}^H_{j+\frac12} - \hat {\bf f}^L_{j+\frac12}$. 
Similar to the previous discussion for \eqref{fake-1d-scheme}, with sufficiently large $\alpha_{j\pm\frac12}$, we have $\widehat{\mathbf{u}}_j^{L,\pm}\in G$.
Then there exist blending parameters
$\theta_{j\pm\frac12}\in[0,1]$ such that $\mathbf{u}_j^\pm\in G$.
Consequently, $\mathbf{u}_j^{n+1}\in G$ under the CFL condition
$\lambda\bigl(\alpha_{j-\frac12}+\alpha_{j+\frac12}\bigr)\le 1$.

The key feature of the method in \cite{Abgrall2026PAMPA1D} and its extension to polygonal meshes in \cite{Abgrall2025PAMPA-Polygon-arXiv} is  a strategy to design effective blending
parameters $\theta_{j\pm\frac12}$ for systems. We briefly review the idea for the 1D Euler equations.
Based on the GQL framework \cite{wu2023geometric}, the invariant domain \eqref{eq:G-Euler}
for the 1D Euler system can be written equivalently as
\[
G^\star
= \bigl\{ \mathbf{u}\in\mathbb{R}^3 : \mathbf{u}\cdot\mathbf{n}^\star>0
\ \ \forall\,\mathbf{n}^\star\in\mathcal{N} \bigr\}, \qquad
\mathcal{N}
= \left\{ \begin{pmatrix} 1 \\ 0 \\ 0 \end{pmatrix} \right\}
\cup
\left\{ \begin{pmatrix} v_\star^{2}/2 \\ -\,v_\star \\ 1 \end{pmatrix} : v_\star\in\mathbb{R} \right\}.
\]
To enforce $\mathbf{u}_j^\pm\in G$ (equivalently $G^\star$), it suffices to require
$\mathbf{u}_j^\pm\cdot\mathbf{n}^\star>0$ for all $\mathbf{n}^\star\in\mathcal{N}$,
for which a convenient choice of blending parameter is
\[
\theta_{j + \frac12} = \min \{ \theta_{j + \frac12}^-, \theta_{j + \frac12}^+ \} ~~
\mbox{with} ~~
\theta_{j\pm \frac12}^\mp
= \min\!\left\{
1,\;
\alpha_{j\pm \frac12}\,
\min_{\mathbf{n}^\star\in\mathcal{N}}
\frac{\widehat{\mathbf{u}}_j^{L,\pm}\cdot\mathbf{n}^\star}{
	\left|\boldsymbol{\delta}_{j\pm\frac12}\cdot\mathbf{n}^\star\right|}
\right\}.
\]
An explicit closed-form expression for the minimizer is given in \cite{Abgrall2026PAMPA1D}.

}

\subsection{Numerical results}

We implement and test
 the fifth order finite difference WENO scheme solving compressible Euler equations with ideal gas EOS \eqref{eq:Euler} with three flux correction methods for enforcing invariant domains:
\begin{enumerate}
    \item The parametrized limiter in \Cref{sec:parameterizedsystem} 
    \item The  Hu--Adams--Shu simple flux limiting in \Cref{sec:hu-shu}
    \item The monolithic convex limiting in \Cref{sec:kuzmin}.
\end{enumerate}

\begin{example}[Leblanc shock tube]\label{Ex:Leblanc}
	This test is a 1D shock tube with $\gamma=1.4$, the initial condition
	\begin{equation*}
		(\rho,v_1,p) = 
		\begin{cases}
			(2,0,10^{9}), & \text{ if } x<0, \\
			(0.001,0,10^{-12}), & \text{ otherwise}, 
		\end{cases}
	\end{equation*}
    and outflow boundary conditions on the domain $[-10,10]$. 
	See \Cref{fig:Euler-Leblanc} for the plots of density and pressure.
	
	\begin{figure}[!thb]
		\centering
		\begin{subfigure}{0.48\textwidth}
			\includegraphics[width=\textwidth]{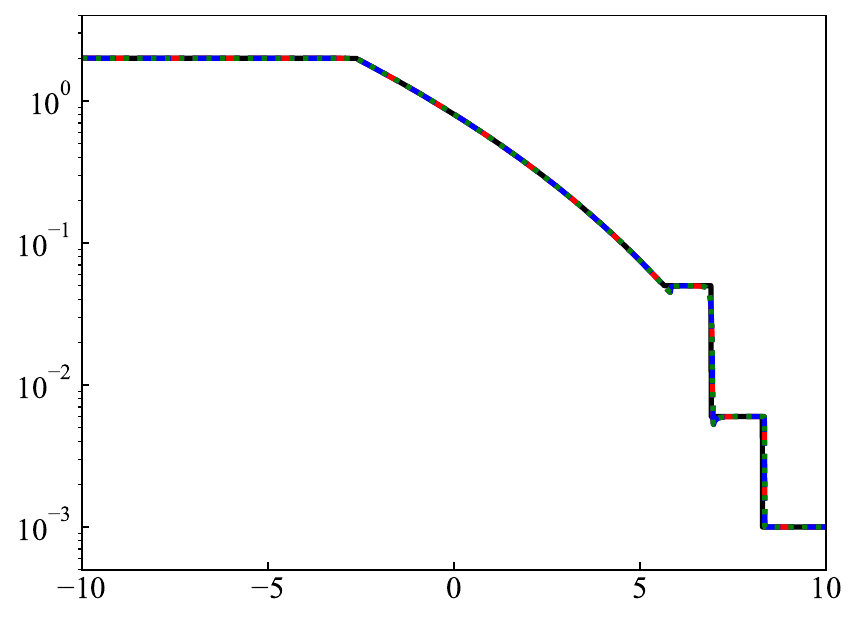}
		\end{subfigure}
		\hfill
		\begin{subfigure}{0.48\textwidth}
			\includegraphics[width=\textwidth]{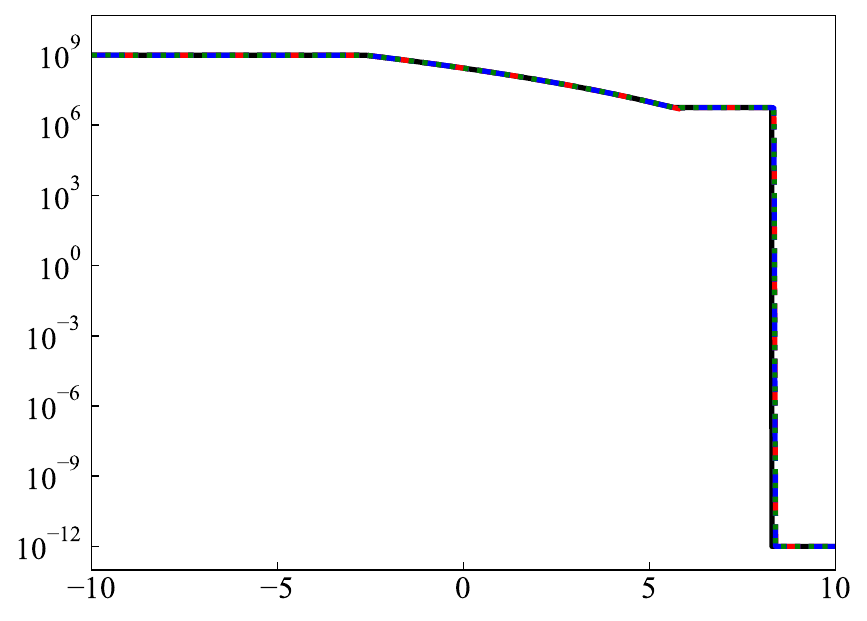}
		\end{subfigure}
		
		\caption{ \Cref{Ex:Leblanc}. Leblanc shock tube: density (left) and pressure 
        (right)
        at $t=0.001$ of the fifth order FD WENO scheme  with IDP flux limiters on $4000$ grid points.  Red dashed line: the parametrized limiter; Blue dash-dot line: the Hu--Adams--Shu limiter; Green dotted line: the monolithic convex limiting. 
		}
		\label{fig:Euler-Leblanc}
	\end{figure}
	
\end{example}

\begin{example}[Double rarefaction with low density and pressure]\label{Ex:DoubleRarefaction}
	This test is a 1D double rarefaction problem with $\gamma=1.4$ and the initial condition 
	\begin{equation*}
		(\rho,v_1,p) = 
		\begin{cases}
			(7,-100,0.01), & \text{ if } x<0.5, \\
			(7,100,0.01), & \text{ otherwise.}
		\end{cases}.
	\end{equation*}
    The exact solution contains perfect vacuum for which high order schemes can easily produce negative density and pressure. 
	The computational domain is $[0,1]$ divided into a quite coarse  mesh of only $100$ uniform cells. The outflow condition is applied on left and right boundaries. 
	\Cref{fig:Euler-DoubleRarefaction} displays the density and velocity at $t=0.003$ obtained by the fifth order FD WENO scheme with IDP flux limiters on $100$ grid points.   
	
	\begin{figure}[!thb]
		\centering
		\begin{subfigure}{0.48\textwidth}
			\includegraphics[width=\textwidth]{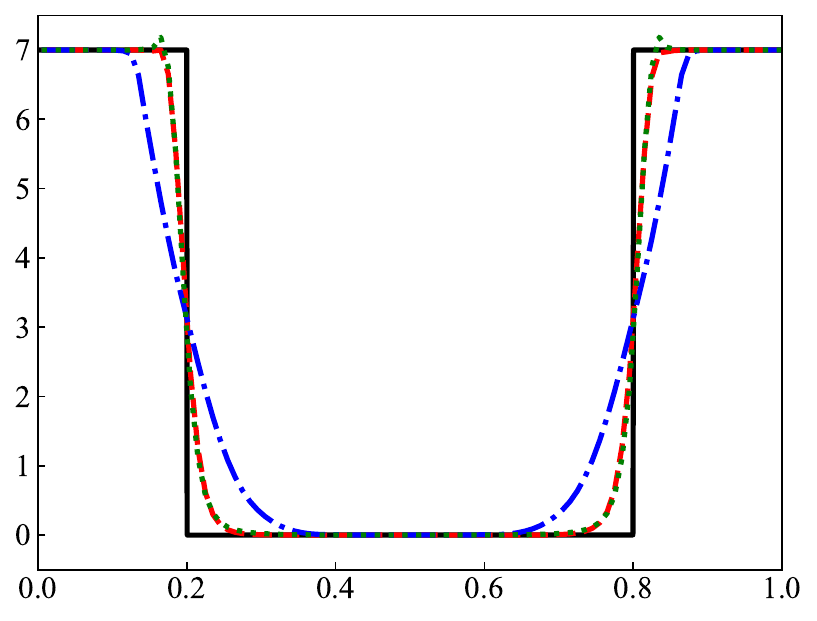}
		\end{subfigure}
		\hfill 
        		\begin{subfigure}{0.48\textwidth}
			\includegraphics[width=\textwidth]{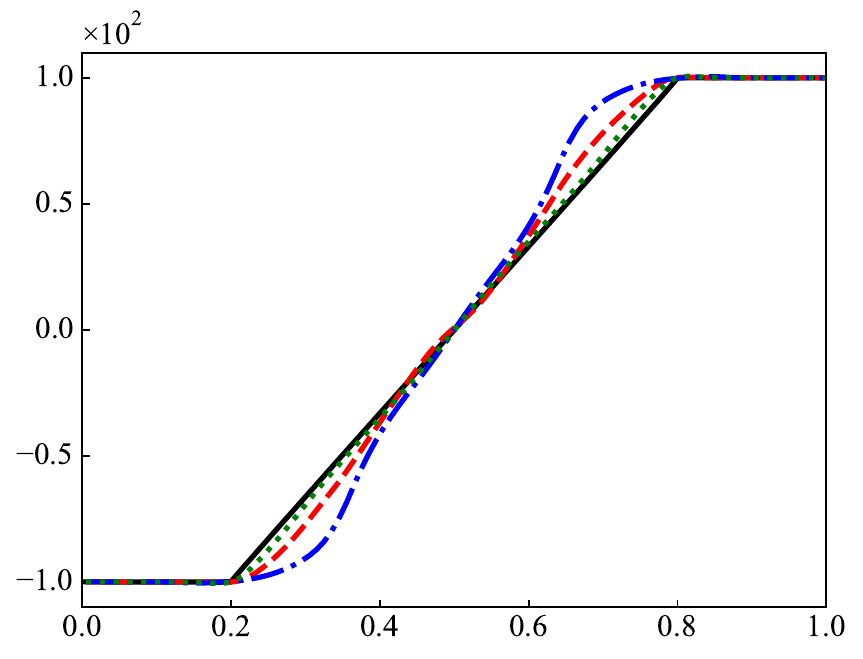}
		\end{subfigure}
		\caption{ \Cref{Ex:DoubleRarefaction}. Double rarefaction: density (left) and velocity (right) at $t=0.003$ of the fifth order FD WENO scheme with IDP flux limiters on $100$ grid points.  Red dashed line: the parametrized limiter; Blue dash-dot line: the Hu--Adams--Shu limiter; Green dotted line: the monolithic convex limiting.
		}
		\label{fig:Euler-DoubleRarefaction}
	\end{figure}

\end{example}

\begin{example}[Shock vortex interaction]\label{Ex:ShockVortex}
This example simulates the interaction of shock and vortex, which involves very low density and low pressure, which was proposed in \cite{cui2024optimal}. 
The computational domain is taken as $[0,2]\times[0,1]$, which is divided into $450 \times 225$ uniform cells. 
\Cref{fig:Euler-ShockVortex} shows the contours of the density and pressure obtained by the fifth order IDP finite difference WENO scheme using Hu--Adams--Shu limiter, convex limiting, and the parametrized limiter, respectively.

\begin{figure}[!thb]
	\centering
	\begin{subfigure}{0.32\textwidth}
		\includegraphics[width=\textwidth]{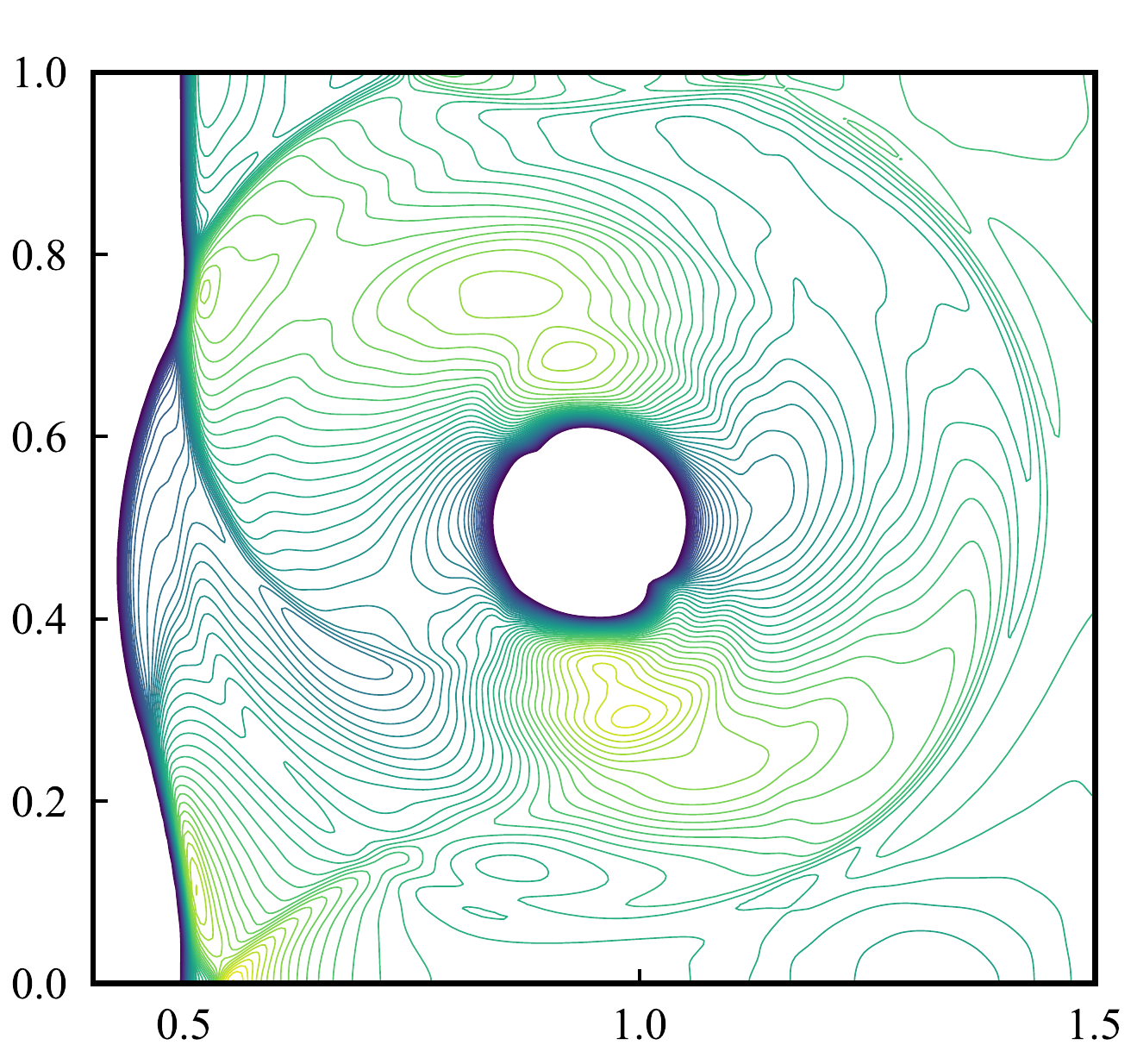}
	\end{subfigure}
	\hfill
	\begin{subfigure}{0.32\textwidth}
		\includegraphics[width=\textwidth]{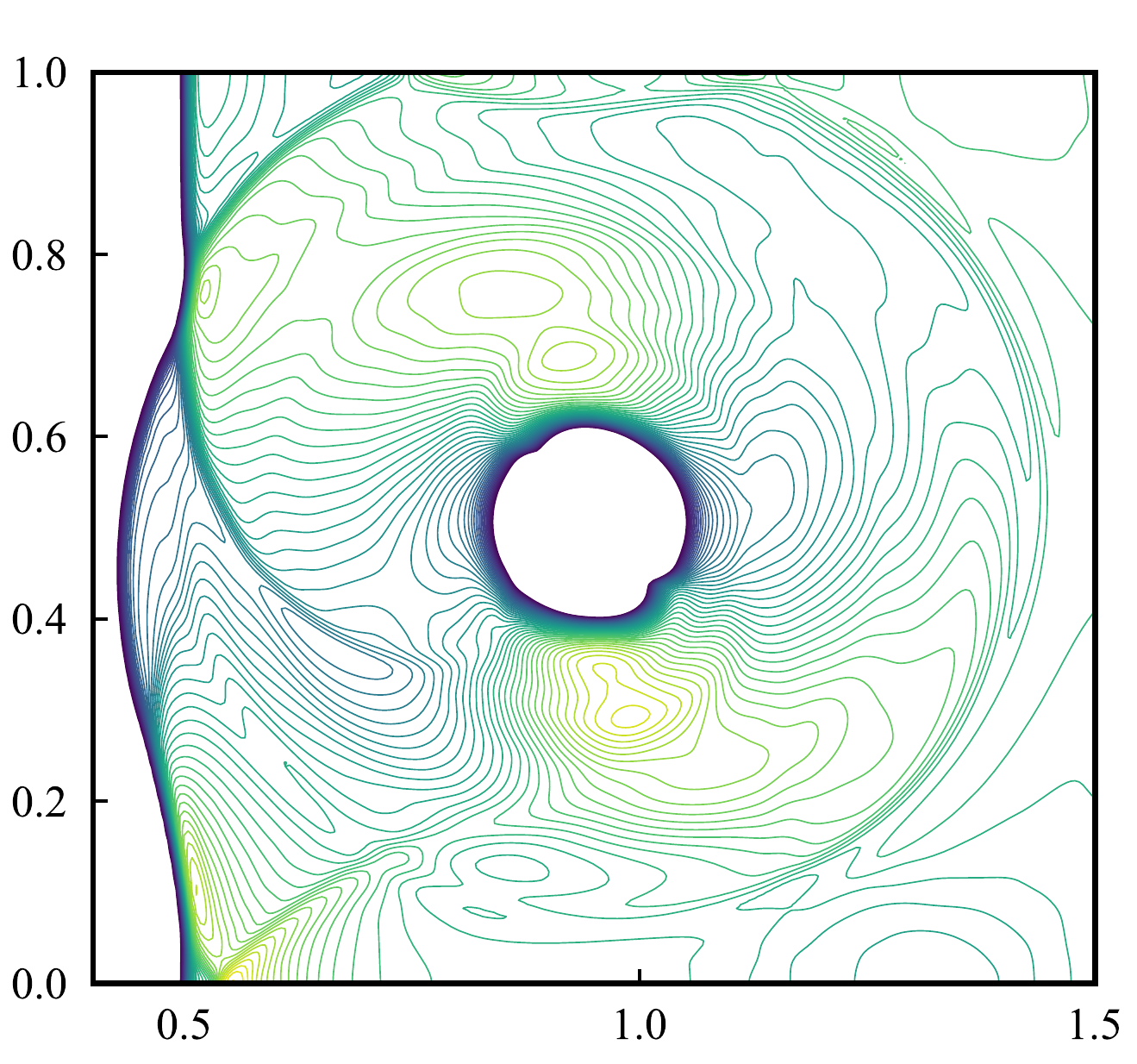}
	\end{subfigure}
	\hfill
	\begin{subfigure}{0.32\textwidth}
		\includegraphics[width=\textwidth]{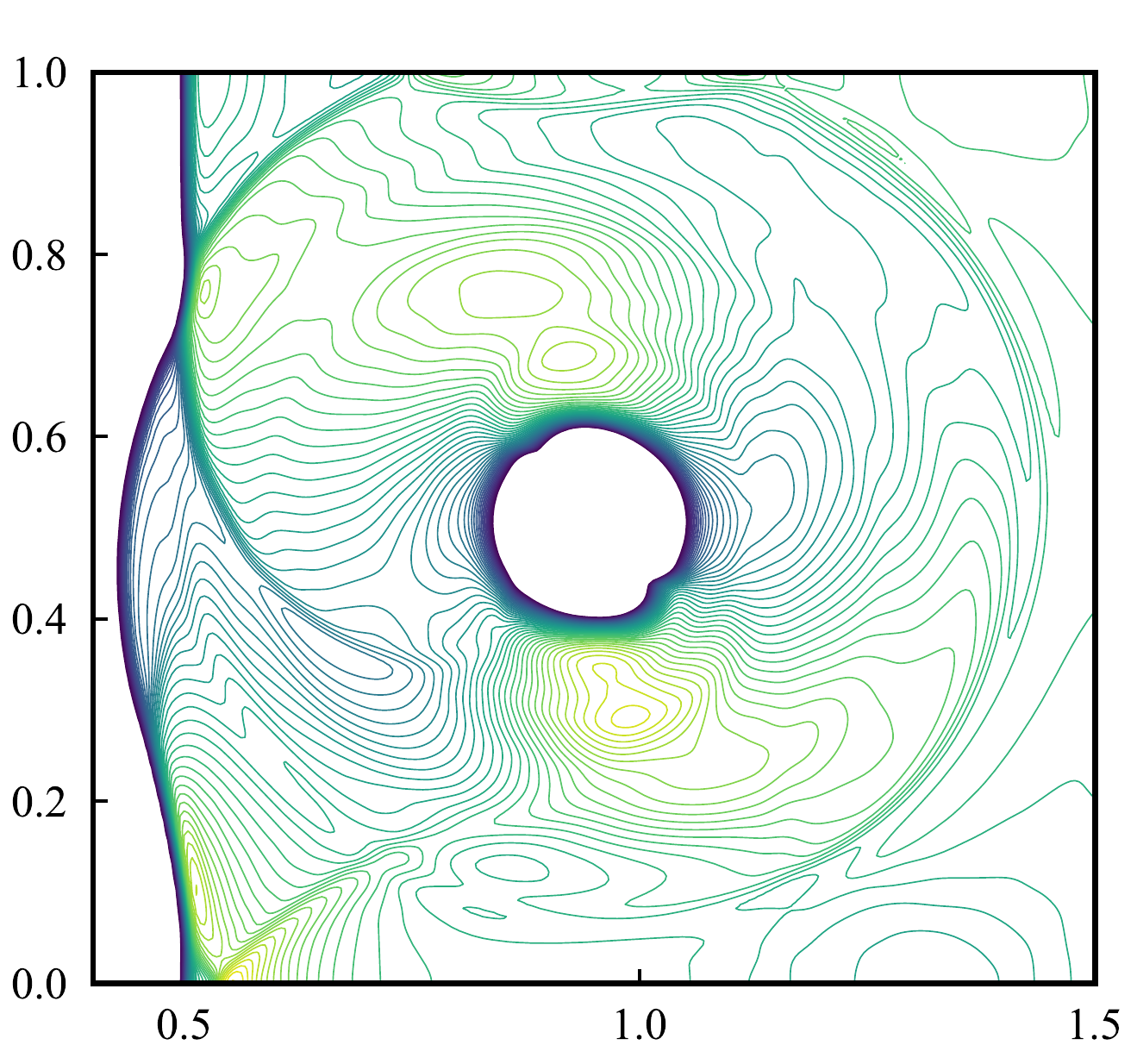}
	\end{subfigure}

	\begin{subfigure}{0.32\textwidth}
		\includegraphics[width=\textwidth]{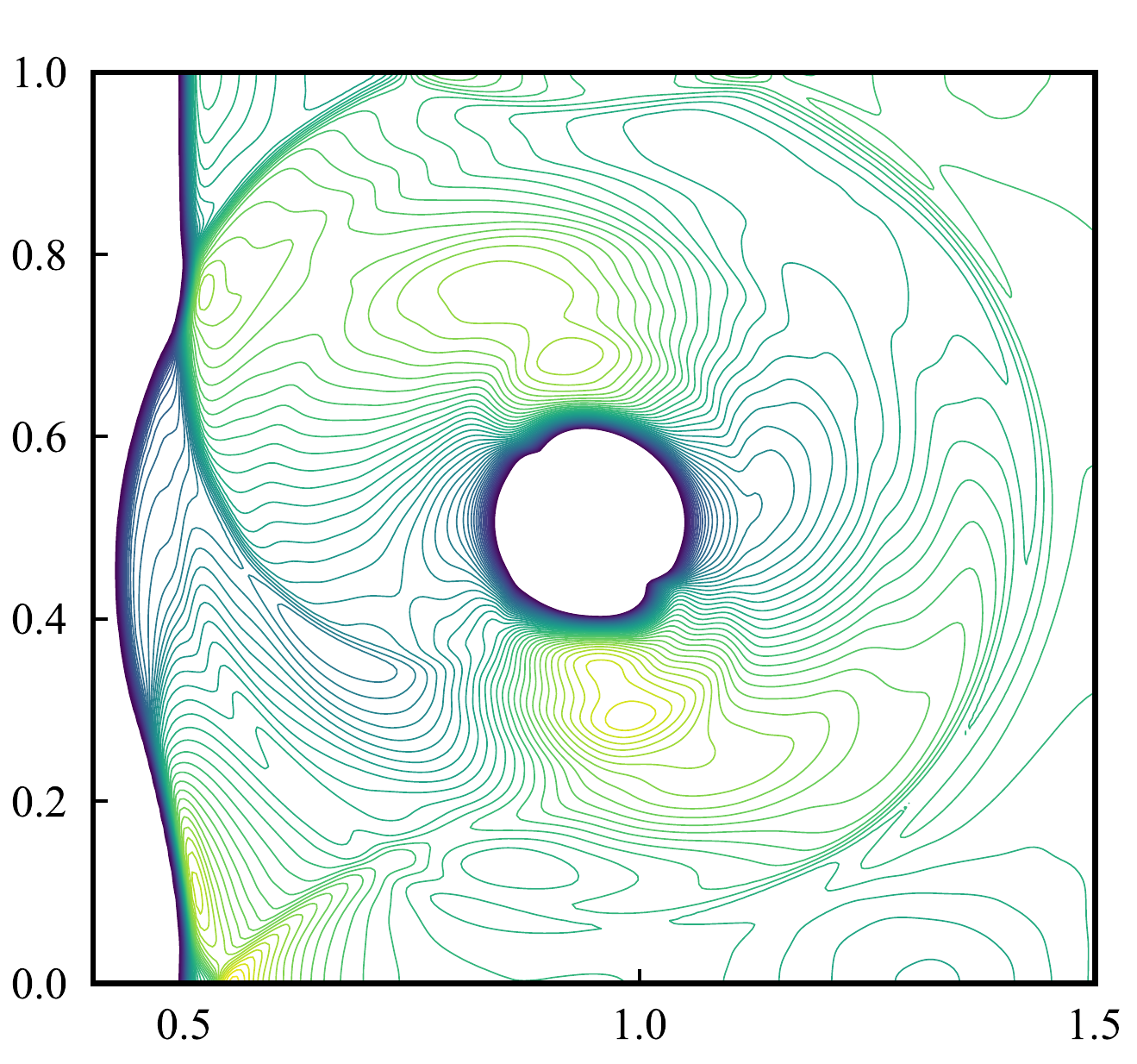}
	\end{subfigure}
	\hfill
	\begin{subfigure}{0.32\textwidth}
		\includegraphics[width=\textwidth]{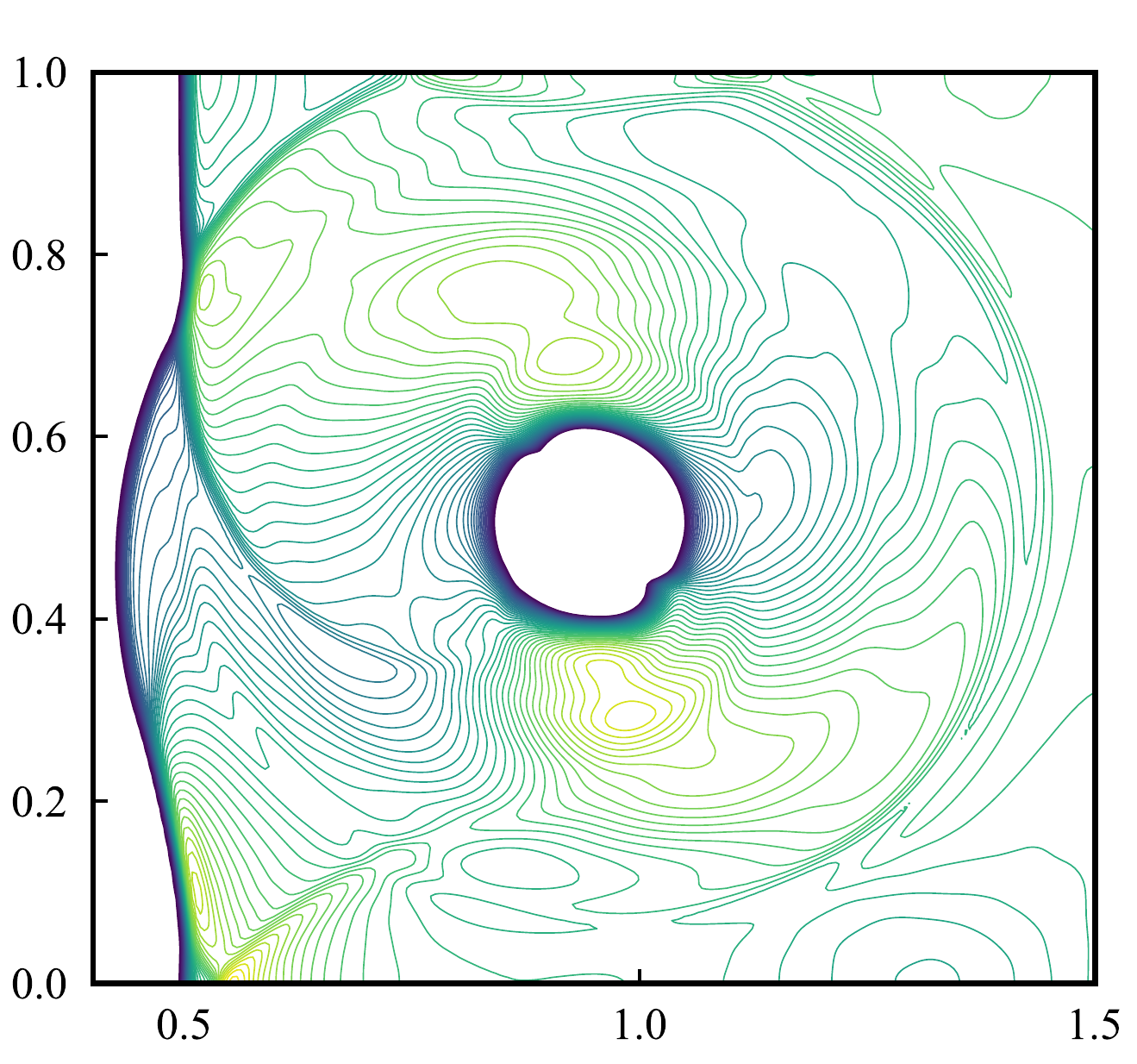}
	\end{subfigure}
	\hfill
	\begin{subfigure}{0.32\textwidth}
		\includegraphics[width=\textwidth]{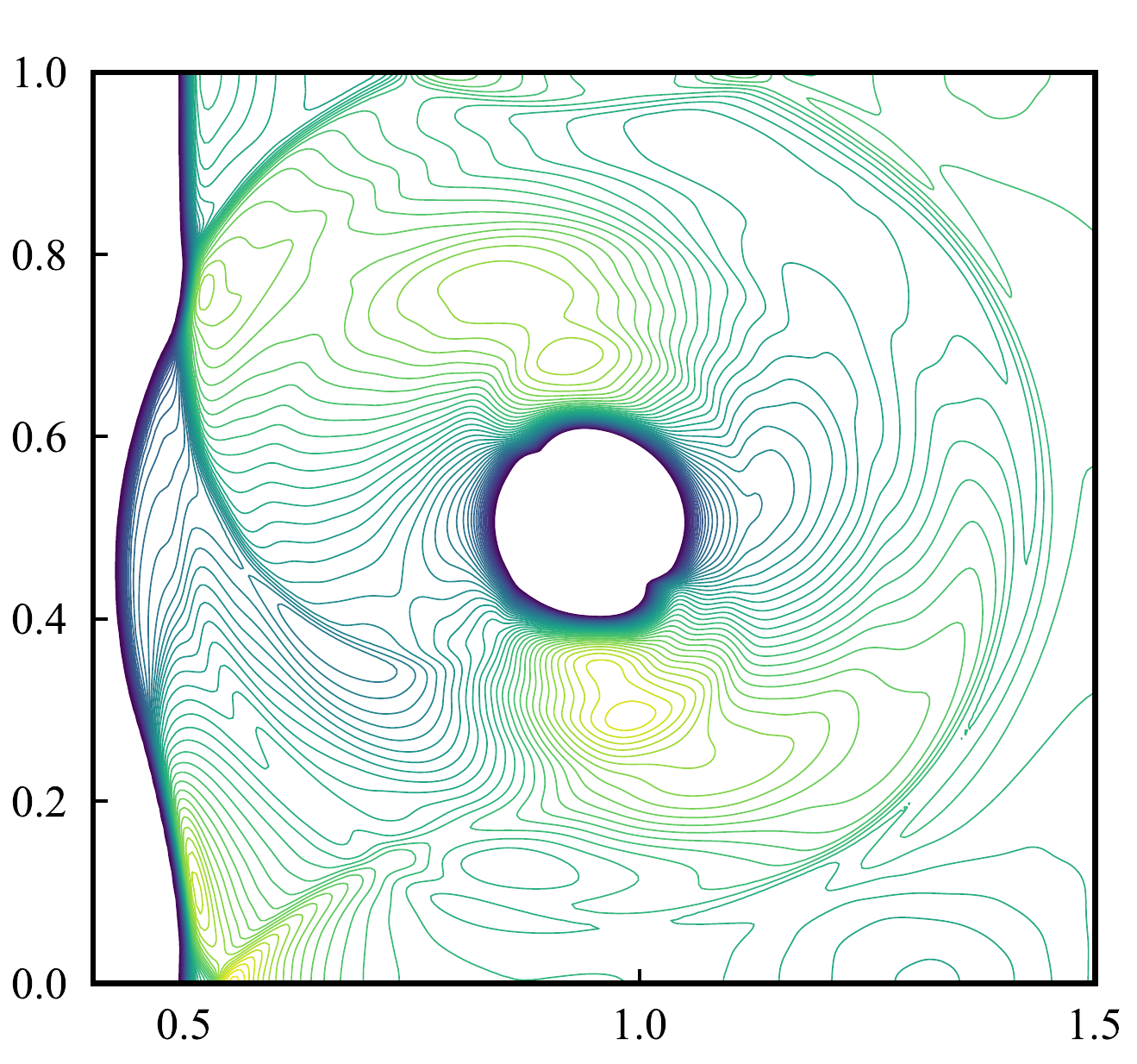}
	\end{subfigure}
	
	\caption{  \Cref{Ex:ShockVortex} The contour plots of density (top) and pressure (bottom) at $t=0.6$
    of the fifth order FD WENO scheme with IDP flux limiters. 
		50 contour lines: from 1.03 to 1.39. 
		Left: the parametrized limiter, middle: the Hu--Adams--Shu limiter, right: the monolithic convex limiting.
	}
	\label{fig:Euler-ShockVortex}
\end{figure}

\end{example}

\begin{example}[Relativistic axisymmetric jet]\label{Ex:RHDJet}
In this example, we simulate a very challenging astrophysical jet problem by solving the axisymmetric version of the relativistic 
hydrodynamic equations \eqref{eq:relativisticEuler}. The adiabatic index is taken as $\gamma=5/3$. The computational domain is set as $[0, 15] \times [0, 75]$, which is divided into $300 \times 1500$ uniform cells. Initially, the domain is full of the static uniform medium with
$$(\rho, u, v, p) = (1.0, 0.0, 0.0, 5.988006089640541\times 10^{-11}).$$
A high-speed relativistic jet with state
$$(\rho_b,{v}_{1,b},{v}_{2,b},p_b) = (0.01, 0.0, 0.999, 5.988006089640541\times 10^{-11})$$
is injected in $z$-direction through nozzle ($r \leq 1$) of the bottom boundary ($z = 0$). In other words,
the fixed inflow condition $(\rho_b, v_{1,b}, v_{2,b}, p_b)$ is applied on $\{r \leq 1, z = 0\}$ of the bottom boundary.
The symmetrical condition is specified on the left boundary $r = 0$, outflow conditions are applied
on other boundaries. For this jet, the classical Mach number is $10,000$, and the relativistic Mach number is about $223,662.719$, which is extremely high. 
\Cref{fig:RHD-Jet} displays the contours of rest-mass density logarithm at $t=100$ obtained by the fifth order IDP finite difference WENO scheme using Hu--Adams--Shu limiter, convex limiting, and the parametrized limiter, respectively.

\begin{figure}[!thb]
	\centering
	\begin{subfigure}{0.32\textwidth}
		\includegraphics[width=\textwidth, height=8cm]{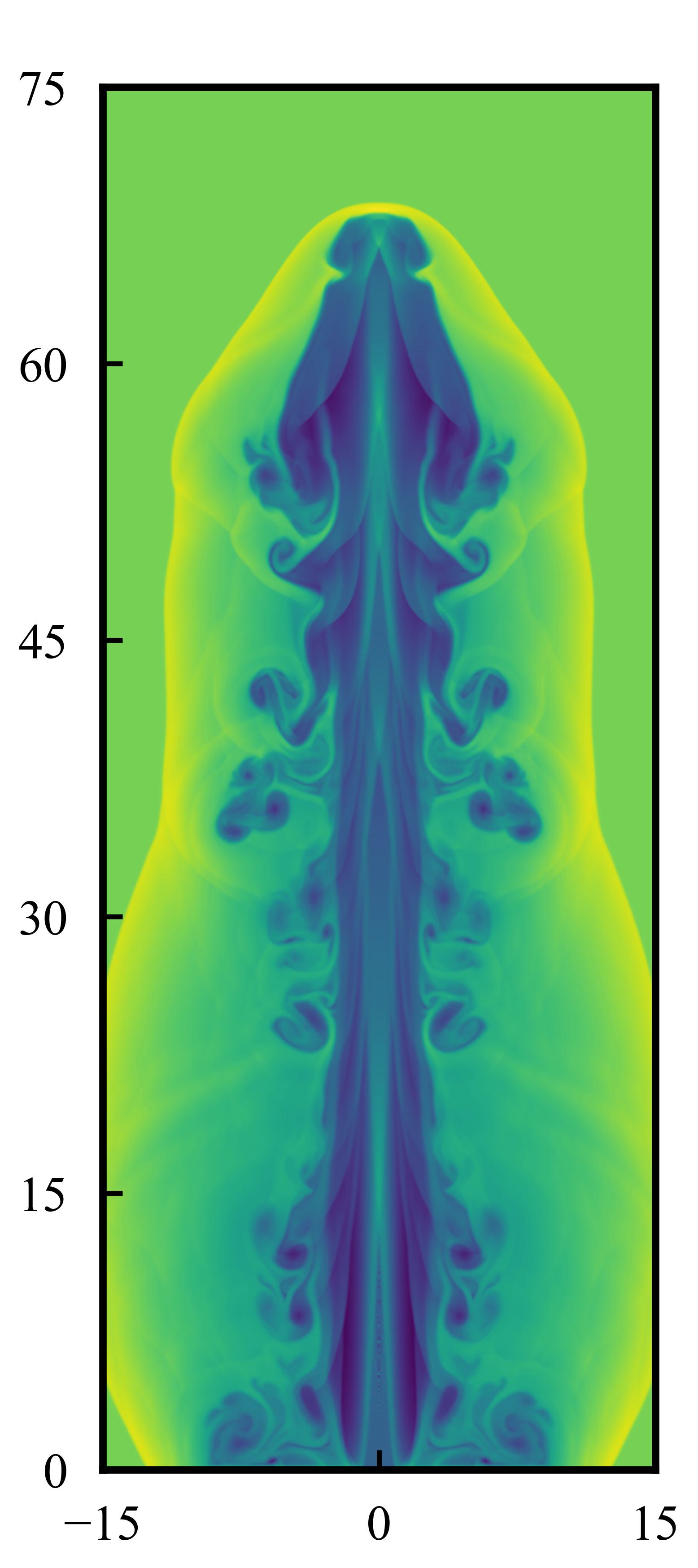}
	\end{subfigure}
	\hfill
	\begin{subfigure}{0.32\textwidth}
		\includegraphics[width=\textwidth, height=8cm]{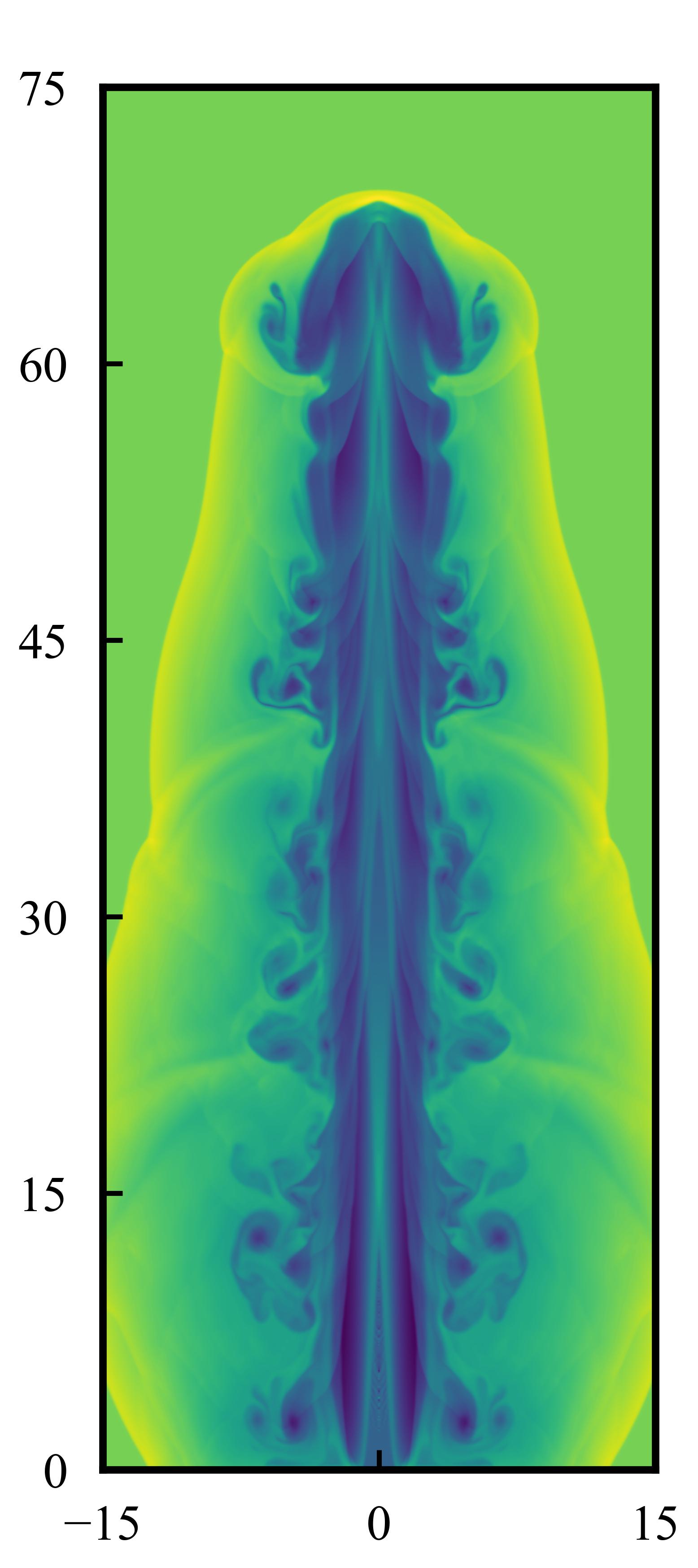}
	\end{subfigure}
	\hfill
	\begin{subfigure}{0.32\textwidth}
		\includegraphics[width=\textwidth, height=8cm]{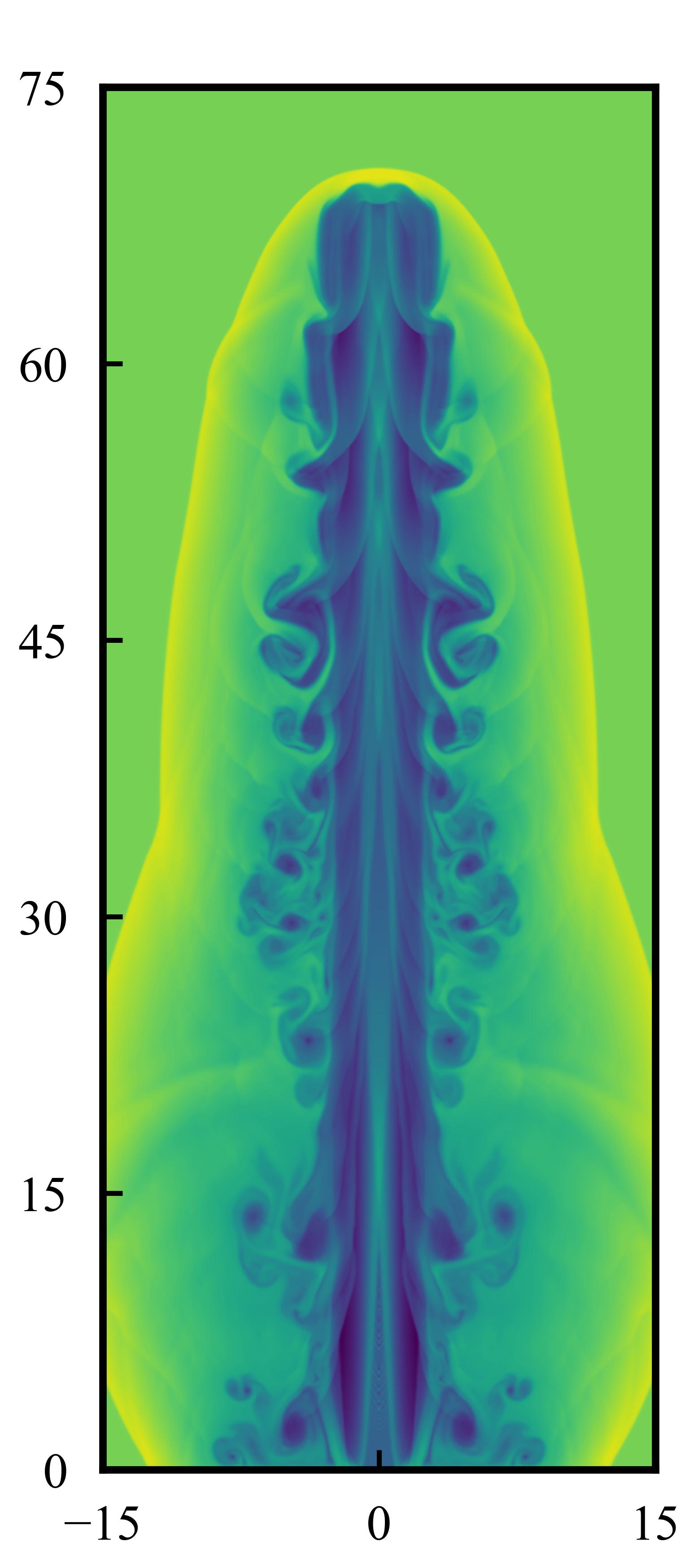}
	\end{subfigure}
	
	\caption{ (Example \ref{Ex:RHDJet}) The contour plots of rest mass density logarithm at $t=100$ of the fifth order FD WENO scheme with three IDP flux limiters. 
		Left: the parametrized limiter, middle: the Hu--Adams--Shu limiter, right: the monolithic convex limiting.
	}
	\label{fig:RHD-Jet}
\end{figure}

\end{example}

\section{Extensions and other approaches}
\label{sec:other}

In this section, we review some extensions and generalizations of the approaches mentioned above, as well as other approaches to enforce invariant domains.
In particular, for compressible MHD equations, the IDP method is more complicated due to the effect of the extra divergence free constraint of the magnetic field on the IDP property, which will be reviewed in \Cref{sec:MHD} with numerical examples shown in \Cref{sec:MHD-numerics}.

\subsection{Other time discretizations}

The Zhang-Shu approach in \Cref{sec:zhang-shu} can be used for any explicit SSP time discretizations. 
In order to use it in other time discretizations, one needs the weak monotonicity, which can be difficult to establish.
The weak monotonicity of backward Euler time stepping for DG methods solving a linear advection was proven in \cite{qin2018implicit}. See \cite{yu2025high} for  scalar convection-diffusion equations.  For DG methods, Lax-Wendroff time stepping was also considered in \cite{moe2017positivity-LW}.
For the flux limiters in \Cref{sec:FCT}, it can be applied to the last time stage of any explicit Runge-Kutta method.   

 {On the other hand, flux limiters are more flexible to use and in general they can be applied to any explicit Runge-Kutta method, e.g., \cite{xiong2013parametrized,kuzmin2022bound}.}
In \cite{ern2022invariant}, every explicit Runge-Kutta method can be made IDP by 
limiting each stage by a forward Euler time step with IDP low order spatial discretization, e.g., the first order schemes in \Cref{sec:firstorderIDP}.
Such a method can also be extended to IMEX (implicit-explicit) schemes \cite{ern2023invariant}, which can be used for convection diffusion equations if the first IMEX is IDP.  {See also \cite{quezada2022maximum} for the Diagonally Implicit Runge--Kutta method.}
 In general, it can be quite difficult to establish an IDP result in an implicit scheme.  
For compressible Navier--Stokes equations, a few implicit and semi-implicit IDP schemes have been constructed in the literature, see  e.g., the second order schemes \cite{grapsas2016unconditionally,guermond2021second-NS}, and  \cite{liu2023positivity-NS} with fourth order spatial accuracy.

\subsection{Lagrangian schemes}
All the schemes mentioned above are for Eulerian schemes defined on a given fixed mesh. 
The techniques and methods in  \Cref{sec:zhang-shu}  and \Cref{sec:FCT} can be extended to Lagrangian type schemes including the semi-Lagrangian and arbitrary Lagrangian–Eulerian  schemes, e.g.,
\cite{qiu2011positivity,rossmanith2011positivity,xiong2014high-Lagrangian,cheng2014positivity,vilar2016positivity,boscheri2017arbitrary,guermond2017-ALE,boscheri2018second,guermond2019arbitrary,guermond2020second,kenamond2021positivity,guermond2023second-ALE}.

\subsection{Subcell limiting methods}

For a finite volume scheme and DG scheme, a subcell limiting is to correct a bad cell solution violating given criteria such as invariant domain, by a first order IDP scheme on subcells of the bad cell. With enough number of subcells in one bad cell, such an IDP method gives a high order accurate correction. 
The subcell limiter in general can be used to preserve more properties such as entropy consideration.
{ See
\cite{dumbser2014posteriori, dumbser2016simple,meister2016positivity, vilar2019posteriori,pazner2021sparse,hajduk2021monolithic,haidar2022posteriori,vilar2024posteriori,vilar2025local} for} subcell limiters for DG and FV methods. 
 
{\color{black}
\subsection{IDP reformation for point values in PAMPA scheme}
The PAMPA method \cite{abgrall2023combination} combines conservative and non-conservative formulations of hyperbolic conservation laws to evolve both cell averages and point values. 
By taking advantage of this flexibility, 
in \cite{Abgrall2025IDP-PAMPA-1D} an automatic IDP reformulation of the non-conservative part was proposed, inspired by Softplus and clipped ReLU mappings from machine learning. 
This  
  yields an unconditionally limiter-free IDP scheme for the point values. In this approach, the Zhang--Shu IDP limiter is still required but solely for the cell midpoint value so that the resulting PAMPA scheme is provably IDP for the updated cell averages.
}

\subsection{Weak monotonicity for convection-diffusion equations}

For preserving bounds or positivity in a scalar equation, the flux limiters in \Cref{sec:FCT} can be easily extended from convection to convection-diffusion equations. 
For the Zhang-Shu method \Cref{sec:zhang-shu}, extensions to convection-diffusion equations would require a 
weak monotonicity result for diffusion equations, which is in general not true for arbitrarily high order DG and FV schemes. 
For special high order schemes, weak monotonicity can still be proven for  diffusion equations, including 
third order direct DG methods \cite{chen2016third, liu2014maximum}, some high order compact finite difference schemes \cite{li2018high}, and a nonstandard finite volume scheme using double cell averages \cite{zhang2012maximum-LiuYY},
all of which are linear schemes, i.e., the scheme is linear when the equation is linear. 
Arbitrarily high order nonlinear DG schemes \cite{sun2018discontinuous,srinivasan2018positivity} can be constructed to be weakly monotone so that the bound-preserving method in \Cref{sec:zhang-shu} still applies.
All these schemes are explicit in time and can be applied to preserve bounds in nonlinear parabolic equations.  

Such weak monotonicity for nonlinear parabolic equations is quite different from the discrete maximum principle finite element methods for linear elliptic and parabolic equations. 
For solving a Poisson equation $-\Delta u=f$, let $-\Delta_h$ denote the discrete Laplacian. If $(-\Delta_h)^{-1}$ is a matrix with positive entries, then we say a scheme is {\it monotone} which implies discrete maximum principle. 
Although high order finite element method is known to violate maximum principle on unstructured meshes,  continuous finite element method with quadratic and cubic polynomial bases on uniform meshes can still be monotone for Laplacian, see \cite{lorenz1977inversmonotonie, li2020monotonicity, cross2023monotonicity} and references therein.  We refer to \cite{barrenechea2024finite,barrenechea2025monotone} for a recent comprehensive review for finite element methods satisfying discrete maximum principle for linear convection diffusion equations.

\subsection{Optimization based approaches for enforcing bounds}

In the literature, bound-preserving limiters and methods can be implemented or formulated as a constrained minimization problem, e.g.,  optimization based limiters for each cell \cite{rider1997constrained,berger2005analysis}. See also \cite{Bochev2012} and references therein. 
There is a natural connection between FCT methods and optimization based method for enforcing bounds and constraints
\cite{liska2010optimization}. In \cite[Section 4.3]{Bochev2012}, it was proven that  Zalesak's original formula \eqref{zelask-1d-scalar} is the minimizer to a global optimization with a modified cost function under box constraints.

There are  advantages of optimization based approaches such as easy treatment for implicit time stepping \cite{van2019positivity-opt-implicit}, flexibility for spatial discretizations \cite{Akil2021}, and easy generalizations to higher order PDEs \cite{liu2024simple-opt-limiter}.   {For preserving bounds of a scalar variable, this has been well studied, e.g.,}
 the remap problem in  arbitrary Lagrangian–Eulerian schemes \cite{bochev2011formulation,bochev2014optimization}.  {In particular, an efficient algorithm was used in \cite{bochev2013fast} to find  the minimizer in $\ell^2$-norm and a direct and cheap construction of one particular minimizer to the $\ell^1$-norm was given in \cite{bradley2019communication}.}
Optimization based postprocessing was also considered by quadratic programming  \cite{guba2014optimization,bochev2020optimization, yee2020quadratic, peterson2024optimization}
as well as first order splitting methods 
\cite{liu2024simple-opt-limiter}.  Gradient descent was used in \cite{zala2023convex}.
In \cite{ruppenthal2023optimal}, Newton's method was used to solve a global optimization problem to find the optimal flux correction in the FCT method for enforcing bounds.
A bound-preserving limiter for total energy was used to enforce positivity of pressure in \cite{liu2024optimization}. 
See also \cite{kirby2024high,keith2024proximal} for enforcing bounds in finite element methods via optimization or variational inequalities for scalar 
convection diffusion problems.
 { On the other hand, for a general system, it is usually difficult to have an efficient optimization based approach with all desired properties enforced. See \cite{liu2025opt} for an optimization based limiter for enforcing the invariant domain set in gas dynamics and global conservation.}

\subsection{Extensions to relativistic hydrodynamics}\label{sec:RHD}

Due to the strong nonlinearity and the effects of curved spacetime in general relativity, the design of IDP schemes encounters several unique challenges. The governing equations of special relativistic hydrodynamics, also known as the relativistic Euler equations, can be written in the form of \eqref{eq:hPDE}:
\begin{equation} \label{eq:relativisticEuler}
		\partial_t	
		\begin{pmatrix}
			D  
			\\
			{\bm m}
			\\
			E	
		\end{pmatrix}
		+ 	 \nabla \cdot
		\begin{pmatrix}
			D {\bm v}
			\\
			{\bm m} \otimes {\bm v} 
			+ p {\bf I}  
			\\
			 {\bm m}
		\end{pmatrix} = {\bf 0},\quad {\bm x}\in \mathbb R^d, 
	\end{equation}	
where $d = 1, 2, 3$ denotes the spatial dimension. The relativistic mass density is $D = \rho W$, with $\rho$ being the rest-mass density and $W = (1 - | {\bm v} |^2)^{-1/2}$ the Lorentz factor. The momentum vector is ${\bm m} = \rho h W^2 {\bm v}$, where $h$ is the specific enthalpy, and the velocity is normalized such that the speed of light is unity. The total energy is given by $E = \rho h W^2 - p$.

An invariant domain of \eqref{eq:relativisticEuler} is
\begin{equation}\label{eq:RHDset}
	G = \left\{ {\bf u} = (D, {\bm m}, E)^\top : D > 0,\; p({\bf u}) > 0,\; | {\bm v}({\bf u}) | < 1 \right\}.
\end{equation}
Here, both ${\bm v}({\bf u})$ and $p({\bf u})$ are nonlinear functions of the conserved variables and lack closed-form expressions, making the design and analysis of IDP schemes nontrivial. To address this difficulty, an equivalent characterization of the set \eqref{eq:RHDset} was proposed in \cite{WU2015539}:
\begin{equation}\label{eq:RHD}
	G = \left\{ {\bf u} = (D, {\bm m}, E)^\top: D > 0,\; q({\bf u}) > 0 \right\},
\end{equation}
where $q({\bf u}) := E - \sqrt{D^2 + |{\bm m}|^2}$ is concave in ${\bf u}$, implying the convexity of $G$.
 Based on this finding,  \Cref{prop:LFS} was rigorously proven in \cite{WU2015539} for the relativistic Euler equations \eqref{eq:relativisticEuler}, and high-order finite difference IDP schemes were developed. Other extensions include high-order IDP (central) DG and finite volume schemes \cite{qin2016bound,wu2016physical,chen2022physical}. See \Cref{Ex:RHDJet} for IDP simulations of a challenging relativistic jet.

Similar to the non-relativistic case discussed in \Cref{ex:entropy}, it was shown in \cite{wu2021minimum,Cui2025LocalMEP-RelEuler} that the relativistic system \eqref{eq:relativisticEuler} also satisfies the minimum entropy principle. Incorporating this leads to another invariant domain { (see \cite{Cui2025LocalMEP-RelEuler} for the general EOS case)}:
\begin{equation}
	G_S = \left\{ 
	{\bf u} = (D, {\bm m}, E)^\top:~ D > 0,\; q({\bf u}) > 0,\; S = \ln \left( \frac{p}{\rho^\gamma} \right) \geq S_{\min}
	\right\},
	\label{G-RHDS}
\end{equation}
where $S({\bf u})$ is not concave, but it was shown in \cite{wu2021minimum} that $D(S - S_{\min})$ is concave in ${\bf u}$, hence $G_S$ remains convex. 
However, \Cref{prop:LFS} (or \Cref{prop:LFS-2D} in 2D) does not hold for $G_S$, and verifying \Cref{prop:RP} (or \Cref{prop:LF-2D}) is highly nontrivial due to the implicit form of $S({\bf u})$.
To overcome this difficulty, the GQL approach can be used to derive an equivalent linear representation of $G_S$:
\begin{equation}
	G_S^\star = \left\{ 
	{\bf u} = (D, {\bm m}, E)^\top:~ D > 0,\; \mathbf{u} \cdot \mathbf{n}^\star + S_{\min} \rho_\star^\gamma > 0~~ \forall \rho_\star > 0, \forall {\bm v}^\star \in \mathbb{B}_1({\bf 0})
	\right\},
	\label{G-RHDS2}
\end{equation}
where $\mathbf{n}^\star := \left( -\sqrt{1 - |{\bm v}^\star|^2} \left( 1 + \frac{S_{\min} \gamma \rho_\star^{\gamma - 1}}{\gamma - 1} \right), -{\bm v}^\star, 1 \right)^\top$ is the inward normal vector to $\partial G_S$, and $\mathbb{B}_1({\bf 0})$ denotes the unit ball in $\mathbb{R}^d$. 
Thanks to the GQL technique, \Cref{prop:wLFS} (\Cref{prop:wLFS-2D} in 2D) was rigorously proven in \cite[Theorem 3.9]{wu2021minimum}, implying \eqref{assumption2} and its multidimensional counterparts without using any assumptions on the exact Riemann solutions. Consequently, high-order numerical schemes preserving the invariant domain $G_S$ in \eqref{G-RHDS} were successfully developed in \cite{wu2021minimum}.

 In general relativity, new challenges arise because the invariant domain depends on the spacetime metric and therefore varies across space, to which  standard convex combination techniques do not directly apply. To address this issue, a new formulation (W-form) of the general relativistic hydrodynamic equations was introduced in \cite{wu2017design}, enabling the generalization of high-order IDP schemes to curved spacetimes \cite{wu2017design,cao2025robust}.

\subsection{Extensions to compressible MHD}\label{sec:MHD}


The MHD equations presents a nonlinear coupling of fluid dynamics with Maxwell's equations; see \eqref{eq:rephrased-idealMHD} for the classical MHD system and \eqref{eq:Rephrased-RMHD} for the relativistic version. These systems involve two key structures: a divergence-free constraint on the magnetic field, and IDP constraints on fluid variables (positivity of density/pressure, subluminal velocity). While divergence-free schemes are well-established, constructing provably IDP methods---especially in multiple dimensions---has been an open problem. 
In the non-relativistic setting, IDP schemes are also called positivity-preserving schemes. Although many techniques and limiters (e.g., \cite{cheng2013positivity-MHD,christlieb2015positivity,christlieb2016high}) were adapted to enforce positivity in ideal MHD, few of the resulting schemes were rigorously and completely proven to be IDP in theory, even for the first-order schemes, especially in multiple dimensions  \cite{christlieb2015positivity}.

The development of IDP schemes for MHD systems encounters unique challenges that do not typically arise in other hyperbolic systems. A fundamental difficulty lies in the lack of understanding of the intrinsic connection---if any---between the IDP property and the divergence-free constraint on the magnetic field. Do such connections exist? If so, how are they expressed mathematically, and 
 how can they be established? These central questions remained open, until the work in \cite{wu2017admissible,wu2018positivity}. 
 Identifying this connection is crucial, as it may provide the foundation for designing provably IDP methods for MHD. 
 Notably, the IDP constraints are pointwise {\em algebraic} structure, whereas the divergence-free constraint is a {\em differential} structure in nature, making it inherently difficult to bridge the two.
 
 This connection was first theoretically revealed in \cite{wu2017admissible,wu2018positivity} for cartesian meshes, in \cite{wu2019provably} for general unstructured meshes, and in \cite{wu2023provably} for central DG schemes on overlapping meshes. It was shown that a discrete divergence-free condition is essential for ensuring the IDP property. Even minor violations of this condition may lead to the loss of the IDP property. Moreover, because of the discrete divergence-free constraint, states at different quadrature points become strongly coupled, rendering classical convex decomposition techniques (decomposing multidimensional schemes into a convex combination of formally first-order schemes), such as \eqref{eq:edefd}, ineffective. 
 Additionally, none of Assumptions 1--6 hold for multidimensional MHD when there is a discontinuity in the normal component of the magnetic field.

This difficulty was ultimately overcome using the GQL approach, yielding the following representation of $G$:
\begin{equation}\label{eq:GQL-MHD}
G^\star = \left\{  {\bf u} \in \mathbb R^{2d+2}: \rho>0,~( {\bf u} - {\bf u}^\star ) \cdot {\bf n}^\star > 0~~\forall {\bf u}^\star  \right\},
\end{equation}
where ${\bf u}^\star \in \partial G$ corresponds to an arbitrary state with thermal pressure $p^\star = 0$, and ${\bf n}^\star$ is an inward normal vector to $\partial G$ at ${\bf u}^\star$; see \cite{wu2018positivity} for ideal MHD and \cite{wu2017admissible} for relativistic MHD. GQL enables precise mathematical formulation of the relationship between the IDP property and the divergence-free condition \cite{wu2017admissible,wu2018positivity,wu2019provably,wu2023provably}. In particular, it was shown that a modified version of \Cref{prop:wLFS} holds:

\begin{theorem}[Validity of a modified \Cref{prop:wLFS} in relativistic MHD]\label{thm:MHD1}
If setting ${\zeta}({\bf u^\star}) = - v_\ell^\star p_m^\star$, where $v_\ell^\star$ and $p_m^\star$ are respectively the $\ell$th component of velocity vector ${\bm v}^\star$ and the magnetic pressure at the state ${\bf u}^\star$, then we have  
	\begin{equation}\label{eq:wLFS-MHD}
		{\bf u}\in G \quad \Longrightarrow \quad  
		\alpha({\bf u}-{\bf u}^\star)\cdot \mathbf{n}^\star 
		\pm  
		{\bf f}_\ell({\bf u})\cdot \mathbf{n}^\star
		> \pm \zeta({\bf u^\star}) \pm B_\ell( {\bm v}^\star \cdot {\bm B}^\star )~~~ \forall {\bf u}^\star, \forall \alpha \ge 1,
	\end{equation}
 where ${\bf f}_\ell$ is the $\ell$th component of the flux ${\bf f}=({\bf f}_1,\dots,{\bf f}_d)$, $\ell=1,\dots,d$. 
\end{theorem}

The additional term $\pm B_\ell({\bm v}^\star \cdot {\bm B}^\star)$ in \eqref{eq:wLFS-MHD} is essential; without it, the inequality reduces to the original \Cref{prop:wLFS}, which does not hold in general in the MHD case. This term captures the fundamental influence of the divergence-free condition on the IDP property.

\begin{example}
	 To illustrate the basic idea, consider the 1D case, where the divergence free condition simplifies to $B_1\equiv const$. 
	 Assume ${\bf u}_L, {\bf u}_R \in G$, then it follows from \eqref{eq:wLFS-MHD} that 
	 \begin{align*}
	 	\alpha({\bf u}_L-{\bf u}^\star)\cdot \mathbf{n}^\star 
	 	+  
	 	{\bf f}_1({\bf u}_L)\cdot \mathbf{n}^\star
	 	&>  \zeta({\bf u^\star}) + B_{1,L} ( {\bm v}^\star \cdot {\bm B}^\star ),
	 	\\
	 	\alpha({\bf u}_R-{\bf u}^\star)\cdot \mathbf{n}^\star 
	 	- 
	 	{\bf f}_1({\bf u}_L)\cdot \mathbf{n}^\star
	 	&> - \zeta({\bf u^\star}) - B_{1,R}( {\bm v}^\star \cdot {\bm B}^\star )
	 \end{align*}
	 If we further assume $B_{1,L} = B_{1,R}$, averaging the two inequalities yields
	 $$
	 \alpha \left( \overline {\bf u} -{\bf u}^\star \right)\cdot \mathbf{n}^\star > \frac12 ( B_{1,L} -B_{1,R})  ( {\bm v}^\star \cdot {\bm B}^\star ) = 0,
	 $$
	 where 
	 $
	 \overline {\bf u} = \frac{ {\bf u}_L + {\bf u}_R  }2 + \frac{ {\bf f}_1 ( {\bf u}_L ) - {\bf f}_1 ( {\bf u}_R )  }{2 \alpha }. 
	 $
	 By the GQL representation \eqref{eq:GQL-MHD}, we conclude
	 \begin{equation}\label{eq:MHDss}
	 {\bf u}_L, {\bf u}_R \in G,~ B_{1,L} = B_{1,R}   \quad  \Longrightarrow \quad   \frac{ {\bf u}_L + {\bf u}_R  }2 + \frac{ {\bf f}_1 ( {\bf u}_L ) - {\bf f}_1 ( {\bf u}_R )  }{2 \alpha }  \in G~~~ \forall \alpha \ge 1.
	 \end{equation}
\end{example}

	 This shows that \eqref{assumption2} holds for relativistic MHD, but only under the additional discrete divergence-free constraint $B_{1,L} = B_{1,R}$. See \cite{wu2018positivity} for the non-relativistic counterpart. 
While this discrete divergence-free condition is trivial in 1D, it becomes significantly more complex in multiple dimensions \cite{wu2019provably,wu2021provably}. 
For instance, a 2D version of \eqref{eq:MHDss} for first-order IDP schemes is given as follows:
\begin{theorem}\label{thm:MHD2}
If ${\bf u}_L, {\bf u}_R, {\bf u}_D, {\bf u}_U \in G$ and the 2D discrete divergence-free constraint $\frac{B_{1,R}-B_{1,L}}{\Delta x} + \frac{B_{2,U}-B_{2,D}}{\Delta y} = 0$ holds, 
then 
\begin{align*} 
	\frac{1}{ \frac{\alpha_x}{\Delta x} + \frac{\alpha_y}{\Delta y} } \Bigg[ & \frac{\alpha_x}{\Delta x}\left( \frac{ {\bf u}_L + {\bf u}_R  }{2} + \frac{ {\bf f}_1 ( {\bf u}_L ) - {\bf f}_1 ( {\bf u}_R )  }{2 \alpha_x } \right) 
	\\
	+ & \frac{\alpha_y}{\Delta y}\left( \frac{ {\bf u}_D + {\bf u}_U  }{2}  + \frac{ {\bf f}_2 ( {\bf u}_D ) - {\bf f}_2 ( {\bf u}_U )  }{2 \alpha_y } \right) \Bigg] \in G~~~ \forall \alpha_x, \alpha_y \ge 1.
\end{align*}
\end{theorem}
It is worth noting that the GQL auxiliary variables in the additional term $\pm B_\ell({\bm v}^\star \cdot {\bm B}^\star)$ play a critical role in linking the IDP and discrete divergence-free properties. 

We refer interested readers to \cite{wu2019provably,wu2021provably} for the multidimensional and higher-order versions of \Cref{thm:MHD1} and \Cref{thm:MHD2} and the corresponding discrete divergence-free constraints on general meshes. Notably, in multiple dimensions, these discrete divergence-free conditions are strongly coupled, involving information from neighboring cells. While global divergence-free discretizations can ensure these constraints, the use of local scaling IDP limiters, such as the Zhang–Shu limiter, typically destroys the global divergence-free property of the magnetic field. Consequently, it is nontrivial to construct high-order {\em strictly conservative} schemes that are both IDP and globally divergence-free.

	 Interestingly, at the continuous level, Wu and Shu discovered that the IDP property of the exact solution is also tightly linked to the divergence-free condition. In \cite{wu2018provably,wu2021provably}, it was shown that if the magnetic field is not divergence-free, even the exact solution of the MHD system may violate pressure positivity. This implies that the set \eqref{eq:rephrased-G-iMHD}, and its relativistic counterpart \eqref{eq:Rephrased-RMHD-G}, is no longer an invariant domain of the PDE. Therefore, when the numerical magnetic field fails to satisfy the divergence-free condition, even an exact PDE solver (assuming it were available) cannot guarantee IDP, highlighting the inherent complexity of constructing genuinely IDP schemes for multidimensional MHD.

 Moreover, in \cite{wu2019provably}, Wu and Shu observed that the symmetrizable modified formulation of the ideal MHD system---originally introduced by Godunov \cite{Godunov1972}---always admits the set \eqref{eq:rephrased-G-iMHD} as an invariant domain, regardless of whether the magnetic field is divergence-free. For the relativistic system, the symmetrizable modified formulation enjoyed the same feature was recently shown in \cite{wu2020entropy}.   
Motivated by this observation, Wu and Shu \cite{wu2018provably,wu2019provably,wu2021provably} proved that the IDP property of numerical schemes based on these symmetrizable modified formulations only depends on a (discrete) locally divergence-free constraint. Crucially, this local constraint is compatible with local scaling IDP limiters, such as the Zhang--Shu limiter. Based on these findings, a series of structure-preserving frameworks, provably IDP and locally divergence-free, have been systematically developed in \cite{wu2018provably,wu2019provably,wu2021provably,wu2023provably,ding2024new,ding2024gql,liu2025structure,ding2025divergence,CaiQiuWu2025NRMHD} for compressible MHD systems using the symmetrizable modified formulation.

 \begin{remark} 
 	{  A localized element-based positivity-preserving FCT approach with divergence cleaning was proposed for continuous finite element discretization of the MHD system  in 
 		\cite{KUZMIN2020109230}.} 
 	A second-order structure-preserving finite element method for ideal MHD was recently proposed in \cite{dao2024structure}. 
 	 This method combines convex limiting with a novel operator splitting technique. {  A constrained transport method provably preserving positivity and divergence-free constraint was further introduced for MHD in \cite{PANG2025114312}.}
\end{remark}

\subsection{Numerical results for ideal MHD and relativistic MHD equations}
\label{sec:MHD-numerics}

\begin{example}[Shock cloud interaction]\label{Ex:ShockCloud}
	In this example,  
	we solve the ideal MHD equations to simulate the interaction of a strong shock wave and a high density cloud with the adiabatic index $\gamma = 5/3$. The computational domain is chosen to be $[0,1]^2$, as in \cite{wu2018provably,wu2019provably,wu2023provably}. 
	The problem is simulated up to $t=0.06$ with $400\times 400$ uniform cells.
	\Cref{fig:MHD-ShockCloud} presents the contours of the density, the thermal pressure, and the magnitude of magnetic field obtained by the third order IDP DG method and fifth order IDP finite volume method. 
    It is worth noting that the schemes would produce nonphysical solutions without using the IDP limiter.
	
	\begin{figure}[!htbp]
		\centering
		\begin{subfigure}{0.48\textwidth}
			\includegraphics[width=\textwidth]{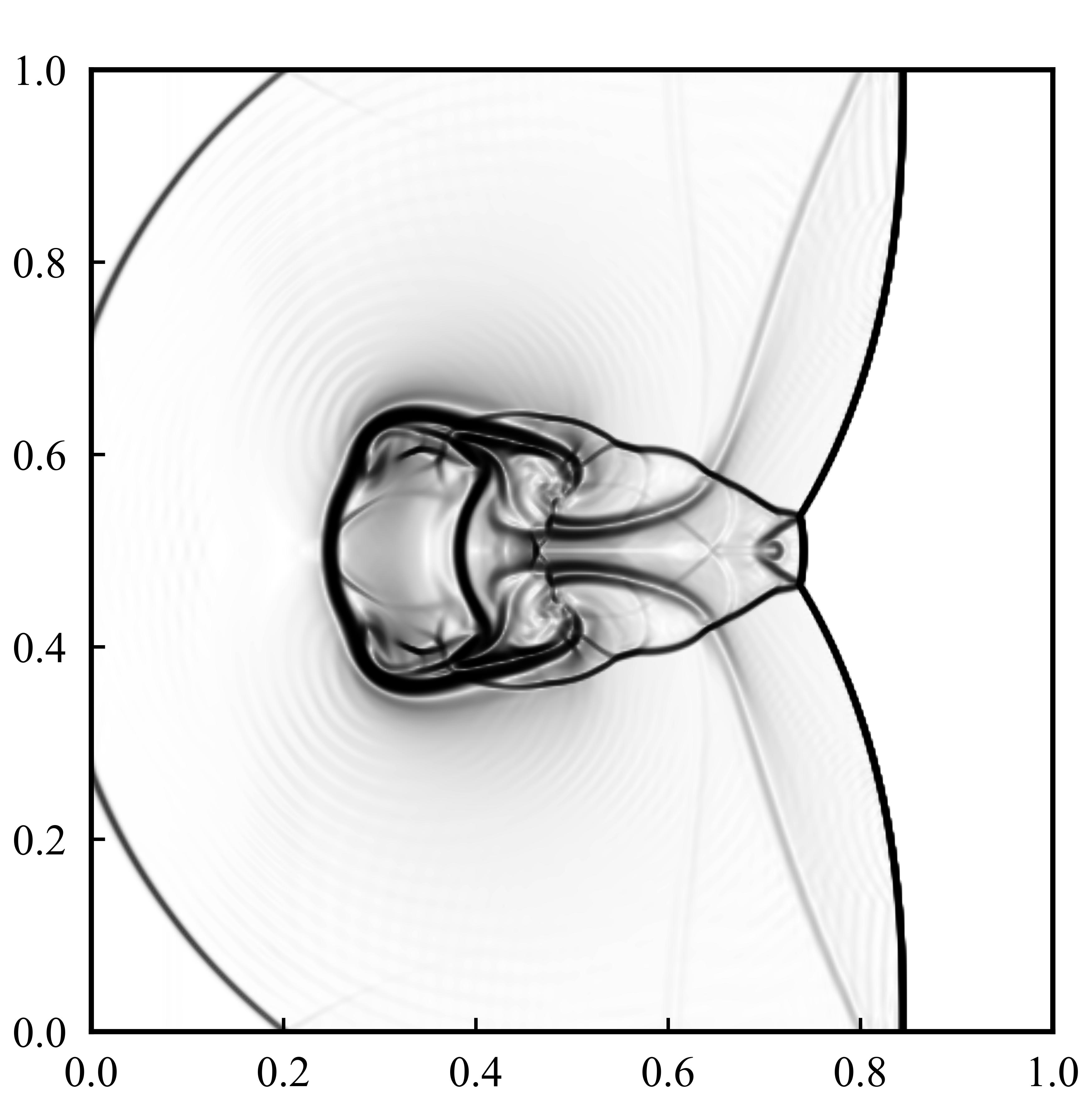}
		\end{subfigure}
           \hfill 
           \begin{subfigure}{0.48\textwidth}
			\includegraphics[width=\textwidth]{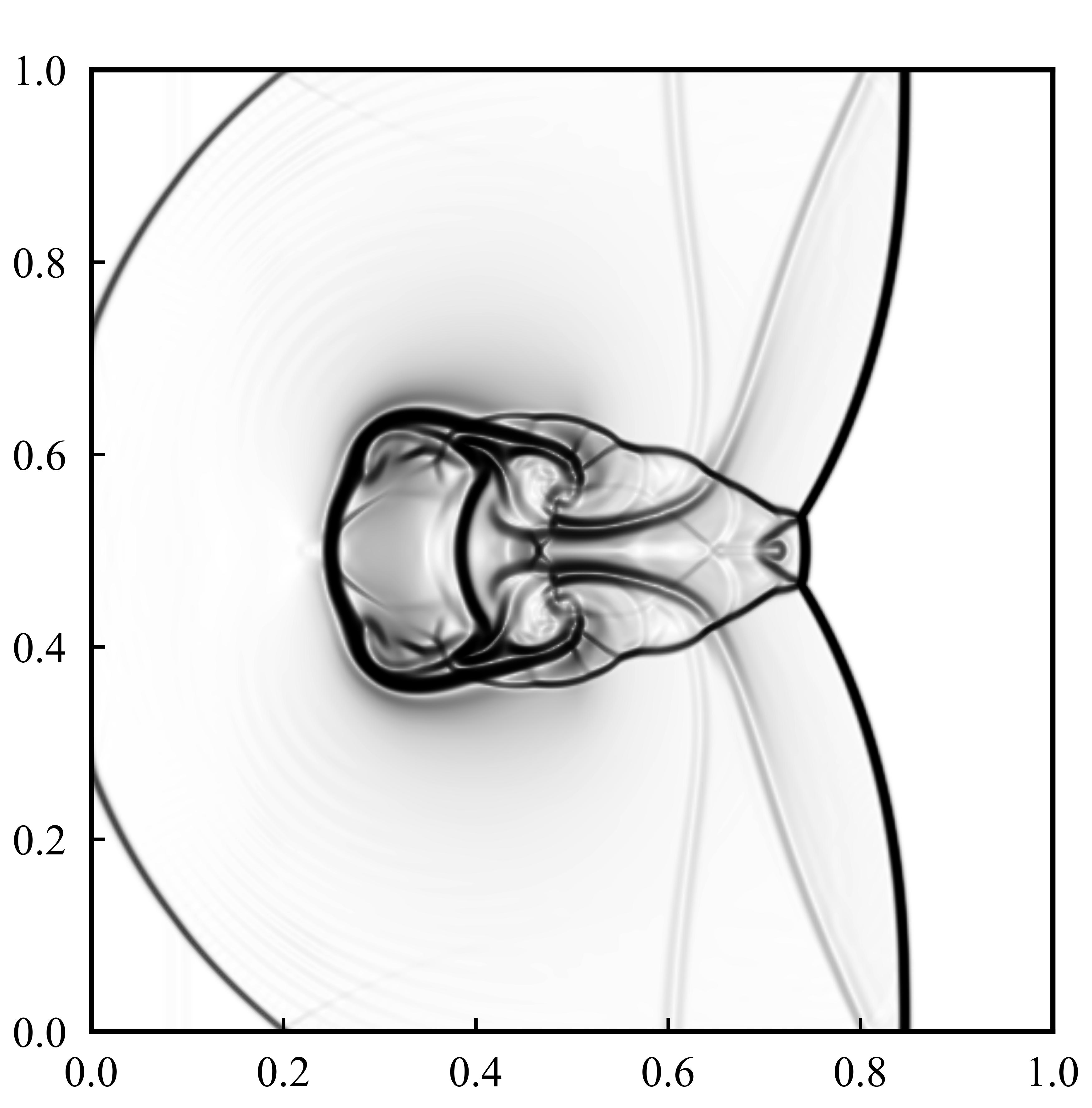}
		\end{subfigure}
            
		\hfill
		\begin{subfigure}{0.48\textwidth}
			\includegraphics[width=\textwidth]{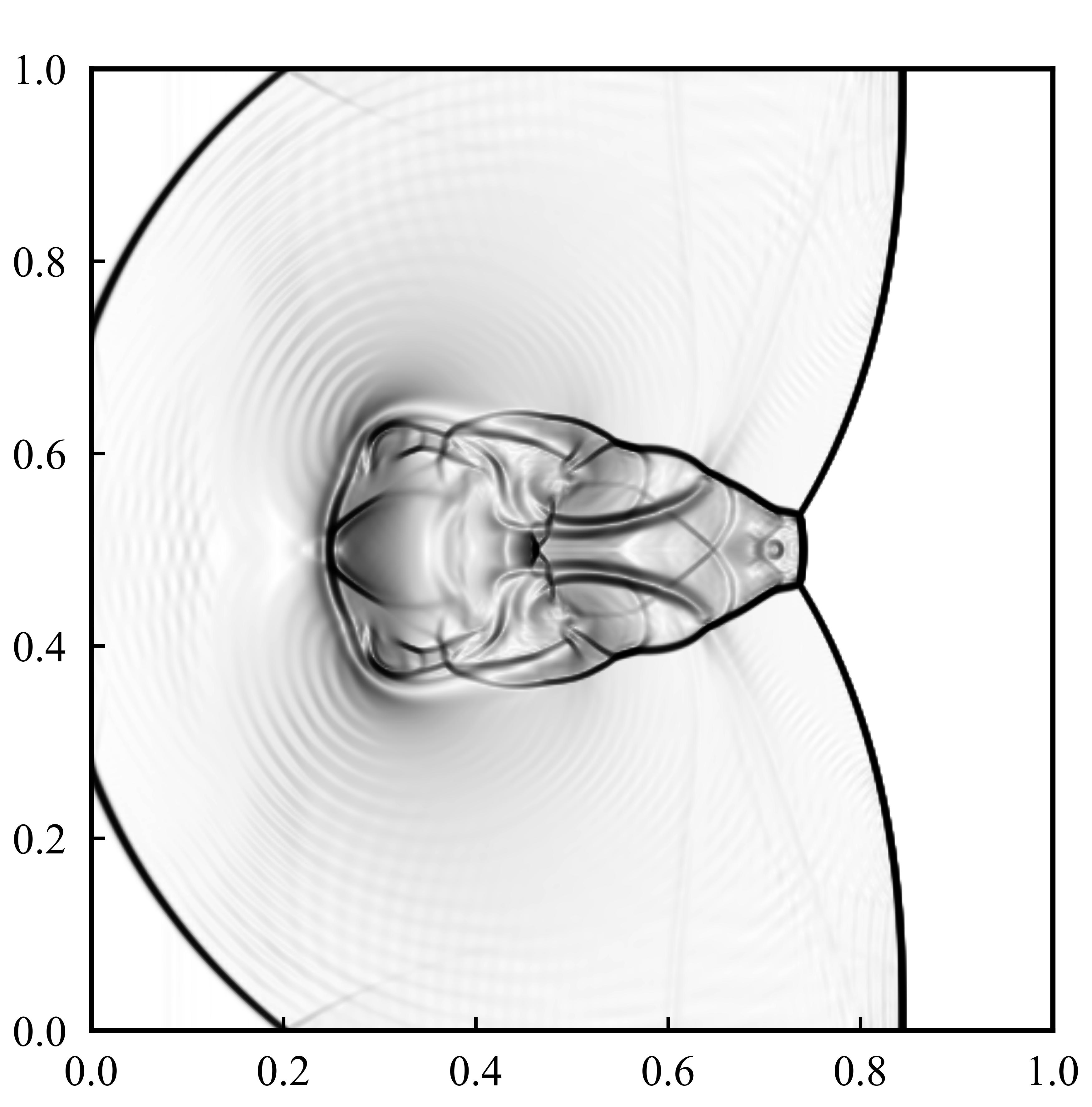}
		\end{subfigure}
        \hfill
		\begin{subfigure}{0.48\textwidth}
			\includegraphics[width=\textwidth]{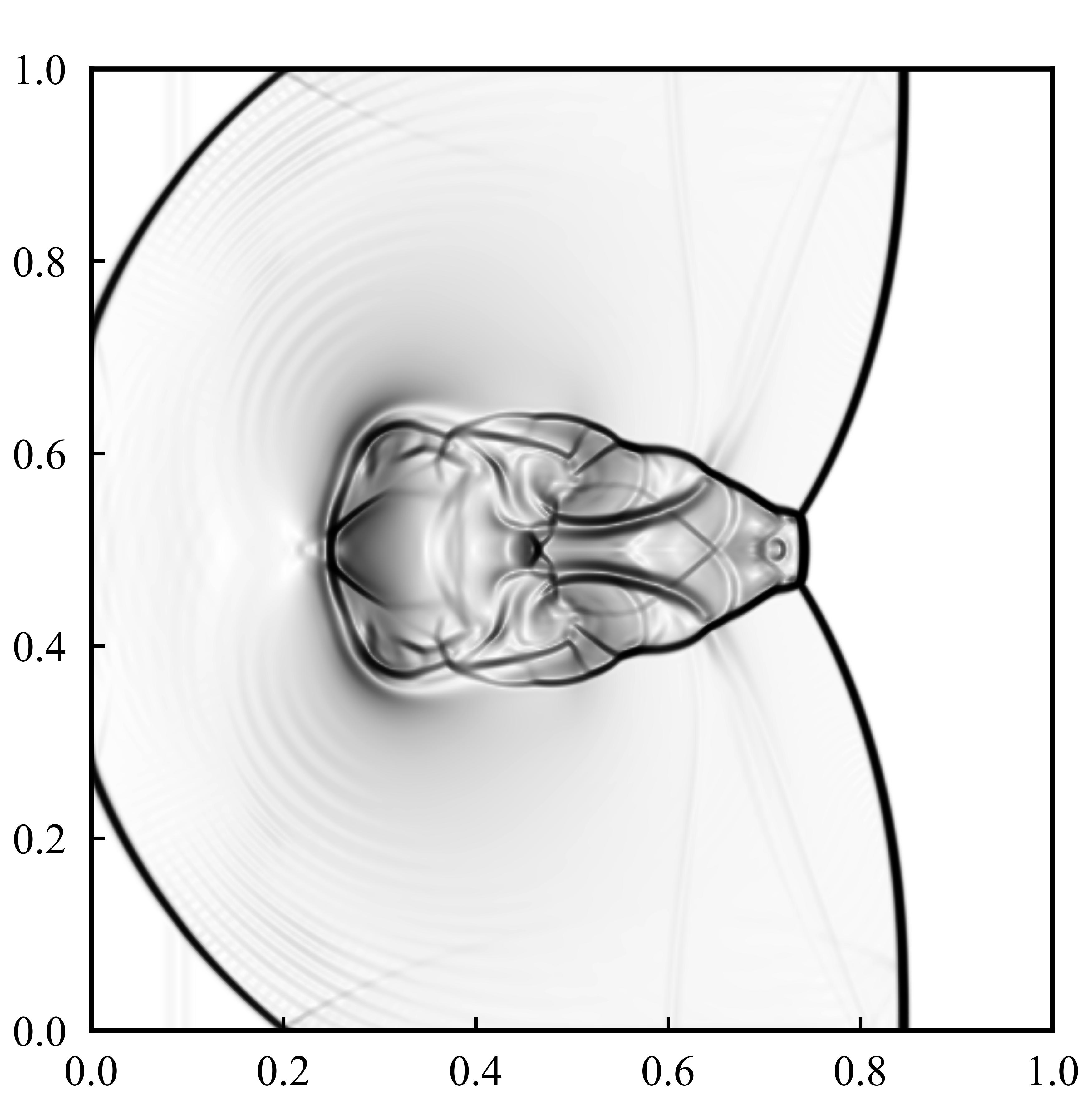}
		\end{subfigure}

		\begin{subfigure}{0.48\textwidth}
			\includegraphics[width=\textwidth]{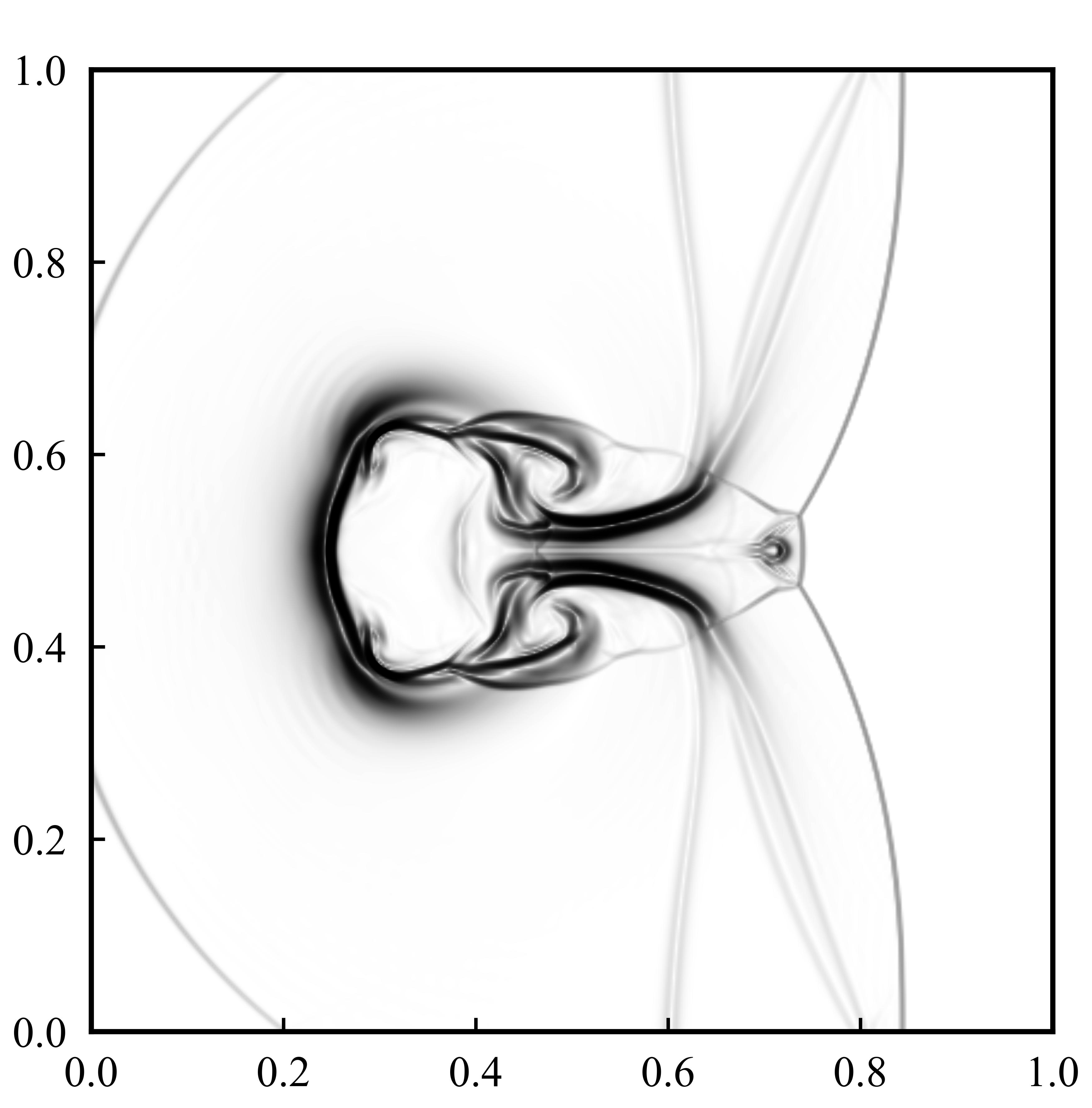}
		\end{subfigure}
		\hfill
		\begin{subfigure}{0.48\textwidth}
			\includegraphics[width=\textwidth]{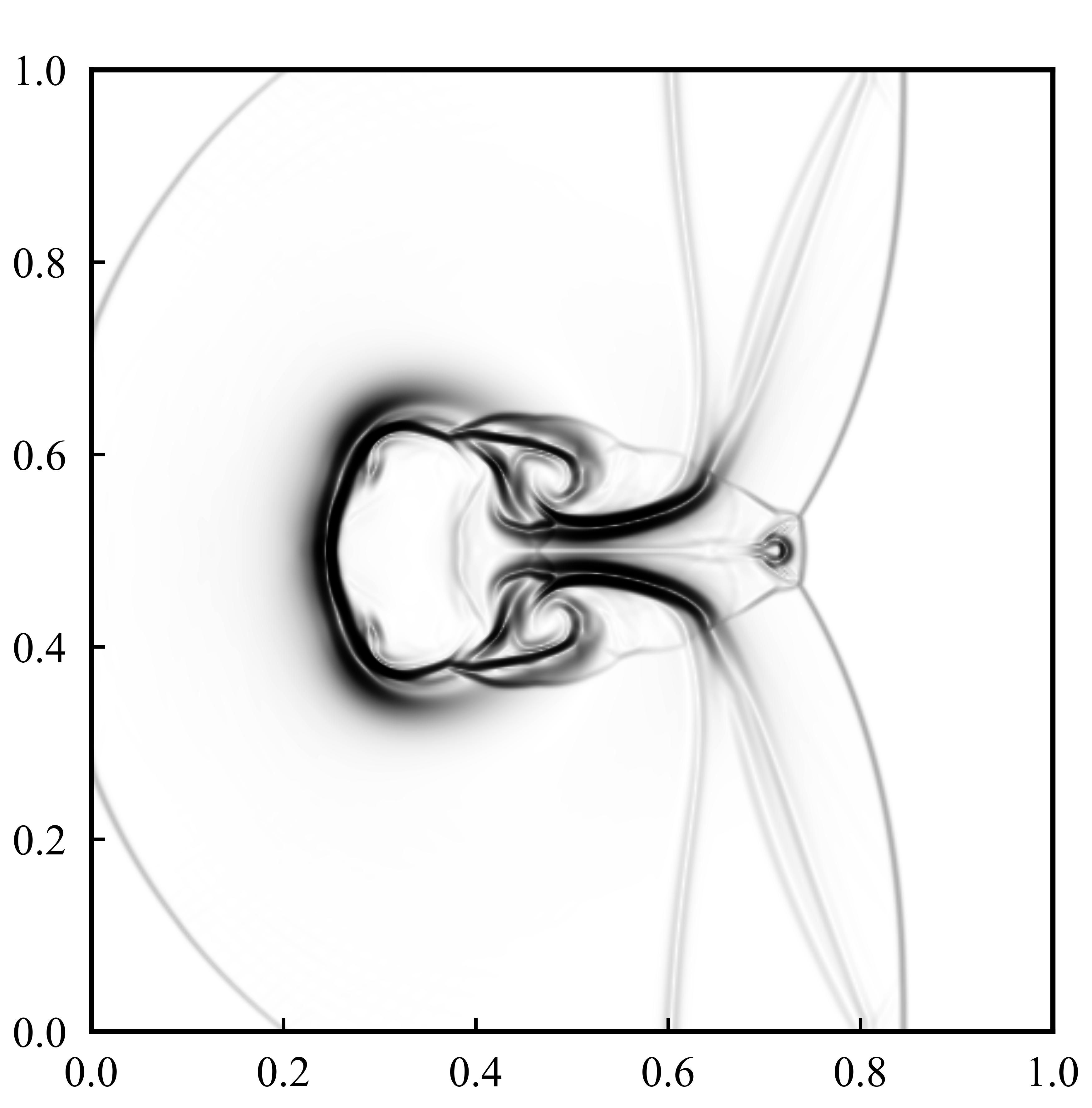}
		\end{subfigure}
		
		\caption{(Example \ref{Ex:ShockCloud}) The density logarithm (top), thermal pressure (middle), and magnitude of magnetic field (bottom).
		Left: third order IDP DG scheme. Right: fifth order IDP finite volume scheme.
		}
		\label{fig:MHD-ShockCloud}
	\end{figure} 
	
\end{example}

\begin{example}[Astrophysical MHD jet with huge Mach number]\label{Ex:Jet}
The high-speed MHD jets are  proposed and first simulated in \cite{wu2018provably,wu2019provably,wu2021provably}. 
	This test simulates an extremely high Mach number jet problem in a strong magnetic field. 
    The adiabatic index is set to be $\gamma=1.4$. 
    Initially, the domain $[-0.5,0.5]\times[0,1.5]$ is full of the ambient plasma with 
	$(\rho, {\bm v}, \mathbf{B}, p) = (0.1\gamma, 0, 0, 0, 0, 2\times 10^5, 0, 1).$ 
	A high-speed jet initially locates at $x\in [-0.05,0.05]$ and $y=0$, it is injected in $y$-direction of the bottom boundary with the inflow jet condition
	$(\rho, {\bm v}, \mathbf{B}, p) = (\gamma, 0, 10^{6}, 0, 0, 2\times 10^5, 0, 1).$ 
	In our test, the computational domain is taken as $[0,0.5]\times [0,1.5]$ and divided into $200\times 600$ cells. For the left boundary $x=0$, the reflecting boundary condition is imposed. The outflow conditions are applied on other boundaries. The final time is $t=1.8\times 10^{-6}$. 
	\Cref{fig:Ex-MHDJet} shows the schlieren images of density logarithm $\log{\rho}$ obtained by the third order IDP DG method and fifth order IDP finite volume method.
	
\begin{figure}[!htbp]
	\centering
	\begin{subfigure}{0.48\textwidth}
		\includegraphics[width=\textwidth]{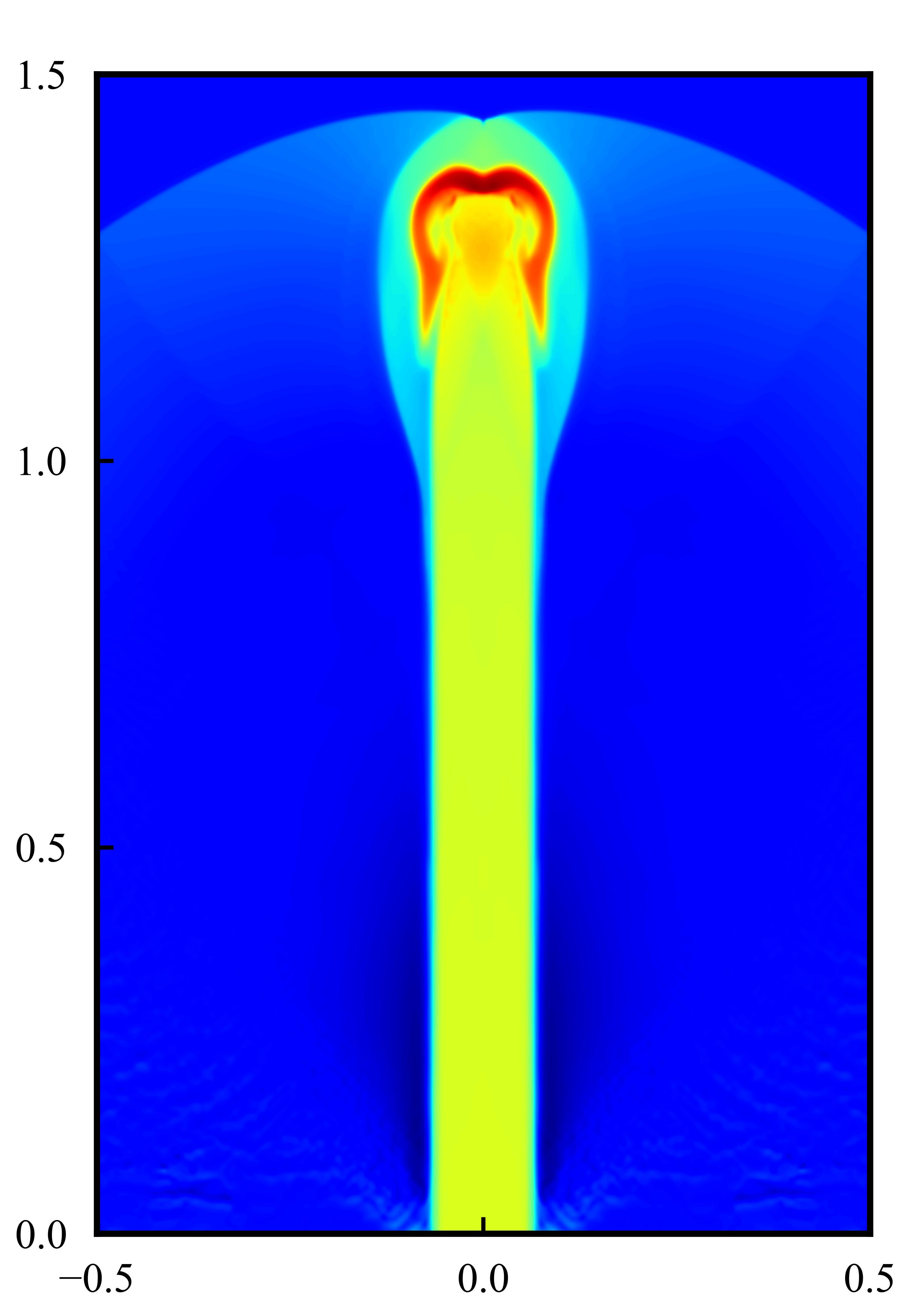}
	\end{subfigure}
	\hfill 
	\begin{subfigure}{0.48\textwidth}
		\includegraphics[width=\textwidth]{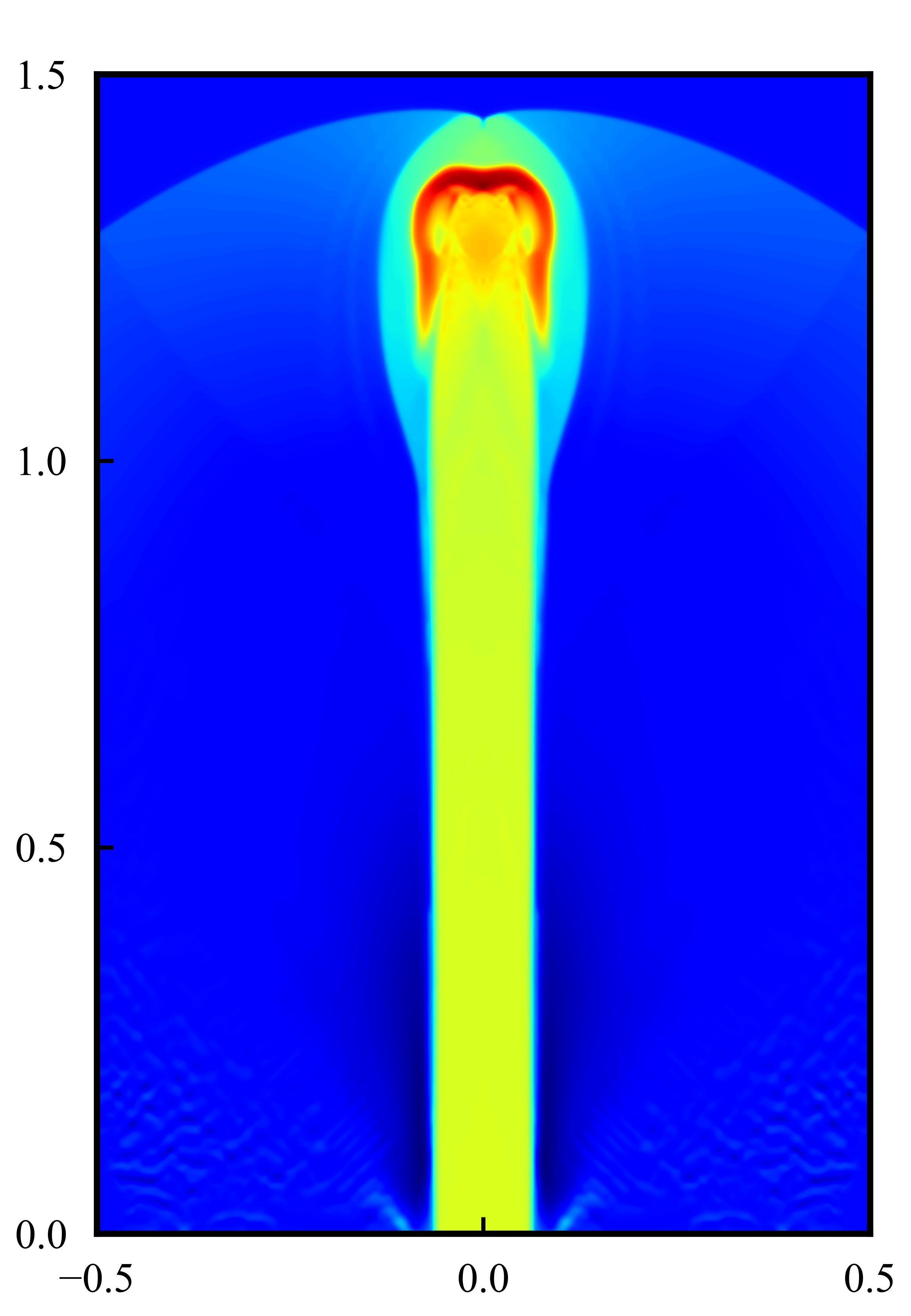}
	\end{subfigure}
	
	\begin{subfigure}{0.48\textwidth}
		\includegraphics[width=\textwidth]{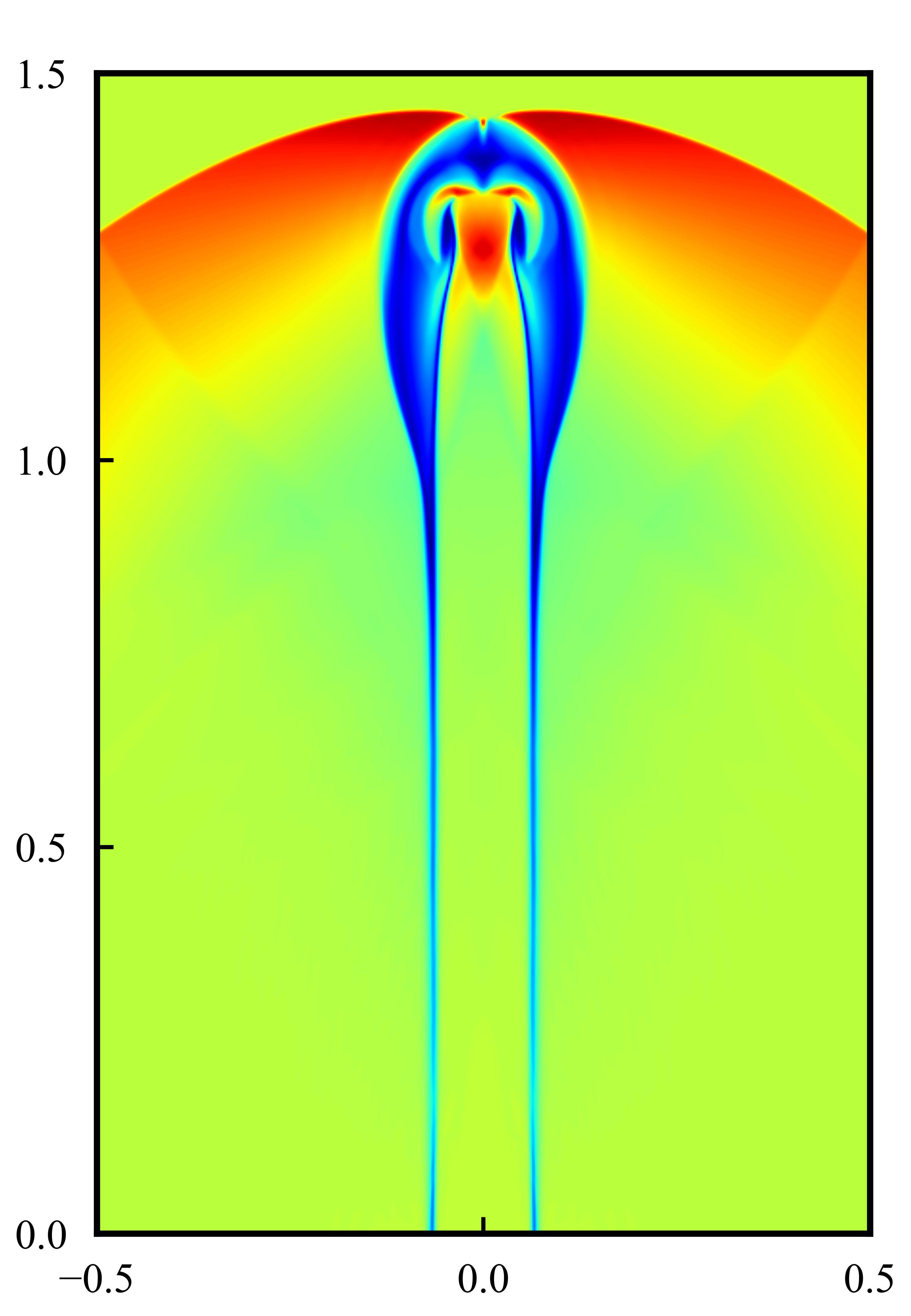}
	\end{subfigure}
	\hfill 
	\begin{subfigure}{0.48\textwidth}
		\includegraphics[width=\textwidth]{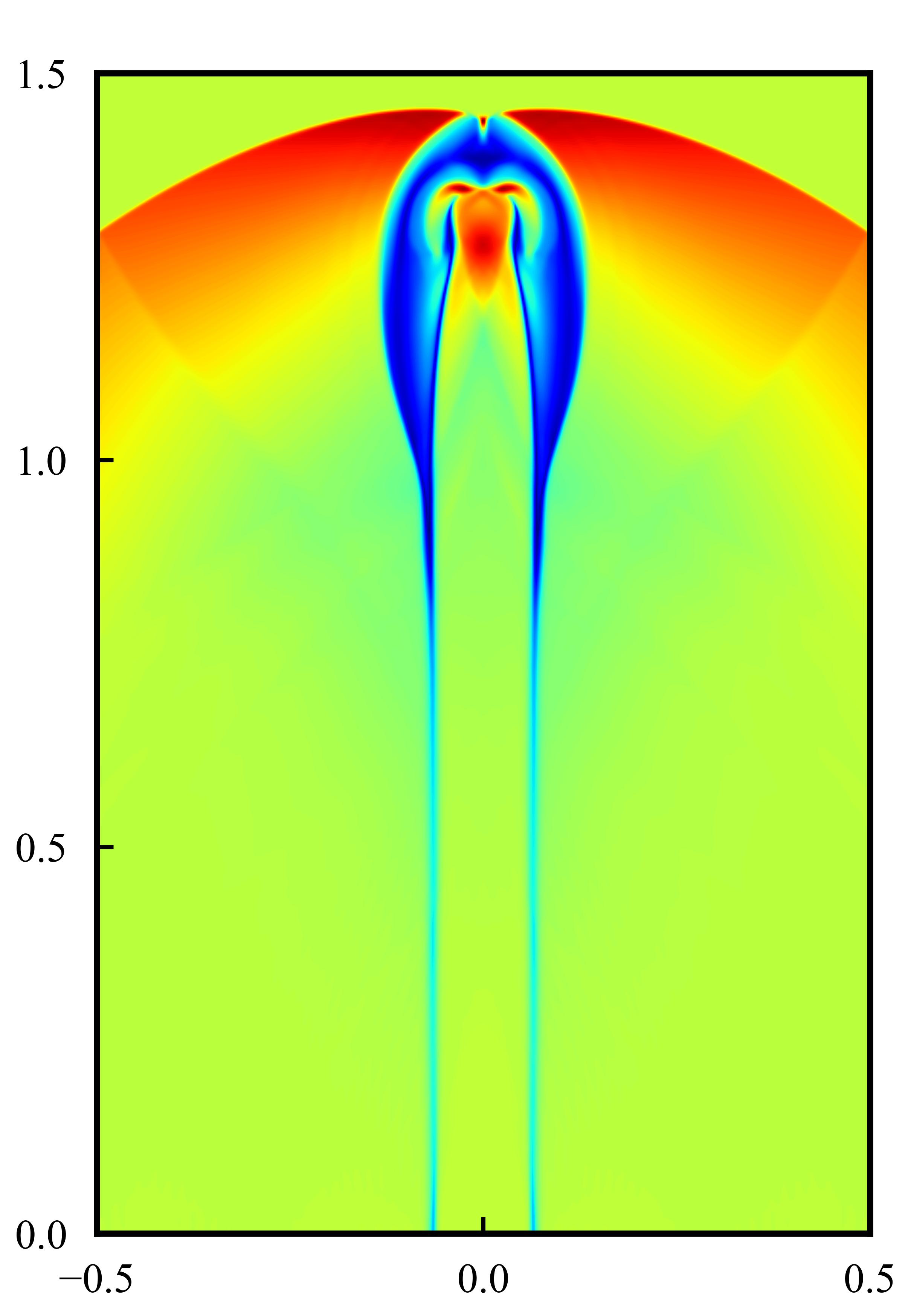}
	\end{subfigure}

	\caption{ (Example \ref{Ex:Jet}) Density logarithm (top) and the magnitude of magnetic
		field (bottom) at $1.8\times 10^{-6}$. Left: third order IDP DG method, right: fifth order IDP finite volume method.
	}
	\label{fig:Ex-MHDJet}
\end{figure} 
	
\end{example}

\begin{example}[Orszag--Tang problem]\label{Ex:OT} 
	We simulates the 2D Orszag--Tang problem of the relativistic MHD equations, 
	 following the setup in \cite{VANDERHOLST2008617,wu2021provably}. 
    In this problem, the initial maximum velocity reaches $0.99$, close to the speed of light. 
 \Cref{fig:RMHD-OT} presents the numerical results, obtained by the third order IDP DG method and fifth order IDP finite volume method 
    with $600\times 600$ uniform cells in the domain $[0,2\pi]^2$, 
    for the logarithm of the rest-mass density $\log{\rho}$ at $t=2.818127$ and $6.8558$. 
	As time progresses, complex wave structures are generated and correctly captured by our method. 
	The results agree with those reported in \cite{wu2021provably,VANDERHOLST2008617}. 
	
	\begin{figure}[!htbp]
		\centering
		\begin{subfigure}{0.48\textwidth}
			\includegraphics[width=\textwidth]{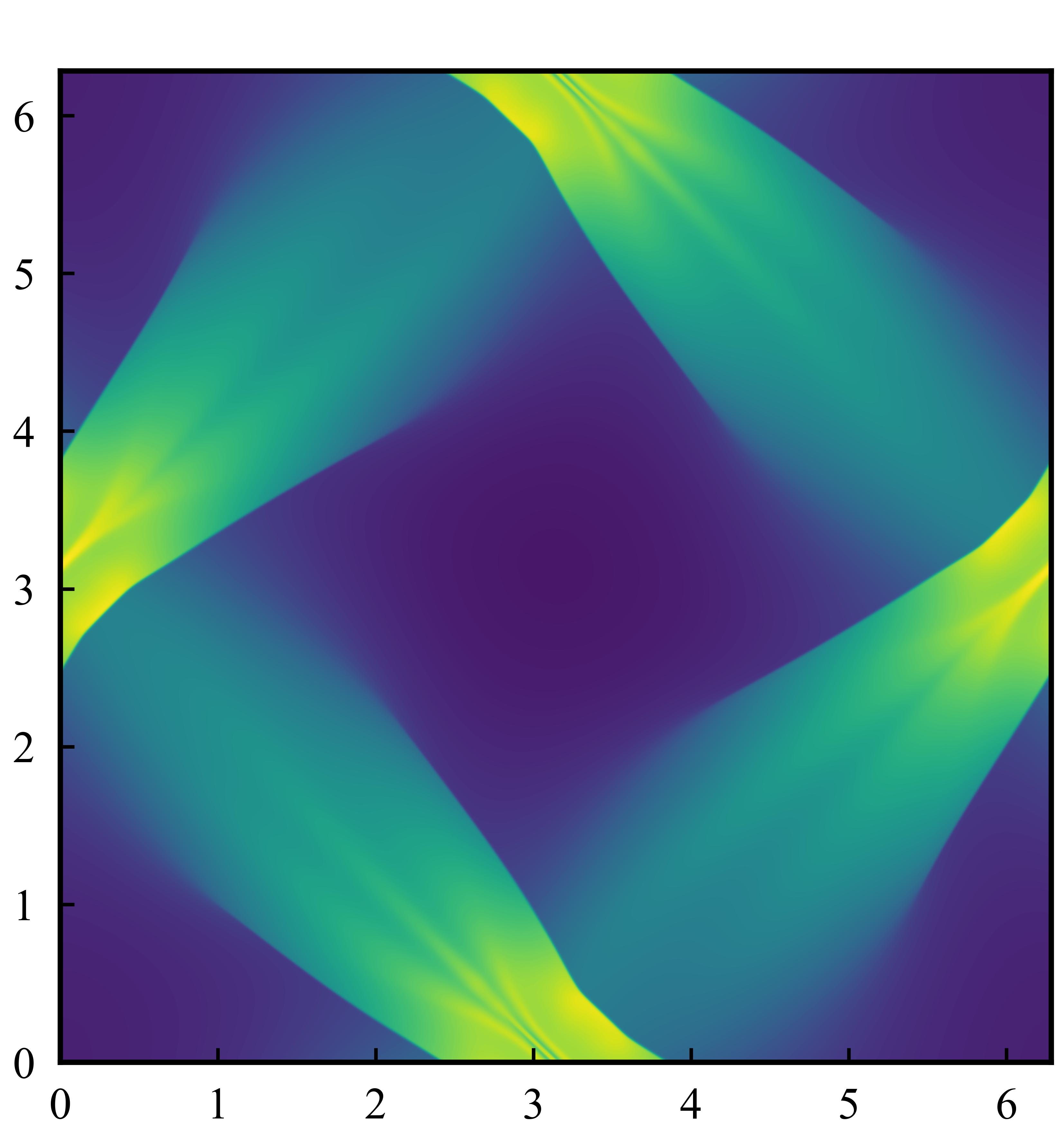}
		\end{subfigure}
		\hfill 
        \begin{subfigure}{0.48\textwidth}
			\includegraphics[width=\textwidth]{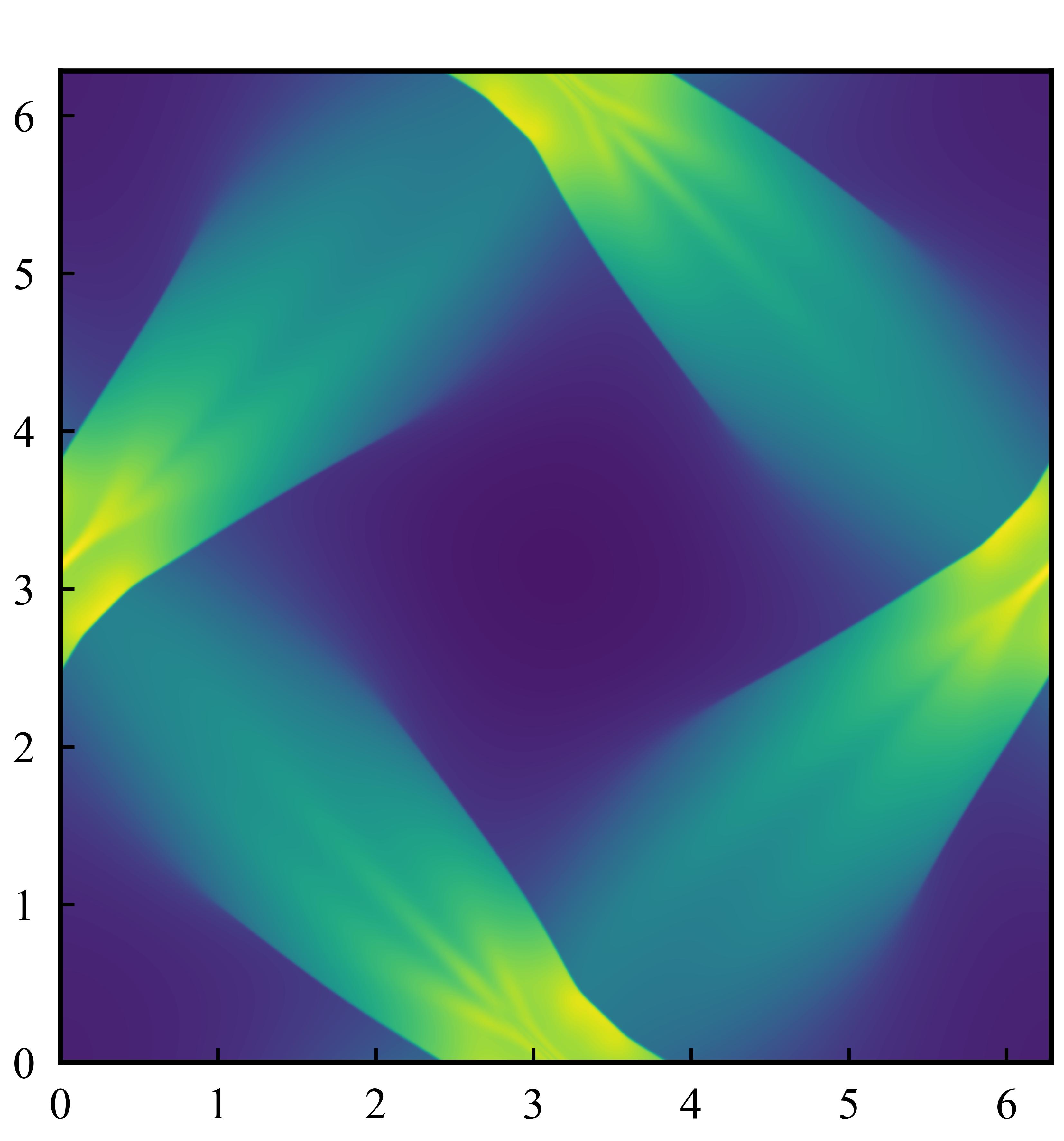}
		\end{subfigure}

		\begin{subfigure}{0.48\textwidth}
			\centering
			\includegraphics[width=\textwidth]{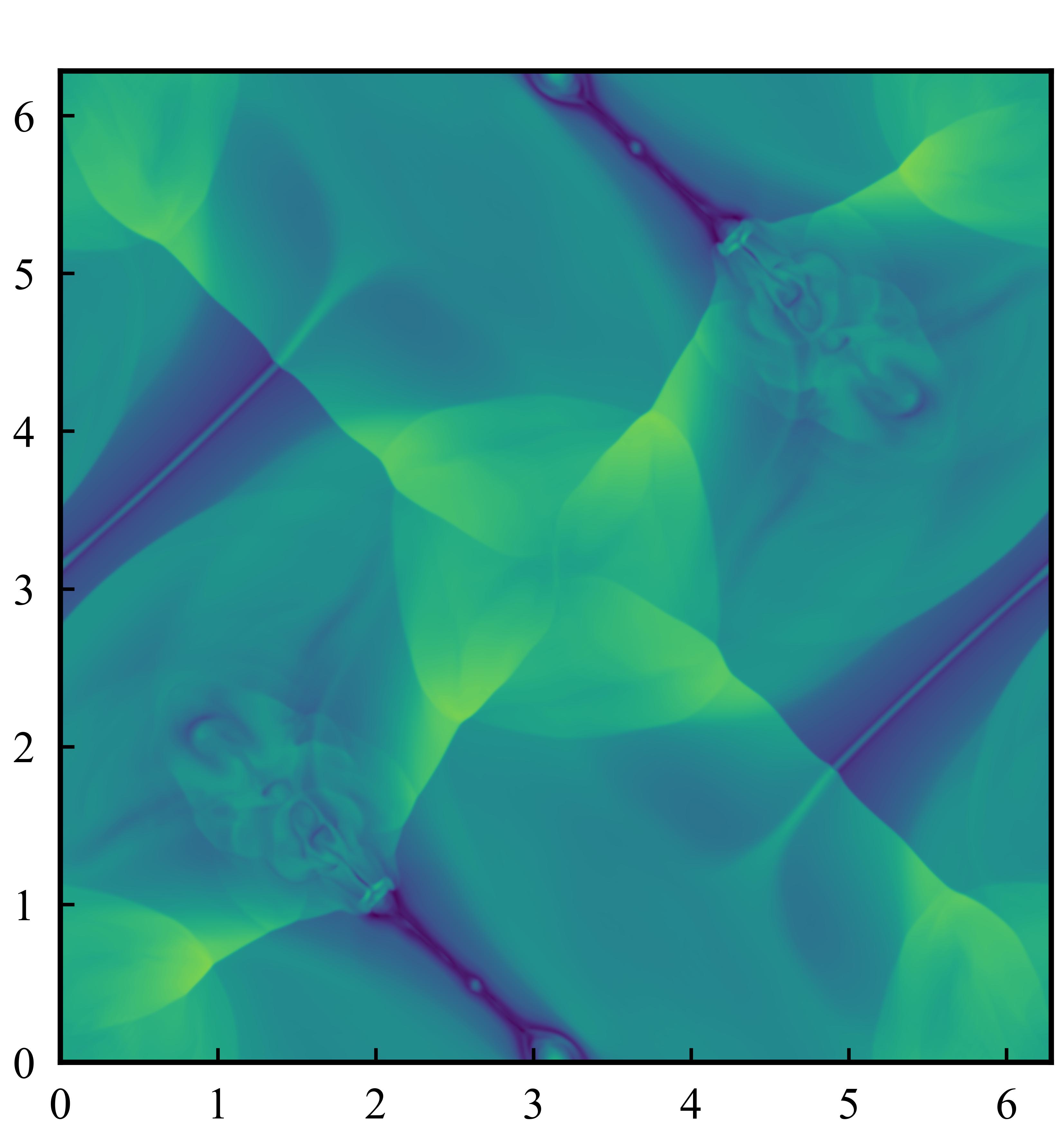}
		\end{subfigure}
		\hfill 	
		\begin{subfigure}{0.48\textwidth}
			\centering
			\includegraphics[width=\textwidth]{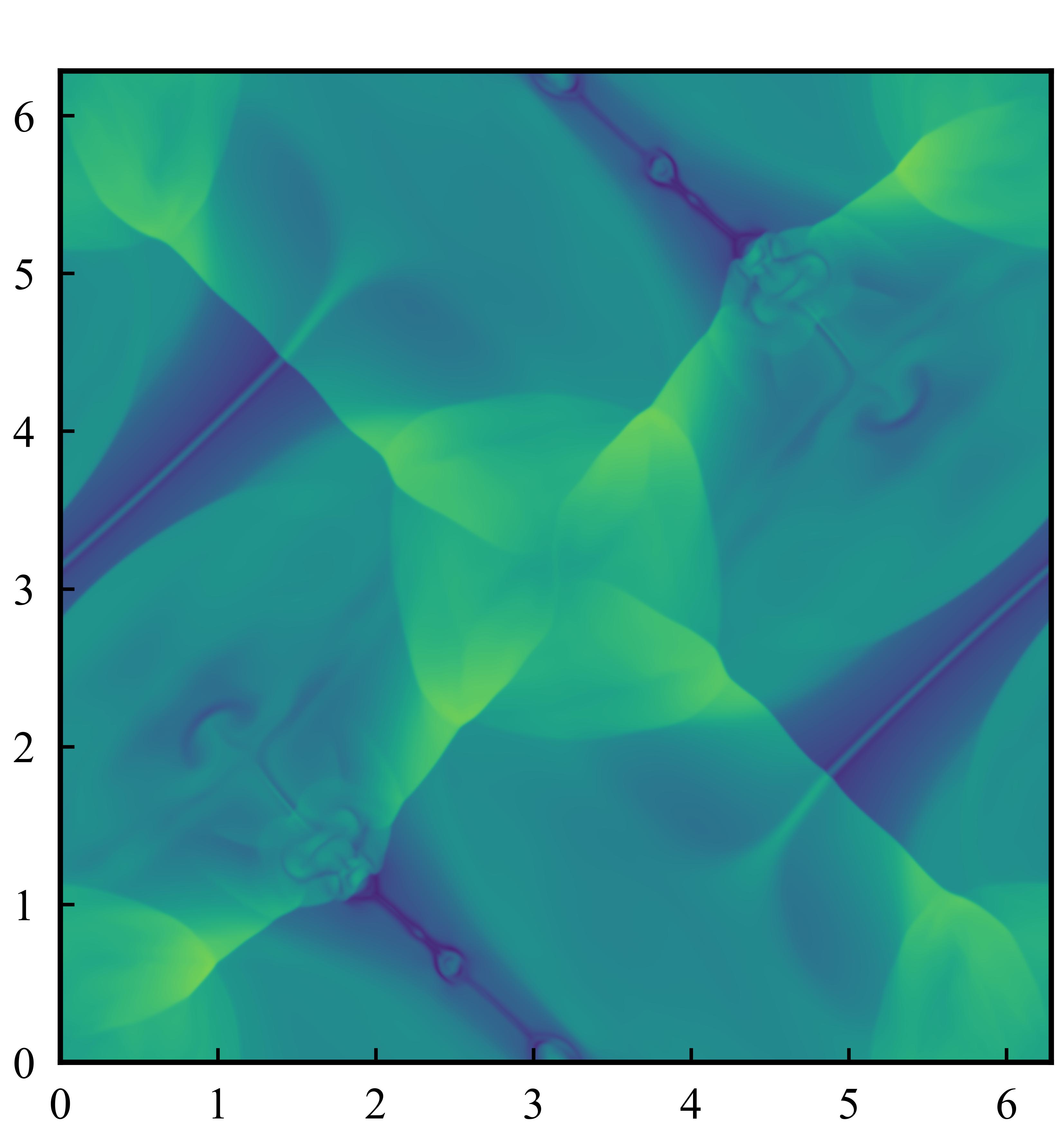}
		\end{subfigure}
		
		\caption{ (Example \ref{Ex:OT})
			Density logarithm $\log{\rho}$ at $t=2.818127$ (top) and $6.8558$ (bottom).
			Left: third order IDP DG method. Right: fifth order IDP finite volume method.
		}
		\label{fig:RMHD-OT}
	\end{figure} 
\end{example}

\begin{example}[Relativistic MHD blast problem]\label{Ex:Blast}
	Relativistic MHD blast wave problems \cite{wu2017admissible,wu2021provably} are widely used to test the robustness of numerical schemes, as nonphysical solutions can easily be produced in numerical simulations. 
    In this paper, we follow the setup in \cite{wu2021provably} and consider the  
    blast problem with a huge strength of magnetic field 2000 in the $x$-direction. 
    \Cref{fig:RMHD-Blast} displays the rest-mass density logarithm, 
    and magnitude of the magnetic field at $t=4$ obtained by the third order IDP DG method and fifth order IDP finite volume method with $400 \times 400$ uniform cells in the domain $\Omega=[-6,6]^2$.

	\begin{figure}[!htbp]
		\centering
		\begin{subfigure}{0.4\textwidth}
			\includegraphics[width=\textwidth]{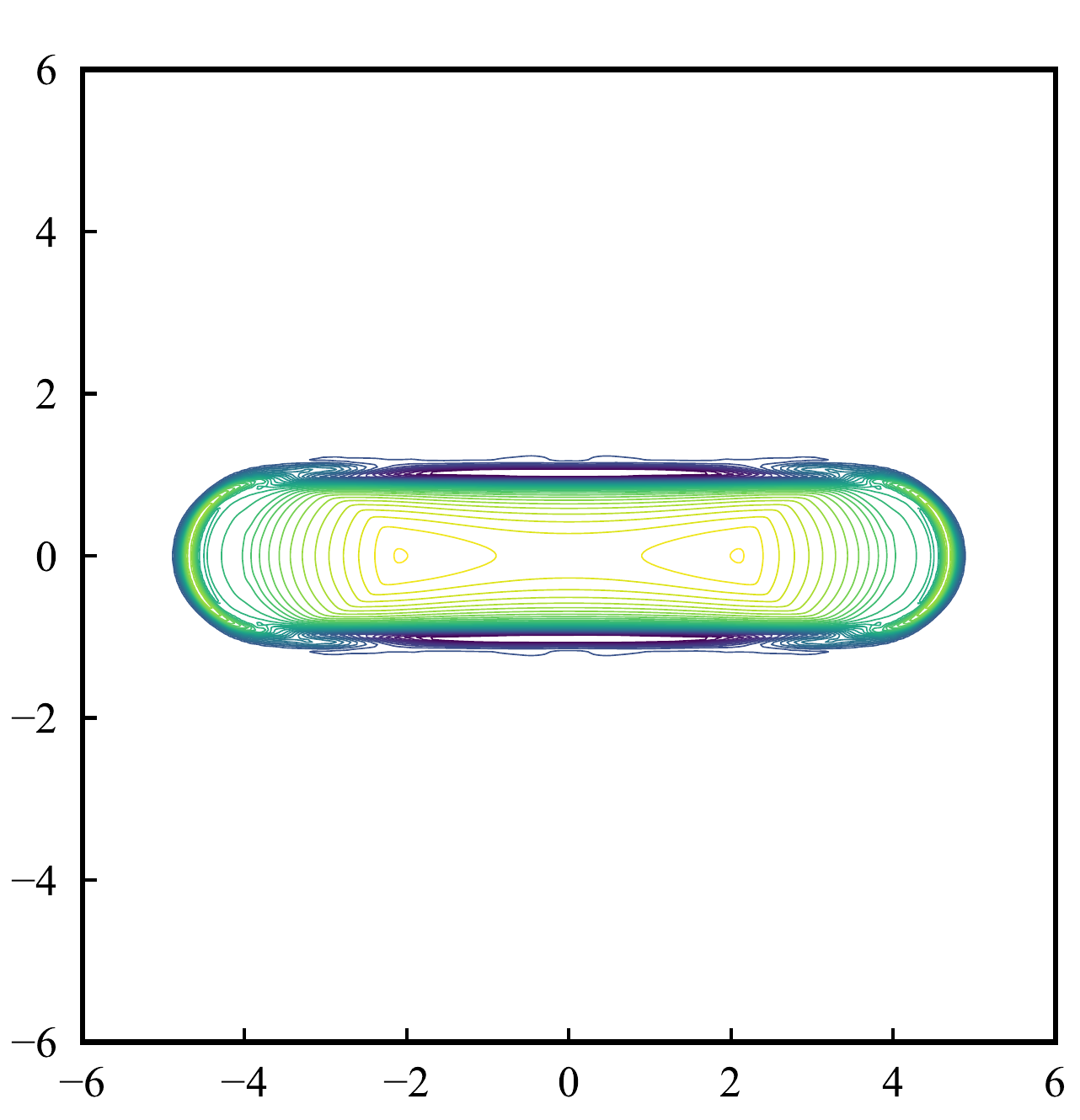}
		\end{subfigure}
		\qquad 
        \begin{subfigure}{0.4\textwidth}
			\includegraphics[width=\textwidth]{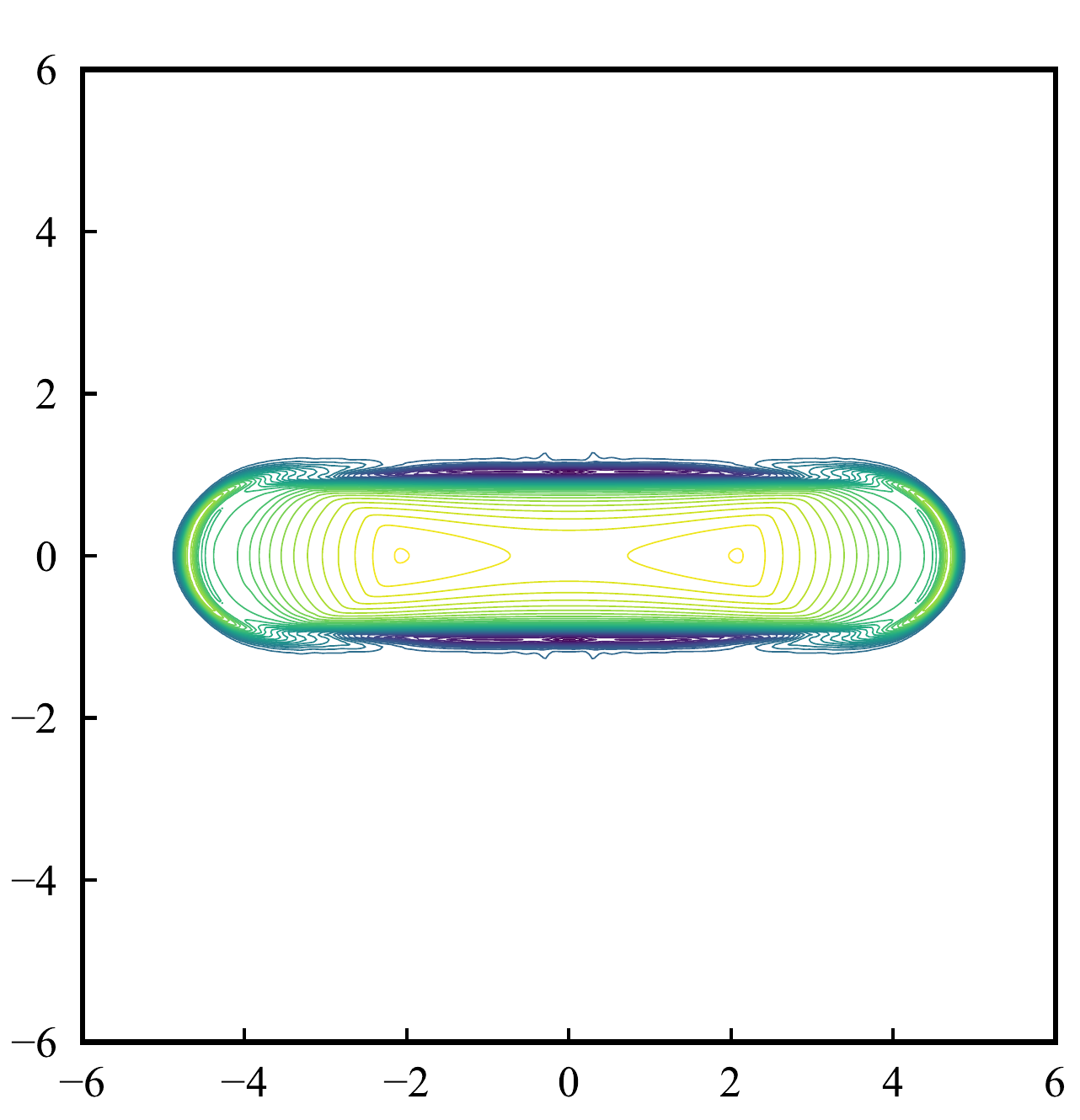}
		\end{subfigure}

		\begin{subfigure}{0.4\textwidth}
			\includegraphics[width=\textwidth]{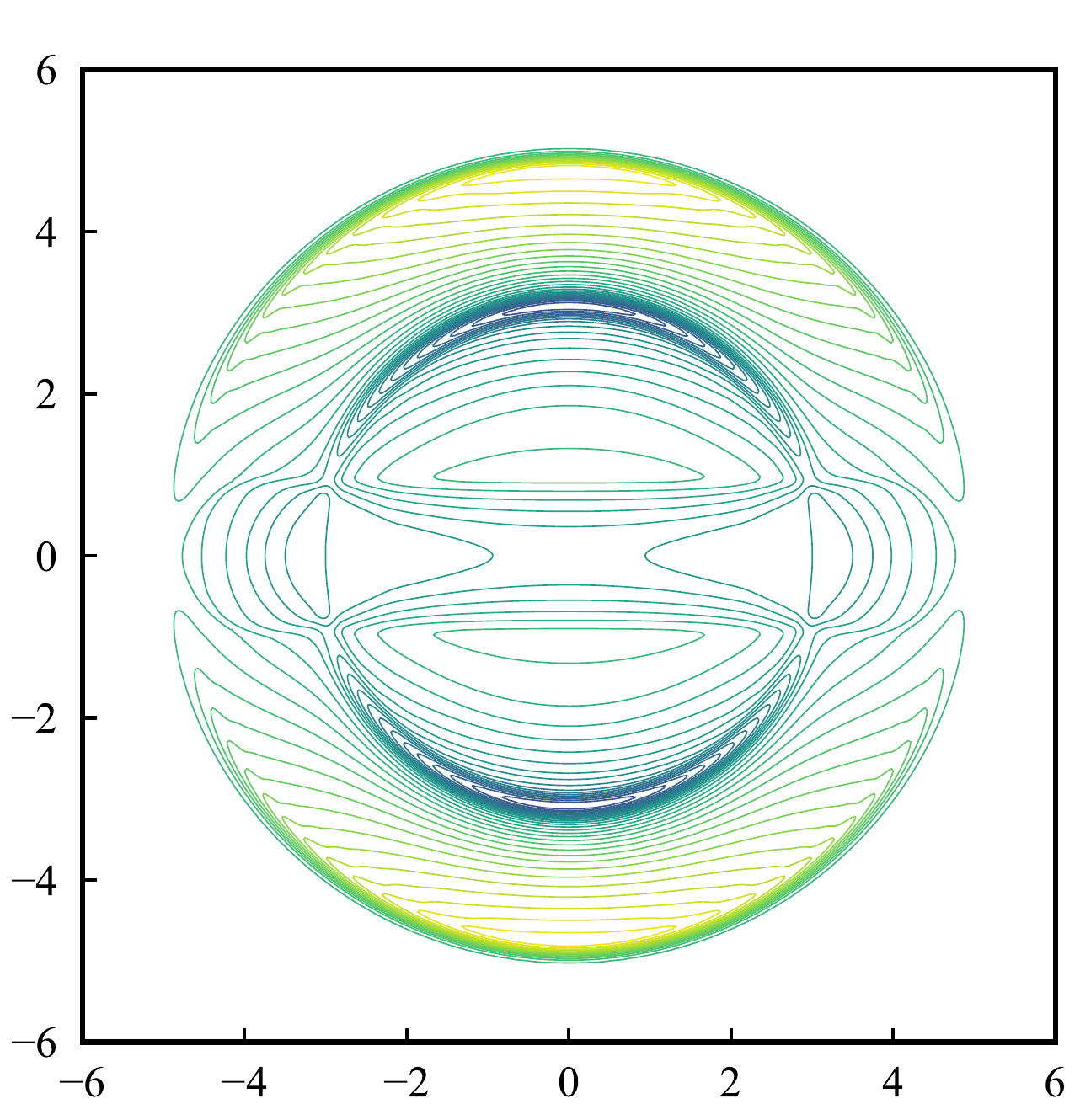}
		\end{subfigure}
		\qquad 
		\begin{subfigure}{0.4\textwidth}
			\includegraphics[width=\textwidth]{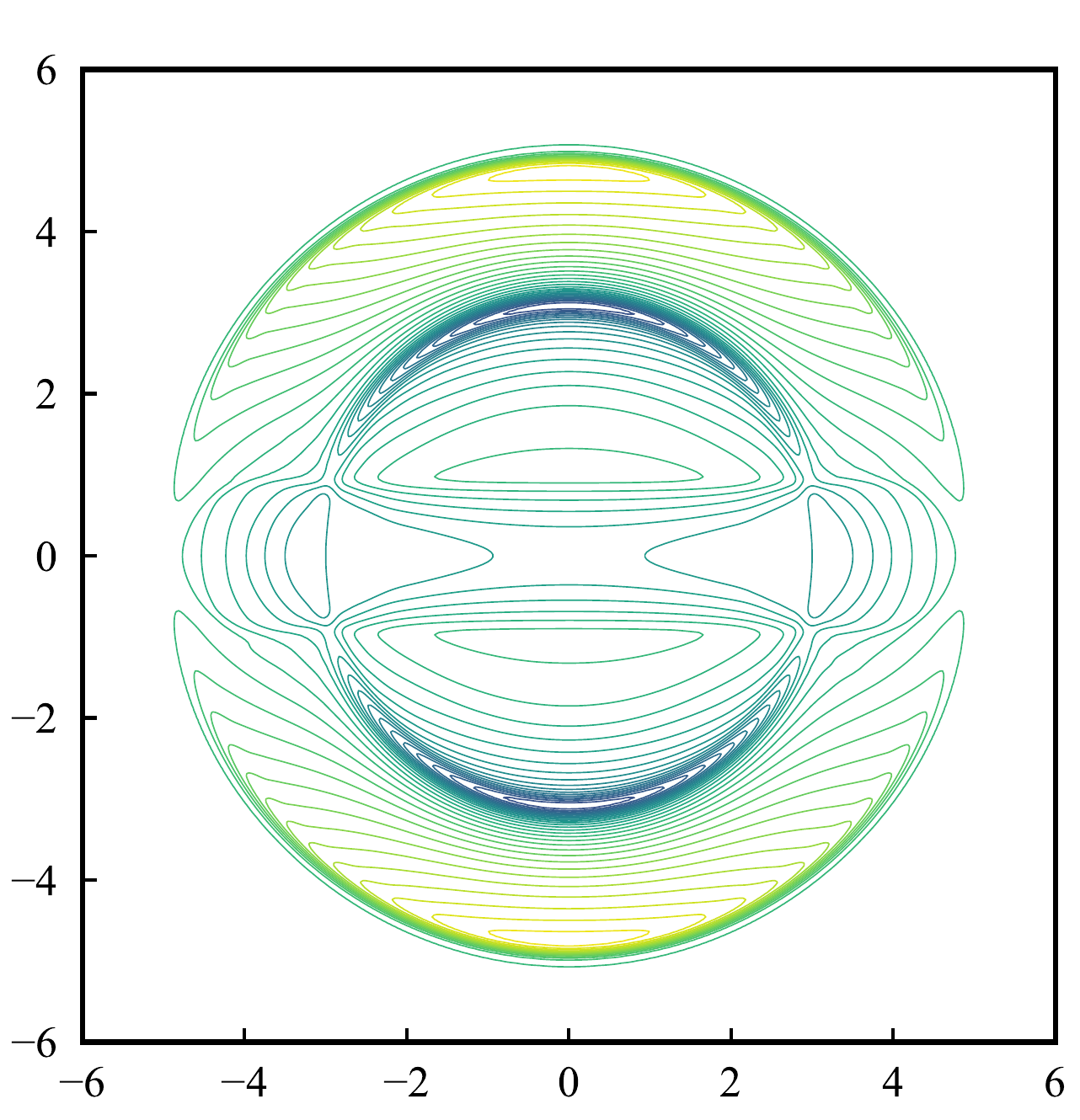}
		\end{subfigure}
		
		\caption{(Example \ref{Ex:Blast}) Contour plots of the rest-mass density logarithm (top), thermal pressure (middle), and magnitude of magnetic field (bottom).  Left: third order IDP DG method. Right: fifth order IDP finite volume method.
		}
		\label{fig:RMHD-Blast}
	\end{figure} 
\end{example}

\section{Concluding remarks}
\label{sec:remarks}
We have presented a comprehensive survey of numerical schemes which are invariant-domain-preserving (IDP)   for hyperbolic systems and related equations. 
We have unified existing techniques and theories for establishing IDP properties in first-order accurate schemes, and given
a systematical review of two popular approaches for constructing high-order IDP schemes, along with recent developments in the field. 
The Zhang--Shu approach leverages the intrinsic weak IDP property of high-order schemes, enabling the design of polynomial limiters that enforce a strong IDP property at point values for high-order finite volume and discontinuous Galerkin methods. 
The flux limiting approaches apply to a broader range of spatial discretizations, including finite difference and continuous finite element methods. 
In addition, we also discussed recent breakthroughs in constructing IDP schemes for more challenging systems, such as the magnetohydrodynamics equations, where maintaining the invariant domain is more delicate due to the complication from  discrete divergence-free constraints. 
Throughout the paper, we have provided new perspectives and insights about existing IDP approaches. 
Extensive examples, including positivity-preserving schemes for gas dynamics and numerical experiments from gas dynamics and magnetohydrodynamics, were presented to illustrate the practical performance and importance of IDP schemes.
In general, preserving the invariant domain remains a cornerstone for developing reliable and physically meaningful numerical methods for hyperbolic systems and related equations. The approaches surveyed in this paper are useful for future research, particularly in the design of high-order, robust, and efficient solvers for more complex applications.

\section*{Acknowledgement}
Xiangxiong Zhang is grateful to Professor Remi Abgrall for discussions on residual distribution schemes. 
The authors would like to thank Dr.~Shengrong Ding and Dr.~Chen Liu for the help on the visualization of some numerical results.

\renewcommand\baselinestretch{0.93}
\bibliographystyle{siamplain}
\bibliography{references}

\begin{thebibliography}{100}

\bibitem{abgrall2001toward}
{\sc R.~Abgrall}, {\em Toward the ultimate conservative scheme: following the
  quest}, J. Comput. Phys., 167 (2001), pp.~277--315.

\bibitem{remi-10}
{\sc R.~Abgrall}, {\em Essentially non-oscillatory residual distribution
  schemes for hyperbolic problems}, J. Comput. Phys., 214 (2006), pp.~773--808.

\bibitem{remi-6}
{\sc R.~Abgrall}, {\em High order schemes for hyperbolic problems using
  globally continuous approximation and avoiding mass matrices}, J. Sci.
  Comput., 73 (2017), pp.~461--494.

\bibitem{remi-2}
{\sc R.~Abgrall}, {\em A general framework to construct schemes satisfying
  additional conservation relations. {Application} to entropy conservative and
  entropy dissipative schemes}, J. Comput. Phys., 372 (2018), pp.~640--666.

\bibitem{remi-11}
{\sc R.~Abgrall}, {\em Some remarks about conservation for residual
  distribution schemes}, Comput. Methods Appl. Math., 18 (2018), pp.~327--351.

\bibitem{abgrall2023combination}
{\sc R.~Abgrall}, {\em A combination of residual distribution and the active
  flux formulations or a new class of schemes that can combine several writings
  of the same hyperbolic problem: application to the {1D E}uler equations},
  Commun. Appl. Math. Comput., 5 (2023), pp.~370--402.

\bibitem{barth}
{\sc R.~Abgrall and T.~Barth}, {\em {Residual distribution schemes for
  conservation laws via adaptive quadrature}}, SIAM J. Sci. Comput., 24 (2003),
  pp.~732--769.

\bibitem{remi-4}
{\sc R.~Abgrall and D.~De~Santis}, {\em Linear and non-linear high order
  accurate residual distribution schemes for the discretization of the steady
  compressible {Navier}-{Stokes} equations}, J. Comput. Phys., 283 (2015),
  pp.~329--359.

\bibitem{remi-5}
{\sc R.~Abgrall, D.~De~Santis, and M.~Ricchiuto}, {\em High-order preserving
  residual distribution schemes for advection-diffusion scalar problems on
  arbitrary grids}, SIAM J. Sci. Comput., 36 (2014), pp.~a955--a983.

\bibitem{Abgrall2025IDP-PAMPA-1D}
{\sc R.~Abgrall, M.~Jiao, Y.~Liu, and K.~Wu}, {\em A novel and simple
  invariant-domain-preserving framework for {PAMPA} scheme: {1D} case}, to
  appear in SIAM J. Sci. Comput., \url{https://arxiv.org/abs/2412.03423}.

\bibitem{Abgrall2026PAMPA1D}
{\sc R.~Abgrall, M.~Jiao, Y.~Liu, and K.~Wu}, {\em Bound-preserving
  {Point-Average-MomentPolynomiAl-Interpreted} ({PAMPA}) scheme:
  {one-dimensional} case}, Commun. Comput. Phys., 39 (2026), pp.~29--58.

\bibitem{remi-7}
{\sc R.~Abgrall, A.~Larat, and M.~Ricchiuto}, {\em Construction of very high
  order residual distribution schemes for steady inviscid flow problems on
  hybrid unstructured meshes}, J. Comput. Phys., 230 (2011), pp.~4103--4136.

\bibitem{Abgrall2025PAMPA-Polygon-arXiv}
{\sc R.~Abgrall, Y.~Liu, and W.~Boscheri}, {\em Bound preserving
  {P}oint-{A}verage-{M}oment {P}olynomi{A}l-interpreted ({PAMPA}) on polygonal
  meshes}, Feb. 2025, \url{https://arxiv.org/abs/2502.10069}.

\bibitem{remi-3}
{\sc R.~Abgrall, M.~Luk{\'a}{\v{c}}ova-Medvid'ov{\'a}, and P.~{\"O}ffner}, {\em
  On the convergence of residual distribution schemes for the compressible
  {Euler} equations via dissipative weak solutions}, Math. Models Methods Appl.
  Sci., 33 (2023), pp.~139--173.

\bibitem{abgrall2002lax}
{\sc R.~Abgrall, K.~Mer, and B.~Nkonga}, {\em A {Lax--Wendroff} type theorem
  for residual schemes}, in Innovative Methods for Numerical Solution of
  Partial Differential Equations, World Scientific, 2002, pp.~243--266.

\bibitem{abgrall2003construction}
{\sc R.~Abgrall and M.~Mezine}, {\em Construction of second order accurate
  monotone and stable residual distribution schemes for unsteady flow
  problems}, J. Comput. Phys., 188 (2003), pp.~16--55.

\bibitem{remi-1}
{\sc R.~Abgrall, P.~{\"O}ffner, and H.~Ranocha}, {\em Reinterpretation and
  extension of entropy correction terms for residual distribution and
  discontinuous {Galerkin} schemes: application to structure preserving
  discretization}, J. Comput. Phys., 453 (2022), p.~110955.

\bibitem{Remi-Roe}
{\sc R.~Abgrall and P.~L. Roe}, {\em High order fluctuation schemes on
  triangular meshes}, J. Sci. Comput., 19 (2003), pp.~3--36.

\bibitem{abgrall2017construction}
{\sc R.~Abgrall, Q.~Viville, H.~Beaugendre, and C.~Dobrzynski}, {\em
  Construction of ap-adaptive continuous residual distribution scheme}, J. Sci.
  Comput., 72 (2017), pp.~1232--1268.

\bibitem{alldredge2015realizability}
{\sc G.~Alldredge and F.~Schneider}, {\em {A realizability-preserving
  discontinuous Galerkin scheme for entropy-based moment closures for linear
  kinetic equations in one space dimension}}, J. Comput. Phys., 295 (2015),
  pp.~665--684.

\bibitem{barrenechea2025monotone}
{\sc G.~R. Barrenechea}, {\em Monotone Discretizations for Elliptic Second
  Order Partial Differential Equations}, Springer Nature, 2025.

\bibitem{barrenechea2017edge}
{\sc G.~R. Barrenechea, E.~Burman, and F.~Karakatsani}, {\em Edge-based
  nonlinear diffusion for finite element approximations of
  convection--diffusion equations and its relation to algebraic flux-correction
  schemes}, Numer. Math., 135 (2017), pp.~521--545.

\bibitem{barrenechea2024finite}
{\sc G.~R. Barrenechea, V.~John, and P.~Knobloch}, {\em Finite element methods
  respecting the discrete maximum principle for convection-diffusion
  equations}, SIAM Rev., 66 (2024), pp.~3--88.

\bibitem{barrenechea2017analysis}
{\sc G.~R. Barrenechea and P.~Knobloch}, {\em Analysis of a group finite
  element formulation}, Appl. Numer. Math., 118 (2017), pp.~238--248.

\bibitem{BarthJespersen1989}
{\sc T.~J. Barth and D.~C. Jespersen}, {\em The design and application of
  upwind schemes on unstructured meshes}, in 27th AIAA Aerospace Sciences
  Meeting, Reno, NV, Jan. 1989, American Institute of Aeronautics and
  Astronautics.
\newblock AIAA Paper 89-0366.

\bibitem{berger2005analysis}
{\sc M.~Berger, M.~Aftosmis, and S.~Muman}, {\em Analysis of slope limiters on
  irregular grids}, in 43rd AIAA Aerospace Sciences Meeting and Exhibit, 2005,
  p.~490.

\bibitem{berthon2006numerical}
{\sc C.~Berthon}, {\em {Numerical approximations of the 10-moment Gaussian
  closure}}, Math. Comp., 75 (2006), pp.~1809--1831.

\bibitem{berthon2008positive}
{\sc C.~Berthon and F.~Marche}, {\em {A positive preserving high order VFRoe
  scheme for shallow water equations: a class of relaxation schemes}}, SIAM J.
  Sci. Comput., 30 (2008), pp.~2587--2612.

\bibitem{bochev2020optimization}
{\sc P.~Bochev, D.~Ridzal, M.~D’Elia, M.~Perego, and K.~Peterson}, {\em
  {Optimization-based, property-preserving finite element methods for scalar
  advection equations and their connection to algebraic flux correction}},
  Comput. Methods Appl. Mech. Engrg., 367 (2020), p.~112982.

\bibitem{bochev2014optimization}
{\sc P.~Bochev, D.~Ridzal, and K.~Peterson}, {\em Optimization-based remap and
  transport: A divide and conquer strategy for feature-preserving
  discretizations}, J. Comput. Phys., 257 (2014), pp.~1113--1139.

\bibitem{bochev2011formulation}
{\sc P.~Bochev, D.~Ridzal, G.~Scovazzi, and M.~Shashkov}, {\em {Formulation,
  analysis and numerical study of an optimization-based conservative
  interpolation (remap) of scalar fields for arbitrary Lagrangian--Eulerian
  methods}}, J. Comput. Phys., 230 (2011), pp.~5199--5225.

\bibitem{Bochev2012}
{\sc P.~Bochev, D.~Ridzal, G.~Scovazzi, and M.~Shashkov}, {\em
  Constrained-optimization based data transfer}, in Flux-Corrected Transport:
  Principles, Algorithms, and Applications, D.~Kuzmin, R.~L{\"o}hner, and
  S.~Turek, eds., Springer Netherlands, Dordrecht, 2012, pp.~345--398.

\bibitem{bochev2013fast}
{\sc P.~Bochev, D.~Ridzal, and M.~Shashkov}, {\em Fast optimization-based
  conservative remap of scalar fields through aggregate mass transfer}, J.
  Comput. Phys., 246 (2013), pp.~37--57.

\bibitem{book1975flux}
{\sc D.~Book, J.~Boris, and K.~Hain}, {\em {Flux-corrected transport II:
  Generalizations of the method}}, J. Comput. Phys., 18 (1975), pp.~248--283.

\bibitem{boris1973flux}
{\sc J.~P. Boris and D.~L. Book}, {\em {Flux-corrected transport. I. SHASTA, a
  fluid transport algorithm that works}}, J. Comput. Phys., 11 (1973),
  pp.~38--69.

\bibitem{boris1976flux}
{\sc J.~P. Boris and D.~L. Book}, {\em {Flux-corrected transport. III.
  Minimal-error FCT algorithms}}, J. Comput. Phys., 20 (1976), pp.~397--431.

\bibitem{boscheri2017arbitrary}
{\sc W.~Boscheri and M.~Dumbser}, {\em {Arbitrary-Lagrangian--Eulerian
  discontinuous Galerkin schemes with a posteriori subcell finite volume
  limiting on moving unstructured meshes}}, J. Comput. Phys., 346 (2017),
  pp.~449--479.

\bibitem{boscheri2018second}
{\sc W.~Boscheri, M.~Dumbser, R.~Loub{\`e}re, and P.-H. Maire}, {\em {A
  second-order cell-centered Lagrangian ADER-MOOD finite volume scheme on
  multidimensional unstructured meshes for hydrodynamics}}, J. Comput. Phys.,
  358 (2018), pp.~103--129.

\bibitem{bradley2019communication}
{\sc A.~M. Bradley, P.~A. Bosler, O.~Guba, M.~A. Taylor, and G.~A. Barnett},
  {\em Communication-efficient property preservation in tracer transport}, SIAM
  J. Sci. Comput., 41 (2019), pp.~C161--C193.

\bibitem{breuss2004correct}
{\sc M.~Breu{\ss}}, {\em {The correct use of the Lax--Friedrichs method}},
  ESAIM Math. Model. Numer. Anal., 38 (2004), pp.~519--540.

\bibitem{buet2006asymptotic}
{\sc C.~Buet and B.~Despres}, {\em Asymptotic preserving and positive schemes
  for radiation hydrodynamics}, J. Comput. Phys., 215 (2006), pp.~717--740.

\bibitem{buet2012asymptotic}
{\sc C.~Buet, B.~Despr{\'e}s, and E.~Franck}, {\em An asymptotic preserving
  scheme with the maximum principle for the {$M_1$} model on distorded meshes},
  C. R. Math. Acad. Sci. Paris, 350 (2012), pp.~633--638.

\bibitem{CaiQiuWu2025NRMHD}
{\sc C.~Cai, J.~Qiu, and K.~Wu}, {\em Provably convergent {Newton--Raphson}
  method: {T}heoretically robust recovery of primitive variables in
  relativistic {MHD}}, SIAM J. Numer. Anal., 63 (2025), pp.~1128--1159.

\bibitem{cao2025robust}
{\sc H.~Cao, M.~Peng, and K.~Wu}, {\em Robust discontinuous {G}alerkin methods
  maintaining physical constraints for general relativistic hydrodynamics}, J.
  Comput. Phys.,  (2025), p.~113770.

\bibitem{castro2006numerical}
{\sc M.~J. Castro, J.~M. Gonz{\'a}lez-Vida, and C.~Par{\'e}s}, {\em {Numerical
  treatment of wet/dry fronts in shallow flows with a modified Roe scheme}},
  Math. Models Methods Appl. Sci., 16 (2006), pp.~897--931.

\bibitem{chen2017entropy}
{\sc T.~Chen and C.-W. Shu}, {\em {Entropy stable high order discontinuous
  Galerkin methods with suitable quadrature rules for hyperbolic conservation
  laws}}, J. Comput. Phys., 345 (2017), pp.~427--461.

\bibitem{chen2022physical}
{\sc Y.~Chen and K.~Wu}, {\em A physical-constraint-preserving finite volume
  {WENO} method for special relativistic hydrodynamics on unstructured meshes},
  J. Comput. Phys., 466 (2022), p.~111398.

\bibitem{chen2016third}
{\sc Z.~Chen, H.~Huang, and J.~Yan}, {\em {Third order
  maximum-principle-satisfying direct discontinuous Galerkin methods for time
  dependent convection diffusion equations on unstructured triangular meshes}},
  J. Comput. Phys., 308 (2016), pp.~198--217.

\bibitem{cheng2014positivity}
{\sc J.~Cheng and C.-W. Shu}, {\em {Positivity-preserving Lagrangian scheme for
  multi-material compressible flow}}, J. Comput. Phys., 257 (2014),
  pp.~143--168.

\bibitem{cheng2013study}
{\sc Y.~Cheng, I.~M. Gamba, and P.~J. Morrison}, {\em {Study of conservation
  and recurrence of Runge--Kutta discontinuous Galerkin schemes for
  Vlasov--Poisson systems}}, J. Sci. Comput., 56 (2013), pp.~319--349.

\bibitem{cheng2013positivity-MHD}
{\sc Y.~Cheng, F.~Li, J.~Qiu, and L.~Xu}, {\em {Positivity-preserving DG and
  central DG methods for ideal MHD equations}}, J. Comput. Phys., 238 (2013),
  pp.~255--280.

\bibitem{christlieb2016high}
{\sc A.~J. Christlieb, X.~Feng, D.~C. Seal, and Q.~Tang}, {\em A high-order
  positivity-preserving single-stage single-step method for the ideal
  magnetohydrodynamic equations}, J. Comput. Phys., 316 (2016), pp.~218--242.

\bibitem{christlieb2015high}
{\sc A.~J. Christlieb, Y.~Liu, Q.~Tang, and Z.~Xu}, {\em {High order
  parametrized maximum-principle-preserving and positivity-preserving WENO
  schemes on unstructured meshes}}, J. Comput. Phys., 281 (2015), pp.~334--351.

\bibitem{christlieb2015positivity}
{\sc A.~J. Christlieb, Y.~Liu, Q.~Tang, and Z.~Xu}, {\em Positivity-preserving
  finite difference weighted {ENO} schemes with constrained transport for ideal
  magnetohydrodynamic equations}, SIAM J. Sci. Comput., 37 (2015),
  pp.~A1825--A1845.

\bibitem{chueh1977positively}
{\sc K.~N. Chueh, C.~C. Conley, and J.~A. Smoller}, {\em Positively invariant
  regions for systems of nonlinear diffusion equations}, Indiana Univ. Math.
  J., 26 (1977), pp.~373--392.

\bibitem{cockburn1990runge}
{\sc B.~Cockburn, S.~Hou, and C.-W. Shu}, {\em {The Runge-Kutta local
  projection discontinuous {G}alerkin finite element method for conservation
  laws. {IV}. {T}he multidimensional case}}, Math. Comp., 54 (1990),
  pp.~545--581.

\bibitem{cotter2016embedded}
{\sc C.~J. Cotter and D.~Kuzmin}, {\em {Embedded discontinuous Galerkin
  transport schemes with localised limiters}}, Journal of Computational
  Physics, 311 (2016), pp.~363--373.

\bibitem{crandall1980monotone}
{\sc M.~G. Crandall and A.~Majda}, {\em Monotone difference approximations for
  scalar conservation laws}, Math. Comp., 34 (1980), pp.~1--21.

\bibitem{cross2023monotonicity}
{\sc L.~J. Cross and X.~Zhang}, {\em {On the monotonicity of Q3 spectral
  element method for Laplacian}}, Ann. Appl. Math., 40 (2024), pp.~161--190.

\bibitem{csik2002conservative}
{\sc A.~Cs{\i}k, M.~Ricchiuto, and H.~Deconinck}, {\em A conservative
  formulation of the multidimensional upwind residual distribution schemes for
  general nonlinear conservation laws}, J. Comput. Phys., 179 (2002),
  pp.~286--312.

\bibitem{cui2023classic}
{\sc S.~Cui, S.~Ding, and K.~Wu}, {\em Is the classic convex decomposition
  optimal for bound-preserving schemes in multiple dimensions?}, J. Comput.
  Phys., 476 (2023), p.~111882.

\bibitem{cui2024optimal}
{\sc S.~Cui, S.~Ding, and K.~Wu}, {\em On optimal cell average decomposition
  for high-order bound-preserving schemes of hyperbolic conservation laws},
  SIAM J. Numer. Anal., 62 (2024), pp.~775--810.

\bibitem{Cui2025LocalMEP-RelEuler}
{\sc S.~Cui, K.~Wu, and L.~Xu}, {\em On local minimum entropy principle of
  high-order schemes for relativistic {E}uler equations}, Math. Comp.,  (2025),
  \url{https://doi.org/10.1090/mcom/4139}.

\bibitem{dao2024structure}
{\sc T.~A. Dao, M.~Nazarov, and I.~Tomas}, {\em A structure preserving
  numerical method for the ideal compressible {MHD} system}, J. Comput. Phys.,
  508 (2024), p.~113009.

\bibitem{deconinck1993multidimensional}
{\sc H.~Deconinck, P.~L. Roe, and R.~Struijs}, {\em {A multidimensional
  generalization of Roe's flux difference splitter for the Euler equations}},
  Comput. \& Fluids, 22 (1993), pp.~215--222.

\bibitem{deconinck1993compact}
{\sc H.~Deconinck, R.~Struijs, G.~Bourgeois, and P.~L. Roe}, {\em Compact
  advection schemes on unstructured meshes}, in Computational Fluid Dynamics,
  VKI Lecture Series 1993-04, 1993.

\bibitem{derigs2016novel-mhd}
{\sc D.~Derigs, A.~R. Winters, G.~J. Gassner, and S.~Walch}, {\em {A novel
  high-order, entropy stable, 3D AMR MHD solver with guaranteed positive
  pressure}}, J. Comput. Phys., 317 (2016), pp.~223--256.

\bibitem{ding2025robust}
{\sc S.~Ding, S.~Cui, and K.~Wu}, {\em Robust {DG} schemes on unstructured
  triangular meshes: {O}scillation elimination and bound preservation via
  optimal convex decomposition}, J. Comput. Phys.,  (2025), p.~113769.

\bibitem{ding2024new}
{\sc S.~Ding and K.~Wu}, {\em {A new discretely divergence-free
  positivity-preserving high-order finite volume method for ideal MHD
  equations}}, SIAM J. Sci. Comput., 46 (2024), pp.~A50--A79.

\bibitem{ding2024gql}
{\sc S.~Ding and K.~Wu}, {\em {GQL-based bound-preserving and locally
  divergence-free central discontinuous Galerkin schemes for relativistic
  magnetohydrodynamics}}, J. Comput. Phys., 514 (2024), p.~113208.

\bibitem{ding2025divergence}
{\sc S.~Ding, K.~Wu, and C.~Yuan}, {\em Divergence-free finite volume {WENO}
  scheme for relativistic magnetohydrodynamics preserving positivity and
  subluminal velocity}, Monthly Not. Roy. Astr. Soc.,  (2025), p.~1167–1190.

\bibitem{du2024high-twolayer}
{\sc C.~Du and M.~Li}, {\em A high-order domain preserving dg method for the
  two-layer shallow water equations}, Comput. \& Fluids, 269 (2024), p.~106140.

\bibitem{DuJuLiQiao2021}
{\sc Q.~Du, L.~Ju, X.~Li, and Z.~Qiao}, {\em Maximum bound principles for a
  class of semilinear parabolic equations and exponential time-differencing
  schemes}, SIAM Rev., 63 (2021), pp.~317--359.

\bibitem{dumbser2016simple}
{\sc M.~Dumbser and R.~Loub{\`e}re}, {\em {A simple robust and accurate a
  posteriori sub-cell finite volume limiter for the discontinuous Galerkin
  method on unstructured meshes}}, J. Comput. Phys., 319 (2016), pp.~163--199.

\bibitem{dumbser2014posteriori}
{\sc M.~Dumbser, O.~Zanotti, R.~Loub{\`e}re, and S.~Diot}, {\em A posteriori
  subcell limiting of the discontinuous {G}alerkin finite element method for
  hyperbolic conservation laws}, J. Comput. Phys., 278 (2014), pp.~47--75.

\bibitem{einfeldt1991godunov}
{\sc B.~Einfeldt, C.-D. Munz, P.~L. Roe, and B.~Sj{\"o}green}, {\em {On
  Godunov-type methods near low densities}}, J. Comput. Phys., 92 (1991),
  pp.~273--295.

\bibitem{endeve2015bound}
{\sc E.~Endeve, C.~D. Hauck, Y.~Xing, and A.~Mezzacappa}, {\em
  {Bound-preserving discontinuous Galerkin methods for conservative phase space
  advection in curvilinear coordinates}}, J. Comput. Phys., 287 (2015),
  pp.~151--183.

\bibitem{ern2022invariant}
{\sc A.~Ern and J.-L. Guermond}, {\em {Invariant-domain-preserving high-order
  time stepping: I. Explicit Runge--Kutta schemes}}, SIAM J. Sci. Comput., 44
  (2022), pp.~A3366--A3392.

\bibitem{ern2023invariant}
{\sc A.~Ern and J.-L. Guermond}, {\em {Invariant-domain preserving high-order
  time stepping: II. IMEX schemes}}, SIAM J. Sci. Comput., 45 (2023),
  pp.~A2511--A2538.

\bibitem{EstepMAMS2000}
{\sc D.~J. Estep, M.~G. Larson, and R.~D. Williams}, {\em Estimating the error
  of numerical solutions of systems of reaction-diffusion equations}, Mem.
  Amer. Math. Soc., {146} ({2000}).

\bibitem{estivalezes1996high}
{\sc J.~Estivalezes and P.~Villedieu}, {\em {High-order positivity-preserving
  kinetic schemes for the compressible Euler equations}}, SIAM J. Numer. Anal.,
  33 (1996), pp.~2050--2067.

\bibitem{fan2021positivity}
{\sc C.~Fan, X.~Zhang, and J.~Qiu}, {\em {Positivity-preserving high order
  finite volume hybrid Hermite WENO schemes for compressible Navier-Stokes
  equations}}, J. Comput. Phys., 445 (2021), p.~110596.

\bibitem{fan2022positivity}
{\sc C.~Fan, X.~Zhang, and J.~Qiu}, {\em {Positivity-preserving high order
  finite difference WENO schemes for compressible Navier-Stokes equations}}, J.
  Comput. Phys., 467 (2022), p.~111446.

\bibitem{fletcher1983group}
{\sc C.~A. Fletcher}, {\em The group finite element formulation}, Comput.
  Methods Appl. Mech. Engrg., 37 (1983), pp.~225--244.

\bibitem{frid2001maps}
{\sc H.~Frid}, {\em Maps of convex sets and invariant regions for
  finite-difference systems of conservation laws}, Arch. Ration. Mech. Anal.,
  160 (2001), pp.~245--269.

\bibitem{Fu2025Bound}
{\sc Q.~Fu, Y.~Gu, A.~Kurganov, and B.-S. Wang}, {\em Bound- and
  positivity-preserving path-conservative central-upwind {AWENO} scheme for the
  five-equation model of compressible two-component flows}, J. Sci. Comput.,
  104 (2025).
\newblock Article Number 94.

\bibitem{Godunov1972}
{\sc S.~K. Godunov}, {\em Symmetric form of the equations of
  magnetohydrodynamics}, Numerical Methods for Mechanics of Continuum Medium, 1
  (1972), pp.~26--34.

\bibitem{godunov1959finite}
{\sc S.~K. Godunov and I.~Bohachevsky}, {\em Finite difference method for
  numerical computation of discontinuous solutions of the equations of fluid
  dynamics}, Matemati{\v{c}}eskij {S}bornik, 47 (1959), pp.~271--306.

\bibitem{goodman1985accuracy}
{\sc J.~B. Goodman and R.~J. LeVeque}, {\em {On the accuracy of stable schemes
  for 2D scalar conservation laws}}, Math. Comp.,  (1985), pp.~15--21.

\bibitem{gottlieb2009high}
{\sc S.~Gottlieb, D.~I. Ketcheson, and C.-W. Shu}, {\em High order strong
  stability preserving time discretizations}, J. Sci. Comput., 38 (2009),
  pp.~251--289.

\bibitem{gottlieb2001strong}
{\sc S.~Gottlieb, C.-W. Shu, and E.~Tadmor}, {\em Strong stability-preserving
  high-order time discretization methods}, SIAM Rev., 43 (2001), pp.~89--112.

\bibitem{grapsas2016unconditionally}
{\sc D.~Grapsas, R.~Herbin, W.~Kheriji, and J.-C. Latch{\'e}}, {\em {An
  unconditionally stable staggered pressure correction scheme for the
  compressible Navier-Stokes equations}}, SMAI J. Comput. Math., 2 (2016),
  pp.~51--97.

\bibitem{Gu2021A}
{\sc Y.~Gu, Z.~Gao, G.~Hu, P.~Li, and L.~Wang}, {\em A robust high order
  alternative {WENO} scheme for the five-equation model}, J. Sci. Comput., 88
  (2021).
\newblock Article Number 12.

\bibitem{guba2014optimization}
{\sc O.~Guba, M.~Taylor, and A.~St-Cyr}, {\em Optimization-based limiters for
  the spectral element method}, J. Comput. Phys., 267 (2014), pp.~176--195.

\bibitem{guermond2021second-NS}
{\sc J.-L. Guermond, M.~Maier, B.~Popov, and I.~Tomas}, {\em {Second-order
  invariant domain preserving approximation of the compressible Navier--Stokes
  equations}}, Comput. Methods Appl. Mech. Engrg., 375 (2021), p.~113608.

\bibitem{guermond2014maximum}
{\sc J.-L. Guermond and M.~Nazarov}, {\em {A maximum-principle preserving C0
  finite element method for scalar conservation equations}}, Comput. Methods
  Appl. Mech. Engrg., 272 (2014), pp.~198--213.

\bibitem{guermond2024finite}
{\sc J.-L. Guermond, M.~Nazarov, and B.~Popov}, {\em {Finite element-based
  invariant-domain preserving approximation of hyperbolic systems: Beyond
  second-order accuracy in space}}, Comput. Methods Appl. Mech. Engrg., 418
  (2024), p.~116470.

\bibitem{guermond2018second-euler}
{\sc J.-L. Guermond, M.~Nazarov, B.~Popov, and I.~Tomas}, {\em {Second-order
  invariant domain preserving approximation of the Euler equations using convex
  limiting}}, SIAM J. Sci. Comput., 40 (2018), pp.~A3211--A3239.

\bibitem{guermond2014second}
{\sc J.-L. Guermond, M.~Nazarov, B.~Popov, and Y.~Yang}, {\em {A second-order
  maximum principle preserving Lagrange finite element technique for nonlinear
  scalar conservation equations}}, SIAM J. Numer. Anal., 52 (2014),
  pp.~2163--2182.

\bibitem{guermond2011entropy}
{\sc J.-L. Guermond, R.~Pasquetti, and B.~Popov}, {\em Entropy viscosity method
  for nonlinear conservation laws}, J. Comput. Phys., 230 (2011),
  pp.~4248--4267.

\bibitem{guermond2016fast}
{\sc J.-L. Guermond and B.~Popov}, {\em {Fast estimation from above of the
  maximum wave speed in the Riemann problem for the Euler equations}}, J.
  Comput. Phys., 321 (2016), pp.~908--926.

\bibitem{guermond2016invariant}
{\sc J.-L. Guermond and B.~Popov}, {\em Invariant domains and first-order
  continuous finite element approximation for hyperbolic systems}, SIAM J.
  Numer. Anal., 54 (2016), pp.~2466--2489.

\bibitem{guermond2017invariant}
{\sc J.-L. Guermond and B.~Popov}, {\em Invariant domains and second-order
  continuous finite element approximation for scalar conservation equations},
  SIAM J. Numer. Anal., 55 (2017), pp.~3120--3146.

\bibitem{guermond2020second}
{\sc J.-L. Guermond, B.~Popov, and L.~Saavedra}, {\em {Second-order invariant
  domain preserving ALE approximation of hyperbolic systems}}, J. Comput.
  Phys., 401 (2020), p.~108927.

\bibitem{guermond2023second-ALE}
{\sc J.-L. Guermond, B.~Popov, and L.~Saavedra}, {\em {Second-order invariant
  domain preserving ALE approximation of Euler equations}}, Commun. Appl. Math.
  Comput., 5 (2023), pp.~923--945.

\bibitem{guermond2017-ALE}
{\sc J.-L. Guermond, B.~Popov, L.~Saavedra, and Y.~Yang}, {\em {Invariant
  domains preserving arbitrary Lagrangian Eulerian approximation of hyperbolic
  systems with continuous finite elements}}, SIAM J. Sci. Comput., 39 (2017),
  pp.~A385--A414.

\bibitem{guermond2019arbitrary}
{\sc J.-L. Guermond, B.~Popov, L.~Saavedra, and Y.~Yang}, {\em Arbitrary
  {Lagrangian-Eulerian} finite element method preserving convex invariants of
  hyperbolic systems}, in Contributions to Partial Differential Equations and
  Applications, B.~N. Chetverushkin, W.~Fitzgibbon, Y.~Kuznetsov,
  P.~Neittaanm{\"a}ki, J.~Periaux, and O.~Pironneau, eds., Springer
  International Publishing, Cham, 2019, pp.~251--272.

\bibitem{guermond2019invariant}
{\sc J.-L. Guermond, B.~Popov, and I.~Tomas}, {\em Invariant domain preserving
  discretization-independent schemes and convex limiting for hyperbolic
  systems}, Comput. Methods Appl. Mech. Engrg., 347 (2019), pp.~143--175.

\bibitem{haidar2022posteriori}
{\sc A.~Haidar, F.~Marche, and F.~Vilar}, {\em {A posteriori finite-volume
  local subcell correction of high-order discontinuous Galerkin schemes for the
  nonlinear shallow-water equations}}, J. Comput. Phys., 452 (2022), p.~110902.

\bibitem{hajduk2021monolithic}
{\sc H.~Hajduk}, {\em {Monolithic convex limiting in discontinuous Galerkin
  discretizations of hyperbolic conservation laws}}, Comput. Math. Appl., 87
  (2021), pp.~120--138.

\bibitem{hajduk2020matrix}
{\sc H.~Hajduk, D.~Kuzmin, T.~Kolev, V.~Tomov, I.~Tomas, and J.~N. Shadid},
  {\em {Matrix-free subcell residual distribution for Bernstein finite
  elements: Monolithic limiting}}, Comput. \& Fluids, 200 (2020), p.~104451.

\bibitem{harten1976finite}
{\sc A.~Harten, J.~M. Hyman, P.~D. Lax, and B.~Keyfitz}, {\em On
  finite-difference approximations and entropy conditions for shocks}, Comm.
  Pure Appl. Math., 29 (1976), pp.~297--322.

\bibitem{hoff1979finite}
{\sc D.~Hoff}, {\em A finite difference scheme for a system of two conservation
  laws with artificial viscosity}, Math. Comp., 33 (1979), pp.~1171--1193.

\bibitem{hoff1985invariant}
{\sc D.~Hoff}, {\em Invariant regions for systems of conservation laws}, Trans.
  Amer. Math. Soc., 289 (1985), pp.~591--610.

\bibitem{hu-adam-2013positivity}
{\sc X.~Y. Hu, N.~A. Adams, and C.-W. Shu}, {\em Positivity-preserving method
  for high-order conservative schemes solving compressible {E}uler equations},
  J. Comput. Phys., 242 (2013), pp.~169--180.

\bibitem{jameson1995positive}
{\sc A.~Jameson}, {\em Positive schemes and shock modelling for compressible
  flows}, Internat. J. Numer. Methods Fluids, 20 (1995), pp.~743--776.

\bibitem{janhunen2000positive}
{\sc P.~Janhunen}, {\em A positive conservative method for magnetohydrodynamics
  based on {HLL} and {Roe} methods}, J. Comput. Phys., 160 (2000),
  pp.~649--661.

\bibitem{jiang1998nonoscillatory}
{\sc G.-S. Jiang and E.~Tadmor}, {\em Nonoscillatory central schemes for
  multidimensional hyperbolic conservation laws}, SIAM J. Sci. Comput., 19
  (1998), pp.~1892--1917.

\bibitem{jiang2018invariant}
{\sc Y.~Jiang and H.~Liu}, {\em Invariant-region-preserving {DG} methods for
  multi-dimensional hyperbolic conservation law systems, with an application to
  compressible {E}uler equations}, J. Comput. Phys., 373 (2018), pp.~385--409.

\bibitem{jiang2013parametrized}
{\sc Y.~Jiang and Z.~Xu}, {\em {Parametrized maximum principle preserving
  limiter for finite difference WENO schemes solving convection-dominated
  diffusion equations}}, SIAM J. Sci. Comput., 35 (2013), pp.~A2524--A2553.

\bibitem{keith2024proximal}
{\sc B.~Keith and T.~M. Surowiec}, {\em {Proximal Galerkin: A
  structure-preserving finite element method for pointwise bound constraints}},
  Found. Comput. Math.,  (2024), pp.~1--97.

\bibitem{kenamond2021positivity}
{\sc M.~Kenamond, D.~Kuzmin, and M.~Shashkov}, {\em {A positivity-preserving
  and conservative intersection-distribution-based remapping algorithm for
  staggered ALE hydrodynamics on arbitrary meshes}}, J. Comput. Phys., 435
  (2021), p.~110254.

\bibitem{ketcheson2020riemann}
{\sc D.~I. Ketcheson, R.~J. LeVeque, and M.~J. Del~Razo}, {\em {Riemann
  problems and Jupyter solutions}}, vol.~16, SIAM, 2020.

\bibitem{khobalatte1992maximum}
{\sc B.~Khobalatte and B.~Perthame}, {\em Maximum principle on the entropy and
  second-order kinetic schemes}, Math. Comp., 62 (1994), pp.~119--131.

\bibitem{KIDDER201784}
{\sc L.~E. Kidder, S.~E. Field, F.~Foucart, E.~Schnetter, S.~A. Teukolsky,
  A.~Bohn, N.~Deppe, P.~Diener, F.~Hébert, J.~Lippuner, J.~Miller, C.~D. Ott,
  M.~A. Scheel, and T.~Vincent}, {\em {SpECTRE: A task-based discontinuous
  Galerkin code for relativistic astrophysics}}, J. Comput. Phys., 335 (2017),
  pp.~84--114.

\bibitem{kirby2024high}
{\sc R.~C. Kirby and D.~Shapero}, {\em High-order bounds-satisfying
  approximation of partial differential equations via finite element
  variational inequalities}, Numer. Math., 156 (2024), pp.~927--947.

\bibitem{klingenberg2017numerical}
{\sc C.~Klingenberg}, {\em Numerical methods for astrophysics}, in Handbook of
  Numerical Analysis, vol.~18, Elsevier, 2017, pp.~465--477.

\bibitem{kurganov2007second}
{\sc A.~Kurganov and G.~Petrova}, {\em {A second-order well-balanced positivity
  preserving central-upwind scheme for the Saint-Venant system}}, Commun. Math.
  Sci., 5 (2007).

\bibitem{kurganov2000new}
{\sc A.~Kurganov and E.~Tadmor}, {\em New high-resolution central schemes for
  nonlinear conservation laws and convection--diffusion equations}, J. Comput.
  Phys., 160 (2000), pp.~241--282.

\bibitem{kuzmin2001positive}
{\sc D.~Kuzmin}, {\em Positive finite element schemes based on the
  flux-corrected transport procedure}, Computational Fluid and Solid Mechanics,
  Elsevier,  (2001), pp.~887--888.

\bibitem{kuzmin2020monolithic}
{\sc D.~Kuzmin}, {\em Monolithic convex limiting for continuous finite element
  discretizations of hyperbolic conservation laws}, Comput. Methods Appl. Mech.
  Engrg., 361 (2020), p.~112804.

\bibitem{kuzmin2020subcell}
{\sc D.~Kuzmin and M.~Q. de~Luna}, {\em Subcell flux limiting for high-order
  {B}ernstein finite element discretizations of scalar hyperbolic conservation
  laws}, Journal of Computational Physics, 411 (2020), p.~109411.

\bibitem{kuzmin2024property}
{\sc D.~Kuzmin and H.~Hajduk}, {\em Property-preserving numerical schemes for
  conservation laws}, World Scientific, 2024.

\bibitem{KUZMIN2020109230}
{\sc D.~Kuzmin and N.~Klyushnev}, {\em {Limiting and divergence cleaning for
  continuous finite element discretizations of the MHD equations}}, J. Comput.
  Phys., 407 (2020), p.~109230.

\bibitem{kuzmin2012flux}
{\sc D.~Kuzmin, R.~L{\"o}hner, and S.~Turek}, {\em Flux-corrected transport:
  principles, algorithms, and applications}, Springer Science \& Business
  Media, 2012.

\bibitem{kuzmin2025-consistency}
{\sc D.~Kuzmin, {Luk\'{a}\v{c}ov\'{a}-Medvid'ov\'{a}}, and P.~{\"{O}ffner}},
  {\em Consistency and convergence of flux-corrected finite element methods for
  nonlinear hyperbolic problems}, J. Numer. Math.,  (2025),
  \url{https://doi.org/10.1515/jnma-2024-0123}.

\bibitem{kuzmin2010failsafe}
{\sc D.~Kuzmin, M.~M{\"o}ller, J.~N. Shadid, and M.~Shashkov}, {\em Failsafe
  flux limiting and constrained data projections for equations of gas
  dynamics}, Journal of Computational physics, 229 (2010), pp.~8766--8779.

\bibitem{kuzmin2004high}
{\sc D.~Kuzmin, M.~M{\"o}ller, and S.~Turek}, {\em {High-resolution FEM--FCT
  schemes for multidimensional conservation laws}}, Comput. Methods Appl. Mech.
  Engrg., 193 (2004), pp.~4915--4946.

\bibitem{Kuzmin2022BoundPreserving}
{\sc D.~Kuzmin, M.~Quezada~de Luna, D.~I. Ketcheson, and J.~Gr{\"u}ll}, {\em
  Bound-preserving flux limiting for high-order explicit runge--kutta time
  discretizations of hyperbolic conservation laws}, J. Sci. Comput., 91 (2022),
  p.~21.

\bibitem{kuzmin2022bound}
{\sc D.~Kuzmin, M.~Quezada~de Luna, D.~I. Ketcheson, and J.~Gr{\"u}ll}, {\em
  {Bound-preserving flux limiting for high-order explicit Runge--Kutta time
  discretizations of hyperbolic conservation laws}}, Journal of Scientific
  Computing, 91 (2022).
\newblock Article Number 21.

\bibitem{kuzmin-turek-2002flux}
{\sc D.~Kuzmin and S.~Turek}, {\em Flux correction tools for finite elements},
  J. Comput. Phys., 175 (2002), pp.~525--558.

\bibitem{lax1960systems}
{\sc P.~Lax and B.~Wendroff}, {\em Systems of conservation laws}, Comm. Pure
  Appl. Math., 13 (1960), pp.~217--237.

\bibitem{leveque1992numerical}
{\sc R.~J. LeVeque}, {\em Numerical methods for conservation laws}, vol.~132,
  Springer, 1992.

\bibitem{levermore1998gaussian}
{\sc C.~D. Levermore and W.~J. Morokoff}, {\em {The Gaussian moment closure for
  gas dynamics}}, SIAM J. Appl. Math., 59 (1998), pp.~72--96.

\bibitem{li2018high}
{\sc H.~Li, S.~Xie, and X.~Zhang}, {\em A high order accurate bound-preserving
  compact finite difference scheme for scalar convection diffusion equations},
  SIAM J. Numer. Anal., 56 (2018), pp.~3308--3345.

\bibitem{li2020monotonicity}
{\sc H.~Li and X.~Zhang}, {\em {On the monotonicity and discrete maximum
  principle of the finite difference implementation of C0-Q2 finite element
  method}}, Numer. Math., 145 (2020), pp.~437--472.

\bibitem{lin2023positivity}
{\sc Y.~Lin, J.~Chan, and I.~Tomas}, {\em {A positivity preserving strategy for
  entropy stable discontinuous Galerkin discretizations of the compressible
  Euler and Navier-Stokes equations}}, J. Comput. Phys., 475 (2023), p.~111850.

\bibitem{liska2010optimization}
{\sc R.~Liska, M.~Shashkov, P.~V{\'a}chal, and B.~Wendroff}, {\em
  {Optimization-based synchronized flux-corrected conservative interpolation
  (remapping) of mass and momentum for arbitrary Lagrangian--Eulerian
  methods}}, J. Comput. Phys., 229 (2010), pp.~1467--1497.

\bibitem{liu2024optimization}
{\sc C.~Liu, G.~T. Buzzard, and X.~Zhang}, {\em {An optimization based limiter
  for enforcing positivity in a semi-implicit discontinuous Galerkin scheme for
  compressible Navier--Stokes equations}}, J. Comput. Phys., 519 (2024),
  p.~113440.

\bibitem{liu2025opt}
{\sc C.~Liu, D.~Milesis, C.-W. Shu, and X.~Zhang}, {\em Efficient
  optimization-based invariant-domain-preserving limiters in solving gas
  dynamics equations}, 2025, \url{https://arxiv.org/abs/2510.21080}.

\bibitem{liu2024simple-opt-limiter}
{\sc C.~Liu, B.~Riviere, J.~Shen, and X.~Zhang}, {\em {A simple and efficient
  convex optimization based bound-preserving high order accurate limiter for
  Cahn--Hilliard--Navier--Stokes system}}, SIAM J. Sci. Comput., 46 (2024),
  pp.~A1923--A1948.

\bibitem{liu2023positivity-NS}
{\sc C.~Liu and X.~Zhang}, {\em {A positivity-preserving implicit-explicit
  scheme with high order polynomial basis for compressible Navier--Stokes
  equations}}, J. Comput. Phys., 493 (2023), p.~112496.

\bibitem{liu2014maximum}
{\sc H.~Liu and H.~Yu}, {\em {Maximum-principle-satisfying third order
  discontinuous Galerkin schemes for Fokker--Planck equations}}, SIAM J. Sci.
  Comput., 36 (2014), pp.~A2296--A2325.

\bibitem{liu2025structure}
{\sc M.~Liu and K.~Wu}, {\em {Structure-preserving oscillation-eliminating
  discontinuous Galerkin schemes for ideal MHD equations: Locally
  divergence-free and positivity-preserving}}, J. Comput. Phys.,  (2025),
  p.~113795.

\bibitem{liu1996nonoscillatory}
{\sc X.-D. Liu and S.~Osher}, {\em {Nonoscillatory high order accurate
  self-similar maximum principle satisfying shock capturing schemes I}}, SIAM
  J. Numer. Anal., 33 (1996), pp.~760--779.

\bibitem{lohmann2016synchronized}
{\sc C.~Lohmann and D.~Kuzmin}, {\em Synchronized flux limiting for gas
  dynamics variables}, J. Comput. Phys., 326 (2016), pp.~973--990.

\bibitem{lohmann2017flux}
{\sc C.~Lohmann, D.~Kuzmin, J.~N. Shadid, and S.~Mabuza}, {\em {Flux-corrected
  transport algorithms for continuous Galerkin methods based on high order
  Bernstein finite elements}}, Journal of Computational Physics, 344 (2017),
  pp.~151--186.

\bibitem{lohner1987finite}
{\sc R.~L{\"o}hner, K.~Morgan, J.~Peraire, and M.~Vahdati}, {\em {Finite
  element flux-corrected transport (FEM--FCT) for the Euler and Navier--Stokes
  equations}}, Internat. J. Numer. Methods Fluids, 7 (1987), pp.~1093--1109.

\bibitem{lorenz1977inversmonotonie}
{\sc J.~Lorenz}, {\em Zur inversmonotonie diskreter probleme}, Numer. Math., 27
  (1977), pp.~227--238.

\bibitem{lv2015entropy}
{\sc Y.~Lv and M.~Ihme}, {\em {Entropy-bounded discontinuous Galerkin scheme
  for Euler equations}}, J. Comput. Phys., 295 (2015), pp.~715--739.

\bibitem{lv2017high-chem}
{\sc Y.~Lv and M.~Ihme}, {\em High-order discontinuous {G}alerkin method for
  applications to multicomponent and chemically reacting flows}, Acta Mech.
  Sin., 33 (2017), pp.~486--499.

\bibitem{mandli2013numerical}
{\sc K.~T. Mandli}, {\em A numerical method for the two layer shallow water
  equations with dry states}, Ocean Modelling, 72 (2013), pp.~80--91.

\bibitem{meena2017positivity}
{\sc A.~K. Meena, H.~Kumar, and P.~Chandrashekar}, {\em {Positivity-preserving
  high-order discontinuous Galerkin schemes for ten-moment Gaussian closure
  equations}}, J. Comput. Phys., 339 (2017), pp.~370--395.

\bibitem{meister2016positivity}
{\sc A.~Meister and S.~Ortleb}, {\em {A positivity preserving and well-balanced
  DG scheme using finite volume subcells in almost dry regions}}, Appl. Math.
  Comput., 272 (2016), pp.~259--273.

\bibitem{mezzacappa2020physical}
{\sc A.~Mezzacappa, E.~Endeve, O.~B. Messer, and S.~W. Bruenn}, {\em Physical,
  numerical, and computational challenges of modeling neutrino transport in
  core-collapse supernovae}, Living Reviews in Computational Astrophysics, 6
  (2020), pp.~1--174.

\bibitem{moe2017positivity-LW}
{\sc S.~A. Moe, J.~A. Rossmanith, and D.~C. Seal}, {\em {Positivity-preserving
  discontinuous Galerkin methods with Lax--Wendroff time discretizations}}, J.
  Sci. Comput., 71 (2017), pp.~44--70.

\bibitem{moujaes2025monolithic}
{\sc P.~Moujaes and D.~Kuzmin}, {\em {Monolithic convex limiting and implicit
  pseudo-time stepping for calculating steady-state solutions of the Euler
  equations}}, J. Comput. Phys., 523 (2025), p.~113687.

\bibitem{Nair2011-book}
{\sc R.~D. Nair, M.~N. Levy, and P.~H. Lauritzen}, {\em Emerging numerical
  methods for atmospheric modeling}, in Numerical Techniques for Global
  Atmospheric Models, P.~Lauritzen, C.~Jablonowski, M.~Taylor, and R.~Nair,
  eds., Springer Berlin Heidelberg, Berlin, Heidelberg, 2011, pp.~251--311.

\bibitem{ni1982multiple}
{\sc R.-H. Ni}, {\em {A multiple-grid scheme for solving the Euler equations}},
  AIAA journal, 20 (1982), pp.~1565--1571.

\bibitem{olbrant2012realizability}
{\sc E.~Olbrant, C.~D. Hauck, and M.~Frank}, {\em {A realizability-preserving
  discontinuous Galerkin method for the M1 model of radiative transfer}}, J.
  Comput. Phys., 231 (2012), pp.~5612--5639.

\bibitem{oran2001numerical}
{\sc E.~S. Oran, J.~P. Boris, and J.~P. Boris}, {\em Numerical simulation of
  reactive flow}, vol.~2, Cambridge University Press, 2001.

\bibitem{PANG2025114312}
{\sc D.~Pang and K.~Wu}, {\em {Provably positivity-preserving constrained
  transport scheme for 2D and 3D ideal magnetohydrodynamics}}, J. Comput.
  Phys., 541 (2025), p.~114312.

\bibitem{pazner2021sparse}
{\sc W.~Pazner}, {\em {Sparse invariant domain preserving discontinuous
  Galerkin methods with subcell convex limiting}}, Comput. Methods Appl. Mech.
  Engrg., 382 (2021), p.~113876.

\bibitem{perthame1992second}
{\sc B.~Perthame}, {\em {Second-order Boltzmann schemes for compressible Euler
  equations in one and two space dimensions}}, SIAM J. Numer. Anal., 29 (1992),
  pp.~1--19.

\bibitem{perthame1994variant}
{\sc B.~Perthame and Y.~Qiu}, {\em {A variant of Van Leer's method for
  multidimensional systems of conservation laws}}, J. Comput. Phys., 112
  (1994), pp.~370--381.

\bibitem{perthame1996positivity}
{\sc B.~Perthame and C.-W. Shu}, {\em {On positivity preserving finite volume
  schemes for Euler equations}}, Numer. Math., 73 (1996), pp.~119--130.

\bibitem{peterson2024optimization}
{\sc K.~Peterson, P.~Bochev, and D.~Ridzal}, {\em {Optimization-based,
  property-preserving algorithm for passive tracer transport}}, Comput. Math.
  Appl., 159 (2024), pp.~267--286.

\bibitem{qin2018implicit}
{\sc T.~Qin and C.-W. Shu}, {\em {Implicit positivity-preserving high-order
  discontinuous Galerkin methods for conservation laws}}, SIAM J. Sci. Comput.,
  40 (2018), pp.~A81--A107.

\bibitem{qin2016bound}
{\sc T.~Qin, C.-W. Shu, and Y.~Yang}, {\em Bound-preserving discontinuous
  {G}alerkin methods for relativistic hydrodynamics}, J. Comput. Phys., 315
  (2016), pp.~323--347.

\bibitem{qiu2011positivity}
{\sc J.-M. Qiu and C.-W. Shu}, {\em {Positivity preserving semi-Lagrangian
  discontinuous Galerkin formulation: Theoretical analysis and application to
  the Vlasov--Poisson system}}, J. Comput. Phys., 230 (2011), pp.~8386--8409.

\bibitem{quezada2022maximum}
{\sc M.~Quezada~de Luna and D.~I. Ketcheson}, {\em Maximum principle preserving
  space and time flux limiting for {Diagonally Implicit Runge--Kutta
  discretizations of scalar convection-diffusion equations}}, Journal of
  Scientific Computing, 92 (2022).
\newblock Article Number 102.

\bibitem{radice2014high}
{\sc D.~Radice, L.~Rezzolla, and F.~Galeazzi}, {\em High-order fully
  general-relativistic hydrodynamics: New approaches and tests}, Classical and
  Quantum Gravity, 31 (2014), p.~075012.

\bibitem{remi-9}
{\sc M.~Ricchiuto and R.~Abgrall}, {\em Explicit {Runge}-{Kutta} residual
  distribution schemes for time dependent problems: second order case}, J.
  Comput. Phys., 229 (2010), pp.~5653--5691.

\bibitem{ricchiuto2005residual}
{\sc M.~Ricchiuto, {\'A}.~Cs{\'\i}k, and H.~Deconinck}, {\em Residual
  distribution for general time-dependent conservation laws}, J. Comput. Phys.,
  209 (2005), pp.~249--289.

\bibitem{rider1997constrained}
{\sc W.~Rider, D.~Kothe, W.~Rider, and D.~Kothe}, {\em Constrained minimization
  for monotonic reconstruction}, in 13th Computational Fluid Dynamics
  Conference, 1997, p.~2036.

\bibitem{zbMATH001505028}
{\sc P.~L. Roe and D.~Sidilkover}, {\em Optimum positive linear schemes for
  advection in two and three dimensions}, SIAM J. Numer. Anal., 29 (1992),
  pp.~1542--1568.

\bibitem{rossmanith2011positivity}
{\sc J.~A. Rossmanith and D.~C. Seal}, {\em {A positivity-preserving high-order
  semi-Lagrangian discontinuous Galerkin scheme for the Vlasov--Poisson
  equations}}, J. Comput. Phys., 230 (2011), pp.~6203--6232.

\bibitem{rueda2024monolithic}
{\sc A.~M. Rueda-Ram{\'\i}rez, B.~Bolm, D.~Kuzmin, and G.~J. Gassner}, {\em
  {Monolithic convex limiting for Legendre-Gauss-Lobatto discontinuous Galerkin
  spectral-element methods}}, Commun. Appl. Math. Comput., 6 (2024),
  pp.~1860--1898.

\bibitem{ruppenthal2023optimal}
{\sc F.~Ruppenthal and D.~Kuzmin}, {\em {Optimal control using flux potentials:
  A way to construct bound-preserving finite element schemes for conservation
  laws}}, J. Comput. Appl. Math., 434 (2023), p.~115351.

\bibitem{schar1996synchronous}
{\sc C.~Sch{\"a}r and P.~K. Smolarkiewicz}, {\em A synchronous and iterative
  flux-correction formalism for coupled transport equations}, J. Comput. Phys.,
  128 (1996), pp.~101--120.

\bibitem{sedov1993-similarity}
{\sc L.~I. Sedov}, {\em Similarity and Dimensional Methods in Mechanics}, CRC
  Press, 10~ed., 1993, \url{https://doi.org/10.1201/9780203739730}.

\bibitem{selmin1993node}
{\sc V.~Selmin}, {\em The node-centred finite volume approach: bridge between
  finite differences and finite elements}, Comput. Methods Appl. Mech. Engrg.,
  102 (1993), pp.~107--138.

\bibitem{selmin1996unified}
{\sc V.~Selmin and L.~Formaggia}, {\em Unified construction of finite element
  and finite volume discretizations for compressible flows}, Internat. J.
  Numer. Methods Engrg., 39 (1996), pp.~1--32.

\bibitem{shi2018local-LW}
{\sc C.~Shi and C.-W. Shu}, {\em On local conservation of numerical methods for
  conservation laws}, Comput. \& Fluids, 169 (2018), pp.~3--9.

\bibitem{shu2009high}
{\sc C.-W. Shu}, {\em High order weighted essentially nonoscillatory schemes
  for convection dominated problems}, SIAM Rev., 51 (2009), pp.~82--126.

\bibitem{smoller2012shock}
{\sc J.~Smoller}, {\em Shock waves and reaction-diffusion equations}, vol.~258,
  Springer Science \& Business Media, 2012.

\bibitem{zbMATH05590395}
{\sc A.~Sommariva and M.~Vianello}, {\em Gauss-{G}reen cubature and moment
  computation over arbitrary geometries}, J. Comput. Appl. Math., 231 (2009),
  pp.~886--896.

\bibitem{srinivasan2018positivity}
{\sc S.~Srinivasan, J.~Poggie, and X.~Zhang}, {\em {A positivity-preserving
  high order discontinuous Galerkin scheme for convection--diffusion
  equations}}, J. Comput. Phys., 366 (2018), pp.~120--143.

\bibitem{struijs1991fluctuation}
{\sc R.~Struijs, H.~Deconinck, and P.~Roe}, {\em {Fluctuation splitting schemes
  for the 2D Euler equations}}, in Computational Fluid Dynamics, VKI Lecture
  Series 1991-01, 1991.

\bibitem{sun2018discontinuous}
{\sc Z.~Sun, J.~A. Carrillo, and C.-W. Shu}, {\em {A discontinuous Galerkin
  method for nonlinear parabolic equations and gradient flow problems with
  interaction potentials}}, J. Comput. Phys., 352 (2018), pp.~76--104.

\bibitem{tadmor1986minimum}
{\sc E.~Tadmor}, {\em A minimum entropy principle in the gas dynamics
  equations}, Appl. Numer. Math., 2 (1986), pp.~211--219.

\bibitem{tang2000positivity}
{\sc H.-Z. Tang and K.~Xu}, {\em Positivity-preserving analysis of explicit and
  implicit {Lax--Friedrichs schemes for compressible Euler equations}}, J. Sci.
  Comput., 15 (2000), pp.~19--28.

\bibitem{tao1999gas}
{\sc T.~Tao and K.~Xu}, {\em {Gas-kinetic schemes for the compressible Euler
  equations: positivity-preserving analysis}}, Zeitschrift f{\"u}r angewandte
  Mathematik und Physik ZAMP, 50 (1999), pp.~258--281.

\bibitem{tong2023class}
{\sc W.~Tong, R.~Yan, and G.~Chen}, {\em {On a class of robust bound-preserving
  MUSCL-Hancock schemes}}, J. Comput. Phys., 474 (2023), p.~111805.

\bibitem{toro2013riemann}
{\sc E.~F. Toro}, {\em {Riemann solvers and numerical methods for fluid
  dynamics: a practical introduction}}, Springer Science \& Business Media,
  2013.

\bibitem{VANDERHOLST2008617}
{\sc B.~{van der Holst}, R.~Keppens, and Z.~Meliani}, {\em A multidimensional
  grid-adaptive relativistic magnetofluid code}, Comput. Phys. Commun., 179
  (2008), pp.~617--627.

\bibitem{van2019positivity-opt-implicit}
{\sc J.~J. van~der Vegt, Y.~Xia, and Y.~Xu}, {\em {Positivity preserving
  limiters for time-implicit higher order accurate discontinuous Galerkin
  discretizations}}, SIAM J. Sci. Comput., 41 (2019), pp.~A2037--A2063.

\bibitem{vilar2019posteriori}
{\sc F.~Vilar}, {\em {A posteriori correction of high-order discontinuous
  Galerkin scheme through subcell finite volume formulation and flux
  reconstruction}}, J. Comput. Phys., 387 (2019), pp.~245--279.

\bibitem{vilar2025local}
{\sc F.~Vilar}, {\em {Local subcell monolithic DG/FV convex property preserving
  scheme on unstructured grids and entropy consideration}}, J. Comput. Phys.,
  521 (2025), p.~113535.

\bibitem{vilar2024posteriori}
{\sc F.~Vilar and R.~Abgrall}, {\em {A posteriori local subcell correction of
  high-order discontinuous Galerkin scheme for conservation laws on
  two-dimensional unstructured grids}}, SIAM J. Sci. Comput., 46 (2024),
  pp.~A851--A883.

\bibitem{vilar2016positivity}
{\sc F.~Vilar, C.-W. Shu, and P.-H. Maire}, {\em {Positivity-preserving
  cell-centered Lagrangian schemes for multi-material compressible flows: From
  first-order to high-orders. Part I: {T}he one-dimensional case}}, J. Comput.
  Phys., 312 (2016), pp.~385--415.

\bibitem{MR3471184}
{\sc F.~Vilar, C.-W. Shu, and P.-H. Maire}, {\em Positivity-preserving
  cell-centered {L}agrangian schemes for multi-material compressible flows:
  {F}rom first-order to high-orders. {P}art {II}: {T}he two-dimensional case},
  J. Comput. Phys., 312 (2016), pp.~416--442.

\bibitem{wang2012robust}
{\sc C.~Wang, X.~Zhang, C.-W. Shu, and J.~Ning}, {\em {Robust high order
  discontinuous Galerkin schemes for two-dimensional gaseous detonations}}, J.
  Comput. Phys., 231 (2012), pp.~653--665.

\bibitem{wu2017design}
{\sc K.~Wu}, {\em Design of provably physical-constraint-preserving methods for
  general relativistic hydrodynamics}, Phys. Rev. D, 95 (2017), p.~103001.

\bibitem{wu2018positivity}
{\sc K.~Wu}, {\em Positivity-preserving analysis of numerical schemes for ideal
  magnetohydrodynamics}, SIAM J. Numer. Anal., 56 (2018), pp.~2124--2147.

\bibitem{wu2021minimum}
{\sc K.~Wu}, {\em Minimum principle on specific entropy and high-order accurate
  invariant-region-preserving numerical methods for relativistic
  hydrodynamics}, SIAM J. Sci. Comput., 43 (2021), pp.~B1164--B1197.

\bibitem{wu2023provably}
{\sc K.~Wu, H.~Jiang, and C.-W. Shu}, {\em {Provably positive central
  discontinuous Galerkin schemes via geometric quasilinearization for ideal MHD
  equations}}, SIAM J. Numer. Anal., 61 (2023), pp.~250--285.

\bibitem{wu2018provably}
{\sc K.~Wu and C.-W. Shu}, {\em A provably positive discontinuous {G}alerkin
  method for multidimensional ideal magnetohydrodynamics}, SIAM J. Sci.
  Comput., 40 (2018), pp.~B1302--B1329.

\bibitem{wu2019provably}
{\sc K.~Wu and C.-W. Shu}, {\em Provably positive high-order schemes for ideal
  magnetohydrodynamics: analysis on general meshes}, Numer. Math., 142 (2019),
  pp.~995--1047.

\bibitem{wu2020entropy}
{\sc K.~Wu and C.-W. Shu}, {\em Entropy symmetrization and high-order accurate
  entropy stable numerical schemes for relativistic {MHD} equations}, SIAM J.
  Sci. Comput., 42 (2020), pp.~A2230--A2261.

\bibitem{wu2021provably}
{\sc K.~Wu and C.-W. Shu}, {\em {Provably physical-constraint-preserving
  discontinuous Galerkin methods for multidimensional relativistic MHD
  equations}}, Numer. Math., 148 (2021), pp.~699--741.

\bibitem{wu2023geometric}
{\sc K.~Wu and C.-W. Shu}, {\em Geometric quasilinearization framework for
  analysis and design of bound-preserving schemes}, SIAM Rev., 65 (2023),
  pp.~1031--1073.

\bibitem{WU2015539}
{\sc K.~Wu and H.~Tang}, {\em {High-order accurate
  physical-constraints-preserving finite difference WENO schemes for special
  relativistic hydrodynamics}}, J. Comput. Phys., 298 (2015), pp.~539--564.

\bibitem{wu2016physical}
{\sc K.~Wu and H.~Tang}, {\em Physical-constraint-preserving central
  discontinuous {G}alerkin methods for special relativistic hydrodynamics with
  a general equation of state}, The Astrophysical Journal Supplement Series,
  228 (2016), p.~3.

\bibitem{wu2017admissible}
{\sc K.~Wu and H.~Tang}, {\em Admissible states and
  physical-constraints-preserving schemes for relativistic magnetohydrodynamic
  equations}, Math. Models Methods Appl. Sci., 27 (2017), pp.~1871--1928.

\bibitem{MR3095289}
{\sc Y.~Xing and X.~Zhang}, {\em Positivity-preserving well-balanced
  discontinuous {G}alerkin methods for the shallow water equations on
  unstructured triangular meshes}, J. Sci. Comput., 57 (2013), pp.~19--41.

\bibitem{xing2010positivity}
{\sc Y.~Xing, X.~Zhang, and C.-W. Shu}, {\em {Positivity-preserving high order
  well-balanced discontinuous Galerkin methods for the shallow water
  equations}}, Advances in Water Resources, 33 (2010), pp.~1476--1493.

\bibitem{xiong2013parametrized}
{\sc T.~Xiong, J.-M. Qiu, and Z.~Xu}, {\em {A parametrized maximum principle
  preserving flux limiter for finite difference RK-WENO schemes with
  applications in incompressible flows}}, J. Comput. Phys., 252 (2013),
  pp.~310--331.

\bibitem{xiong2015high-DG-confussion}
{\sc T.~Xiong, J.-M. Qiu, and Z.~Xu}, {\em {High order
  maximum-principle-preserving discontinuous Galerkin method for
  convection-diffusion equations}}, SIAM J. Sci. Comput., 37 (2015),
  pp.~A583--A608.

\bibitem{xiong2016parametrized}
{\sc T.~Xiong, J.-M. Qiu, and Z.~Xu}, {\em {Parametrized positivity preserving
  flux limiters for the high order finite difference WENO scheme solving
  compressible Euler equations}}, J. Sci. Comput., 67 (2016), pp.~1066--1088.

\bibitem{xiong2014high-Lagrangian}
{\sc T.~Xiong, J.-M. Qiu, Z.~Xu, and A.~Christlieb}, {\em {High order maximum
  principle preserving semi-Lagrangian finite difference WENO schemes for the
  Vlasov equation}}, J. Comput. Phys., 273 (2014), pp.~618--639.

\bibitem{xu1999monotone}
{\sc J.~Xu and L.~Zikatanov}, {\em A monotone finite element scheme for
  convection-diffusion equations}, Math. Comp., 68 (1999), pp.~1429--1446.

\bibitem{xu2014parametrized}
{\sc Z.~Xu}, {\em Parametrized maximum principle preserving flux limiters for
  high order schemes solving hyperbolic conservation laws: {O}ne-dimensional
  scalar problem}, Math. Comp., 83 (2014), pp.~2213--2238.

\bibitem{xu2017bound}
{\sc Z.~Xu and X.~Zhang}, {\em Bound-preserving high-order schemes}, in
  Handbook of Numerical Analysis, vol.~18, Elsevier, 2017, pp.~81--102.

\bibitem{yee2020quadratic}
{\sc B.~C. Yee, S.~S. Olivier, T.~S. Haut, M.~Holec, V.~Z. Tomov, and P.~G.
  Maginot}, {\em {A quadratic programming flux correction method for high-order
  DG discretizations of SN transport}}, J. Comput. Phys., 419 (2020),
  p.~109696.

\bibitem{yu2025high}
{\sc K.~Yu, J.~Cheng, Y.~Liu, and C.-W. Shu}, {\em High-order implicit
  maximum-principle-preserving local discontinuous galerkin methods for
  convection--diffusion equations}, J. Comput. Appl. Math.,  (2025), p.~116660.

\bibitem{Akil2021}
{\sc V.~Zala, R.~M. Kirby, and A.~Narayan}, {\em Structure-preserving nonlinear
  filtering for continuous and discontinuous {G}alerkin spectral/hp element
  methods}, SIAM J. Sci. Comput., 43 (2021), pp.~A3713--A3732.

\bibitem{zala2023convex}
{\sc V.~Zala, A.~Narayan, and R.~M. Kirby}, {\em Convex optimization-based
  structure-preserving filter for multidimensional finite element simulations},
  J. Comput. Phys., 492 (2023), p.~112364.

\bibitem{zalesak1979fully}
{\sc S.~T. Zalesak}, {\em {Fully multidimensional flux-corrected transport
  algorithms for fluids}}, J. Comput. Phys., 31 (1979), pp.~335--362.

\bibitem{zalesak2012design}
{\sc S.~T. Zalesak}, {\em The design of Flux-Corrected Transport ({FCT})
  algorithms for structured grids}, Springer, 2012.

\bibitem{MR2493559}
{\sc L.~Zhang, T.~Cui, and H.~Liu}, {\em A set of symmetric quadrature rules on
  triangles and tetrahedra}, J. Comput. Math., 27 (2009), pp.~89--96.

\bibitem{zhang2017positivity}
{\sc X.~Zhang}, {\em On positivity-preserving high order discontinuous
  {G}alerkin schemes for compressible {N}avier--{S}tokes equations}, J. Comput.
  Phys., 328 (2017), pp.~301--343.

\bibitem{zhang2012maximum-LiuYY}
{\sc X.~Zhang, Y.~Liu, and C.-W. Shu}, {\em Maximum-principle-satisfying high
  order finite volume weighted essentially nonoscillatory schemes for
  convection-diffusion equations}, SIAM J. Sci. Comput., 34 (2012),
  pp.~A627--A658.

\bibitem{zhang2010genuinely}
{\sc X.~Zhang and C.-W. Shu}, {\em A genuinely high order total variation
  diminishing scheme for one-dimensional scalar conservation laws}, SIAM J.
  Numer. Anal., 48 (2010), pp.~772--795.

\bibitem{ZHANG20103091}
{\sc X.~Zhang and C.-W. Shu}, {\em On maximum-principle-satisfying high order
  schemes for scalar conservation laws}, J. Comput. Phys., 229 (2010),
  pp.~3091--3120.

\bibitem{zhang2010positivity}
{\sc X.~Zhang and C.-W. Shu}, {\em On positivity-preserving high order
  discontinuous {G}alerkin schemes for compressible {E}uler equations on
  rectangular meshes}, J. Comput. Phys., 229 (2010), pp.~8918--8934.

\bibitem{zhang2011maximum}
{\sc X.~Zhang and C.-W. Shu}, {\em Maximum-principle-satisfying and
  positivity-preserving high-order schemes for conservation laws: survey and
  new developments}, Proceedings of the Royal Society A: Mathematical, Physical
  and Engineering Sciences, 467 (2011), pp.~2752--2776.

\bibitem{zhang2012minimum}
{\sc X.~Zhang and C.-W. Shu}, {\em A minimum entropy principle of high order
  schemes for gas dynamics equations}, Numer. Math., 121 (2012), pp.~545--563.

\bibitem{zhang2012positivity}
{\sc X.~Zhang and C.-W. Shu}, {\em {Positivity-preserving high order finite
  difference WENO schemes for compressible Euler equations}}, J. Comput. Phys.,
  231 (2012), pp.~2245--2258.

\bibitem{zhang2012maximum}
{\sc X.~Zhang, Y.~Xia, and C.-W. Shu}, {\em {Maximum-principle-satisfying and
  positivity-preserving high order discontinuous Galerkin schemes for
  conservation laws on triangular meshes}}, J. Sci. Comput., 50 (2012),
  pp.~29--62.

\end{thebibliography}
\end{document}